\documentclass[10pt]{article}
\usepackage[utf8]{inputenc}

\usepackage[margin=1in]{geometry} 
\usepackage{graphicx}

\usepackage{booktabs} 
\usepackage{array} 
\usepackage{paralist} 
\usepackage{verbatim}
\usepackage{mathrsfs}
\usepackage{amssymb}
\usepackage{amsthm}
\usepackage{amsmath,amsfonts,amssymb}
\usepackage{esint}
\usepackage{graphics}
\usepackage{enumerate}
\usepackage{mathtools}
\usepackage{xfrac}
\usepackage{bbm}
\usepackage{caption}
\usepackage{subcaption}
\usepackage{tikz-cd}
\usepackage{adjustbox}

\usepackage[colorlinks=true, pdfstartview=FitV, linkcolor=blue, citecolor=blue, urlcolor=blue]{hyperref}

\numberwithin{equation}{section}
\numberwithin{figure}{section}

\newtheorem{theorem}{Theorem}[section]

\newtheorem{corollary}[theorem]{Corollary}
\newtheorem{proposition}[theorem]{Proposition}
\newtheorem{lemma}[theorem]{Lemma}
\theoremstyle{definition}
\newtheorem{definition}[theorem]{Definition}

\newtheorem{remark}[theorem]{Remark}

\newcommand*{\N}{\ensuremath{\mathbb{N}}}

\newcommand*{\Z}{\ensuremath{\mathbb{Z}}}

\newcommand*{\R}{\ensuremath{\mathbb{R}}}
\newcommand*{\Zd}{\ensuremath{\mathbb{Z}^d}}
\newcommand*{\Rd}{\ensuremath{\mathbb{R}^d}}

\newcommand{\eps}{\varepsilon}

\renewcommand*{\tilde}{\widetilde}
\renewcommand*{\hat}{\widehat}

\newcommand{\ep}{\eps}

\newcommand{\E}{\mathbb{E}}

\DeclareSymbolFont{boldoperators}{OT1}{cmr}{bx}{n}
\SetSymbolFont{boldoperators}{bold}{OT1}{cmr}{bx}{n}
\edef\bar{\unexpanded{\protect\mathaccentV{bar}}\number\symboldoperators16}
\renewcommand{\a}{\mathbf{A}}

\definecolor{labelkey}{rgb}{0,0,1}

\newcommand{\indc}{{\boldsymbol{1}}}

\usepackage{titlesec}

\newcommand{\addperiod}[1]{#1.}
\titleformat{\section}
   {\centering\normalfont\Large}{\thesection.}{0.5em}{}
\titleformat*{\subsection}{\bfseries}
\titleformat{\subsubsection}[runin]
  {\normalfont\bfseries}
  {\thesubsubsection.}
  {0.5em}
  {\addperiod}
\titleformat*{\subsubsection}{\bfseries}
\titleformat*{\paragraph}{\bfseries}
\titleformat*{\subparagraph}{\large\bfseries}

\title{Hydrodynamic limit for a class of degenerate convex $\nabla \varphi$-interface models}

\author{
Paul Dario
\thanks{CNRS and LAMA, Universit\'e Paris-Est Cr\'eteil, Cr\'eteil, France.
{\footnotesize paul.dario@u-pec.fr.}
}
}
\date{ }

\usepackage[nottoc,notlot,notlof]{tocbibind}

\begin{document}

\maketitle

\begin{abstract}
    We study the Langevin dynamics corresponding to the $\nabla \varphi$-interface model with a degenerate convex interaction potential satisfying a polynomial growth assumption. Following the work of the author and Armstrong~\cite{armstrong2022quantitative}, we interpret these Langevin dynamics as a nonlinear parabolic equation forced by white noise and apply homogenization methods to derive a quantitative hydrodynamic limit. This result quantifies and extends to a class of degenerate convex potentials the seminal result of Funaki and Spohn~\cite{FS}. In order to handle the degeneracy of the potential, we make use of the notion of \emph{moderated environment} originally introduced by Mourrat and Otto~\cite{MO16} and further developed by Biskup and Rodriguez~\cite{biskup2018limit} to study the properties of solutions of parabolic equations with degenerate coefficients (and of the corresponding random walks).
\end{abstract}

\setcounter{tocdepth}{1}
\tableofcontents

\section{Introduction}

In this article, we study the $\nabla \varphi$ interface model defined as follows. Given a dimension $d \geq 2$ and a finite set $\Lambda \subseteq \Zd$, we consider a scalar field $\varphi : \Lambda \to \R$ which is interpreted as a discretized interface embedded in $\R^{d+1}$ (where $\varphi(x)$ is the height of the interface at the vertex $x \in \Lambda$, see Figure~\ref{fig:GFFSample4}). The set of discrete interfaces $\Omega_\Lambda := \left\{ \varphi : \Lambda \to \R \right\} \simeq \R^\Lambda$ is then equipped with a probability distribution given by the formula
\begin{equation} \label{def.gradphifinitevol}
    \mu_{\Lambda}(d \varphi) := \frac{1}{Z_\Lambda} \exp \left( - \sum_{x \in \Lambda} V(\nabla \varphi(x)) \right) \prod_{x \in \Lambda} d \varphi(x),
\end{equation}
where $Z_\Lambda$ is the constant chosen so that $\mu_{\Lambda}$ is a probability distribution and where we used the following conventions and notation:
\begin{itemize}
    \item Given a function $\varphi \in \Omega_\Lambda,$ we implicitly extend it by $0$ outside the set $\Lambda$ and define its discrete gradient according to the formula, for any $x \in \Lambda$,
    \begin{equation*}
        \nabla \varphi(x) := ( \varphi(x + e_1) - \varphi(x), \ldots,
        \varphi(x + e_d) - \varphi(x)) \in \R^d.
    \end{equation*}
    \item The map $V \in C^2(\Rd)$ is a convex interaction potential whose second derivative satisfies the following growth assumption: there exist an exponent $r > 2$ and three constants $c_-, c_+ \in (0 , \infty)$ and $R_0 \geq 1$ such that
    \begin{equation} \label{AssPot}
        \forall \, x \in \R^d \mbox{ with } |x| \geq R_0, ~ c_- |x|^{r-2} I_d \leq D_p^2V(x) \leq c_+ |x|^{r-2} I_d, \tag{A}
    \end{equation}
    where $D_p^2V(x)$ denotes the Hessian of $V$ at the point $x \in \Rd$, the identity matrix is denoted by $I_d$, the inequalities are understood as inequalities for symmetric matrices and $|x|$ is the Euclidean norm of $x \in \Rd$.
\end{itemize}

\begin{figure}
        \centering
        \includegraphics[scale=0.8]{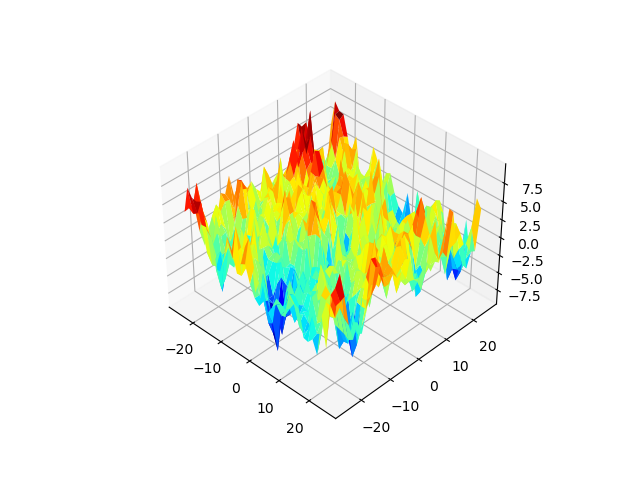}
    \caption{A random interface sampled accoding to the Gibbs measure~\eqref{def.gradphifinitevol}.} \label{fig:GFFSample4}
\end{figure}

\begin{figure}
\begin{minipage}{0.5\textwidth}
        \centering
        \includegraphics[width=0.8\linewidth, height=0.2\textheight]{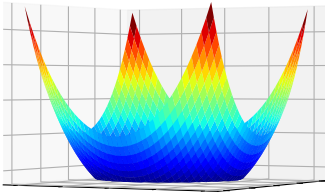}
\end{minipage}
\begin{minipage}{0.5\textwidth}
        \centering
        \includegraphics[width=0.8\linewidth, height=0.3\textheight]{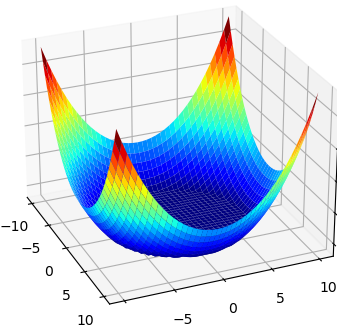}
    \end{minipage}
    \caption{An example of potential satisfying the Assumption~\eqref{AssPot}: it is convex and exhibits a flat part where its Hessian is degenerate.}
\end{figure}

The static properties of the model have been extensively studied (see~\cite{F05, V06} and Section~\ref{section.bib} for a more detailed account of the literature). In this article, we are interested in the dynamic properties of the model, starting from the observation that the Gibbs measure~\eqref{def.gradphifinitevol} is naturally associated with the following Langevin dynamics
\begin{equation} \label{eq:introSDE}
   \left\{ \begin{aligned}
    d \varphi(t , x) & = \nabla \cdot D_p V(\nabla \varphi)(t , x) dt + \sqrt{2} dB_t(x) & \mbox{for}~ (t , x) \in (0 , \infty) \times \Lambda, \\
    \varphi(t , x) & = 0 & \mbox{for}~ (t , x) \in (0 , \infty) \times \partial^+ \Lambda.
    \end{aligned}
    \right.
\end{equation}
where $\left\{ B_t(x) \, : \, t \geq 0, \, x \in \Lambda \right\}$ is a collection of independent Brownian motions, in the second line, we denoted by $\partial^+ \Lambda := \left\{ y \notin \Lambda \, : \, \exists y \in \Lambda, \, x \sim y \right\}$ the discrete external boundary of the set $\Lambda$ and we made use of the notation for the discrete divergence introduced in Section~\ref{sec.defellipticequation}. Specifically, the system of stochastic differential equations~\eqref{eq:introSDE} has $\mu_\Lambda$ as unique invariant measure. 

A typical question in statistical mechanics is then to describe the macroscopic behaviour of the Gibbs measure~\eqref{def.gradphifinitevol} and the Langevin dynamics~\eqref{eq:introSDE}. In this direction, an important theorem known as the hydrodynamic limit and originally established, in the case of uniformly convex potentials, by Funaki and Spohn~\cite{FS} (and extended by Nishikawa~\cite{nishikawa2003hydrodynamic} from the periodic to the Dirichlet boundary conditions) asserts that under a suitable large-scale limit, the Langevin dynamics converge to a deterministic profile $h$ which evolves according to the nonlinear parabolic equation
\begin{equation} \label{eq:hydroeqh}
    \partial_t h - \nabla \cdot D_p \bar \sigma (\nabla h) = 0,
\end{equation}
where $\bar \sigma : \Rd \to \R$ is a deterministic strictly convex function called the surface tension of the model (see~\cite[Proposition 1.1]{FS} or~\eqref{def.surfacetensionintro} below).

The purpose of this article is to extend the result of~\cite{FS} to the class of potentials satisfying the Assumption~\eqref{AssPot}. Before stating our main result, we need to introduce some additional results and notation:
\begin{itemize}
    \item For mostly technical convenience, we will not work with Dirichlet boundary conditions but with \emph{periodic} boundary conditions. To this end, we denote by $\mathbb{T} := \R^d /\Z^d$ the $d$-dimensional (unit and continuous) torus. For any fixed parameter $\ep \in (0,1)$, we let $\mathbb{T}^\ep$ be a discretization of mesh size $\ep$, see Figure~\ref{fig:torus} (for simplicity, we will always assume that $\ep^{-1}$ is an integer and that the discrete torus $\mathbb{T}^\ep$ contains exactly $\ep^{-d}$ vertices).
    \item For any initial condition $f \in C^\infty(\mathbb{T})$ and any $\ep \in (0,1)$, we define the (suitably rescaled) Langevin dynamics started from $f$ on the discretized torus $\mathbb{T}^\ep$
    \begin{equation} \label{eq:introSDErescaled}
   \left\{ \begin{aligned}
    d u^\ep(t , x) & = \nabla^\ep \cdot D_p V(\nabla^\ep u^\ep)(t , x) dt + \sqrt{2} dB^\ep_t(x) & \mbox{for}~ (t , x) \in (0 , \infty) \times \mathbb{T}^\ep, \\
    u^\ep(0 , x) & = f(x) & \mbox{for}~ x \in \mathbb{T}^\ep,
    \end{aligned}
    \right.
\end{equation}
where we used the notation $\nabla^\ep := \ep^{-1} \nabla$ and $\nabla^\ep \cdot :=  \ep^{-1} \nabla \cdot$ for the rescaled gradient and divergence, and where we wrote $B^\ep_t(x) := \ep B_{\ep^{-2 }t} (\ep^{-1}x)$ (N.B. the scaling properties of the Brownian motion ensure that $B^\ep(x)$ is a Brownian motion, this convention is set so that the same Brownian motions can be used in~\eqref{eq:introSDE} and~\eqref{eq:introSDErescaled}).
\end{itemize}

\begin{figure}
        \centering
        \includegraphics[scale=0.7]{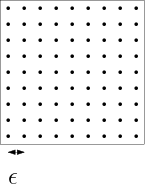}
    \caption{The discretized torus $\mathbb{T}^\ep$ with $\ep = 1/9$.\label{fig:torus}}
\end{figure}

\begin{theorem}[Quantitative hydrodynamic limit] \label{main.thm}
    Let us fix a dimension $d \geq 2$, a convex potential $V \in C^2(\Rd)$ satisfying the Assumption~\eqref{AssPot} and a smooth initial condition $f \in C^\infty(\mathbb{T})$.
   There exist two constants $\lambda_- , \lambda_+ \in (0,\infty)$ and a convex function $\bar \sigma \in C^{1,1}(\Rd)$ satisfying
    \begin{equation} \label{eq:conv.barsigma}
        \lambda_- (|x|^{r-1} + 1) I_d \leq D_p^2 \bar \sigma (x) \leq  \lambda_+ (|x|^{r-1} + 1) I_d ~~\mbox{for almost every}~ x \in \Rd,
    \end{equation}
    such that, if we let $\bar u$ be the solution of the deterministic nonlinear parabolic equation
    \begin{equation} \label{def.barumainthm}
   \left\{ \begin{aligned}
    \partial_t \bar u - \nabla \cdot D_p \bar \sigma (\nabla \bar u ) & = 0  & \mbox{for}~ (t , x) \in (0 , \infty) \times \mathbb{T}, \\
     \bar u(0 , \cdot) & = f & \mbox{for}~ x \in \mathbb{T},
    \end{aligned}
    \right.
    \end{equation}
    then there exist two constants $c := c(d, V , f) >0$ and $C := C(d , V , f) < \infty$ and an exponent $\theta := \theta(d, r) > 0$ such that, for any $\ep \in (0,1),$
    \begin{equation} \label{eq:quantihydro}
        \int_0^1 \ep^d \sum_{x \in \mathbb{T}^\ep} \left| u^\ep(t , x) - \bar u (t , x) \right|^2 \, dt \leq  \mathcal{O}_{\Psi , c} \left( C \ep^{\theta} \right).
    \end{equation}
\end{theorem}

\begin{figure}
\begin{subfigure}{.5\linewidth}
\centering
\includegraphics[scale=0.6]{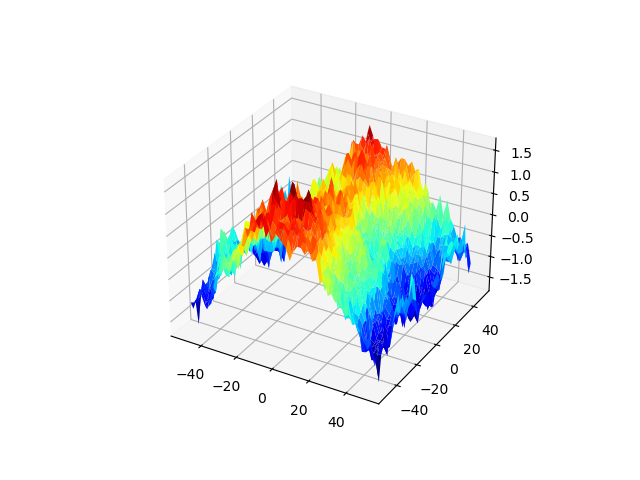}
\end{subfigure}
\begin{subfigure}{.5\linewidth}
\centering
\includegraphics[scale=0.6]{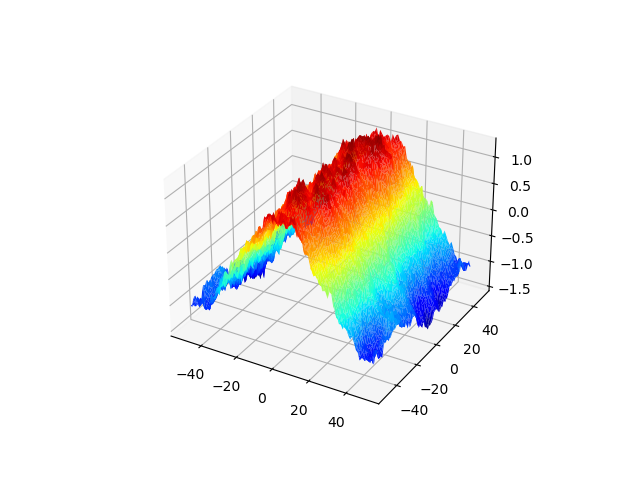}
\end{subfigure}\\[1ex]
\begin{subfigure}{\linewidth}
\centering
\includegraphics[scale=0.65]{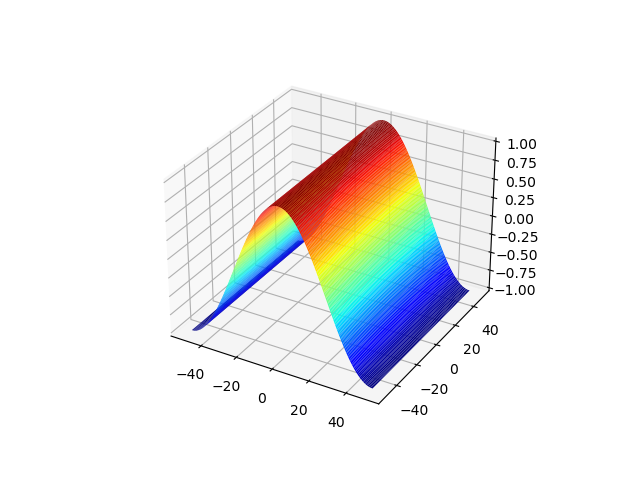}
\end{subfigure}
\caption{An illustration of Theorem~\ref{main.thm}: the two pictures on the first line are a realization of the Langevin dynamics started from a smooth initial data at time $t = 1$ with $\ep = 1/50$ (left) and $\ep = 1/100$ (right). As $\ep \to 0$, the dynamics concentrate around a smooth deterministic profile drawn on the second line.} \label{fig:figure1.3}
\end{figure}

\begin{remark}
    Let us make a few remarks about the previous theorem:
    \begin{itemize}
    \item On the right-hand side of~\eqref{eq:quantihydro} we use the notation ``$X \leq \mathcal{O}_{\Psi , c}(A)$" as shorthand for the statement
    \begin{equation} \label{eq:StochOrlicz}
        \mathbb{P} \left[ X \geq t A \right] \leq \exp \left( - c \left| \ln t \right|^{\frac{r}{r-2}} \right) ~~ \forall t \in [1 , \infty).
    \end{equation}
    Since the exponent $r/(r-2)$ is strictly larger than $1$, the right-hand side of~\eqref{eq:StochOrlicz} decays faster than any polynomial. This implies that all the moments of the random variable on the left-hand side of~\eqref{eq:quantihydro} are finite.
    \item We assumed that the exponent $r$ in Assumption~\eqref{AssPot} is strictly larger than $2$, but the same result (with a simpler proof and a stronger stochastic integrability estimate) should hold in the case $r = 2$.
    \item The argument gives an explicit value of the order of $\theta \simeq 1/(100 d r)$ for the rate of convergence. While we believe that the argument could be optimised so as to improve the exponent, obtaining the optimal rate of convergence (which should correspond to the value $\theta = 1$ with a logarithmic correction in two dimensions) seems (at least) much more technical to obtain (see the discussion in Section 1.5.2 of~\cite{armstrong2022quantitative}).
    \item In a similar way to Funaki and Spohn~\cite[Proposition 1.1]{FS}, the surface tension $\bar \sigma$ can be defined as the following limit, for any $p \in \Rd$,
    \begin{equation} \label{def.surfacetensionintro}
        \bar \sigma(p) := \lim_{L \to \infty} - \frac{1}{(2L+1)^d} \ln \frac{\int_{\Omega^\circ_L} \exp \left( - \sum_{x \in \mathbb{T}_L} V(p + \nabla \varphi(x)) \right) \, d\varphi}{\int_{\Omega^\circ_L} \exp \left( - \sum_{x \in \mathbb{T}_L} V(\nabla \varphi(x)) \right) \, d\varphi},
    \end{equation}
    where $\mathbb{T}_L := (\Z / (2L+1)\mathbb{Z})^d$ denotes the discrete torus, $\Omega^\circ_L := \left\{ \varphi : \mathbb{T}_L \to \R \, : \, \sum_{x \in \mathbb{T}_L} \varphi(x) =0 \right\}$ and the integral is computed with respect to the Lebesgue measure on $\Omega^\circ_L$. The inequality~\eqref{eq:conv.barsigma} asserts that the surface tension is strictly convex for this class of potentials.
    \item A fundamental feature of~\eqref{eq:conv.barsigma} is that the inequality holds for almost every $x \in \Rd$, while the potential $V$ is only assumed to satisfy the lower bound outside a compact set. This is an instance of a \emph{convexification} phenomenon which has already been observed and exploited in the literature (see e.g., ~\cite{grunewald2009two, CDM09, CD12, AKM16, dizdar2018quantitative}). This gain of convexity has the following consequence in terms of elliptic regularity: the solutions of the parabolic equation~\eqref{eq:hydroeqh} are known to possess good regularity properties (see, e.g., Proposition~\ref{prop.reghomogenizedsolution}) which is not the case of the solutions of the equation $\partial_t h - \nabla \cdot D_p V(\nabla h) = 0$. This observation is an important ingredient for the proof of Theorem~\ref{main.thm}.
    \item We expect that the constants $\lambda_-$ and $\lambda_+$ are different from the constants $c_-$ and $c_+$. In particular the constant $\lambda_-$ can be much smaller than the constant $c_-$; this is due to the existence of a ``flat" region $|x| \leq R_0$ where no assumption is made on the convexity of $V$.
    \item While we only prove that the surface tension $\bar \sigma$ is $C^{1,1}(\Rd)$ (which only implies that its Hessian is defined almost everywhere), it is reasonable to believe that the surface tension $\bar \sigma$ is in fact twice-continuously differentiable and that its second derivative satisfies the upper and lower bounds~\eqref{eq:conv.barsigma} everywhere. A possible way to prove it would be to adapt the techniques of~\cite{AW, armstrong2022quantitative} which establish the $C^2$ regularity of the surface tension in the case of uniformly convex potentials.
    \item In most of the literature, the $\nabla \varphi$-interface model is introduced through a slightly different formalism: the potential is defined to be a (convex) function $U : \R \to \R$ and the Gibbs measure is given by the identity
\begin{equation} \label{def.gradphifinitevoloriginal}
    \mu_{U, \Lambda}(d \varphi) := \frac{1}{Z_{U,\Lambda}} \exp \left( - \sum_{ x  \in \Lambda } \left( \sum_{i = 1}^d U( \nabla_i \varphi(x)) \right) \right) \prod_{x \in \Lambda} d \varphi(x).
\end{equation}
While the two models~\eqref{def.gradphifinitevol} and~\eqref{def.gradphifinitevoloriginal} have similar definitions, an important distinction has to be made: while the potential associated with the model~\eqref{def.gradphifinitevol} satisfies the \emph{isotropic} growth condition $V(\nabla \varphi) \geq c |\nabla \varphi|^{r} - C$, the potential associated with~\eqref{def.gradphifinitevoloriginal} can only satisfy an \emph{anisotropic} growth condition of the form $\sum_{i = 1}^d U (\nabla_i \varphi) \geq \sum_{i = 1}^d \left| \nabla_i \varphi  \right|^r - C$ (depending on the assumptions on $U$). 

This anisotropy should affect the properties of the surface tension of the model~\eqref{def.gradphifinitevoloriginal} (in particular, it should not satisfy the inequality~\eqref{eq:conv.barsigma}), and eventually affect the regularity properties of the solution of the equation~\eqref{def.barumainthm}. We decided to focus in this article on the model~\eqref{def.gradphifinitevol}.
\end{itemize}
\end{remark}

\subsection{Related works} \label{section.bib}

We mention in this section some of the important results about the $\nabla \varphi$ interface model, but the list is certainly not exhaustive and we refer the interested reader to the review articles~\cite{F05, V06} on the topic. The study of the $\nabla \varphi$-interface model was initiated by Brascamp, Lieb and Lebowitz~\cite{BLL75} who studied the typical height of the interface (depending on the dimension and the potential). This article is devoted to the hydrodynamic limit and we mention that, beside the article of Funaki and Spohn~\cite{FS}, important approaches have been proposed by Guo, Papanicolaou, Varadhan~\cite{guo1988nonlinear} and Yau~\cite{yau1991relative}, and more recently in the contributions~\cite{grunewald2009two, F2012, dizdar2018quantitative} (in these three references, the proofs make important use of logarithmic Sobolev inequalities). We additionally refer to the recent works~\cite{giunti2022quantitative, gu2024quantitative, funaki2024quantitative} where quantitative homogenization methods are used to study interacting particles systems (and in particular obtain quantitative hydrodynamic limits).

Other properties of the model have been successfully investigated.  In the uniformly convex setting, the scaling limit of the model was identified by Brydges and Yau~\cite{BY} in a perturbative setting, and by Naddaf, Spencer~\cite{NS} and Giacomin, Olla, Spohn~\cite{GOS} for general uniformly convex potentials. Large deviation estimates and concentration inequalities were established by Deuschel, Giacomin and Ioffe~\cite{DGI00}, and sharp decorrelation estimates for the discrete gradient of the field were obtained by Delmotte and Deuschel~\cite{DD05}. The scaling limit of the field in finite-volume was established by Miller~\cite{Mi}. More recently, the $C^2$ regularity of the surface tension and the fluctuation-dissipation relation were proved by Armstrong and Wu~\cite{AW} (see also the recent subsequent work of Wu~\cite{wu2022local}) and by Adams and Koller~\cite{adams2023hessian}, and Deuschel and Rodriguez~\cite{deuschel2022isomorphism} identified the scaling limit of the square of the gradient field. We conclude this paragraph by mentioning that the maximum of the interface has been investigated in the recent articles~\cite{BWu2020, wu2019subsequential, schweiger2024tightness}.

The case of non-uniformly convex potentials was studied in the high temperature regime by Cotar, Deuschel and M\"{u}ller~\cite{CDM09}, who established the strict convexity of the surface tension, and by Cotar and Deuschel~\cite{CD12} who proved the uniqueness of ergodic Gibbs measures, obtained sharp estimates on the decay of covariances and identified the scaling limit of the model (see also~\cite{DNV19} for the hydrodynamic limit).
The strict convexity of the surface tension in the low temperature regime was established by Adams, Koteck{\'y} and M\"{u}ller~\cite{AKM16} through a renormalization group argument. This renormalization group approach was further developed in~\cite{adams2019cauchy} to obtain a (form of) verification of the Cauchy-Born rule for these models (we additionally refer to the works of Hilger~\cite{hilger2016scaling, hilger2020scaling, hilger2020decay} for additional results in this line of research). In~\cite{BK07}, Biskup and Koteck{\'y} showed the possible non-uniqueness of infinite-volume, shift-ergodic gradient Gibbs measures for some nonconvex interaction potentials, and Biskup and Spohn~\cite{BS11} proved that, for an important class of nonconvex potentials, the scaling limit of the model is a Gaussian free field (see Armstrong and Wu~\cite{AW23} for the scaling limit of the SOS-model using a similar strategy). We finally mention the recent works of Magazinov and Peled~\cite{magazinov2020concentration}, who established sharp localization and delocalization estimates for a class of convex degenerate potentials $V$, the one of Andres and Taylor~\cite{AT21} who identified the scaling limit of the model for a class of convex potentials satisfying the assumption $\inf V'' \geq \lambda > 0,$ and the recent work of Sellke~\cite{Sel24} who obtained sharp upper bounds for the localization/delocalization of the interface for a broad class of potentials. Another important related model is the so-called integer-valued Gaussian free field (where the interface $\varphi$ is assumed to take values in $\Z$ instead of $\R$), for which the scaling limit was recently identified (in two dimensions and at high temperature) in a series of breakthrough articles by Bauerschmidt, Park and Rodriguez~\cite{bauerschmidt2022discreteI, bauerschmidt2022discreteII}.

We finally refer to the thesis of Sheffield~\cite{Sh} for many additional results and techniques on this class of models (including large deviations principles for the random interface, proof of the strict convexity of the surface tension, the introduction of the cluster swapping, etc.).

\subsection{Sketch of proof}

The proof of Theorem~\ref{main.thm} relies on a combination of ideas and techniques developed in four different articles~\cite{armstrong2022quantitative, MO16, biskup2018limit, D23U}. Each of them is described in a subsection below.

\subsubsection{The Langevin dynamics as a stochastic homogenization problem} \label{section:articlewithScott}

The starting point of this article is an analogy between hydrodynamic limit for Langevin dynamics and stochastic homogenization of nonlinear equation which was first exploited in~\cite{armstrong2022quantitative}. 

To be more specific, in the standard problem of stochastic homogenization of nonlinear elliptic equations (see e.g.~\cite{AS, AFK1, AFK2, FN19, CG21}), one considers a random Lagrangian $L : (x , p) \mapsto L(x , p)$ with $x , p \in \Rd$ and assume that $L$ is uniformly convex in the $p$ variable. One is then interested in studying the large-scale behaviour of the solutions of the nonlinear elliptic equation
\begin{equation} \label{eq:20370303}
    \nabla \cdot D_p L(x , \nabla u) = 0 ~\mbox{in}~ \Rd.
\end{equation}
 The standard homogenization theorem~\cite{DM1,DM2} asserts that, under some suitable assumptions on the law of the Lagrangian, there exists a deterministic \emph{effective Lagrangian} $p \mapsto \bar L(p)$ such that any solution of~\eqref{eq:20370303} is well-approximated over large scales by a solution $\bar u$ of the equation
\begin{equation*}
    \nabla \cdot D_p \bar L(\nabla \bar u) = 0 ~\mbox{in}~ \Rd.
\end{equation*}
The starting point of our analysis is the observation that the Langevin dynamic~\eqref{eq:introSDE} can be viewed as a (discrete) nonlinear parabolic equation with noise, where the randomness is not encoded in the Lagrangian but externally through the Brownian motions.

In comparison to the homogenization theorem mentioned above, the hydrodynamic limit for the $\nabla \phi$ model~\cite{FS} states that the solutions of the Langevin dynamics~\eqref{eq:introSDE} are well-approximated over large-scales by the solution of the deterministic equation
\begin{align*}
\partial_t \bar u - \nabla \cdot D_p \bar \sigma (\nabla \bar u) = 0.
\end{align*}
The hydrodynamic limit can thus be viewed as a homogenization theorem, where the surface tension $\bar \sigma$ plays the role of the effective Lagrangian. The main objective of~\cite{armstrong2022quantitative} was to make this analogy rigorous and to prove (under the assumption that the potential is uniformly convex) a quantitative version of the hydrodynamic limit using the classical tool used in stochastic homogenization, namely the \emph{two-scale expansion}.

An important ingredient in the implementation of a two-scale expansion is the first-order corrector. In the case of the Langevin dynamic~\eqref{eq:introSDE}, we will make use of a finite-volume version of this quantity introduced in Definition~\ref{Prop:Langevin} below (we may refer to this function as either the (Langevin) dynamic ot the (first-order) corrector). Two properties are important on the first-order corrector in order to implement a two-scale expansion: the sublinearity of the first-order corrector and of the weak-norm of its flux. Sections~\ref{sec:section5sublinearity} and~\ref{sec:section5sublinearityflux} are devoted to the proofs of these properties.

Compared to the article~\cite{armstrong2022quantitative}, the main additional difficulty is the non-uniform convexity of the potential allowed by the Assumption~\eqref{AssPot}. Uniform convexity is both useful to study the Gibbs measure~\eqref{def.gradphifinitevol}, via for instance the Brascamp-Lieb inequality, and the Langevin dynamic, via elliptic regularity estimates such as the Caccioppoli inequality or the Nash-Aronson estimate.

In order to compensate for the lack of uniform ellipticity, we rely on a technique introduced by  Mourrat and Otto~\cite{MO16} and further developed by Biskup and Rodriguez~\cite{biskup2018limit} which is discussed in more details below.

\subsubsection{Parabolic equations with degenerate coefficients and the moderated environment} \label{sec:intromoderatedenvt}

The articles of Mourrat and Otto~\cite{MO16} and Biskup and Rodriguez~\cite{biskup2018limit} are devoted to the following problem (N.B. the formalism has been slightly tweaked to match the one used in this article): they consider solutions of the discrete parabolic elliptic equations
\begin{equation} \label{eq:parabolicintro}
    \partial_t u - \nabla \cdot \a \nabla u = 0,
\end{equation}
where $\a : (0 , \infty) \times \Zd \to \mathcal{S}^+(\R)$ is an environment (see Section~\ref{sec.defellipticequation}). Under the assumption that the matrix $\a$ is uniformly elliptic (i.e., its eigenvalues lie between two strictly positive constants), many properties can be established regarding the regularity of the solutions of~\eqref{eq:parabolicintro}. An interesting and fruitful line of research consists in extending these results of regularity to environments which are degenerate but random and whose law satisfies some suitable assumptions (see~\cite{biskupsurvey} for a review). 

The articles~\cite{MO16, biskup2018limit} belong to this line of research and work under the following assumption on the law of the environment: if we let $\Lambda_-(t , x)$ be the smallest eigenvalue of $\a(t, x)$ and define
\begin{equation*}
    \mathbf{m}(t , x) := \int_0^\infty \frac{\Lambda_-(t + s , x)}{(1 + s)^4} \, ds,
\end{equation*}
then this random variable is almost surely strictly positive and satisfies
\begin{equation} \label{eq:assumptionmoderatedintro}
    \E \left[ \mathbf{m}(t , x)^{-q} \right] < \infty ~~ \mbox{for some explicit exponent } q > 1.
\end{equation}
The random variable $\mathbf{m}$ is called \emph{the moderated environment}. In words the assumption~\eqref{eq:assumptionmoderatedintro} states that, if the matrix $\a$ is degenerate at some point $(t , x) \in (0 , \infty) \times \Zd$, then, with high probability, it will not remain degenerate for a very long time.

Under the assumption~\eqref{eq:assumptionmoderatedintro}, Mourrat and Otto~\cite{MO16} obtained on diagonal heat kernel estimates and Biskup and Rodriguez~\cite{biskup2018limit} obtained an $L^\infty$-regularity estimate which is then used to derive a quenched invariance principle for the random walk evolving in the random environment $\a$. 

The general strategy of the present article is to combine the ideas and techniques of~\cite{MO16, biskup2018limit} with the strategy presented in Section~\ref{section:articlewithScott} to establish the hydrodynamic limit for the Langevin dynamic in the setting of a degenerate potential satisfying the Assumption~\eqref{AssPot}. 

An important step in the implementation of this strategy is to verify the moment assumption~\eqref{eq:assumptionmoderatedintro} in the case when the environment $\a$ is given by the formula 
\begin{equation} \label{eq:evtHSintro}
    \a(t , x) := D_p^2 V(\nabla \varphi(t , x)),
\end{equation}
where $(t , x) \mapsto \varphi(t , x)$ is the Langevin dynamic introduced in~\eqref{eq:introSDE} (or more specifically, the one with periodic boundary condition introduced in Definition~\eqref{Prop:Langevin}). This part of the argument relies on the technique developed in~\cite{D23U} to which the next section is dedicated.

\subsubsection{Forcing the fluctuations of the Langevin dynamics} \label{subsection1.2.3}

The core of the proof of the moment assumption~\eqref{eq:assumptionmoderatedintro} when $\a$ is given by~\eqref{eq:evtHSintro} is the following statement on the Langevin dynamic: for any $R > 0$, there exists a constant $c_R > 0$ depending only on $d$ and $R$ such that, for any time $T \geq 1$ and any $x \in \Lambda$,
\begin{equation} \label{eq:22440901}
    \mathbb{P} \left[  \, \forall t \in [0 , T], \, \left| \nabla \varphi(t , x) \right| \leq R \,  \right] \leq 2 \exp \left( - c_R \left| \ln T \right|^{\frac{r}{r-2}} \right).
\end{equation}
The inequality~\eqref{eq:22440901} states that the probability that the gradient of the Langevin dynamic remains in any fixed compact set for a time $T \gg 1$ decays super-polynomially fast in $T$ (as the exponent $r / (r - 2)$ is strictly larger than $1$). Combining this result with the Assumption~\eqref{AssPot} (and some technical work) shows that the probability for the environment $\a(\cdot , x)$ to remain degenerate for a long time $T \gg 1$ decays super-polynomially fast in $T$. This turns out to be sufficiently strong to deduce the moment bound~\eqref{eq:assumptionmoderatedintro} (in fact, it implies that all the moments of the moderated environment are finite).

\begin{figure}
        \centering
        \includegraphics{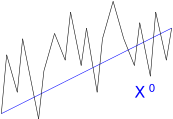}
    \caption{A realization of a Brownian motion (in black) and an increment (in blue).}
\end{figure}

In the rest of this section, we give a brief sketch of the proof of the inequality~\eqref{eq:22440901}, trying to highlight the main ideas without insisting on the technical details. We will in particular present the argument in the simpler setting where $r = 2$ (i.e., the Hessian of $V$ remains bounded). The proof relies on three observations:
\begin{enumerate}
    \item[(i)] The Langevin dynamic $\varphi$ can be seen as a \emph{deterministic function} of the Brownian motions.
    \item[(ii)]  For any $x \in \Zd$, the Brownian motion $B_t(x)$ can be decomposed into a sum of independent increments and Brownian bridges by defining, for any $n \in \N$ and any $t \in [n , n+1]$,
    \begin{equation} \label{def.XnandWnsketch}
            X_n(x) := B_{n+1}(x) -  B_{n}(x) ~~\mbox{and}~~ W_n(t , x) := B_t(x) - B_{n}(x) - (t-n) X_n(x).
    \end{equation}
    \item[(iii)] Using~\eqref{def.XnandWnsketch}, we may rewrite the first line of~\eqref{eq:introSDE} as follows, for any $n \in \N$,
 \begin{equation*}
  \begin{aligned}
    d \varphi(t , x) & = \nabla \cdot D_p V(\nabla \varphi(t , x)) dt  + \sqrt{2} X_n(x) dt + \sqrt{2} dW_n(t , x) &~~\mbox{for}~~(t , x) \in (n , n+1) \times \Lambda.
    \end{aligned}
\end{equation*}
\end{enumerate}
The strategy is then to \emph{differentiate} the Langevin dynamic $\varphi(t , y)$ (for some fixed $(t , y) \in (0 , \infty) \times \Lambda$) with respect to an increment $X_n(x)$. An explicit computation (see Proposition~\ref{prop:propsLangevin}) yields the identity
\begin{equation*}
    \frac{\partial \varphi(t , x)}{\partial X_n(x)}  = \sqrt{2} \int_{n}^{n+1} P_\a \left( t , y ; s , x \right)  \, ds,
\end{equation*}
where $P_\a \left( \cdot , \cdot ; s , x \right)$ is the heat kernel started at time $s$ from the vertex $x$ in the environment $\a$ (see Section~\ref{sec:sectionheatkernel} for a formal definition in the periodic setting).

The strategy is then to use the properties of the heat-kernel to prove that there exists an explicit constant $c > 0$ such that
\begin{equation*}
    \frac{\partial \varphi(n+1 , x)}{\partial X_n(x)} \geq c > 0.
\end{equation*}
This inequality means that the partial derivative of the value $\varphi(n+1 , x )$ with respect to the increment $X_n(x)$ is lower bounded by $c > 0$, and thus that $\varphi(n+1 , x )$ is a random variable which is \emph{sensitive} to the increment $X_n(x)$. In fact a similar statement can be established on the \emph{gradient} of the dynamic $\nabla \varphi(n+1 , x )$. Combining this result with the observation that the distribution of the increment $X_n(x)$ is unbounded (as it is Gaussian), we obtain that for any $R > 1$, there exists a constant $c'_R >0$ such that
\begin{equation*}
    \mathbb{P} \left[ \left| \nabla \varphi(n+1 , x) \right| \leq R \right] \leq 1 - c'_R.
\end{equation*}
Using that the increments of the Brownian motion are independent, the previous inequality can be iterated to obtain a result of the form: for any $N \in \N$,
\begin{equation} \label{eq:introexprate}
    \mathbb{P} \left[ \forall n \in \{ 1 , \ldots , N \}, ~ \left| \nabla \varphi (n , x) \right| \leq R \right] \leq ( 1 - c'_R)^N,
\end{equation}
which implies, under the additional assumption $r = 2$, that the left-hand side of~\eqref{eq:22440901} decays exponentially fast in $T$ (this is much stronger than the super-polynomial decay on the right-hand side of~\eqref{eq:22440901}). 

This argument is essentially a proof and could be implemented (in the case $r = 2$) to obtain an exponential decay for the probability of the gradient of the dynamic to remain in a bounded set for a long time. In the case $r > 2$ considered in this article, the unboundedness of the Hessian of $V$ causes some additional technical difficulties, which result in a deterioration of the stochastic integrability estimate from the exponential rate on the right-hand side of~\eqref{eq:introexprate} to the super-polynomial rate on the right-hand side of~\eqref{eq:22440901}.

\subsection{Convention for constants} Throughout this article, the symbols $C$ and $c$ denote strictly positive constants with $C$ larger than $1$ and $c$ smaller than $1$. We allow these constants to vary from line to line, with $C$ increasing and $c$ decreasing. These constants may depend only on the dimension $d$, the potential $V$ and, in Section~\ref{sec:section6}, on the initial condition $f$. We specify the dependency of the constants and exponents by writing, for instance, $C := C(d , V)$ to mean that the constant $C$ depends on the parameters $d$ and $V$.

\section{Preliminary results and notation}

\subsection{Notation}

We unfortunately must introduce quite a bit of notation, particularly since we are making use of techniques and results from different settings (discrete parabolic equations, stochastic homogenization, statistical mechanics). The reader is encouraged to skim and consult as a reference.

\subsubsection{General notation}

We consider the hypercubic lattice $\Zd$, the real vector space $\Rd$ in dimension $d \geq 2$ and denote by $(e_1 , \ldots , e_d)$ the canonical basis of $\Rd$. For $x , y \in \Rd$, we use the notation $x \cdot y$ to refer to the Euclidean scalar product on the spaces $\Rd$ (or $\R^{2d}$). We denote by $\left|\cdot \right|$ the Euclidean norm on $\Rd$ and write $\left|\cdot \right|_+ := \left|\cdot \right|+1$. Given two vertices $x , y \in \mathbb{Z}^d$, we write $x \sim y$ if $|x - y| = 1$. We denote by $\mathcal{S}^+(\Rd)$ the set of $(d \times d)$ symmetric matrices with positive eigenvalues.

Given two real numbers $a , b$, we denote by $a \wedge b := \min(a , b)$ and by $a \vee b := \max(a , b)$, and by $\lfloor a \rfloor$ and $\lceil a \rceil$ the floor and ceiling of $a$ respectively. We denote by $\indc_A$ the indicator function of a set $A$.

Given an integer $L \in \N$, we introduce the box and parabolic cylinder
\begin{equation*}
    \Lambda_L := \{-L, \ldots, L\}^d
    \subseteq \Z^d
    ~~\mbox{and}~~ Q_L := (-L^2 , 0) \times \Lambda_L.
\end{equation*}
We denote by $|\Lambda_L | := (2L+1)^d$ the cardinality (or volume) of the box $\Lambda_L$ and by $Q_L := L^2 (2L+1)^d$ the volume of the parabolic cylinder. More generally, given a finite set $U \subseteq \Zd$ and a bounded interval $I \subseteq \R$, we denote by $|U|$ the cardinality of $U$ and by $|I \times U| = |I| \times |U|$ (where $|I|$ is the Lebesgue measure of $I$) the volume of the parabolic cylinder $I \times U$.

We denote by $\partial^+ \Lambda_L$ the outer boundary of $\Lambda_L$, i.e., $\partial^+ \Lambda_L := \left\{ y \in \Zd \setminus \Lambda \, : \, \exists x \in \Lambda_L, \, y \sim x \right\}$.

As mentioned above, we let $\mathbb{T}_L := (\Z/(L+1) \Z)^d$ be the $d$-dimensional discrete torus and denote by $|\mathbb{T}_L| := (2L+1)^d$. We note that we may identify the vertices of $\mathbb{T}_L$ with the ones of $\Lambda_L $ (N.B. the torus is only used in this article to emphasize that we work with periodic boundary conditions)

\subsubsection{Brownian motions} \label{sectionBM}

Throughout this paper, we consider a collection of independent Brownian motions $\left\{ B_t(x) \, : \, t \geq 0, \, x \in \Zd \right\}$. For a technical reason (in Section~\ref{sec:introdLangevindyn}), we will need to have a definition for a Brownian motion defined for any time $t \in \R$ (and not only for the positive times). We will thus extend the previous definition as follows: we consider a second collection of independent Brownian motions $\left\{ B_t^1(x) \, : \, t \geq 0, \, x \in \Zd \right\}$ (which are independent of $(B_t(x))_{x \in \Zd, t \geq 0}$) and set, for any $t \in (- \infty , 0)$,
\begin{equation*}
    B_t (x) = B_{-t}^1(x).
\end{equation*}
This gives a reasonable definition of a Brownian motion defined on $\R$, since the trajectories are continuous, for any $t , s \in \R$, $B_t - B_s$ is a Gaussian random variable of variance $|t - s|$ and for any $t , s , t_1 , s_1 \in \R$ with $s < t < s_1 < t_1$ the random variables $B_t - B_s$ and $B_{t_1} - B_{s_1}$ are independent. Let us additionally note that the following property holds: for any $T \in \R$ and any $x \in \Zd$, the process $\left\{ B_{T + t}(x) - B_T(x) \, : \, t \geq 0 \right\}$ is a Brownian motion.

We denote by $\mathbb{P}$ the law of these Brownian motions and by $\E$ the corresponding expectation. We will denote by
\begin{equation*} 
    \tilde{B}_t(x) := B_t(x) - \frac{1}{\left| \Lambda_L \right|} \sum_{x \in \Lambda_L} B_t(x) ~~\mbox{for}~ t \in \R.
\end{equation*}
The collection $\{ \tilde B_t(x) \, : \, t \in  \R, \, x \in \mathbb{T}_L \}$ (identifying the vertices of the box $\Lambda_L$ with the ones of the torus $\mathbb{T}_L$) is a Brownian motion on the Euclidean vector space $\Omega^\circ_L := \left\{ \varphi : \mathbb{T}_L \to \R \, : \, \sum_{x \in \mathbb{T}_L} \varphi(x) =0 \right\}$ (equipped with the $L^2$ scalar product).

\medskip

\subsubsection{Discrete differential and elliptic operators} \label{sec.defellipticequation}

Given a function $\varphi : \Zd \to \mathbb{R}$ and a vertex $x \in \Zd$, we define the discrete gradient $\nabla \varphi : \Zd \to \mathbb{R}^d$ according to the formula,
\begin{align} \label{def:discretegradient}
    \nabla \varphi (x) &:= \left(  \nabla_1 \varphi(x) , \ldots, \nabla_d \varphi(x)\right) \\
                    & = \left(  \varphi(x + e_1) - \varphi(x ), \ldots,  \varphi(x + e_d) -  \varphi(x ) \right) \in \mathbb{R}^d, \notag
\end{align}
A vector field is a map $\vec{F} := (\vec{F}_1 , \ldots, \vec{F}_d) : \Zd \to \Rd$. The divergence of a vector field is the map $\nabla \cdot \vec{F} : \Zd \to \R$ defined according to the identity, for any $x \in \Zd$,
\begin{equation} \label{def:discretediv}
    \nabla \cdot \vec{F}(x) := \sum_{i = 1}^d (\vec{F}_i(x) - \vec{F}_i(x - e_i)).
\end{equation}
The discrete divergence is defined so as to satisfy the following integration by parts property: for any finitely supported function $v : \Zd \to \R$,
\begin{equation*} 
    \sum_{x \in \Zd} (\nabla \cdot \vec{F}(x)) v(x) = - \sum_{x \in \Zd} \vec{F}(x) \cdot \nabla v(x).
\end{equation*}
An environment is a function $\a : \Zd \to \mathcal{S}^+(\R^d)$.
We introduce the discrete elliptic operator, for any function $u : \Zd \to \R$,
\begin{equation} \label{def:divagrad}
    \nabla \cdot \a \nabla u (x) := \sum_{i = 1}^d e_i \cdot \left( \a(x) \nabla u(x) - \a(x - e_i) \nabla u(x- e_i) \right).
\end{equation}
The notation is consistent with the definition of the discrete divergence~\eqref{def:discretediv}, and the formula~\eqref{def:divagrad} can be obtained equivalently by applying the discrete divergence to the vector field $x \mapsto \a(x) \nabla u(x)$. As a consequence, the following property holds: for any functions $u , v : \Zd \to \R$ with $u$ or $v$ finitely supported
\begin{equation} \label{eq:discreteIPP}
    \sum_{x \in \Zd} \left( \nabla \cdot \a \nabla u (x) \right) v(x) = - \sum_{x \in \Zd} \nabla v(x) \cdot \a(x) \nabla u (x).
\end{equation}
All the previous definitions are implicitly extended from $\Zd$ to the discrete torus $\mathbb{T}_L$ and to functions depending on time.

\subsubsection{The potential $V$}

In this article, we let $V \in C^2(\Rd)$ be a fixed convex potential satisfying the Assumption~\eqref{AssPot}. We denote by $D_p V : \Rd \to \Rd$ its gradient and $D_p^2 V : \Rd \to \mathcal{S}(\Rd)$ its Hessian, i.e.,
\begin{equation*}
    \forall p \in \Rd, ~ D_p V(p) := \left( \frac{\partial V}{\partial p_1}(p) , \ldots, \frac{\partial V}{\partial p_d}(p) \right) ~~\mbox{and}~~ D_p^2 V(p) := \left( \frac{\partial^2 V}{\partial p_i \partial p_j}(p) \right)_{i , j \in \{ 1 , \ldots, d\}}.
\end{equation*}
We use the notation $D_p V$ instead of the more standard $\nabla V$ in order to reserve the notation $\nabla$ for the discrete gradient~\eqref{def:discretegradient} (for functions defined on $\Zd$).

For $p \in \Rd$, we denote by $\Lambda_+(p)$ to be the largest eigenvalue of the symmetric positive matrix $D^2_p V(p)$, i.e.,
\begin{equation*}
    \Lambda_+(p) := 1 \vee \sup_{\substack{\xi \in \Rd \\ | \xi | \leq 1}} \xi \cdot D^2_p V(p) \xi.
\end{equation*}
 The maximum with $1$ is added so that $\Lambda_+$ is always larger than $1$ (this is to simplify the notation in the proofs below). For the smallest eigenvalue, we adopt a more general definition
\begin{equation} \label{eq:defLambda-}
    \Lambda_-(p) := \inf_{q \in \Rd} \inf_{\substack{\xi \in \Rd \\ | \xi | \leq 1}} \int_0^1  \xi \cdot D^2_p V((1-t)p + t q) \xi  \, dt.
\end{equation}
Note that $\Lambda_-(p)$ is smaller than the smallest eigenvalue of the $D^2_p V(p)$ (this is obtained by taking $q = p$ in the previous definition).

The following lemma provides provides upper and lower bounds on the values $\Lambda_+(p)$ and $\Lambda_-(p)$. They are simple consequences of the Assumption~\eqref{AssPot}.

\begin{lemma} \label{lemma.upperandlowerboundLambda}
There exist three constants $R_1 := R_1(d , V) \in (2 , \infty)$, $C := C(d , V) < \infty$ and $c := c(d , V) > 0$ such that the following hold:
\begin{itemize}
\item For any slope $p \in \Rd$,
\begin{equation*}
    \Lambda_+(p) \leq C |p|^{r-2}_+.
\end{equation*}
\item For any slope $p \in \Rd$ satisfying $|p| \geq R_1/2$,
\begin{equation*}
     \Lambda_-(p) \geq c |p|^{r-2}_+ + 1.
\end{equation*}
\item For any slope $p \in \Rd$, any $q \in \Rd$ with $|q| \geq R_1$ and any $\xi \in \Rd$,
\begin{equation*}
    \int_0^1  \xi \cdot D^2_p V((1-t)p + t q) \xi  \, dt \geq |\xi|^2.
\end{equation*}
\end{itemize}
\end{lemma}

\subsubsection{Norms and Sobolev spaces} \label{subsecfunctions}

In what follows, we let $q \in (1 , \infty)$ be an exponent and denote by $q' = q/(q-1)$ the conjugate exponent of $q$. Given an integer $L \in \N$ and a function $u : \Lambda_L \to \R$, we define the following scaled norms:
\begin{itemize}
    \item \emph{$L^2$-norm:}
    $
    \left\| u \right\|_{\underline{L}^2 \left( \Lambda_L \right)}^2 :=  \frac{1}{\left| \Lambda_L \right|} \sum_{x \in \Lambda_L} \left| u(x)\right|^2,
    $ 
    \item \emph{$L^q$-norm:}
    $
    \left\| u \right\|_{\underline{L}^q \left( \Lambda_L \right)}^q :=  \frac{1}{\left| \Lambda_L \right|} \sum_{x \in \Lambda_L} \left| u(x)\right|^q,
    $
    \item \emph{$W^{1,q}$-norm:}
    $
        \left\| u \right\|_{\underline{W}^{1,q}(\Lambda_L)} :=  \frac{1}{L} \left\| u \right\|_{\underline{L}^q (\Lambda_L)} + \left\| \nabla u \right\|_{\underline{L}^q (\Lambda_L)},
    $
    \item \emph{$W^{-1,q}$-norm:}
        $
        \left\| u \right\|_{W^{-1,q}(\Lambda_L)} := \sup \left\{ \frac{1}{\left| \Lambda_L \right|} \sum_{x \in \Lambda_L} u(x) v(x) \, : \,   \left\| v \right\|_{\underline{W}^{1,q'}(\Lambda_L)} \leq 1 \, \mbox{and} \, v = 0 ~ \mbox{on} ~ \partial^+ \Lambda_L \right\}.
        $
\end{itemize}

In the parabolic setting, given an integer $L \in \N$, a function $u : Q_L \to \R$ and a vector field $\vec{F}  : Q_L \to \R$, we define their average value over $Q_L$ according to the formulae
\begin{equation*}
    \left( u \right)_Q := \frac{1}{|Q_L|} \int_{(-L^2 , 0)} \sum_{x \in \Lambda_L} u(t , x) \, dt ~~ \mbox{and}~~ ( \vec{F} )_{Q_L} := \frac{1}{|Q_L|} \int_{(-L^2 , 0)} \sum_{x \in \Lambda} \vec{F}(t , x) \, dt.
\end{equation*}

We then define the following norms and Sobolev spaces, for any function $u : Q_L \to \R$,
\begin{itemize}
    \item \emph{$L^2$-norm:}
    $
    \left\| u \right\|_{\underline{L}^2 \left( Q_L \right)}^2 :=  \frac{1}{L^2} \int_{(-L^2 , 0)}\left\| u(t , \cdot)\right\|^2_{\underline{L}^2(\Lambda_L)} \, dt,
    $
    \item \emph{$L^q$-norm:}
    $
    \left\| u \right\|_{\underline{L}^q \left( Q_L \right)}^q :=  \frac{1}{L^2} \int_{(-L^2 , 0)}\left\| u(t , \cdot)\right\|^q_{\underline{L}^q(\Lambda_L)} \, dt,
    $
    \item \emph{$L^qW^{1,q}$-norm:}
    $
        \left\| u \right\|_{\underline{L}^q((-L^2 , 0) , \underline{W}^{1,q}(\Lambda_L))}^q :=  \frac{1}{L^2} \int_{(-L^2 , 0)}\left\| u(t , \cdot)\right\|^q_{\underline{W}^{1,q}(\Lambda_L)} \, dt,
    $
    \item \emph{$L^{q}W^{-1,q}$-norm:}
         $
        \left\| u \right\|_{\underline{L}^{q}((-L^2 , 0) , \underline{W}^{-1,q}(\Lambda_L))}^{q} := \frac{1}{L^2} \int_{(-L^2 , 0)}\left\| u(t , \cdot)\right\|^{q}_{\underline{W}^{-1,q}(\Lambda_L)} \, dt,
        $
    \item  \emph{$W^{1,q}_{\mathrm{par}}$-norm:}
    $
    \left\| u \right\|_{\underline{W}^{1,q}_{\mathrm{par}}(Q_L)} := \frac{1}{L} \left\| u \right\|_{\underline{L}^q \left( Q_L \right)} + \left\| \nabla u \right\|_{\underline{L}^q \left( Q_L \right)} + \left\| \partial_t u \right\|_{\underline{L}^{q}((-L^2 , 0) , \underline{W}^{-1,q}(\Lambda_L))},
    $
    \item  \emph{$\hat{W}^{-1,q}_{\mathrm{par}}$-norm:}
    $
    \left\| u \right\|_{\hat{W}^{-1,q}_{\mathrm{par}}(Q_L)} := \sup \left\{ \frac{1}{\left| Q_L \right|} \int_{(-L^2,0)}\sum_{x \in \Lambda_L} u(t , x) v(t , x) \, dt : \,  \left\| v \right\|_{\underline{W}^{1,q'}_{\mathrm{par}}(Q_L)} \leq 1 \right\}.
    $
\end{itemize}

\subsection{Parabolic equations}

This section introduces some notation pertaining to discrete parabolic equations. We will fix an integer $L \in \mathbb{N}$ and a time dependent environment $\a :\R \times \mathbb{T}_L \to \mathcal{S}^+(\R^d)$. We denote below the discrete Dirac $\delta_x : \mathbb{T}_L \to \R$ to be the function given by the identity $\delta_x(x) = 1$ and $\delta_x(y) = 0$ for $y \in \mathbb{T}_L \setminus \{ 0 \}$.

\subsubsection{Heat kernel} \label{sec:sectionheatkernel}

\begin{definition}[Heat kernel on the torus]
For a time $s \in \mathbb{R}$ and a vertex $y \in \mathbb{T}_L$, we define the discrete heat kernel $P_\a(\cdot , \cdot ; s , y) : (s , \infty) \times \mathbb{T}_L \to \R$ started from $y$ at time $s$  to be the solution of the discrete parabolic equation
\begin{equation} \label{eq:defheatkernelPa}
    \left\{ \begin{aligned}
        \partial_t P_\a - \nabla \cdot \a \nabla P_\a & = 0 &~~\mbox{in}~~ &(s , \infty) \times \mathbb{T}_L, \\
        P_\a (s , \cdot ; s , y ) & = \delta_{x} - \frac{1}{\left| \mathbb{T}_L \right|} &~~\mbox{in}~~ &\mathbb{T}_L.
    \end{aligned} \right.
\end{equation}
\end{definition}

\begin{remark}
    Let us make a few remarks about the previous definition:
    \begin{itemize}
        \item We choose here a definition of the heat kernel in the torus (i.e., with periodic boundary conditions) as it will be useful to study the periodic Gibbs measure introduced in Definition~\ref{def:periodicGibbsmeasure} (or the periodic Langevin dynamic defined below).
        \item As it will be useful in the proofs below, we extend the definition of the heat kernel $P_\a(\cdot , \cdot ; s , x)$ to all the times $t \in \R$ by setting 
        \begin{equation} \label{eq:conventionheatnegativetime}
            P_\a(t , x; s , y) := 0 ~~\mbox{if}~~ t < s.
        \end{equation}
        \item The reason we added the term $1/\left| \mathbb{T}_L \right|$ in the initial condition is due to the periodic boundary condition. A property that is convenient for the heat kernel to satisfy is that it converges to $0$ as the time tends to infinity (under reasonable assumptions on the environment $\a$), but in the periodic setting, we have the following (preservation of mass) property (obtained by summing the first inequality of~\eqref{eq:defheatkernelPa} over all the vertices of $\mathbb{T}_L$ and applying a discrete integration by parts)
        \begin{equation*}
            \partial_t \sum_{x \in \mathbb{T}_L} P_\a (t , x ; s , y ) = 0 ~~ \implies ~~ t \mapsto \sum_{x \in \mathbb{T}_L} P_\a (t , x ; s , y ) \mbox{ is constant} .
        \end{equation*}
        In particular, the heat kernel can only converge to $0$ if the initial condition has an average value equal to $0$. Subtracting the term $1/\left|\mathbb{T}_L\right|$ to the Dirac $\delta_y$ ensures that this property holds.
    \end{itemize}
   
\end{remark}

The following proposition collects two elementary properties of the heat kernel (the proof of~\eqref{eq:energyestimateheatkernel} is obtained by multiplying the first line of~\eqref{eq:defheatkernelPa} by $P_\a$, summing over the vertices of the torus and performing the discrete integration by parts~\eqref{eq:discreteIPP}).

\begin{proposition}[Properties of the heat kernel] \label{prop:propheatkernel}
For any $(s , y) \in \R \times \mathbb{T}_L$ and any time $t \geq s$, one has the identity
\begin{equation} \label{eq:energyestimateheatkernel}
     \partial_t \left\| P_\a (t , \cdot ; s , y) \right\|_{L^2 (\mathbb{T}_L)}^2 = -  \sum_{x \in \mathbb{T}_L} \nabla P_\a (t, x ; s , y) \cdot \a(t , x) \nabla P_\a(t , x; s , y ).
\end{equation}
By the assumption $\a(t , x) \in \mathcal{S}^+(\Rd)$, the term on the right-hand side is negative, and thus the following properties hold:
\begin{itemize}
    \item \textit{Decay of the $L^2$ norm:} the map $t \mapsto \left\| P_\a (t , \cdot ; s , y) \right\|_{L^2 (\mathbb{T}_L)}$ is decreasing.
    \item \textit{Pointwise and $L^2$ bounds on the heat kernel:} for any time $t \geq s$ and any vertex $x \in \mathbb{T}_L$,
    \begin{equation*}
        \left| P_\a (t , x ; s , y) \right| \leq \left\| P_\a (t , \cdot ; s , y) \right\|_{L^2 (\mathbb{T}_L)} \leq \left\| P_\a (s , \cdot ; s , y ) \right\|_{L^2 (\mathbb{T}_L)} = \left( 1 - \frac{1}{\left| \mathbb{T}_L \right|} \right)^{\sfrac 12} \leq 1.
    \end{equation*}
\end{itemize}
\end{proposition}

We complete this section with a caveat: while this class of discrete parabolic equations shares many properties with its continuous counterpart, the maximum principle does not seem to hold for these equations.

\subsubsection{Parabolic Caccioppoli inequality}

In this section, we state the discrete and parabolic version of the Poincar\'e inequality. Since we did not make any specific assumptions on the environment $\a$ (and on its eigenvalues), we keep a dependency in the environment on both the left and right-hand sides of~\eqref{eq:discreteparabCaccioppoli}.

\begin{proposition}[Parabolic Caccioppoli inequality] \label{prop:paraboliccacciop}
There exists a constant $C := C(d)< \infty$ such that, for any solution of the parabolic equation
\begin{equation*}
    \partial_t u - \nabla \cdot \a \nabla u = 0 ~~\mbox{in}~~ Q_{2L},
\end{equation*}
one has the estimate
\begin{equation} \label{eq:discreteparabCaccioppoli}
    \int_{-L^2}^0 \sum_{x \in \Lambda_L} \nabla u(t , x) \cdot \a(t  , x) \nabla u(t , x) \leq \frac{C}{L^2} \int_{-4L^2}^0 \sum_{x \in \Lambda_{2L}} \Lambda_+(t , x) |u(t , x)|^2 \, dt.
\end{equation}
\end{proposition}

\noindent 
The proof of this inequality follows the standard strategy (introducing a compactly supported cutoff function $\eta$ and testing the parabolic equation with the test function $\eta^2 u$). We thus omit the proof here (a proof can be found in the continuum setting in~\cite[Lemma B.3]{ABM}, among other possible references, the adaptation to the discrete setting and degenerate environment is mostly notational).

\subsubsection{Parabolic multiscale Poincar\'e inequality}

The following inequality provides a convenient control over the $\hat{W}^{-1,q}_{\mathrm{par}}$ of a function in terms of its average values over parabolic cylinders of different sizes. It is a discrete and $L^r$ version of~\cite[Proposition 3.6]{ABM} which is itself a parabolic version of an inequality which first appeared in~\cite[Proposition 6.1]{AKM1}. A proof of this inequality, which is an adaptation of the one of~\cite[Proposition 3.6]{ABM}, can be found in Appendix~\ref{App.multiscale}. Before stating the result, we introduce the following notation: for each pair of integers $m , n \in \mathbb{N}$ with $m \leq n$,
\begin{equation*}
    \mathcal{Z}_{m,n} := (3^{2m} \Z \times 3^m \Zd) \cap Q_{3^n}.
\end{equation*}
The collection $\left( z + Q_{3^m} \right)_{z \in \mathcal{Z}_{m   ,n}}$ is a partition of the parabolic cylinder $Q_{3^n}$.

\begin{proposition}[Parabolic multiscale Poincar\'e inequality] \label{prop.multiscalePoinc}
There exists a constant $C := C(d , r) < \infty$ such that, for any integer $n \in \N$ and any function $f : Q_{3^n} \to \R$, 
\begin{equation*}
    \left\| f \right\|_{\underline{\hat{W}}^{-1,r}_{\mathrm{par}} (Q_{3^n})} \leq C  \left\| f \right\|_{\underline{L}^{r} (Q_{3^n})} + C \sum_{m = 0}^n 3^m \left( \left| \mathcal{Z}_{m,n} \right|^{-1}  \sum_{z \in \mathcal{Z}_{m,n}} \left| \left( f \right)_{z + Q_{3^m}}  \right|^r \right)^{\sfrac 1r}.
\end{equation*}
\end{proposition}

\begin{remark}
    In the proofs below, we will use this inequality when $f$ is a vector field (i.e., valued in $\mathbb{R}^d$, in which case one can measure its $\hat{W}^{-1,q}_{\mathrm{par}}$-norm by considering the maximum of the sum of the $\hat{W}^{-1,q}_{\mathrm{par}}$-norms of each of its components).
\end{remark}

\subsection{Gibbs measure and Langevin dynamics}

In this section, we introduce the Gibbs measure and Langevin dynamics on the torus $\mathbb{T}_L$ with prescribed slope. The analysis of the properties of these dynamics (typical size, weak norm of the flux, fluctuations in time, etc.) occupies a large part of this article and is an essential ingredient in the proof of Theorem~\ref{main.thm}.

\subsubsection{Gibbs measure and surface tension} \label{subsec2.1.6}

An important role is played in this article by the Gibbs measure~\eqref{def.gradphifinitevol}. There are various possibilities to define it (or specifically, various boundary conditions which can be imposed), and we will make extensive use of a tilted version of the Gibbs measure to which periodic boundary conditions are imposed (N.B. the same measure has been introduced and used by Funaki and Spohn~\cite{FS} to derive the hydrodynamic limit in the uniformly convex setting).

Before stating the definition, we recall the notation $\Omega^\circ_L := \left\{ \varphi : \mathbb{T}_L \to \R \, : \, \sum_{x \in \mathbb{T}_L} \varphi(x) =0 \right\}$ for periodic functions whose average value is equal to $0$ on the torus $\mathbb{T}_L$. We denote by $d\varphi$ the Lebesgue measure on this space (N.B. the space $\Omega^\circ_L$ is finite dimensional, it can thus be equipped with a Lebesgue measure once a scalar product has been specified, and we consider here the $L^2$-scalar product).

\begin{definition}[Periodic Gibbs measure $\mu_{L , p}$] \label{def:periodicGibbsmeasure}
Given a sidelength $L \in \N$ and a slope $p \in \Rd$, we define the periodic Gibbs measure on the torus $\mathbb{T}_L$ with slope $p \in \Rd$ according to the formula
\begin{equation*}
    \mu_{L , p} (d \varphi) := \frac{1}{Z_{L,p}} \exp \left( - \sum_{y \in \mathbb{T}_L} V(p + \nabla \varphi(y)) \right) d \varphi ~~\mbox{with}~~ Z_{L,p} := \int_{\Omega^\circ_L} \exp \left( - \sum_{x \in \mathbb{T}_L} V(p + \nabla \varphi(x)) \right) \, d\varphi.
\end{equation*}
We denote by $\E_{L,p}$ and $\mathrm{Var}_{L , p}$ the expectation and variance with respect to the measure $\mu_{L , p}$.
\end{definition}

\begin{remark}
A convenient property satisfied by the measure $\mu_{L , p}$ is that it is invariant under translation.
\end{remark}

An important role is then played by the finite-volume surface tension (as it approximates the surface tension $\bar \sigma$). Its definition is stated below.

\begin{definition}[Finite volume surface tension] \label{def.surfacetension}
Given a sidelength $L \in \N$ and a slope $p \in \Rd$, we define the finite-volume surface tension $\bar \sigma_L(p)$ according to the formula
\begin{equation*}
    \bar \sigma_L(p) := - \frac{1}{\left| \mathbb{T}_L \right|} \ln \frac{Z_{L,p}}{Z_{L,0}}.
\end{equation*}
\end{definition}

We state below an identity for the gradient of $\bar \sigma_L(p)$ which can be obtained by explicitly differentiating the previous definition with respect to the slope $p \in \Rd$ and using the translation invariance of the measure $\mu_{L,p}$ (N.B. the right-hand side below does not depend on $x \in \mathbb{T}_L$ due to this translation invariance)
\begin{equation*}
    \forall x \in \mathbb{T}_L, ~ D_p \bar \sigma_L(p) := \E_{L , p} \left[ D_p V(\nabla \varphi(x) )  \right].
\end{equation*}

\subsubsection{Langevin dynamic in a torus} \label{sec:introdLangevindyn}

\begin{definition}[Stationary Langevin dynamic on the torus] \label{Prop:Langevin}
For almost every realization of the Brownian motions $\{ \tilde B_t(x) \, : \, t \in \R, \, x \in  \mathbb{T}_L \}$ and every slope $p \in \Rd$, there exists a unique function $\varphi_{L } (\cdot , \cdot ; p) : \R \times \mathbb{T}_L \to \R$ satisfying the properties:
\begin{itemize}
\item[(i)] Average value: For any time $t \in \R$, the average value of the function $\varphi_{L }(t , \cdot ; p)$ is equal to $0$, i.e., $$\sum_{x \in \mathbb{T}_L} \varphi_{L}(t , x ;p) = 0.$$
\item[(ii)] Growth at $t \to -\infty$: for any $x \in \mathbb{T}_L$, the function $t \mapsto \frac{1}{t} \varphi_{L}(t , x ;p)$ converges to $0$ as $t \to -\infty$.
\item[(iii)] Stochastic differential equations: the function $\varphi_{L} (\cdot , \cdot ; p) $ is a solution to the system of stochastic differential equations
\begin{equation} \label{eq:def.Langevintorus}
     d \varphi_{L}(t , x ;p) = \nabla \cdot D_p V(p + \nabla \varphi_{L}(\cdot , \cdot ; p)) (t , x) \, dt  + \sqrt{2} d \tilde{B}_t(x) \hspace{5mm} \mbox{for} \hspace{5mm} (t , x) \in \R \times \mathbb{T}_L.
\end{equation}
We call the function $\varphi_{L} (\cdot , \cdot ; p) $ the Langevin dynamic on the torus $\mathbb{T}_L$ with slope $p \in \Rd$.
\end{itemize}
\end{definition}

\begin{remark}
    Let us make a few remarks about the previous definition:
    \begin{itemize}
    \item The item (iii) is understood in the following sense: for any pair of times $t_0 , t_1 \in \R$ with $t_0 \leq t_1$ and any $x \in \mathbb{T}_L,$
    \begin{equation} \label{eq:reqldefSDE}
       \varphi_{L } (t_1 , x;p) - \varphi_{L } (t_0 , x;p) = \int_{t_0}^{t_1} \nabla \cdot D_p V(p + \nabla \varphi_{L}(\cdot , \cdot ; p)) (t , x) \, dt  + \sqrt{2} \tilde{B}_{t_1}(x) - \sqrt{2} \tilde{B}_{t_0}(x).
    \end{equation}
    \item Defining a Langevin dynamic on the torus $\mathbb{T}_L$ is equivalent to solving the system of stochastic differential equations~\eqref{eq:def.Langevintorus} on the box $\Lambda_L$ with periodic boundary conditions (see Remark~\ref{rem:remark6.3} of Section~\ref{sec:section5sublinearityflux} and Section~\ref{sec:section6} below).
    \item The item (iii) implies that the function $ t \mapsto \sum_{x \in \mathbb{T}_L} \varphi_{L}(t , x ;p)$ is constant (by summing both sides of~\eqref{eq:def.Langevintorus} over the vertices of $\mathbb{T}_L$).
    \item The growth condition (ii) could be weakened if needed.
    \item An almost (and more common in the literature) equivalent way to define these dynamics is to solve the stochastic differential equations~\eqref{eq:def.Langevintorus} on the positive times and start from a random initial profile distributed according to the Gibbs measure $\mu_{L , p}$.
    \item We will frequently consider the dynamics as functions of the Brownian motions (see Proposition~\ref{prop:propsLangevin} below). To emphasize this dependency, we will use the notation $\varphi_{L}(t , x ;p)(\{ B_t(x) \, : \, t \in \R, \, x \in \Zd \})$.
     \item For a slope $p \in \Rd$, we introduce the random time-dependent environment 
    \begin{equation*}
        \a(t , x ; p) := D_p^2 V \left(p + \nabla \varphi_{L}  (t , x;p) \right) \in \mathcal{S}^+(\Rd) \hspace{5mm} \mbox{for} \hspace{3mm} (t , x) \in \R \times \mathbb{T}_L.
    \end{equation*}
    \end{itemize}
\end{remark}

We collect in the following proposition some properties satisfied by the Langevin dynamics. Before stating the result, we mention that one of the main approach developed in this article is to differentiate the dynamics with respect to the two parameters on which they depend: the slope $p \in \Rd$ and the Brownian motions $\{ \tilde B_t(x) \, : \, t \geq 0, \, x \in \Zd \}$. Differentiating with respect to the slope is performed as follows: for any $(t , x) \in \R \times \mathbb{T}_L$, any $p, \lambda \in \Rd$,
\begin{equation*}
    w_{L, p , \lambda} (t , x) = \lim_{\ep \to 0} \frac{\varphi_{L}(t , x ; p + \ep \lambda) - \varphi_{L}(t , x ; p)}{\ep}.
\end{equation*}
Differentiating with respect to the Brownian motions is performed the following way: for any $x \in \mathbb{T}_L$ and any pair of times $s , t \in \R$ with $s < t$, we define the increment $X_{s , t}(y) := \tilde B_t(y) - \tilde B_s(y)$. We then introduce the piecewise affine function: for $T \in \R$,
\begin{equation*}
    g_{s,t} (T) := 
    \left\{ \begin{aligned}
        0 &~~~\mbox{if}~ T \leq s, \\
        \frac{T - s}{t - s} & ~~~\mbox{if}~ s \leq T \leq t, \\
        1 & ~~~\mbox{if}~ T \geq t,
    \end{aligned} \right.
\end{equation*}
and define the derivative of a real-valued random variable $Z$ depending on the Brownian motions with respect to the increment $X_{s , t}(x)$ as follows
\begin{multline} \label{eq:diffwithrespecttoBM}
    \frac{\partial Z}{\partial X_{s , t}(y)} (\{ \tilde B_t(x)  \, : \, t \in \R, \, x \in \mathbb{T}_L \}) \\
    := \lim_{\ep \to 0} \frac{Z (\{ \tilde B_t(x) + \ep \delta_y(x) g_{s,t}(t) \, : \, t \in \R, \, x \in \mathbb{T}_L \}) - Z (\{ \tilde B_t(x) \, : \, t \in \R, \, x \in \mathbb{T}_L \})}{\ep}.
\end{multline}

\begin{proposition}[Stationarity, ergodicity, reversibility and differentiability] \label{prop:propsLangevin}
For any sidelength $L \in \N$ and any slope $p \in \Rd$, the Langevin dynamic $\varphi_{L}$ satisfies the following properties:
\begin{itemize}
    \item Distribution: for any time $t \in \R$, the random map $\varphi_{L }(t , \cdot ; p) : \mathbb{T}_L \to \R$ is distributed according to the Gibbs measure $\mu_{L , p}$.
    \item Stationarity: the law of the Langevin dynamic is stationary with respect to space and time translations.
    \item Ergodicity: the law of the dynamic is ergodic with respect to time translations.
    \item Differentiability with respect to the slope: for any pair $(t , x) \in \R \times \mathbb{T}_L$ and any $p, \lambda \in \Rd$, the function $w_{L , p , \lambda}$ is the unique stationary solution (with finite second moments) to the parabolic equation
    \begin{equation} \label{eqstatpardifferentiated}
        \partial_t  w_{L , p , \lambda} - \nabla \cdot \a(\cdot ; p )( \lambda +  \nabla w_{L , p , \lambda} ) = 0   \hspace{5mm} \mbox{in} \hspace{3mm}  \R \times \mathbb{T}_L.
    \end{equation}
    \item Differentiability with respect to the Brownian motions: for any pair of times $(t , s) \in \R$ with $t < s$ and any vertex $x \in \mathbb{T}_L$, one has the identity (using~\eqref{eq:conventionheatnegativetime} if applicable)
    \begin{equation*}
            \frac{\partial \varphi_{L } (\cdot , \cdot ; p)}{\partial X_{t , s} (x)} = \int_t^s P_{\a(\cdot , p)} \left(\cdot , \cdot ; s' , x \right) \, ds'.
    \end{equation*}
\end{itemize}
\end{proposition}

In the finite setting (the underlying space being the torus $\mathbb{T}_L$), the proof of these properties follows fairly standard arguments. A detailed sketch of proof can be found in Appendix~\ref{App.B}.

\subsubsection{Helffer-Sj\"{o}strand representation}

The Helffer-Sj\"{o}strand representation formula is a powerful tool to study the $\nabla \varphi$-interface model which was originally introduced by Helffer and Sj\"{o}strand~\cite{HS}, and then used by Naddaf and Spencer~\cite{NS} and Giacomin, Olla and Spohn~\cite{GOS} in order to identify the scaling limit of the model. In this article, we will use this inequality to obtain much less refined information: it is used in Proposition~\ref{prop:sublincorr} of Section~\ref{sec:section5sublinearity} to obtain quantitative estimates on the typical size of the Langevin dynamic. In particular, we only state below the version of the result we need for the proof below, but emphasize that more general versions can be found in~\cite{HS, NS, GOS}.

\begin{proposition}[Helffer-Sj\"{o}strand representation formula~\cite{HS, NS, GOS}]
For any sidelength $L \in \N$ and any slope $p \in \Rd$, one has the identity
\begin{equation*}
        \mathrm{Var}_{L , p} \left[ \varphi (0) \right] = \E \left[ \int_0^\infty P_{\a(\cdot ; p)} (t , 0) \, dt \right].
\end{equation*}
\end{proposition}

\subsection{Log-concavity}

Under the assumption that the potential $V$ is convex, the Gibbs measure $\mu_{L,p}$ is a log-concave probability distribution. This class of measures have been extensively studied in the literature. We collect below two of their important properties (the preservation of log-concavity under marginalization and the Efron's monotonicity theorem) which are important inputs in Section~\ref{sec:section3moderated} (and specifically in the proof of Proposition~\ref{prop.prop2.3}).

\subsubsection{Log-concave measures}

We start with the definition of a log-concave probability measure.

\begin{definition}[Log-concave measure]
For $n \in \N$, a Borel probability measure $\mu$ in $\R^n$ is called log-concave if for any pair of compact convex sets $A , B \subseteq \R^n$ and any $\lambda \in (0 , 1)$, one has the inequality
\begin{equation*}
    \mu(\lambda A + (1 - \lambda)B) \geq \mu(A)^\lambda \mu(B)^{1-\lambda}
\end{equation*}
where $\lambda A + (1 - \lambda)B := \left\{ \lambda x + (1-\lambda)y \, : \, x \in A, y \in B \right\}$.
\end{definition}

Any probability measure which is absolutely continuous with respect to the Lebesgue measure on $\R^n$ and whose density is log-concave (i.e., its logarithm is a concave function) is a log-concave probability measure. In particular, for any $L \in \N$ and any slope $p \in \Rd$, the measure $\mu_{L , p}$ is a log-concave probability measure on the space $\Omega_L^\circ$.

\subsubsection{Preservation of log-concavity under marginalization}

A first (fundamental) property of log-concave measures is that any marginal of a log-concave probability distribution is log-concave (equivalently, log-concavity is preserved under marginalization). This result is a consequence of the Pr\'ekopa-Leindler inequality~\cite{prekopa1971logarithmic, prekopa1973logarithmic, leindler1972certain}.

\begin{proposition}[Pr\'ekopa-Leindler~\cite{prekopa1971logarithmic, prekopa1973logarithmic, leindler1972certain}] \label{prop:prekopa-leindler}
Any marginal distribution of a log-concave distribution is also log concave.
\end{proposition}

As a direct consequence of Proposition~\ref{prop:prekopa-leindler}, we obtain the following result.

\begin{corollary}
For any sidelegnth $L \in \N$ and any a slope $p \in \Rd$, if we let $\varphi : \mathbb{T}_L \to \R$ be a random variable distributed according to the Gibbs measure $\mu_{L , p}$, then, for any $x \in \mathbb{T}_L$ and any index $i \in \{1 , \ldots, d\}$, the real-valued random variables $\varphi(x)$ and $\nabla_i \varphi(x) = \varphi(x + e_i)  - \varphi(x)$ have log-concave distributions.
\end{corollary}

Log-concavity is used in this article to upgrade stochastic integrability. Specifically, we will make use of the following property of log-concave distributions: if $X$ is a real-valued random variable whose distribution is log-concave then (for some explicitly computable constant $C < \infty$)
\begin{equation} \label{eq:logconcavityimpliesmoments}
    \E \left[ \left| X \right| \right] \leq 1 \implies \forall K \geq 1,~ \mathbb{P} \left[ X \geq CK  \right] \leq \exp \left( - K \right).
\end{equation}
This property follows from the observation that any concave function decaying to minus infinity must decay at least linearly fast.

\subsubsection{Efron's monotonicity theorem}

The second result we need pertaining to log-concave measures is the Efron's monotonicity theorem for pairs of independent log-concave random variables. This result is due to Efron~\cite{Efron1965}.

\begin{theorem}[Efron's monotonicity theorem~\cite{Efron1965}] \label{th:Efronmonotonicity}
    Let $(X , Y)$ be a pair of independent, real-valued  and log-concave random variables and let $\Psi : \R^2 \to \R$ be a function which is nondecreasing in each of its arguments, then the conditional expectation
    \begin{equation*}
            \E \left[ \Psi(X , Y) \, | \, X + Y = s \right] ~\mbox{is nondecrasing in}~ s.
    \end{equation*}
\end{theorem}

\subsection{Maximal inequalities}

In this section, we recall some classical properties of maximal functions. We let $(\Omega , \mathcal{F} , \mathbb{P})$ be a probability space, and let $(\tau_t)_{t \in \R}$ be a measure preserving action of $\Z$ on this space. For every measurable function $f : \Omega \to \R$, we define the maximal function
\begin{equation*} 
    M (f) := \sup_{T \geq  1} \frac{1}{T} \int_0^T f(\tau_t \omega).
\end{equation*}
We next record the $L^q$ maximal inequality, which can be obtained as a consequence of the weak type~$(1,1)$ estimate~\cite[Theorem 3.2]{AK81} with the Marcinkiewicz interpolation theorem (see~\cite[Appendix D]{taylor2006measure}). The result is stated and used in~\cite[Appendix A]{MO16}.

\begin{proposition}[$L^p$ Maximal inequality] \label{propmaximalineq}
For any $q \in ( 1 , \infty]$, there exists a constant $C := C(q,d) < \infty$ such that, for any $f \in L^q(\Omega)$,
\begin{equation*}
    \left\| M (f) \right\|_{L^q(\Omega)} \leq C  \left\| f \right\|_{L^q(\Omega)}.
\end{equation*} 
\end{proposition}

\begin{remark}
We will use this result when $\Omega := C(\R) \times \mathbb{T}_L$ is the space of trajectories of the Langevin dynamic, $\mathcal{F}$ is the $\sigma$-algebra generated by the projections, $\mathbb{P}$ is the law of the Langevin dynamic with slope $p \in \Rd$ and $\tau_s$ is the time shift $\tau_s \varphi (t , x ; p) = \varphi (t + s , x ; p)$ (the fact the the operator $\tau_s$ preserves the measure $\mathbb{P}$ is a consequence of the stationarity property stated in Proposition~\ref{prop:propsLangevin}).
\end{remark}

\subsection{Stochastic integrability}

\begin{definition}
    Let $X$ be a random variable. For any exponent $s > 0$ and any constant $K \in (0 , \infty)$, we write
    \begin{equation*}
        X \leq \mathcal{O}_s(K) ~~ \iff  ~~ \mathbb{P} \left[ |X| \geq t K \right] \leq \exp \left( - t^s \right) ~~ \forall t \in [1 , \infty)
    \end{equation*}
    and, for any constant $c > 0$, we write
    \begin{equation*}
        X \leq \mathcal{O}_{\Psi,c}(K) ~~ \iff  ~~  \mathbb{P} \left[ |X| \geq t K \right] \leq \exp \left( - c \left| \ln t \right|^{\frac{r}{r-2}} \right) ~~ \forall t \in [1 , \infty).
    \end{equation*}
\end{definition}

We collect below some useful properties of this notation. The proofs of these results for the $\mathcal{O}_s$ notation can be found in~\cite[Appendix A]{AKMbook} (and the proofs can be extended to the $\mathcal{O}_{\Psi,c}$ notation).

\begin{proposition}[Properties of the $\mathcal{O}_s$ and $\mathcal{O}_{\Psi,c}$ notation] \label{prop:prop2.20}
For any $s > 0$ and any $c > 0$, the notation $\mathcal{O}_s$ and $\mathcal{O}_{\Psi,c}$ satisfy the following properties:
\begin{itemize}
    \item Comparison: there exists $C := C(s , c)$ such that: $X \leq \mathcal{O}_s(K) \implies X \leq \mathcal{O}_{\Psi,c}(C K)$.
    \item Summation: there exists a constant $C := C(s)$ (resp. $C := C(c)$) such that, for any collection $X_1, \ldots, X_N$ and any $K_1 , \ldots, K_N$ satisfying $X_i \leq \mathcal{O}_s(K_i)$ (resp. $X_i \leq \mathcal{O}_{\Psi,c}(K_i)$), $X_1 + \ldots + X_n \leq \mathcal{O}_s(C K_1 + \ldots + C K_N)$ (resp. $X_1 + \ldots + X_n \leq \mathcal{O}_{\Psi , c}(C K_1 + \ldots + C K_N)$).
    \item Integration: Let $t \mapsto X(t)$ be a continuous random function and let $I \subseteq \R$ be a bounded interval of $\R$. Then there exists a constant $C := C(s)$ ($C := C(c)$) such that if $X(t) \leq \mathcal{O}_s(K)$ (resp. $X(t) \leq \mathcal{O}_{\Psi, c}(K)$) for any $t \in I$, then $\int_I X(t) \, dt \leq \mathcal{O}_s(C |I| K)$ (resp. $\int_I X(t) \, dt \leq \mathcal{O}_{\Psi, c}(C |I| K)$).
    \item Product: there exists $C := C(s)$ (resp. $C := C(c)$) such that for any pair $X_1 , X_2$ and $K_1 , K_2$ such that $X_i \leq \mathcal{O}_s(K_i)$ (resp. $X_i \leq \mathcal{O}_{\Psi , c}(K_i)$), then $X_1 X_2 \leq \mathcal{O}_{s/2}(C K_1 K_2)$ (resp. $X_1 X_2 \leq \mathcal{O}_{\Psi , c/2^{r/(r-2)}}(C K_1 K_2)$).
    \item Powers: For any $\alpha > 0$, any random variable $X$ and any constant $K_1$ such that $X \leq \mathcal{O}_s(K_1)$ (resp. $X \leq \mathcal{O}_{\Psi , c}(K_1)$), one has $|X|^\alpha \leq \mathcal{O}_{s/\alpha}(K_1^\alpha)$ (resp. $|X|^\alpha \leq \mathcal{O}_{\Psi , c/\alpha^{r/(r-2)}}(K_1^\alpha)$).
    \item Maximum: There exists a constant $C := C(s)$ (resp. $C := C(c)$) such that for any collection $X_1, \ldots, X_N$ and any $K \geq 1$ satisfying $X_i \leq \mathcal{O}_s(K)$ (resp. $X_i \leq \mathcal{O}_{\Psi , c}(K)$), $\max_{i = 1 , \ldots, N} X_i \leq \mathcal{O}_s(C (\ln N)^{1/s} K)$ (resp. $\max_{i = 1 , \ldots, N} X_i \leq \mathcal{O}_{\Psi , c}(e^{C (\ln N)^{(r-2)/r}} K)$).
    \item Concentration: For any $s \in (1,2)$, there exists a constant $C := C(s)$ such that for any collection $X_1, \ldots, X_N$ of independent random variables satisfying $X_i \leq \mathcal{O}_s(K)$ (for some $K \geq 1$) and $\E[X_i] = 0$, $\sum_{i = 1}^{N} X_i \leq \mathcal{O}_{s}(C \sqrt{N} K)$).
\end{itemize}
\end{proposition}

\section{The moderated environment} \label{sec:section3moderated}

In this section, we formalize the argument presented in Section~\ref{subsection1.2.3}. Specifically, we establish the following results:
\begin{itemize}
    \item We first obtain a (presumably sharp) stochastic integrability estimate for the gradient of an interface sampled according to the Gibbs measure $\mu_{L , p}$ (see Proposition~\ref{prop.prop2.3});
    \item We then show a fluctuation estimate for the Langevin dynamics, asserting that the probability that they remain in a bounded set for a long time is small (see Proposition~\ref{prop3.4});
    \item We then define the moderated environment associated with the Langevin dynamics and show that it satisfies good stochastic integrability estimates (see Proposition~\ref{prop:stochintmoderated}).
\end{itemize}

\subsection{Stochastic integrability for the gradient of the interface}

In this section, we establish a stochastic integrability estimate for the gradient of an interface sampled according to the Gibbs measure $\mu_{L , p}$. As in~\cite[Proposition 3.1]{D23U}, the proof is based on the Efron's monotonicity theorem for log-concave measure and a coupling argument, originally due to Funaki and Spohn~\cite{FS}, for the Langevin dynamics. The main important feature of Proposition~\ref{prop.prop2.3} is that the decay on the right-hand side of~\eqref{eq:stochintwithslope} is super-Gaussian (as $r > 2$).

\begin{proposition} \label{prop.prop2.3}
There exist two constants $c := c(d , V) > 0$ and $C := C(d , V) < \infty$ such that, for any $L\in \N$, any $p \in \Rd$, any $x \in \mathbb{T}_L$ and any $K \in (0 , \infty)$,
\begin{equation} \label{eq:stochintwithslope}
    \mu_{L , p} \left[ \left| \nabla \varphi(x) \right| \geq K \right] \leq C \exp \left( - c  K^r \right).
\end{equation}
\end{proposition}

\begin{remark}
    Since the model is defined on the torus, and thus translation invariant, the law of the random variable $\left| \nabla \varphi(x) \right| $ does not depend on the vertex $x \in \mathbb{T}_L$.
\end{remark}

\begin{remark}
   Compared to~\cite[Proposition 3.1]{D23U}, the main feature of the previous proposition is that the estimate is uniform over the slope $p \in \Rd$. In fact the result is suboptimal in this aspect as the random variable $\left| \nabla \varphi(x) \right| $ should concentrate around $0$ as the norm of the slope increases (this can be seen in inequality~\eqref{eq:26120827}), and the inequality should thus improve as the norm of the slope gets larger. While we believe that the proof could be optimised to capture this phenomenon, we did not try to do so to minimize the technicality of the argument.
\end{remark}

\begin{proof}
Let us select a sidelength $L \in \N$, a slope $p = (p_1 , \ldots, p_d) \in \Rd$ and a vertex $x \in \mathbb{T}_L$. Without loss of generality, we may assume that $|p_1| = \max_{1 \leq i \leq d} |p_i|$. We note that this assumption implies that $|p_1| \geq |p| / \sqrt{d}$. We will prove the following inequality: there exist $c := c(d , V) > 0$ and $C := C(d , V) < \infty$ such that, for any $K \in (0 , \infty)$,
\begin{equation} \label{eq:10241601}
    \mu_{L , p} \left[ \left| \nabla_1 \varphi(x) \right| \geq K \right] \leq 
    C \exp \left( - c  K^r \right).
\end{equation}
The inequality~\eqref{eq:stochintwithslope} can be deduced from~\eqref{eq:10241601} by using the upper bound $ \left| \nabla_1 \varphi(x) \right| \leq \left| \nabla \varphi(x) \right|$.

The rest of the argument is devoted to the proof of~\eqref{eq:10241601}. We first recall that, by the translation invariance of the Gibbs measure $\mu_{L , p}$, for any index $i \in \{ 1 , \ldots, d \}$,
\begin{equation} \label{eq:expegradphiis0}
    \E_{L , p} \left[ \nabla_i \varphi(x)  \right] = \E_{L , p} \left[ \varphi(x+ e_i) -   \varphi(x+ e_i) \right] = \E_{L , p} \left[ \varphi(x+ e_i) \right] -  \E_{L , p} \left[ \varphi(x) \right]  =0.
\end{equation}
We split the argument into three steps.

\medskip

\textit{Step 1. Bound on the random variable $|\nabla \varphi(x)|$.}

\medskip

\textit{Substep 1.1. Bound on the $L^2$-norm.} We first prove the following upper bound on the $L^2$-norm of the random variable $\left| \nabla \varphi(x) \right|$: there exists a constant $C := C(d , V) < \infty$ such that 
\begin{equation} \label{eq:26120827}
    \E_{L , p} \left[ \left| \nabla \varphi(x) \right|^2  \right] \leq \frac{C}{|p|_+^{r-2}}.
\end{equation}
The proof of~\eqref{eq:26120827} is based on the following identity: for any vertex $y \in \mathbb{T}_L$,
\begin{equation} \label{eq:08312612}
    \E \left[ \varphi(y) \nabla \cdot D_p V(p + \nabla  \varphi)(y) \right] = - \frac{\left| \mathbb{T}_L \right| -1 }{\left| \mathbb{T}_L \right|}.
\end{equation}
The identity~\eqref{eq:08312612} is a consequence of the following general identity (which follows from an integration by parts): for any continuously differentiable probability density $f : \R^n \to [0 , \infty)$ such that $|z|f(z)$ tends to $0$ at infinity and $z \to (1 + |z|) \nabla f(z)$ is integrable, and for any index $i \in \{ 1 , \ldots , n \}$,
\begin{equation} \label{identity:f}
    \int_{\R^n} z_i \frac{d f}{d z_i} (z) \, dz = -1.
\end{equation}
We will apply this result in the following setting:
\begin{itemize}
    \item[(i)] \textit{Underlying space:} we consider the vector space $\Omega^\circ_{L}$. Its dimension is $|\mathbb{T}_L| -1 $.
    \item[(ii)] \textit{Probability density:} we consider the density $$f( \varphi ) := \frac{1}{Z_{L , p}} \exp \left( - \sum_{y \in \mathbb{T_L}} V(p + \nabla \varphi (y)) \right).$$
    \item[(iii)] \textit{Coordinates:} we denote by $\varphi_x := \delta_x - \frac{1}{\left| \mathbb{T}_L \right|} \in \Omega^\circ_{L}$ and observe that the following identity holds
\begin{equation*}
    \frac{d f}{d \varphi_x}(\varphi) = \frac{1}{Z_{L , p}} \nabla \cdot D_pV(p + \nabla  \varphi)(x) \exp \left( - \sum_{y \in \mathbb{T_L}} V(p + \nabla \varphi (y)) \right)
\end{equation*}
    as well as the identities
\begin{equation*}
    \varphi(x) =  \sum_{y \in \mathbb{T}_L} \varphi(y) \varphi_x(y)    ~~ \mbox{and}~~ \left\| \varphi_x \right\|_{L^2(\mathbb{T}_L)} = \left(\frac{\left| \mathbb{T}_L \right|-1}{\left| \mathbb{T}_L \right|}\right)^{\frac 12}.
\end{equation*}
\end{itemize}
In the setting above, the identity~\eqref{identity:f} becomes
\begin{equation*}
 \mathbb{E}_{L,p} \left[ \varphi(x) \nabla \cdot D_p V(p + \nabla  \varphi)(x)\right] = -\frac{\left| \mathbb{T}_L \right| -1}{\left| \mathbb{T}_L \right|}.
\end{equation*}
Summing the inequality~\eqref{eq:08312612} over the vertices $x \in \mathbb{T}_L$ and performing a discrete integration by parts (see~\eqref{eq:discreteIPP}), we deduce that
\begin{equation*}
    \E_{L,p} \left[  \sum_{y \in\mathbb{T}_L} D_pV(p + \nabla \varphi(y)) \cdot \nabla \varphi(y) \right] = \left| \mathbb{T}_L \right| -1.
\end{equation*}
We next record the following identity (since $D_p V(p)$ is deterministic, it is a direct consequence of~\eqref{eq:expegradphiis0})
\begin{equation*} 
        \E_{L,p} \left[ D_pV(p ) \cdot \nabla \varphi(x) \right] =0.
\end{equation*}
Combining the two previous identities, we obtain that
\begin{equation*}
    \E_{L,p} \left[  \sum_{y \in\mathbb{T}_L} \left( D_pV(p + \nabla \varphi(y)) - D_pV(p) \right) \cdot \nabla \varphi(y) \right] = \left| \mathbb{T}_L \right| -1.
\end{equation*}
Using the Assumption~\eqref{AssPot} on the potential $V$, we see that, for any realization $\varphi \in \Omega^\circ_{L}$, and any vertex $y \in \mathbb{T}_L$,
\begin{align*}
    \left( D_pV(p + \nabla \varphi(y)) - D_pV(p) \right) \cdot \nabla \varphi(y) & = \left( \left( \int_{0}^1  D_p^2 V(p + s \nabla \varphi(y)) \, ds \right) \nabla \varphi(y)  \right) \cdot \nabla \varphi(y) \\
    & \geq c |p|^{r-2}_+  \left| \nabla \varphi(y) \right|^2 - C.
\end{align*}
A combination of the two previous displays yields the inequality
\begin{equation*}
    \E_{L,p} \left[  \sum_{y \in \mathbb{T}_L} \left| \nabla \varphi(y) \right|^2  \right] \leq \frac{C }{|p|_+^{r-2}} \left| \mathbb{T}_L \right|.
\end{equation*}
Using that the translation invariance of the measure $\mu_{L,p}$, we deduce that
\begin{equation*}
    \E_{L,p} \left[ \left| \nabla \varphi(x) \right|^2 \right] = \E_{L,p} \left[ \frac{1}{\left| \mathbb{T}_L \right|} \sum_{y \in \mathbb{T}_L} \left| \nabla \varphi(y) \right|^2 \right] \leq \frac{C}{|p|_+^{r-2}}.
\end{equation*}

\medskip

\textit{Substep 1.2. Upgrading to exponential moments using log-concavity.} We note that, since the Gibbs measure $\mu_{L, p}$ is log-concave, we can apply the Pr\'ekopa-Leindler inequality~\cite{prekopa1971logarithmic, prekopa1973logarithmic, leindler1972certain} to deduce that the distributions of the random variables $\nabla_1 \varphi
(x), \ldots, \nabla_d \varphi
(x)$ are also log-concave. This implies that their tails decay at least exponentially fast on the scale of their standard deviation. In particular, we have the inequality: for any $K \geq 1$,
\begin{equation*} 
   \mu_{L,p} \left[ \left| \nabla \varphi(t , x) \right| \geq K \right] \leq \exp \left( - c |p|^{(r-2)/2} K \right).
\end{equation*}
In particular, for any exponent $\alpha \in [1 , \infty)$,
\begin{equation} \label{eq:gradphilogconcave}
     \E_{L , p} \left[ \left| \nabla \varphi(x) \right|^{2\alpha} \right] \leq \frac{C_\alpha}{|p|_+^{\alpha(r-2)}}.
\end{equation}
Our goal is then to upgrade the decay on the right-hand side of~\eqref{eq:gradphilogconcave} from exponential to the super exponential rate~\eqref{eq:stochintwithslope} (while losing the factor involving the norm of the slope $p$).

\medskip

\textit{Step 2. Perturbing the potential $V$.} We let $\mathcal{V} : \R \to \R$ be a twice continuously differentiable convex function satisfying the following properties:
\begin{itemize}
\item[(i)] \textit{Lower bound on the growth $\mathcal{V}$:} we assume that $\mathcal{V}(p_1) = 0$ and that there exist two constants $C := C(d , V) < \infty$ and $c := c(d , V) > 0$ such that $\mathcal{V}(z_1) \geq c |z_1 - p_1|^{r} - C$.
\item[(ii)] \textit{Upper bound on the growth of $\mathcal{V}':$} we assume that $\mathcal{V}'(p_1) =0$ and that $\left| \mathcal{V}'(z_1) \right| \leq |z_1 - p_1|^{r-1}$,
\item[(iii)] \textit{Upper bound on the growth of $\mathcal{V}'':$} we assume that the function $z= (z_1 , \ldots, z_n) \mapsto V(z) - \mathcal{V}(z_1)$ is convex.
\end{itemize}
We note that the function $\mathcal{V}$ is allowed to depend on the value of the slope $p_1$, but the constants should only depend on the dimension $d$ and the potential $V$. The existence of the function $\mathcal{V}$ is guaranteed by the Assumption~\eqref{AssPot} on the potential $V$.

We then introduce the collection of convex potentials~$(V_{y})_{y \in \mathbb{T}_L}$ defined as follows: for any vertex $y \in \mathbb{T}_L$ and any $z = (z_1 , \ldots, z_n) \in \R^n$,
\begin{equation*}
    V_{y} (z) := \left\{ \begin{aligned}
    V(z) &~\mbox{if}~ y \neq x, \\
    V(z) - \mathcal{V}(z_1)  &~\mbox{if}~ y = x, \\
    \end{aligned} \right.
\end{equation*}
and let $\varphi^x : \mathbb{T}_L \to \R$ be a random interface distributed according to the Gibbs measure
\begin{equation} \label{eq:09310312}
    \mu^x_{L,p}(d \varphi) := \frac{1}{Z_{L,p}^x} \exp \left( - \sum_{y \in \mathbb{T}_L } V_{y} \left( p+ \nabla \varphi(y) \right) \right)  d \varphi.
\end{equation}
Since the potentials~$(V_{y})_{y \in \mathbb{T}_L}$ are all convex, the measure~\eqref{eq:09310312} is log-concave, and thus the random variables $\nabla_1 \varphi^x(x), \ldots, \nabla_d \varphi^x(x)$ are also log-concave. 

We next prove the following estimate: there exists a constant $C := C(d , V) < \infty$ such that
\begin{equation} \label{eq:10110312}
    \E \left[\left| \nabla \varphi^x(x) \right|^2  \right] \leq \frac{C}{|p|_+^{r-2}}.
\end{equation}
The proof of~\eqref{eq:10110312} is based on a coupling argument for Langevin dynamics. We consider the stationary Langevin dynamic associated with the measure $\mu^x_{L,p}$, i.e., the stationary solution of the stochastic differential equation (the existence of this dynamic can be proved using the same arguments as the ones presented in Appendix~\ref{App.B})
    \begin{equation} \label{eq:21252711}
    d \varphi_{L}^x (t , y ; p) = \nabla \cdot D_p V_{y}(p+ \nabla \varphi_{L}^x) (t , y ; p) + \sqrt{2} d \tilde B_t(y) ~~~\mbox{for}~~~ (t , y) \in \R \times \mathbb{T}_L.
\end{equation}
We next couple the two dynamics~\eqref{eq:21252711} and~\eqref{eq:def.Langevintorus} by assuming that they are driven by the same Brownian motions. Subtracting the two dynamics, we observe that the difference $u := \varphi_{L}(\cdot , \cdot ; p) - \varphi^x_{L}(\cdot , \cdot ; p)$ solves the parabolic equation
\begin{equation} \label{eq:12390312}
    \partial_t u - \nabla \cdot \a \nabla u  = \nabla \cdot \left[ \left(D_p V_{y} - D_p V \right)(p+\nabla \varphi_{L}) (\cdot , \cdot ; p) \right] \hspace{5mm} \mbox{in} ~[0 , \infty] \times \mathbb{T}_L,
\end{equation}
with the definition
\begin{align*}
    \a(t , y) & :=
    \int_0^1 D_p^2 V_{y}(p+ s \nabla \varphi_L(t , y; p ) + (1-s) \nabla \varphi_L^x(t , y;p)) \, ds \\
    & = \int_0^1 D_p^2 V_{y}(p+ \nabla \varphi_L(t , y; p ) - (1-s) \nabla u(t , y)) \, ds  \in \mathcal{S}^+(\Rd).
\end{align*}
Noting that the potentials $V_{y}$ and $V$ are only different at the vertex $x$, and that their difference is given by the function $z = (z_1 , \ldots, z_d) \mapsto \mathcal{V}(z_1)$, we may use an energy estimate on the equation~\eqref{eq:12390312} (i.e., multiply both besides of~\eqref{eq:12390312} by $u$, sum over the vertices $x \in \mathbb{T}_L$, integrate over the times $t \in [0 , T]$ and perform a discrete integration by parts) and obtain, for any $T \geq 0$,
\begin{multline} \label{eq:13110312}
    \sum_{x \in \mathbb{T}_L}  \left| u(T , x) \right|^2 + \int_{0}^T \sum_{y \in  \mathbb{T}_L } \nabla u(t ,y) \cdot \a(t , y) \nabla u(t ,y)  \, dt \\ \leq C \int_0^T \left| \mathcal{V}'\left(p_1 + \nabla_1 \varphi_L(t , x ; p) \right) \right| \left| \nabla_1 u(t,x) \right| \, dt + C \sum_{x \in \mathbb{T}_L}  \left| u(0 , x) \right|^2.
\end{multline}
The inequality~\eqref{eq:12390312} implies the following inequality (forgetting the first term on the left-hand side and the sum in the integral which both contribute positively to the left-hand side)
\begin{equation*}
    \int_{0}^T  \nabla u(t ,x) \cdot \a(t , x) \nabla u(t ,x)  \, dt \leq C \int_0^T \left| \mathcal{V}'\left(p_1 + \nabla_1 \varphi_L(t , x) \right) \right| \left| \nabla_1 u(t,x) \right| \, dt + C \sum_{x \in \mathbb{T}_L}  \left| u(0 , x) \right|^2.
\end{equation*}
Using the definition of the function $\Lambda_-$ introduced~\eqref{eq:defLambda-}, we have the inequality
\begin{equation*}
    \nabla u(t ,x) \cdot \a(t , x) \nabla u(t ,x)  \geq  \Lambda_-(p + \nabla \varphi_L(t , x ; p)) \left| \nabla u(t ,x) \right|^2.
\end{equation*}
Additionally, it follows from the definition of the environment $\a$ and the third property of Lemma~\ref{lemma.upperandlowerboundLambda} that there exists a constant $C := C(d , V) < \infty$ such that, if $\left| \nabla u(t,x) \right| \geq C$, then
\begin{equation*}
    \nabla u(t ,x) \cdot \a(t , x) \nabla u(t ,x) \geq \left| \nabla u(t ,x) \right|^2.
\end{equation*}
The two previous inequalities can be combined so as to obtain the following statement: there exists a constant $C := C(d, V) < \infty$ such that
\begin{equation} \label{eq:13100312}
    \nabla u(t ,x) \cdot \a(t , x) \nabla u(t ,x)  \geq  \frac{1}{2} (\Lambda_-(p + \nabla \varphi_L(t , x ; p)) +1 ) \left| \nabla u(t ,x) \right|^2  - C.
\end{equation}
We then substitute~\eqref{eq:13100312} into~\eqref{eq:13110312} and apply the Cauchy-Schwarz inequality
\begin{equation*}
    (\Lambda_-(p + \nabla \varphi_L(t , x ; p)) +1) \int_{0}^T \left| \nabla u(t ,x) \right|^2 \, dt  \leq C T   + C \int_0^T \left| \mathcal{V}'(p_1+ \nabla_1 \varphi_{L}(t,x;p))\right|^2 \, dt + C \sum_{x \in \mathbb{T}_L }  \left| u(0 , x) \right|^2.
\end{equation*}
Using the definition $u := \varphi_{L}(\cdot , \cdot ; p) - \varphi^x_{L} (\cdot , \cdot ; p)$, we thus obtain
\begin{multline*}
     (\Lambda_-(p + \nabla \varphi_L(t , x ; p)) +1)  \int_{0}^T \left| \nabla \varphi_{L}^x(t ,x; p) \right|^2 \, dt \leq  C T + C\sum_{x \in \mathbb{T}_L }  \left| u(0 , x) \right|^2 \\ 
    + \int_0^T \left( \mathcal{V}'( p_1+ \nabla_1 \varphi_{L}(t,x;p))^2  + (\Lambda_-(p + \nabla \varphi_L(t , x ; p)) +1) \left|\nabla \varphi_{L}(t,x;p)\right|^2\right) \, dt  .
\end{multline*}
Dividing both sides of the inequality by $(\Lambda_-(p + \nabla \varphi_L(t , x ; p)) +1)$, taking the expectation, and using the time stationarity of the gradients $\nabla \varphi_{L} (\cdot , \cdot ; p )$ and $\nabla \varphi_{L}^x (\cdot , \cdot ; p )$, we deduce that, for any $T > 0$,
\begin{align*}
    \E \left[\left| \nabla \varphi^x_L(0 ,x ; p) \right|^2  \right] & \leq C \E \left[ \frac{\mathcal{V}'( p_1 + \nabla_1 \varphi_L(0 , x;p) )^2 + 1}{ \Lambda_-(p + \nabla \varphi_L(0 , x ; p)) +1 }  + \left|\nabla_1 \varphi_L(0,x)\right|^2 \right] + \frac{C}{T} \sum_{x \in \mathbb{T}_L}  \E \left[ \left| u(0 , x) \right|^2 \right].
\end{align*}
Taking the limit $T \to \infty$, using the bound~\eqref{eq:26120827} and the Cauchy-Schwarz inequality, we obtain
\begin{align*}
    \E \left[\left| \nabla \varphi^x_L(0 ,x;p) \right|^2  \right] & \leq  \frac{C}{|p|_+^{r-2}} + C \E \left[ \frac{\mathcal{V}'( p_1 + \nabla_1 \varphi_L(0 , x;p) )^2 +1}{\Lambda_-(p + \nabla \varphi_L(0 , x ; p)) +1 } \right] \\
    & \leq \frac{C}{|p|_+^{r-2}} + C \E \left[ (\mathcal{V}'( p_1 + \nabla_1 \varphi_L(0 , x;p) )^2 + 1)^2\right]^{1/2} \E \left[ \frac{1}{(\Lambda_-(p + \nabla \varphi_L(0 , x ; p)) +1 )^2} \right]^{1/2}.
\end{align*}
We next estimate the first term on the right-hand side. Using the assumption (ii) on the function $\mathcal{V}$ and the inequality of~\eqref{eq:gradphilogconcave} (with $\alpha = r-1 > 1$), we deduce that
\begin{equation*}
    \E \left[ (\mathcal{V}'( p_1 + \nabla_1 \varphi_L(0 , x) )^2 + 1)^2 \right] \leq C \E \left[ \left|\nabla_1 \varphi_L(0 , x) ) \right|^{4(r-1)} +1 \right]  \leq  \frac{C}{|p|^{2(r-1)(r-2)}_+} + C \leq C.
\end{equation*}
For the second term, the result of Lemma~\ref{lemma.upperandlowerboundLambda} implies that there exists a constant $c := c(d , V) > 0$ such that
\begin{equation*}
    \Lambda_-(p + \nabla \varphi_L(t , x ; p)) + 1 \geq 
        \left\{ \begin{aligned}
        c |p|_+^{r-2}  & ~~\mbox{if}~ |\nabla \varphi_L(t , x ; p)| \leq \frac{|p|_+}{2}, \\
        1 & ~~\mbox{otherwise.} 
        \end{aligned} \right.
\end{equation*}
Thus
\begin{equation*}
    \E \left[ \frac{1}{(\Lambda_-(p + \nabla \varphi_L(t , x ; p)) +1 )^2} \right] \leq \frac{C}{|p|^{2(r-2)}_+} + \E \left[ \indc_{\{  \nabla \varphi_L(t , x ; p)) \geq |p|_+/2 \}} \right].
\end{equation*}
Using the inequality~\eqref{eq:gradphilogconcave} (with $K = |p|_+/2$), we obtain
\begin{equation*}
    \E \left[ \frac{1}{(\Lambda_-(p + \nabla \varphi_L(t , x ; p)) +1 )^2} \right]^{1/2} \leq \frac{C}{|p|^{r-2}_+} + e^{- c |p|_+^{r/2}} \leq \frac{C}{|p|^{r-2}_+}.
\end{equation*}
Combining the previous inequalities completes the proof of~\eqref{eq:10110312}.

\medskip

\textit{Step 3. Applying Efron's monotonicity theorem.} In the next step of the proof, we let $Y$ be a real-valued random variable whose law is given by
\begin{equation*}
    \mu_Y := \frac{1}{Z_Y} \exp \left( - \mathcal{V} \left( p_1 +  y \right) \right) dy \hspace{5mm} \mbox{with} \hspace{5mm} Z_Y := \int_\R \exp \left( - \mathcal{V} \left( p_1 +  y \right) \right) dy.
\end{equation*}
We couple the random variables $Y$ and $\varphi^x$ by assuming that they are independent. Using the assumption on the function $\mathcal{V}$, the independence of $Y$ and $\nabla_1 \varphi^x(x)$ and the bound~\eqref{eq:10110312}, we deduce that there exists a constant $c := c(d , V) > 0$ such that
\begin{align} \label{eq:15210312}
    \mathbb{P} \left[ Y \geq \nabla_1 \varphi^x(x) \right] & \geq \mathbb{P} \left[ \left\{ Y \geq 2 \E \left[\left| \nabla_1 \varphi^x(x) \right|  \right]  \right\} \cap \left\{ \nabla_1 \varphi^x(x) \leq 2 \E \left[\left| \nabla_1 \varphi^x(x) \right|  \right]   \right\} \right] \\
    & = \mathbb{P} \left[ \left\{ Y \geq 2 \E \left[\left| \nabla_1 \varphi^x(x) \right|  \right]  \right\} \right)  \mathbb{P} \left(\left\{ \nabla_1 \varphi^x(x) \leq 2 \E \left[\left| \nabla_1 \varphi^x(x) \right|  \right]   \right\} \right] \notag \\
    & \geq c. \notag
\end{align}
We next rely on the observation that the law of random variable $\nabla_1 \varphi(x)$ (where $\varphi$ is distributed according to the measure~$\mu_{L,p}$) is equal to the law of the random variable $Y$ conditionally on the event $\left\{ Y - \nabla_1 \varphi^x(x) = 0 \right\}$. This property is a consequence of the following observation: if $X$ and $Z$ are two independent real-valued random variables with bounded continuous densities $f$ and $g$ then the law of $X$ conditionally on the event $\{ X - Z = 0\}$ has a density proportional to the function $fg$. In particular, for any non-negative function $F : \R \to [0 , \infty)$, one has the identity
\begin{equation} \label{eq:1519031222}
    \E_{L,p} \left[ F(\nabla_1 \varphi(x)) \right] = \E \left[ F(Y) \,  | \,  Y - \nabla_1 \varphi^x(x) = 0 \right].
\end{equation}
Recalling that $c > 0$ the constant in Assumption (i) on the growth of the function $\mathcal{V}$, we introduce the function
\begin{equation*}
    F(z) := \left\{ \begin{aligned}
    0 &~\mbox{if}~ z \leq  0, \\
    \exp \left( \frac{c z^r}{2} \right) &~\mbox{if}~ z \geq 0.
    \end{aligned} \right.
\end{equation*}
Let us note that the function $F$ is nonnegative and increasing.
Assumption (i) on the growth of the function $\mathcal{V}$ implies that there exists a constant $C := C(d, V) < \infty$ such that
\begin{equation} \label{eq:15190312}
    \E \left[ F(Y) \right] = \frac{1}{Z_Y} \int_\R F(z) \exp \left( - \mathcal{V}(p_1 + z)\right) \, dz \leq C.
\end{equation}
We then note that the Efron's monotonicity theorem applied to the pair of independent random variables $(Y , \nabla_1 \varphi^x(x))$, the nonnegativity and monotonicity of the function $F$ imply the almost sure inequality
\begin{equation} \label{eq:15400312}
    \E \left[ F(Y) \,  | \,  Y - \nabla_1 \varphi^x(x) = 0 \right] \indc_{\left\{ Y - \nabla_1 \varphi^x(x) \geq 0 \right\}} \leq \E \left[ F(Y) \,  | \,  Y - \nabla_1 \varphi^x(x) \right].
\end{equation}
Combining the bound~\eqref{eq:15400312} with the lower bound~\eqref{eq:15210312}, the identity~\eqref{eq:1519031222} and the inequality~\eqref{eq:15190312} yields the existence of a constant $C := C(d , V) < \infty$ such that
\begin{align} \label{eq:15420312}
    \E_{L,p} \left[ F(\nabla_1 \varphi(x) ) \right] & = \E \left[ F(Y) \,  | \,  Y - \nabla_1 \varphi^x(x) = 0 \right] \\
    & \leq \frac{1}{\mathbb{P} \left( Y -\nabla_1 \varphi^x(x) \geq 0 \right)} \E \left[ \E \left[ F(Y) \,  | \,  Y - \nabla_1 \varphi^x(x) \right] \right] \notag \\
    & \leq \frac{1}{\mathbb{P} \left( Y - \nabla_1 \varphi^x(x) \geq 0 \right)} \E \left[ F(Y)\right] \notag \\
    & \leq C. \notag
\end{align}
The inequality~\eqref{eq:15420312} implies that there exist two constants $C := C(d , V) < \infty$ and $c := c(d , V) > 0$ such that, for any $K \geq 1$,
\begin{equation} \label{eq:15490312}
    \mu_{L , p} \left[  \nabla_1 \varphi (x) > K \right] \leq C \exp \left( - c K^{r} \right).
\end{equation}
The same argument applied with the potential $\tilde V(z) :=  V(- z)$ yields the upper bound, for any $K \geq 1$,
\begin{equation} \label{eq:15500312}
    \mu_{L , p}  \left[  \nabla_1 \varphi^x (x) < - K \right] \leq C \exp \left( - c K^{r} \right).
\end{equation}
Combining~\eqref{eq:15490312} and~\eqref{eq:15500312} completes the proof of Proposition~\ref{prop.prop2.3}.
\end{proof}

\subsubsection{Stochastic integrability for the Langevin dynamic} \label{fluctuationofgradphi}

In this section, we extend the result of the previous section to the stationary Langevin dynamic using (essentially) a union bound.

\begin{proposition} \label{propositionsubdynamic}
There exist two constants $c := c(d , V) > 0$ and $C := C(d , V) < \infty$ such that, for any $p \in \Rd$, any $T \geq 1$ and any $K \geq 1$,
\begin{equation} \label{eq:13510512}
    \mathbb{P}\left[ \sup_{t \in [0 , T]}  \left| \nabla \varphi_{L} (t , x ; p) \right| \geq K \right] \leq C |p|^{r-1}_+ T \exp \left( - c K^r \right).
\end{equation}
\end{proposition}

\begin{remark}
Once again, the result is not optimal as the right-hand side of~\eqref{eq:13510512} should improve (instead of deteriorate) as $|p| \to \infty$.
\end{remark}

\begin{proof}
Fix $K \geq 1$ and let $N :=  |p|_+^{r-1} K^{r}$. We have the inclusion of events
\begin{multline} \label{eq:11580512}
    \left\{ \sup_{t \in [0 , T]}  \left| \nabla \varphi_{L} (t , x ; p)  \right| \geq K  \right\} \subseteq \left\{ \sup_{n \in \left\{ 0 , \ldots, \lfloor T N \rfloor \right\}} \left| \nabla \varphi_{L} \left(\frac{n}{N} , x ; p \right) \right| \geq \frac{K}{2} \right\} \\ 
    \bigcup \left\{ \sup_{n \in \left\{ 0 , \ldots, \lfloor T N \rfloor \right\}} \sup_{t \in \left[ \frac{n}{N} , \frac{n+1}{N} \right]} \left| \nabla \varphi_{L} \left(t , x ; p \right) - \nabla \varphi_{L} \left(\frac{n}{N} , x ; p \right) \right| \geq \frac{K}{2} \right\}.
\end{multline}
We then bound the probabilities of the two terms on the right-hand side separately. For the first one, we use a union bound together with the result of Proposition~\ref{prop.prop2.3} and the identity $N :=   |p|_+^{r-1} K^{r} $ to obtain
\begin{align} \label{eq:13540512}
    \mathbb{P}\left[ \sup_{n \in \left\{ 0 , \ldots, \lfloor T N \rfloor \right\}} \left| \nabla \varphi_{L} \left(\frac{n}{N} , x ;p \right) \right| \geq \frac{K}{2} \right] & \leq  \sum_{n = 0}^{\lfloor T N \rfloor} \mathbb{P}\left[  \left| \nabla \varphi_{L} \left(\frac{n}{N} , x ; p \right) \right| \geq \frac{K}{2} \right] \\
    & \leq C  |p|_+^{r-1} K^r T \exp \left( - c K^r \right) \notag \\
    & \leq  C  |p|_+^{r-1} T \exp \left( - c K^r \right), \notag
\end{align}
where we reduced the value of the constant $c$ in the third line to absorb the polynomial factor $K^r$. For the second term on the right-hand side of~\eqref{eq:11580512}, we first fix an integer $n \in \left\{ 0 , \ldots, \lfloor T N \rfloor \right\}$ and use the definition of the Langevin dynamic~\eqref{eq:reqldefSDE} to write
\begin{equation*}
     \nabla \varphi_{L} \left(t , x ; p \right) - \nabla \varphi_{L} \left(\frac{n}{N} , x ; p \right) = \int_{\frac{n}{N}}^t \nabla \left( \nabla \cdot D_p V(p + \nabla \varphi_{L} (\cdot , \cdot ; p)) \right) (s , x) \, ds + \nabla B_{t}(x) - \nabla B_{\frac{n}{N}}(x).
\end{equation*}
This implies
\begin{multline} \label{eq:13340512}
    \sup_{t \in \left[ \frac{n}{N} , \frac{n+1}{N} \right]} \left| \nabla \varphi_{L} \left(t , x ;p \right) - \nabla \varphi_{L} \left(\frac{n}{N} , x ; p \right) \right| \\ \leq \int_{\frac{n}{N}}^{\frac{n+1}{N}} \left| \nabla \left( \nabla \cdot D_p V(p + \nabla \varphi_{L} (\cdot , \cdot ; p)) \right) (s , x) \right| \, ds +  \sup_{t \in \left[ \frac{n}{N} , \frac{n+1}{N} \right]} \left| \nabla B_{t}(x) - \nabla B_{\frac{n}{N}}(x) \right|.
\end{multline}
Using the definition of the discrete gradient and Assumption~\eqref{AssPot}, we see that
\begin{align*}
    \left| \nabla \left( \nabla \cdot D_pV(\nabla \varphi_L (\cdot , \cdot ; p)) \right) (s , x) \right| & \leq \sum_{y \sim x} \left| D_pV (p + \nabla \varphi_{L} (t,y ; p)) \right| \\
    & \leq C |p|_+^{r-1} + C \sum_{y \sim x} \left|\nabla \varphi_{L} (t , y ;p) \right|^{r-1}.
\end{align*}
Using Proposition~\ref{prop:prop2.20} ``Integration" and noting that $|p|_+^{r-1} / N = 1/ K^r \leq 1$, we deduce that
\begin{align} \label{eq:13350512}
    \mathbb{P}\left[ \int_{\frac{n}{N}}^{\frac{n+1}{N}} \left| \nabla \left( \nabla \cdot D_pV(p + \nabla \varphi_L(\cdot , \cdot ; p)) \right) (s , x) \right| \, ds \geq \frac{K}{4} \right] & \leq C \exp \left( - c (N K )^{\frac{r}{r-1}} \right) \\
    & \leq C \exp \left( - c K^{r}\right). \notag
\end{align}
The supremum of the Brownian motions can be estimated by noting that the difference of two independent Brownian motions is equal in law (up to a multiplicative constant equal to $\sqrt{2}$) to a Brownian motion. This leads to the inequality
\begin{align} \label{eq:13360512}
    \mathbb{P}\left[  \sup_{t \in \left[ \frac{n}{N} , \frac{n+1}{N} \right]} \left| \nabla B_{t}(x) - \nabla B_{\frac{n}{N}}(x) \right| \geq \frac{K}{4} \right] & \leq
    \mathbb{P}\left[ \sup_{t \in \left[ 0 , 1 \right]} \left| B_{t} \right| \geq c \sqrt{N}K \right] \\ & \leq C \exp \left( - c N K^2 \right) \notag \\
    & \leq C \exp \left( - c K^r \right). \notag
\end{align}
Combining the inequalities~\eqref{eq:13340512},~\eqref{eq:13350512} and~\eqref{eq:13360512} with a union bound, we deduce that
\begin{align} \label{eq:13570512}
    \lefteqn{\mathbb{P}\left[  \sup_{n \in \left\{ 0 , \ldots, \lfloor T N \rfloor \right\}} \sup_{t \in \left[ \frac{n}{N} , \frac{n+1}{N} \right]} \left| \nabla \varphi_{L} \left(t , x ;p \right) - \nabla \varphi_{L} \left(\frac{n}{N} , x ;p \right) \right| \geq \frac{K}{2} \right] } \qquad & \\ & \leq \sum_{n = 0}^{ \lceil T N \rceil} \mathbb{P}\left[ \sup_{t \in \left[ \frac{n}{N} , \frac{n+1}{N} \right]} \left| \nabla \varphi_{L} \left(t , x ;p \right) - \nabla \varphi_{L} \left(\frac{n}{N} , x ; p \right) \right| \geq \frac{K}{2} \right] \notag \\
    & \leq C NT \exp \left( - c K^r \right)\notag  \\
    & \leq C |p|^{r-1}_+ K^r T \exp \left( - c K^r \right) \notag \\
    & \leq  C |p|^{r-1}_+ T \exp \left( - c K^r \right). \notag
\end{align}
Combining~\eqref{eq:11580512},~\eqref{eq:13540512} and~\eqref{eq:13570512} completes the proof of~\eqref{eq:13510512}.
\end{proof}

\subsection{A fluctuation estimate for the Langevin dynamic}

Building upon the stochastic integrability estimate for the dynamic established in Proposition~\ref{propositionsubdynamic}, we prove that the gradient of dynamic cannot remain contained in a bounded set for a long time. In the following statement, we will use the value $R_1$ introduced in Lemma~\ref{lemma.upperandlowerboundLambda} (but similar conclusions would hold with more general constants).

\begin{proposition}[Fluctuation for the Langevin dynamic] \label{prop3.4}
There exist two constants $C := C(d , V) < \infty$ and $c := c(d , V) >0$ such that, for any $T \geq 1$ and any vertex $x \in \mathbb{T}_L$, 
\begin{equation} \label{upperboundfluctuation}
    \mathbb{P}\left[ \, \forall t \in [0 , T], \, \left| p + \nabla \varphi_{L} (t , x ; p)  \right| \leq R_1 \, \right] \leq C \exp \left( - c \left( \ln T \right)^{\frac{r}{r-2}} \right).
\end{equation}
\end{proposition}

\begin{remark}
    The proof is in fact almost identical to the one of~\cite[Proposition 3.3]{D23U}, the only (nontrivial) difference is that we prove an estimate which holds uniformly over the slopes $p \in \Rd$.
\end{remark}

\begin{proof} The argument is split into different steps.

\medskip

\textit{Step 1. Reducing the problem to large times.}

\medskip

We fix a vertex $x \in \mathbb{T}_L$ and prove the following estimate: there exist two constants $C := C(d,V) < \infty$ and $c := c(d , V) > 0$ and a time $T_0 := T_0(d , V) < \infty$ such that, for any $T \geq  e^{T_0 |p|^{r-2}_+}$,
\begin{equation} \label{upperboundfluctuation2}
    \mathbb{P}\left[ \, \forall t \in [0 , T], \, \left| p + \nabla \varphi_{L} (t , x ; p)  \right| \leq R_1\, \right] \leq C \exp \left( - c \left( \ln T \right)^{\frac{r}{r-2}} \right).
\end{equation}
The constant $T_0$ will only depend on the parameters $d$ and $V$. It will be chosen following three constraints in the proof below (they are stated at the beginning of Steps 2 and 4).

The bound~\eqref{upperboundfluctuation} can be deduced from~\eqref{upperboundfluctuation2}. Indeed, for $T \leq e^{T_0 |p|^{r-2}_+}$, the inequality~\eqref{upperboundfluctuation} can proved directly as follows
\begin{align*}
     \mathbb{P}\left[ \, \forall t \in [0 , T], \, \left| p + \nabla \varphi_{L} (t , x ; p)  \right| \leq R_1 \, \right] & \leq  \mathbb{P}\left[ \, \left| p + \nabla \varphi_{L}(0 , x;p) \right| \leq R_1 \, \right] \\
     & \leq  \mathbb{P}\left[ \, \left| \nabla \varphi_{L}(0 , x;p) \right| \geq |p| - R_1 \, \right]. \\
\end{align*}
Using that the dynamic at time $0$ is distributed according to the Gibbs measure $\mu_{L , p}$ together with the Proposition~\ref{prop.prop2.3}, we obtain
\begin{align*}
     \mathbb{P}\left[ \, \forall t \in [0 , T], \, \left| p + \nabla \varphi_{L} (t , x ; p)  \right| \leq R_1 \, \right]  & \leq C \exp \left( - c(|p| - R_1)^r \right) \\
     & \leq C \exp \left( - c |p|^r \right) \\
     & \leq C \exp \left( - c (\ln T)^{\frac{r}{r-2}} \right),
\end{align*}
where in the second line, we increased the constant $C$ and reduced the exponent $c$ to absorb the constant $R_1$ (using that $R_1$ only depends on the potential $V$).

\medskip

\textit{Step 2. Setting up the argument.} 

\medskip

Let us fix $T \geq e^{T_0 |p|^{r-2}_+}$ and set $N := (\ln T)/R_1^2$ (various constraints will be made on $T_0$ below). We impose here a first constraint on the time $T_0$ and assume that it is chosen large enough so that, for any $T \geq e^{T_0}$, the following computation can be performed (the requirement plays a role in the last inequality)
\begin{align} \label{eq:19132612}
    \mathbb{P} \left[ \left| B_{1/N} \right| \geq \frac{4}{3} R_1 \right] & = \mathbb{P} \left[ \left| B_{1} \right| \geq \frac{4}{3} \sqrt{\ln T} \right] \\
    & = \frac{2}{\sqrt{2\pi}} \int_{\frac{4}{3}\sqrt{\ln T}}^\infty e^{-\frac{x^2}{2}} \, dx \notag  \\
    & \geq \frac{2}{\sqrt{2\pi}} \int_{\frac{4}{3}\sqrt{\ln T}}^{\frac{4}{3}\sqrt{\ln T} + 1} e^{-\frac{x^2}{2}} \, dx \notag \\
    & \geq  \frac{2}{\sqrt{2\pi}}  \exp \left( - \frac{1}{2} \left( \frac{4}{3} \sqrt{\ln T} + 1 \right)^2  \right) \notag \\
    & \geq \frac{1}{T^{\sfrac{9}{10}}}. \notag
\end{align}
We next decompose the Brownian motions into mutually independent Brownian bridges and increments. To be more specific, we introduce the following notation:
\begin{itemize}
    \item For each $k \in \Z$ and each $y \in \mathbb{T}_L$, we let $W_k(\cdot ; y)$ be the Brownian bridge defined by the formula
    \begin{equation} \label{notation:Brownianbridges}
        \forall t \in \left[0, \frac{1}{N}\right], \hspace{5mm} W_k(t ; y) := B_{t + \frac{k}{N}}(y) - B_{\frac{k}{N}}(y) - N t (B_{\frac{k+1}{N}}(y) - B_{\frac{k}{N}}(y)).
    \end{equation}
    We will denote by $\mathcal{W} := \left\{  W_k (\cdot ; y) \, : \, k \in \Z, y \in \mathbb{T}_L \right\}$ the collection of Brownian bridges.
    \item For each $k \in \N$ and each $y \in \mathbb{T}_L$, we denote by $X_k(y)$ the increment
    \begin{equation} \label{notation:Increments}
            X_k(y) := B_{\frac{k+1}{N}}(y) - B_{\frac{k}{N}}(y).
    \end{equation}
    We will denote by $\mathcal{X} := \left\{ X_k(y) \, : \, k \in \Z, \, y \in \mathbb{T}_L\right\}$ the set of all the increments. For $l \in \N \times \mathbb{T}_L$, the set $\mathcal{X}_{l} := \left\{ X_k(y) \, : \, k \in \Z, \, y \in \mathbb{T}_L, (k,y) \neq (l , x) \right\}$ denotes the collection of all the increments except~$X_l(x)$ (recalling that the vertex $x$ is fixed in the argument).
\end{itemize}
We then introduce the notation 
$$\mathcal{R} := (\mathcal{X}, \mathcal{W}).$$
The set of all possible pairs $\mathcal{R}$ will be denoted by 
$$\Omega := \R^{\Z \times \mathbb{T}_L} \times C\left( \left[0,\frac 1N\right] , \R\right)^{\Z \times \mathbb{T}_L}.$$
Since the dynamic $\left\{ \varphi_L(t , x ; p) \, : \, t \geq 0, \, x \in \mathbb{T}_L \right\}$ can interpreted as deterministic functions of $\mathcal{R} \in \Omega$, we  will write
\begin{equation*}
    \varphi_L(t , x ; p) := \varphi_L(t , x ; p) \left(  \mathcal{R} \right).
\end{equation*}
For $ l\in \Z $, we denote by $\mathcal{R}_{l} := (\mathcal{X}_{l}, \mathcal{W})$ and by $\Omega_{l}$ the set of possible values for $\mathcal{R}_{l}$. We have the identities $\mathcal{R} = (X_l(x), \mathcal{R}_{l})$ and $\Omega = \R \times \Omega_{l}$. To emphasize the dependency of the dynamic on the increment $X_l(x)$, we will write
\begin{equation} \label{eq:15342911}
    \varphi_L(t , x ; p) = \varphi_L(t , x;p) \left( X_l(x), \mathcal{R}_{l} \right).
\end{equation}
We denote by $\mathcal{F}_{\mathcal{R}, l}$ the $\sigma$-algebra generated by $\mathcal{R}_{l}$ and note that the increment $X_{l}(x)$ is independent of the $\sigma$-algebra $\mathcal{F}_{\mathcal{R}, l}$.

\medskip

\textit{Step 3. Introducing the bad events $(A_l)_l$ and estimating the probability of their intersection.}

\medskip

For any~$l \in \N$, we introduce the following random subset of $\R$ (depending on the collection $\mathcal{R}_{l}$),
\begin{equation} \label{eq:11200112}
    \mathcal{A}_l (\mathcal{R}_{l}) := \left\{  X \in \R \, : \, \left| p+ \nabla \varphi_{L} \left(\frac{l+1}{N} , x ;p \right) \left( X, \mathcal{R}_{l} \right) \right| \leq R_1  \right\} \subseteq \R,
\end{equation}
where we used the notation introduced in~\eqref{eq:15342911}. In words, the set $\mathcal{A}_l (\mathcal{R}_{l})$ is the set of all possible values for the increment $X_l(x)$ such that the norm of the gradient of the dynamic $\varphi_L(\cdot , \cdot ; p)$ computed at time $(l+1)/N$ at the vertex $x$ with Brownian bridges and increments given by $\mathcal{R} = (X_l(x) , \mathcal{R}_{l})$ belongs to the interval $[-R_1 , R_1]$. 

We finally introduce the event $A_l \subseteq \Omega$ defined as follows
\begin{equation} \label{eq:11210112}
    A_l := \left\{ \mathcal{R} := (X_l(x),  \mathcal{R}_{l}) \in \Omega \, : \,  X_l(x) \in \mathcal{A}_l (\mathcal{R}_{l}) ~~\mbox{and}~~ \frac{1}{\sqrt{2 \pi N}}\int_{\mathcal{A}_l (\mathcal{R}_{l})} e^{- \frac{x^2}{2N}} \, dx \leq 1- \frac{1}{T^{\sfrac{9}{10}}} \right\}.
\end{equation}
Since the law of the increment $X_l(x)$ is Gaussian of variance $1/N$ and since $X_l(x)$ is independent of the set $\mathcal{R}_{l}$, we have the almost sure upper bound
\begin{equation} \label{eq:28111220}
    \E \left[ \indc_{A_{l}}  \big| \mathcal{F}_{\mathcal{R}, l} \right] \leq 1- \frac{1}{T^{\sfrac{9}{10}}}.
\end{equation}
We next estimate the probability for the intersection of all the events $A_{l}$ for $l \in \left\{ 0 , \ldots, \lfloor NT \rfloor \right\}$ and prove the following stretched exponential decay in the time $T$,
\begin{equation} \label{eq:11492811}
    \mathbb{P}\left[ \bigcap_{l = 0}^{ \lfloor N T \rfloor} A_l \right] \leq \exp \left( - T^{\sfrac{1}{10}} \right).
\end{equation}
The proof of~\eqref{eq:11492811} is obtained by consecutive conditioning. We first note that, since the dynamic~$\varphi_L(t , x ;p)$ depends only on the increments $X_l(x)$ and the Brownian bridges $W_l(\cdot; y)$ such that $t \geq \frac{l}{N}$, the events $(A_0, , \ldots, A_{\lfloor NT \rfloor - 1})$ do not depend on the increment $X_{\lfloor NT \rfloor}(x)$, and are thus measurable with respect to the $\sigma$-algebra $\mathcal{F}_{\mathcal{R}, \lfloor NT \rfloor}$. Combining this observation with the upper bound~\eqref{eq:28111220}, we obtain
\begin{align*}
    \mathbb{P}\left[ \bigcap_{l = 0}^{\lfloor NT \rfloor } A_l \right] & = \E \left[ \prod_{l = 0}^{\lfloor NT \rfloor } \indc_{A_l} \right] \\
    & = \E \left[ \E \left[ \prod_{l = 0}^{\lfloor NT \rfloor } \indc_{A_l}  \bigg| \mathcal{F}_{\mathcal{R}, \lfloor NT \rfloor}\right]\right] \\
    & =  \E \left[\left(  \prod_{l = 0}^{\lfloor NT \rfloor - 1}  \indc_{A_l} \right) \times \E \left[ \indc_{A_{\lfloor NT \rfloor }}  | \mathcal{F}_{\mathcal{R}, \lfloor NT \rfloor} \right]\right] \\
    & \leq \left( 1- \frac{1}{T^{\sfrac{9}{10}}} \right)\mathbb{P}\left[ \bigcap_{l = 0}^{\lfloor NT \rfloor -1} A_l \right].
\end{align*}
We may then iterate the previous computation, noting that, for any $l \in \left\{0 , \ldots, \lfloor NT \rfloor-1 \right\}$, the events $(A_1 , \ldots, A_{l})$ are measurable with respect to the $\sigma$-algebra $\mathcal{F}_{\mathcal{R} , l +1 }$. This leads to the upper bound, for~$T_0$ sufficiently large depending only on $d$ and $V$ (imposing here a second constraint on $T_0$) so that, for any $T \geq e^{T_0}$, $\lfloor NT \rfloor +1 \geq T (\ln T)/R_1^2 \geq T$,
\begin{equation*}
    \mathbb{P}\left[ \bigcap_{l = 0}^{ \lfloor N T \rfloor} A_l \right] \leq \left( 1- \frac{1}{T^{\sfrac{9}{10}}} \right)^{\lfloor NT \rfloor +1}
     \leq  \exp \left( - \frac{\lfloor NT \rfloor +1}{T^{\sfrac{9}{10}}} \right) 
     \leq \exp \left( - T^{\sfrac{1}{10}} \right).
\end{equation*}

\medskip

\textit{Step 4. Introducing a bad event and estimating its probability.}

\medskip

We next note that, by the identity $N := \ln T /R_1^2$ and the Assumption~\eqref{AssPot}, we may find a time $T_G := T_G(d , V) < \infty$ and a constant $C_G := C_G(d , V) < \infty$ such that the following implication holds: for any $T \geq T_G $ and any slope $p \in \Rd$,
\begin{equation} \label{eq:8.570112}
    \left| p + \nabla \varphi_{L} (t , x ; p) \right| \leq \frac{(\ln T)^{\frac{1}{r-2}}}{C_G} \implies   \Lambda_+ (t , x ;p) \leq \frac{N}{8d}.
\end{equation}
We impose here a third constraint on the time $T_0$ and assume that it is larger than $T_G$. We then define the interval $I_T$
\begin{equation*}
    I_T := \left[ - \frac{(\ln T)^{\frac{1}{r-2}}}{(\sqrt{2}(4d)^2)C_G} , \frac{(\ln T)^{\frac{1}{r-2}}}{(\sqrt{2}(4d)^2)C_G}  \right]
\end{equation*}
as well as the ``bad" event (or, to be precise, the complementary of a ``good" event $G_T$ for notational convenience)
\begin{equation*}
    G_T^c := \left\{ \mathcal{R} \in \Omega \, : \, \sup_{t \in [0,T]} \sum_{y \sim x} \left| p + \nabla \varphi_L (t , y ; p)(\mathcal{R}) \right| \geq \frac{(\ln T)^{\frac{1}{r-2}}}{2 C_G} \right\} \bigcup \bigcup_{k=0}^{\lfloor NT \rfloor} \left\{  X_k(x) \notin I_T    \right\}.
\end{equation*}
We then show that the probability of the event $G_T^c$ is close to $0$. To this end, we impose the fourth and final constraint on the time $T_0$ and assume that it satisfies the inequality
\begin{equation*}
    T_0 \geq (4 C_G)^{r-2}.
\end{equation*}
This choice implies that, for any $T \geq e^{T_0 |p|^{r-2}_+}$, 
\begin{equation*}
    \frac{(\ln T)^{\frac{1}{r-2}}}{2 C_G}  \geq 2 |p|_+.
\end{equation*}
We next estimate the probability of the event $G_T^c$ by using Proposition~\ref{propositionsubdynamic}, that the law of the increments $\left\{ X_k(y) \, : \, 1 \leq k \leq \lfloor NT \rfloor \right\}$ is Gaussian of variance $1/N = R_1^2/\ln T$ and a union bound. We obtain
\begin{equation} \label{eq:18092811}
    \mathbb{P}\left[ G_T^c \right] \leq C |p|^{r-1}_+ T \exp \left( - c \left( \ln T \right)^{\frac{r}{r-2}} \right) + C NT \exp \left( - c \left( \ln T \right)^{\frac{r}{r-2}} \right) \leq  C \exp \left( - c \left( \ln T \right)^{\frac{r}{r-2}} \right).
\end{equation}

\medskip

\textit{Step 5. Proving that the Langevin dynamic fluctuates in the complement of the bad events.}

\medskip

We will now prove the inclusion of events
\begin{equation} \label{eq:18082811}
    \left\{ \mathcal{R} \in \Omega \, : \, \forall t \in [0 , T], \, \left| p+ \nabla \varphi_{L} (t , x ; p)  \right| \leq R_1  \right\} \subseteq \bigcap_{l = 0}^{ \lfloor N T \rfloor} A_l \cup G_T^c.
\end{equation}
Proposition~\ref{upperboundfluctuation} is then obtained by combining~\eqref{eq:11492811},~\eqref{eq:18092811},~\eqref{eq:18082811} and a union bound.

By Proposition~\ref{prop:propsLangevin}, we know that, for any $y \in \mathbb{T}_L$, the derivative of the $\nabla \varphi_{L} (t , y; p)$ with respect to the increment $X_l(x)$ is given by the following identity
\begin{equation*}
    \frac{\partial \nabla \varphi_{L} (t , y ; p)}{\partial X_l(x)} =  \sqrt{2} N \int_{\frac{l}{N}}^{\frac{l+1}{N}} \nabla P_{\a(\cdot ; p)} (t , y ; s , x) \, ds.
\end{equation*}
The pointwise upper bound on the heat kernel stated in Proposition~\ref{prop:propheatkernel} implies the following estimate on its gradient, for any vertex $y \in \mathbb{T}_L $ and any pair of times $(t , s) \in \R \times \R$ with $t \geq s$,
\begin{align} \label{eq:13462811}
    \left| \nabla P_{\a(\cdot ; p)} \left( t , y ; s , x  \right) \right| & \leq \sum_{i=1}^d \left|  P_{\a(\cdot ; p)} \left( t , y ; s , x  \right) \right| + \left|  P_{\a(\cdot ; p)} \left( t , y + e_i ; s , x  \right) \right| \\
    & \leq 4d. \notag
\end{align}
A combination of the previous displays implies the following bound, for any $\mathcal{R} = (X_l(x), \mathcal{R}_{l}) \in \Omega$ and any $(t,y)  \in \R \times \mathbb{T}_L$,
\begin{equation} \label{eq:08250112}
    \left| \frac{\partial \nabla \varphi_L(t , y ; p)}{\partial X_l(x)} (X_l(x), \mathcal{R}_{l}) \right| \leq 4 \sqrt{2}d.
\end{equation}
We then fix a realization of the increments and Brownian bridges $\mathcal{R} := (X_{l}(x), \mathcal{R}_{l}) \in \Omega$ and assume that $\mathcal{R} \in G_T$. We first claim that, for any increment $X \in I_T$,
\begin{equation} \label{eq:08260112}
   \sup_{t \in [0 , T]}\sum_{y \sim x} \left| p + \nabla \varphi_L \left( t, y ; p \right) \left( X , \mathcal{R}_{l} \right) \right| \leq \frac{(\ln T)^{\frac{1}{r-2}}}{C_G}.
\end{equation}
To prove~\eqref{eq:08260112}, we first use~\eqref{eq:08250112} and deduce that
\begin{equation*}
    \sup_{t \in [0 , T]}\sum_{y \sim x} \left| \nabla \varphi_L \left( t, y ; p \right)(X , \mathcal{R}_{l}) -  \nabla \varphi_L \left( t, y; p  \right)(X_{l}(x) , \mathcal{R}_{l}) \right|  \leq \sqrt{2} (4 d)^2 \left| X - X_{l}(x) \right|
     \leq  \frac{(\ln T)^{\frac{1}{r-2}}}{2C_G}.
\end{equation*}
By the assumption $(X_{l}(x), \mathcal{R}_{l}) \in G_T$, we have that
\begin{equation*}
     \sup_{t \in [0 , T]} \sum_{y \sim x} \left| p + \nabla \varphi_L \left( t, y ; p\right)(X_{l}(y) , \mathcal{R}_{l}) \right| \leq \frac{(\ln T)^{\frac{1}{r-2}}}{2 C_G}.
\end{equation*}
A combination of the two previous displays with the triangle inequality yields, for any $X \in I_T$,
\begin{equation*}
     \sup_{t \in [0 , T]} \sum_{y \sim x} \left| p + \nabla \varphi_L \left( t, y ; p \right)(X , \mathcal{R}_{l})\right| \leq  \frac{(\ln T)^{\frac{1}{r-2}}}{C_G}.
\end{equation*}
Using the definition of the constant $C_G$ and the implication~\eqref{eq:8.570112}, we have proved the following result: for any $T \geq T_G$, any $\mathcal{R} := (X_l(x), \mathcal{R}_{l}) \in G_T$ and any increment $X \in I_T$, one has the upper bound
\begin{equation*}
    \sup_{t \in [0 , T]} \sum_{y \sim x}  \Lambda_+ (t , y ; p)(X , \mathcal{R}_{l,y}) \leq \frac{N}{8d}.
\end{equation*}
The previous upper bound is useful as it can be used to control the derivative in time of the heat kernel. Indeed, using the identity $\partial_t P_{\a(\cdot ; p)} = \nabla \cdot {\a(\cdot ; p)} \nabla P_{\a(\cdot ; p)}$ together with the bound~\eqref{eq:13462811}, we obtain the estimate, for any pair of times $(s , t)  \in \R \times \R$ with $t \geq s$,
\begin{align*}
    \left| \partial_t \nabla P_{\a(\cdot ; p)} (t , x ; s , y) \right|
    & \leq \sum_{y \sim x}  \left| \a(t , y ; p)  \nabla P_{\a(\cdot ; p)} (t , y ; s , x) \right| \\
    & \leq \sum_{y \sim x} \Lambda_+ (t , y ; p) \left| \nabla P_{\a(\cdot ; p)} (t , y ; s , x) \right| \\
    & \leq 4d \sum_{y \sim x} \Lambda_+ (t , y ; p) \\
    & \leq \frac{N}{2}.
\end{align*}
We next consider the index $i = 1$ (but note that the argument would be identical with any index $i \in \{ 1 , \ldots, d \}$) and note that the following identity holds $ \nabla_1 P_{\a(\cdot ; p)} (s , x ; s , x)= -1$. From this observation, we deduce that, for any $\mathcal{R} := (X_l(x), \mathcal{R}_{l}) \in G_T$ and any increment $X \in I_T$,
\begin{align*}
    \frac{\partial \nabla_1 \varphi_L\left( \frac{l+1}{N} , x ; p\right)}{\partial X_l(x)} (X, \mathcal{R}_{l}) & = \sqrt{2} N \int_{\frac{l}{N}}^{\frac{l+1}{N}} \nabla_1  P_{\a(\cdot ; p)} \left(\frac{l+1}{N}, x ; s , x \right) (X, \mathcal{R}_{l}) \, ds  \\
    &  \leq - \sqrt{2} N \int_{\frac{l}{N}}^{\frac{l+1}{N}} 1 - \frac{N}{2} \left(\frac{l+1}{N} - s \right) \, ds \\
    & \leq - \frac{3}{4}.
\end{align*}
This upper bound on the derivative of the gradient of the dynamic implies that, for any $(X_l(x) , \mathcal{R}_{l}) \in G_T$, the function
\begin{equation*}
     X \mapsto - \nabla_1 \varphi_L\left( \frac{l+1}{N} , x ; p \right)(X , \mathcal{R}_{l}) - \frac{3}{4} X ~~\mbox{is increasing on the interval}~I_T.
\end{equation*}
This implies the following upper bound on the Lebesgue measure of the set $\mathcal{A}_l (\mathcal{R}_{l}) \cap I_T$,
\begin{equation*}
    \left| \mathcal{A}_l (\mathcal{R}_{l}) \cap I_T \right| \leq \frac{8}{3} R_1,
\end{equation*}
which then yields the estimate, for any $T \geq  e^{T_0 |p|^{r-2}_+}$ (in particular, the computation~\eqref{eq:19132612} applies),
\begin{equation*}
    \frac{1}{\sqrt{2 \pi N}}\int_{\mathcal{A}_l (\mathcal{R}_{l})} e^{- \frac{x^2}{2N}} \, dx  \leq 1 -  \frac{1}{\sqrt{2 \pi N}}\int_{I_T \setminus [-\frac{4}{3}R_1 , \frac{4}{3} R_1]} e^{- \frac{x^2}{2N}} \, dx \leq 1 -  \frac{1}{T^{\sfrac{9}{10}}}.
\end{equation*}
From the definitions~\eqref{eq:11200112} and~\eqref{eq:11210112}, the previous inequality implies, for any $l \in \{1 , \ldots , \lfloor NT \rfloor \}$,
\begin{equation*}
    G_T \cap A_l = G_T \cap \left\{ \mathcal{R} \in \Omega \, : \,  \left|p +  \nabla \varphi_L \left(\frac{l+1}{N} , x ; p \right) \left(\mathcal{R} \right) \right| \leq R_1 \right\}.
\end{equation*}
Taking the intersection over $l \in \{1 , \ldots , \lfloor NT \rfloor \}$ completes the proof of~\eqref{eq:18082811}.
\end{proof}

\subsection{The moderated environment}

In this section, we introduce the moderated environment (following the presentation of Section~\ref{sec:intromoderatedenvt}) and establish its main features: specifically, we introduce it in Section~\ref{sec:defmoderatedevt}, establish a stochastic integrability estimate in Section~\ref{sec:stochintestmoderate} (which implies that all its moments are finite), and show how it can be used to ``moderate" the heat equation in Section~\ref{sec:moderationheatequation} (this last part follows closely the articles~\cite{MO16, biskup2018limit}).

\subsubsection{Definition} \label{sec:defmoderatedevt}

Following the insight of Mourrat and Otto~\cite{MO16}, we introduce the \emph{moderated environment}. We first introduce the two functions
\begin{equation*}
    k_t := \frac{\delta}{(1 + t)^{4}} \hspace{5mm} \mbox{and} \hspace{5mm}  K_t :=  k_t + \int_{t}^{\infty} s k_s \, ds,
\end{equation*}
where $\delta := \delta(d)> 0$ is chosen sufficiently small so that, for any $t , s' \in (0 , \infty)$ with $s' \geq t$,
\begin{equation} \label{K*KsmallerthanK}
    \int_{t}^{s'} K_{s - t} K_{s' - s} \, ds \leq K_{s' - t} ~~ \mbox{and}~~ \int_{0}^{\infty} K_{s} \, ds \leq 1.
\end{equation}
Equipped with these functions, we introduce the moderated environment.

\begin{definition}[Moderated environment for the Langevin dynamic] \label{def:moderatedLangevin}
    Given an integer $L \in \N$ and a slope $p \in \Rd$, we introduce the moderated environment
    \begin{equation*}
    \mathbf{m}(t , x ; p) :=  |p|^{r-2}_+ \int_t^{\infty} k_{|p|^{r-2}_+ (s-t)} \frac{\Lambda_- (s, x ; p)  \wedge |p|^{r-2} }{( s-t)^{-1} \sum_{y \sim x}\int_t^s \left( |p|^{r-2} + \Lambda_+ \left( s' , x ; p\right)  \right) \, ds'} \, ds.
\end{equation*}
\end{definition}

\begin{remark}
Let us make two remarks about the previous definition:
\begin{itemize}
\item By the stationarity of the Langevin dynamic (and thus of the environment $\a$), the law of the random variable $\mathbf{m}(t,x;p)$ does not depend on $t$ nor on $x$.
\item The definition of the moderated environment $ \mathbf{m}(\cdot , \cdot ; p)$ has been scaled with respect to the norm of the slope $p \in \Rd$: with this definition, the random variable $\mathbf{m}(t , x ; p) $ is typically of order $1$ (see Proposition~\ref{prop:stochintmoderated}, this is different from $\a(t , x ; p)$ which is a symmetric matrix whose eigenvalues are typically of order $|p|_+^{r-2}$) and this scaling has been chosen so as to obtain the inequality stated in Proposition~\ref{prop4.3} (with a constant independent of the norm $|p|$ on the right-hand side).
\end{itemize}
\end{remark}

\subsubsection{Stochastic integrability} \label{sec:stochintestmoderate}

In the following proposition, we obtain some control over the probability of the moderated environment to be small or large.

\begin{proposition}[Stochastic integrability for the moderated environment] \label{prop:stochintmoderated}
There exist two constants $c := c(d , V) > 0$ and $C := C(d , V) < \infty$ such that, for any $K \geq 1$, any $L \in \N$ and any $(t , x , p) \in \R \times \Zd \times \Rd$,
\begin{equation} \label{eq:18510612}
    \mathbb{P} \left[ \mathbf{m}(t , x ; p)  \geq K \right] \leq C \exp \left( - c K^{\frac{r}{r-2}} \right)
\end{equation}
and
\begin{equation} \label{eq:18500612}
    \mathbb{P} \left[ \mathbf{m}(t , x ; p )^{-1} \geq K \right] \leq C \exp \left( - c (\ln K)^{\frac{r}{r-2}} \right).
\end{equation}
\end{proposition}

\begin{remark}
    The same proof can be used to show the slightly upgraded estimate: for any $K \geq 1$, any $L \in \N$ and any $(t , x , p) \in \R \times \Zd \times \Rd$,
    \begin{equation} \label{ineq:infmoderated}
    \mathbb{P} \left[ \inf_{\substack{s \in (t , t + 1/|p|_+^{r-2}) }} \mathbf{m}(s , x ; p)  \geq K \right] \leq C \exp \left( - c K^{\frac{r}{r-2}} \right).
\end{equation}
    This estimate will be used in the proof below.
\end{remark}

\begin{remark} \label{Rem:remark3.13}
The inequalities~\eqref{eq:18500612} and~\eqref{eq:18510612} imply that all the moments of the random variables $\mathbf{m}(t , x ; p)$ and $\mathbf{m}(t , x ; p)^{-1}$ are finite and bounded uniformly in the slope $p \in \Rd$: for any exponent $\gamma \in (1 , \infty)$, there exists a constant $C_\gamma := C(d , V , \gamma) < \infty$ such that for any $(t , x , p) \in \R \times \Zd \times \Rd$,
    \begin{equation*}
        \E \left[  \mathbf{m}(t , x ; p)^\gamma  \right] + \E \left[  \mathbf{m}(t , x ; p)^{-\gamma} \right] \leq C_\gamma.
    \end{equation*}
\end{remark}

\begin{proof}

     Using Proposition~\ref{prop.prop2.3}, Lemma~\ref{lemma.upperandlowerboundLambda} and Proposition~\ref{prop:prop2.20} ``Integration", there exist two constants $C := C(d , V) < \infty$ and $c :=  c(d , V)  > 0$ such that, for any $(t , x , p) \in \R \times \Zd \times \Rd$ and any $s > t$,
    \begin{multline} \label{eq:13030103}
        \mathbb{P} \left[ \Lambda_- (t , x ; p) \leq c |p|_+^{r-2} \right] +  \mathbb{P} \left[ ( s-t)^{-1} \sum_{y \sim x}\int_t^s \left( |p|_+^{r-2}  + \Lambda_+ \left( s' , y ; p\right)  \right) \, ds' \geq C |p|_+^{r-2} \right]  \\
        \leq C \exp \left(  - c |p|^r \right).
    \end{multline}
    For later purposes, we let $C_0$ and $c_0$ be two constants depending on $d$ and $V$, chosen large and small enough respectively, so that, for any $(t , x , p) \in \R \times \Zd \times \Rd$,
    \begin{equation*}
    \mathbb{P} \left[ \left| D_p V \left( p + \nabla \varphi_L(t , x ; p) \right) \right| \leq C_0 |p|_+^{r-1} \right] \leq C_0 \exp \left( - c_0 |p|^r_+ \right),
    \end{equation*}
    and, for any $T \geq 1$,
    \begin{equation} \label{eq:182804044}
    \mathbb{P} \left[ \sup_{t \in [0, T]} \left| D_p V \left( p + \nabla \varphi_L(t , x ; p) \right) \right| \leq C_0 |p|_+^{r-1} \right] \leq C_0 T \exp \left( - c_0 |p|^r_+ \right).
    \end{equation}
    The existence of these constants is guaranteed by Proposition~\ref{prop.prop2.3} and Proposition~\ref{propositionsubdynamic}.

    \medskip
    
    We break the argument into several steps.

    \medskip

    \textit{Step 1. Proof of the inequality~\eqref{eq:18510612}.}

    \medskip

    For the upper bound~\eqref{eq:18510612}, we note that, by Proposition~\ref{prop:prop2.20} ``Integration", it is sufficient to prove that there exist two constants $C := C(d , V) < \infty$ and $c := c(d , V) > 0$ such that, for any $(t , x , p) \in \R \times \Zd \times \Rd$ and any $s > t$,
    \begin{equation} \label{eq:09220403}
        \mathbb{P} \left[ \frac{\Lambda_- (s, x ; p)  \wedge |p|_+^{r-2}  }{( s-t)^{-1} \sum_{y \sim x}\int_t^s \left( |p|_+^{r-2}  + \Lambda_+ \left( s' , y ; p\right)  \right) \, ds'} \geq K \right] \leq  C \exp \left( - c K^{\frac{r}{r-2}} \right).
    \end{equation}

    \textit{Substep 1.1. Small values of $K$.} 
    
    \medskip
    
    We note that it is sufficient to prove the inequality~\eqref{eq:09220403} under the assumption that $K$ is larger than the constant $C_0$. This can be achieved by increasing the value of the constant $C$ in~\eqref{eq:09220403} so that the right-hand side is larger than $1$ for any $K \leq C_0$.

    \medskip

    \textit{Substep 1.2. Intermediate values of $K$.} 
    
    \medskip

    We prove the inequality~\eqref{eq:09220403} in the case $C_0 \leq K \leq |p|_+^{r-2}$. If the assumption $C_0 \leq |p|_+^{r-2}$ is not satisfied, then this case can be disregarded. Under the assumption $K \geq C_0$, we can use the inclusion of events
    \begin{equation*}
        \left\{ \frac{\Lambda_- (s, x ; p)  \wedge |p|_+^{r-2}  }{( s-t)^{-1} \sum_{y \sim x}\int_t^s \left( |p|_+^{r-2}  + \Lambda_+ \left( s' , y ; p\right)  \right) \, ds'} \geq K \right\} \subseteq \left\{ \Lambda_- (s, x ; p)  \geq C_0 |p|_+^{r-2} \right\}.
    \end{equation*}
    Using Proposition~\ref{eq:13030103} and the assumption $K \leq |p|_+^{r-2}$, we deduce that
    \begin{align*}
             \mathbb{P} \left[  \frac{\Lambda_- (s, x ; p)  \wedge  |p|_+^{r-2} }{( s-t)^{-1} \sum_{y \sim x}\int_t^s \left(  |p|_+^{r-2} + \Lambda_+ \left( s' , y ; p\right)  \right) \, ds'} \geq K \right] & \leq C \exp \left( - c |p|_+^r \right) \\
              & \leq C \exp \left( - c K^{\frac{r}{r-2}} \right). 
    \end{align*}
    This completes the proof of the upper bound~\eqref{eq:18510612} in the case $C_0 \leq K \leq |p|_+^{r-2}$. 

     \medskip
        
    \textit{Substep 1.3. Large values of $K$.}

    \medskip
    
    We now prove~\eqref{eq:18510612} in the case $K \geq |p|_+^{r-2}$. In this setting, we write similarly
    \begin{equation*}
        \left\{ \frac{\Lambda_- (s, x ; p)  \wedge |p|_+^{r-2}  }{( s-t)^{-1} \sum_{y \sim x}\int_t^s \left( |p|_+^{r-2}  + \Lambda_+ \left( s' , y ; p\right)  \right) \, ds'} \geq K \right\} \subseteq \left\{ \Lambda_- (s, x ; p)  \geq K \right\}.
    \end{equation*}
    Using Proposition~\ref{prop.prop2.3}, we obtain that the probability of the event on the right-hand side is smaller than $C \exp \left( - c K^{\frac{r}{r-2}} \right)$. The proof of the inequality~\eqref{eq:18510612} is complete.

    \medskip

    \textit{Step 2. Proof of the inequality~\eqref{eq:18500612}}

    \medskip

    We next prove the (more difficult) inequality~\eqref{eq:18500612}. As it was the case for the proof of the inequality~\eqref{eq:18510612}, we split the argument into different cases depending on the value of the constant $K$. To this end, we introduce a constant $K_0$ which shall only depend on the dimension $d$ and the potential $V$. Its specific value obeys two constraints: the first one is that it is larger than the ratio $C_0 / c_0$ (and in particular, larger than $C_0$ since we assumed $c_0 \in (0 , 1]$). The second condition is stated in~\eqref{eq:secondcondK0} below.

    \medskip

    \textit{Substep 2.1. Small values of $K$.} 

    \medskip
    
    First, if the constant $K$ is smaller than $K_0$, we may increase the value of the constant $C$ in~\eqref{eq:18500612} so that the inequality holds, by ensuring that the right-hand side is larger than $1$ for any $K \leq K_0$. 

    \medskip

    \textit{Substep 2.2. Intermediate values of $K$.}

    \medskip
    
    We next assume that
    \begin{equation*}
       \exp ( |p|_+^{r-2} ) \geq K_0 ~~ \mbox{and} ~~ K_0 \leq K \leq \exp ( |p|_+^{r-2} ).
    \end{equation*} 
    If the first inequality is not satisfied, then this case can be disregarded. By Proposition~\ref{prop:prop2.20} ``Integration", it is sufficient to show that, for any $(t , x , p) \in \R \times \Zd \times \Rd$ and any $s > t$,
    \begin{equation} \label{eq:10120403}
        \mathbb{P} \left[ \frac{\Lambda_- (s, x ; p)  \wedge |p|_+^{r-2}  }{( s-t)^{-1} \sum_{y \sim x}\int_t^s \left( |p|_+^{r-2}  + \Lambda_+ \left( s' , y ; p\right)  \right) \, ds'} \leq K^{-1} \right] \leq  C \exp \left( - c \left( \ln  K \right)^{\frac{r}{r-2}} \right).
    \end{equation}
    To prove the inequality~\eqref{eq:10120403}, we use the assumption $K \geq C_0 / c_0$ and write
    \begin{multline*}
        \left\{ \frac{\Lambda_- (s, x ; p)  \wedge |p|_+^{r-2}  }{( s-t)^{-1} \sum_{y \sim x}\int_t^s \left( |p|_+^{r-2}  + \Lambda_+ \left( s' , y ; p\right)  \right) \, ds'} \leq K^{-1} \right\} \subseteq \\
       \left\{ \Lambda_- (s, x ; p) \leq c_0 |p|_+^{r-2} \right\}  \cup \left\{ ( s-t)^{-1} \sum_{y \sim x}\int_t^s \left( |p|_+^{r-2}  + \Lambda_+ \left( s' , y ; p\right)  \right) \, ds' \geq C_0 |p|_+^{r-2} \right\}.
    \end{multline*}
    Taking the probability on both sides and using the inequality~\eqref{eq:13030103},
    \begin{align*}
        \lefteqn{ \mathbb{P} \left[  \frac{\Lambda_- (s, x ; p)  \wedge |p|_+^{r-2}  }{( s-t)^{-1} \sum_{y \sim x}\int_t^s \left( |p|_+^{r-2}  + \Lambda_+ \left( s' , y ; p\right)  \right) \, ds'} \leq K \right] } \qquad & \\ &
        \leq  \mathbb{P} \left[ \Lambda_- (s, x ; p) \leq c_0 |p|_+^{r-2}  \right] +   \mathbb{P} \left[ ( s-t)^{-1} \sum_{y \sim x}\int_t^s \left( |p|_+^{r-2}  + \Lambda_+ \left( s' , y ; p\right)  \right) \, ds' \geq C_0 |p|_+^{r-2} \right] \\
        & \leq 2 C_0 \exp \left( - c_0 |p|^r \right).
    \end{align*}
    Using the assumption $K \leq \exp ( |p|_+^{r-2} )$, we deduce~\eqref{eq:10120403} from the previous display.

    \medskip

    \textit{Substep 2.3. Large values of $K$.}

    \medskip
    
We finally treat the case $K \geq \exp ( |p|_+^{r-2} ) \vee K_0$. To ease the notation, we only prove the result for $t= 0$ and $x=0$. This can be done without loss of generality by the stationarity (with respect to both space and time) of the Langevin dynamic. We first prove the following inclusion of events: there exists a constant $c_1 := c_1(d, V) > 0$ such that, for any $T \geq 1$,
\begin{align} \label{eq:11030712}
    \left\{ \mathbf{m}(0 , 0 ; p) \leq \frac{c_1}{|p|_+^{4(r-2)} T^{6}} \right\} \subseteq & \left\{ \sup_{t \in [0,T]}  \left| p +  \nabla \varphi_L\left( t, 0 ; p \right) \right| \leq  R_1 \right\} \notag \\
    \qquad & \bigcup \left\{ \sup_{t \in [0,T]} \Lambda_+ (t , 0 ; p) + \sum_{y \sim 0} \left| D_p V \left( p +  \nabla \varphi_L\left( t , y ; p \right) \right) \right|  \geq \frac{T}{2} \right\} \notag \\
    \qquad & \bigcup \left\{ \sup_{\substack{t , t' \in [0,T]\\ |t - t'| \leq \frac{1}{T}}} \left| \nabla B_{t'} \left( 0\right) - \nabla B_{t} \left( 0 \right) \right| \geq \frac{1}{2} \right\}.
\end{align}
The inclusion~\eqref{eq:11030712} asserts that, for the moderated environment $ \mathbf{m}(0 , 0 ; p)$ to be small, the norm $ \left| p +  \nabla \varphi_L\left( t, 0 ; p \right) \right|$ must remain smaller than $R_1$ for a long time (this behaviour is ruled out by Proposition~\ref{prop3.4}), or must behave very irregularly, this condition is represented by the second and third events on the right-hand side of~\eqref{eq:11030712}, and happens with small probability.

We first prove~\eqref{eq:11030712}. This inclusion is equivalent to the following implication: for any $T \geq 1$,
\begin{multline} \label{eq:17260712}
    \sup_{t \in [0,T]}  \left| p +  \nabla \varphi_L\left( t, 0 ; p \right) \right| \geq  R_1, ~ \sup_{t \in [0,T]} \Lambda_+ (t , 0 ; p)  + \sum_{y \sim 0} \left| D_p V \left( p +  \nabla \varphi_L\left( t , y ; p \right) \right) \right| \leq \frac{T}{2} \\ ~ \mbox{and} ~ \sup_{\substack{t , t' \in [0,T]\\ |t - t'| \leq \frac{1}{T}}} \left| \nabla B_{t'} \left( 0 \right) - \nabla B_{t} \left( 0 \right) \right| \leq \frac12 \implies \mathbf{m}(0 , 0 ; p) \geq \frac{c_1}{|p|_+^{4(r-2)} T^{6}}.
\end{multline}
We assume that the three conditions on the left-hand side of~\eqref{eq:17260712} are satisfied and fix a time $t \in [0 , T]$ such that $\left| p + \nabla \varphi_L\left( t, 0 ; p \right) \right| \geq  R_1$. Using the definition of the Langevin dynamic~\eqref{eq:reqldefSDE}, we see that, for any time $s \in \left[t - \frac{1}{2T} , t + \frac{1}{2T} \right]$,
\begin{align*}
    \left| \nabla \varphi_L(s , 0 ; p) - \nabla \varphi_L(t , 0 ; p) \right| &
    \leq \left| \int_{t}^s \nabla \left( \nabla \cdot D_p V( p + \nabla \varphi_L(\cdot , \cdot ; p)) \right)(s',0)  \, ds' \right| + \left| \nabla B_s(0) - \nabla B_t(0)\right| \\
    & \leq \int_{t - \frac{1}{2T}}^{t + \frac{1}{2T}} \sum_{y \sim 0}\left| D_p V \left( p + \nabla \varphi_L\left( s' , 0 ; p \right) \right) \right| \, ds' + \frac12 \\
    & \leq 1.
\end{align*}
Using the assumption $R_1 \geq 2$ (which can be made without loss of generality, see Lemma~\ref{lemma.upperandlowerboundLambda}), we deduce that, for any $s \in \left[t - \frac{1}{2T} , t + \frac{1}{2T} \right]$, $\left| p + \nabla \varphi_L(s , 0 ; p) \right| \geq \frac{R_1}{2}$. Consequently, for any  $s \in \left[t - \frac{1}{2T} , t + \frac{1}{2T} \right]$,
\begin{equation*}
    \Lambda_-(s , 0 ; p) \geq 1.
\end{equation*}
The left-hand side of~\eqref{eq:17260712} yields the upper bound, for any $s \in [0,T]$,
\begin{equation*}
     \Lambda_+(s , 0 ; p) \leq \frac{T}{2}.
\end{equation*}
Combining the two previous displays with the definition of the moderated environment and the definition of the function~$k$, we deduce that, for any $T \geq 1$,
\begin{align*}
    \mathbf{m}(0, 0 ; p) & = |p|^{r-2}_+ \int_0^{\infty} k_{|p|^{r-2}_+ s} \frac{\Lambda_- (s, 0 ; p)  \wedge |p|^{r-2} }{s^{-1} \sum_{y \sim 0}\int_0^s \left( |p|^{r-2} + \Lambda_+ \left( s' , y ; p\right)  \right) \, ds'} \, ds \\
    & \geq  |p|^{r-2}_+ \int_{t - \frac{1}{2T}}^{t + \frac{1}{2T}} k_{|p|^{r-2}_+ s} \frac{\Lambda_- (s, 0 ; p)  \wedge |p|^{r-2} }{s^{-1} \sum_{y \sim 0}\int_0^s \left( |p|^{r-2} + \Lambda_+ \left( s' , y ; p\right)  \right) \, ds'} \, ds  \\
    & \geq \frac{|p|^{r-2}_+}{|p|^{r-2}_+ + T}  \int_{t - \frac{1}{2T}}^{t + \frac{1}{2T}} k_{|p|^{r-2}_+ s} ds \\
    & \geq \frac{c_1}{|p|_+^{4(r-2)} T^{6}}.
\end{align*}
The proof of~\eqref{eq:17260712}, and thus of~\eqref{eq:11030712} is complete. For $K \geq \exp ( |p|_+^{r-2} ) \vee K_0,$ we denote by $T_K$ the unique nonnegative solution to the equation
\begin{equation} \label{eq:secondcondK0}
    K = \frac{|p|_+^{4(r-2)} T_K^{6}}{c_1}.
\end{equation}
We impose here the second condition on the constant $K_0$: we assume that it is sufficiently large so that $T_K \geq 4 C_0 |p|_+^{r-1}$. This can be done thanks to the assumption $K \geq \exp ( |p|_+^{r-2} )$ and since the exponential grows faster than any polynomial. Note that these assumptions imply that $\ln K \leq C \ln T_K$ for a large constant $C := C(d , V) < \infty$.

Using the inclusion of events~\eqref{eq:11030712}, we see that~\eqref{eq:18500612} can be proved by estimating the probabilities of the three events on the right-hand side of~\eqref{eq:11030712}. For the first event, we use Proposition~\ref{prop3.4} and write
\begin{align*}
    \mathbb{P} \left[ \sup_{t \in [0,T_K]}  \left| p +  \nabla \varphi_L\left( t, 0 ; p \right) \right| \leq  R_1 \right] & \leq C \exp \left( - c \left( \ln T_K \right)^{\frac{r}{r-2}} \right)\\
    & \leq C \exp \left( - c \left( \ln K \right)^{\frac{r}{r-2}} \right).
\end{align*}
For the second term, we use the inequalities~\eqref{eq:13030103} and~\eqref{eq:182804044} (together with Assumption~\eqref{AssPot}) to obtain that
\begin{align*}
     \mathbb{P} \left[ \sup_{t \in [0,T_K]} \Lambda_+ (t , 0 ; p) + \sum_{y \sim 0} \left| D_p V \left( p +  \nabla \varphi_L\left( t , y ; p \right) \right) \right|  \geq \frac{T_K}{2}   \right] & \leq C |p|_+^{r-1} T_K \exp \left( - c  T_K^{\frac{r}{r-1}}\right) \\
     & \qquad + C |p|_+^{r-1} T_K \exp \left( - c  T_K^{\frac{r}{r-2}}\right) \\ &
    \leq C \exp \left( - c \left( \ln K \right)^{\frac{r}{r-2}} \right).
\end{align*}
For the third term, we note that
\begin{align*}
     \mathbb{P} \left[ \sup_{\substack{t , t' \in [0,T_K]\\ |t - t'| \leq \frac{1}{T_K}}} \left| \nabla B_{t'} \left( 0 \right) - \nabla B_{t} \left( 0 \right) \right| \geq \frac{1}{2} \right] & 
    \leq  \sum_{l = 0}^{\lceil T_K^2 \rceil} \mathbb{P} \left( \sup_{t \in \left[ \frac{l-1}{T_K} , \frac{l+1}{T_K} \right] } \left| \nabla B_t(0) - \nabla B_{\frac{l}{T_K}}(0) \right| \geq \frac{1}{4} \right) \\
    & \leq (T_K^2+2)  \mathbb{P} \left( \sup_{t \in \left[ 0 , \frac{2}{T_K} \right] } \left| \nabla B_t(0) - \nabla B_{\frac{1}{T_K}}(0) \right| \geq \frac{1}{4} \right) \\
    & \leq (T_K^2+2)  \mathbb{P} \left( \sup_{t \in \left[ 0 , 2 \right] } \left| \nabla B_t(0) - \nabla B_{1}(0) \right| \geq \frac{\sqrt{T_K}}{4} \right) \\
    & \leq C (T_K^2+2) \exp \left( - c T_K \right).
\end{align*}
We then estimate (crudely) the last term on the right-hand side so as to obtain
\begin{equation*}
    \mathbb{P} \left[ \sup_{\substack{t , t' \in [0,T_K]\\ |t - t'| \leq \frac{1}{T_K}}} \left| \nabla B_{t'} \left( 0\right) - \nabla B_{t} \left( 0 \right) \right| \geq \frac{1}{2} \right] \leq C \exp \left( - c \left( \ln K \right)^{\frac{r}{r-2}} \right).
\end{equation*}
Combining the three previous displays with~\eqref{eq:11030712} yields
\begin{equation*}
     \mathbb{P} \left[ \mathbf{m}(t , 0 ; p)^{-1} \geq K \right]  \leq   C \exp \left( - c \left( \ln K \right)^{\frac{r}{r-2}} \right).
\end{equation*}
This implies~\eqref{eq:18500612}.
\end{proof}

\subsection{Moderation for solutions of the heat equation} \label{sec:moderationheatequation}

In this section, we adapt the arguments of Mourrat and Otto~\cite[Proposition 4.6]{MO16} to environments which are not bounded from above. Using the terminology introduced in~\cite[Definition 3.1]{MO16}, we show that the environment $\a$ is $(w , C K)$-moderate. Specifically, we establish Proposition~\ref{prop4.3} whose proof closely follows the one of~\cite{MO16}.

In the statement below, for a fixed vertex $x \in \mathbb{T}_L$, we will use the notation $\sum_{y \sim_2 x}$ to sum over all the vertices $y \in \mathbb{T}_L$ are at distance at most $2$ from $x$ (for the Euclidean norm $|\cdot|$).

\begin{proposition} \label{prop4.3}
There exists a constant $C := C(d) > 0$ such that, for every $t \geq 0$ and every solution $u : (0 , \infty) \times \mathbb{T}_L \to \R$ of the parabolic equation\begin{equation*}
    \partial_t u - \nabla \cdot \a(\cdot ; p) \nabla u = 0 ~~~\mbox{in} ~~~ (0 , \infty) \times \mathbb{T}_L,
\end{equation*}
one has the inequality, for any $(t , x) \in \R \times \mathbb{T}_L$,
\begin{equation*}
    \mathbf{m}(t , x ; p) \left| \nabla u(t, x) \right|^2 \leq C  \sum_{y \sim_2 x}  \int_t^\infty K_{|p|^{r-2}_+ (s - t)} \nabla u(s , y) \cdot \a(s , y ; p) \nabla u(s , y) \, ds.
\end{equation*}
\end{proposition}

\begin{remark}
    This inequality is already known and plays also an important role in the article of Biskup and Rodriguez~\cite{biskup2018limit} (see Lemma 2.11 there) to identify the scaling limit of a random walk evolving in a degenerate environment. The proof is added below for completeness (and because some minor modifications need to be incorporated in the argument to take into account that the environment is not bounded from above).
\end{remark}

\begin{proof}
The proof closely follows the ones of~\cite[Proposition 4.6]{MO16},~\cite[Lemma 2.11]{biskup2018limit} and~\cite[Proposition 4.3]{D23U} with some additional technicalities to take into account that the environment is not bounded from above (compared to~\cite{MO16, biskup2018limit}) and the scaling with respect to the slope $p$ has to be taken into account in the analysis (compared to both~\cite{MO16, biskup2018limit, D23U}).

To ease the notation, we assume that $t = 0$. We first estimate
\begin{align} \label{eq:16110612}
\mathbf{m}(0 , x ; p) \left| \nabla u(0 , x) \right|^2 & =  |p|_+^{r-2} \int_0^{\infty} k_{|p|_+^{r-2} s} \frac{\Lambda_-(s,x) \wedge |p|_+^{r-2} }{s^{-1} \sum_{y \sim x} \int_0^s (\Lambda_+ (s',y) + |p|_+^{r-2}) \, ds'} |\nabla u(0 , x)|^2 \, ds \\
                        & \leq 2 |p|_+^{r-2}  \int_t^{\infty} k_{|p|_+^{r-2} s} \frac{\Lambda_-(s,x) \wedge |p|_+^{r-2} }{s^{-1} \sum_{y \sim x } \int_0^s (\Lambda_+ (s',y) + |p|_+^{r-2}) \, ds'} |\nabla u(s , x)|^2 \, ds \notag \\
                        & \quad + 2 |p|_+^{r-2}  \int_t^{\infty} k_{|p|_+^{r-2} s} \frac{\Lambda_-(s,x) \wedge |p|_+^{r-2}}{s^{-1} \sum_{y \sim x } \int_0^s (\Lambda_+ (s',y) + |p|_+^{r-2}) \, ds'} |\nabla u(s , x) - \nabla u(0 , x)|^2 \, dt. \notag
\end{align}
The first term on the right-hand side can be estimated by upper bounding the numerator and lower bounding the denominator as follows
\begin{multline*}
\int_0^{\infty} k_{|p|_+^{r-2}  s} \frac{\Lambda_-(s,x) \wedge |p|_+^{r-2} }{s^{-1} \sum_{y \sim x } \int_0^s (\Lambda_+ (s',y) + |p|_+^{r-2}) \, ds'} |\nabla u(s , x)|^2 \, ds \\ \leq  \frac{1}{|p|_+^{r-2}}  \int_0^{\infty} k_{|p|_+^{r-2}  s} \Lambda_-(s,x ; p) | \nabla u(s , x)|^2 \, ds.
\end{multline*}
Using that $\Lambda_-(s,x ; p)$ is smaller than the smallest eigenvalue of the matrix $\a(s , x ; p)$, we can write
\begin{align}  \label{eq:16110312}
     \lefteqn{|p|_+^{r-2} \int_0^{\infty} k_{|p|_+^{r-2}  s} \frac{\Lambda_-(s,x) \wedge |p|_+^{r-2} }{s^{-1} \sum_{y \sim x } \int_0^s (\Lambda_+ (s',y) + |p|_+^{r-2}) \, ds'} |\nabla u(s , x)|^2 \, ds } \qquad & \\ & \leq \int_0^{\infty} k_{|p|_+^{r-2}  s} \Lambda_-(s,x ; p) | \nabla u(s , x)|^2 \, ds \notag \\
     & \leq \int_0^{\infty} k_{|p|_+^{r-2}  s}  \nabla u(s , x)  \cdot \a(s , x ; p) \nabla u(s , x) \, ds. \notag
\end{align}
We next estimate the second term on the right-hand side of~\eqref{eq:16110612}. To this end, we use the identity $\partial_t u = \nabla \cdot \a(\cdot ; p) \nabla u$ and write
\begin{align*}
    \left| \nabla u(s , x) - \nabla u(0 , x) \right|^2 & \leq C \sum_{y \sim x} ( u(s , y) - u(0 , y) )^2  \\
    & \leq C \sum_{y \sim x}  \left(  \int_0^s \nabla \cdot \a \nabla u(s' , y)\, ds' \right)^2.
\end{align*}
Applying the Cauchy-Schwarz inequality to the terms on the right-hand side, we obtain
\begin{align*}
    \left( \int_0^s \nabla \cdot \a \nabla u(s' , y) \, ds' \right)^2 & \leq C \sum_{y' \sim y} \left( \int_0^s \left| \a(s' , y' ; p) \nabla u(s' , y') \right| \, ds' \right)^2 \\
    & \leq C \sum_{y' \sim y} \left( \int_0^s \Lambda_+ (s' , y' ; p) \, ds' \right) \left( \int_0^s \nabla u(s' , y') \cdot \a(s' , y') \nabla u(s' , y')  \, ds' \right).
\end{align*}
Combining the two previous displays yields
\begin{align*}
    \left| \nabla u(s , x) - \nabla u(0 , x) \right|^2 & \leq C \sum_{y' \sim y} \sum_{y \sim x} \left( \int_0^s \Lambda_+ (s' , y' ; p) \, ds' \right) \left( \int_0^s \nabla u(s' , y') \cdot \a(s' , y ; p) \nabla u(s' , y)  \, ds' \right) \\
    & \leq C \sum_{y \sim_2 x} \left( \int_0^s \Lambda_+ (s' , y ; p) \, ds' \right) \left( \int_0^s \nabla u(s' , y) \cdot \a(s' , y ; p) \nabla u(s' , y)  \, ds' \right).
\end{align*}
We thus obtain
\begin{multline*}
    \frac{\Lambda_-(s,x) \wedge |p|_+^{r-2}}{s^{-1} \sum_{y \sim x } \int_0^s (\Lambda_+ (s',y ; p) + |p|_+^{r-2}) \, ds'} \left| \nabla u(s , x) - \nabla u(0 , x) \right|^2 \\
    \leq C s \sum_{y \sim x} \int_0^s \nabla u (s' , y) \cdot \a(s', y ; p) \nabla u(s' , y)  \, ds'.
\end{multline*}
Combining the previous estimate with~\eqref{eq:16110612} and~\eqref{eq:16110312}, we deduce that
\begin{align*}
    \mathbf{m}(0 , x ; p) \left| \nabla u(0 , x) \right|^2 &
    \leq C \int_0^{\infty} k_{|p|_+^{r-2}  s}  \nabla u(s , x)  \cdot \a(s , x ; p) \nabla u(s , x) \, ds \\
    & \qquad + C    \sum_{y \sim x} \int_{0}^\infty |p|_+^{r-2} s k_{|p|_+^{r-2} s} \int_0^s  \nabla u (s' , y) \cdot \a(s', y ; p) \nabla u(s' , y) \, ds' \\
    &  \leq C   \sum_{y \sim x} \int_0^\infty K_{|p|_+^{r-2} s} \nabla u (s , y) \cdot \a(s, y ; p) \nabla u(s , y)  \, ds.
\end{align*}
The proof of Proposition~\ref{prop4.3} is complete.
\end{proof}

\section{Sublinearity of the Langevin dynamic} \label{sec:section5sublinearity}

In this section, we establish the following sublinearity estimates on the Langevin dynamic. The proof makes use of the Helffer-Sj\"{o}strand representation formula, together with the results established on the moderated environment (the argument is similar to the one of~\cite{MO16, D23U}, with some adaptations to take into account the nature of the parabolic equation studied here; specifically the lack of maximum principle).

\begin{proposition}[Sublinearity for the Langevin dynamic] \label{prop:sublincorr}
There exists a constant $C := C(d , V) < \infty$ such that the following estimate holds: for any integer $L \in \N$, any slope $p \in \Rd$ and any $(t , x) \in \R \times \mathbb{T}_L$,
    \begin{equation} \label{eq:sublinearitycorr}
         \left| \varphi_L(t , x ; p) \right| \leq  \mathcal{O}_1 (C L^{7/8}) . 
    \end{equation}
\end{proposition}

\begin{remark}
Let us make three remaks about the previous result:
\begin{itemize}
\item The exponent 7/8 on the right-hand side of~\eqref{eq:sublinearitycorr} is not optimal (this quantity should grow like the square-root of a logarithm of $L$ in two dimensions and be bounded in dimensions $3$ and higher). Similarly, the stochastic integrability is certainly not optimal.
\item For the purpose of obtaining a quantitative version of the hydrodynamical limit, any exponent strictly smaller than $1$ is enough (and in particular $7/8$ satisfies this property). We decided to minimise the technical complexity of the proof rather than optimise the right-hand side of~\eqref{eq:sublinearitycorr}. Nevertheless, we believe that the arguments presented below can be improved to obtain sharper estimates (although obtaining the optimal result seems to require a substantial additional amount of work).
\item The dependency of the right-hand side in the slope $p \in \Rd$ is also suboptimal as the estimate should improve as $|p| \to \infty$.
\end{itemize}
\end{remark}

\begin{proof}
In order to prove Proposition~\ref{prop:sublincorr}, it is enough to prove the upper bound on the variance
\begin{equation} \label{eq:11360202}
    \mathrm{Var}_{L , p} \left[ \varphi(0) \right] \leq C L^{7/8}.
\end{equation}
Indeed, by the definition of the periodic measure $\mu_{L , p}$, we have the identity $\E_{L , p} \left[ \varphi(0) \right] = 0$, and the inequality~\eqref{eq:11360202} is in fact an estimate on the $L^2$-norm of the random variable $\varphi(0)$. Combining this observation with the log-concavity of the random variable $\varphi(0)$, we can upgrade the stochastic integrability estimate from $L^2$ to exponential moments (see~\eqref{eq:logconcavityimpliesmoments}).

In order to prove the inequality~\eqref{eq:11360202}, we use the Helffer-Sj\"{o}strand representation formula
\begin{equation} \label{eq:HScorrectorsubadditive}
    \mathrm{Var}_{L , p} \left[ \varphi(0) \right] = \E \left[ \int_0^\infty P_{\a(\cdot ; p)}(t , 0) \, dt \right],
\end{equation}
and estimate the term on the right-hand side. To ease the notation, we fix a slope $p \in \Rd$ and write $\a$ instead of $\a(\cdot ; p)$.

We follow the proof of~\cite[Proof of Theorem 3.2]{MO16}, we introduce the notation
\begin{equation*}
    \mathcal{E}_t := \sum_{x \in \mathbb{T}_L} P_\a (t , x)^2, \hspace{5mm}  \mathcal{D}_t = \sum_{x \in \mathbb{T}_L} \nabla P_\a(t , x) \cdot \a(t , x) \nabla P_\a(t , x),
\end{equation*}
as well as the moderated quantities
\begin{equation*}
    \bar{\mathcal{E}}_t := \int_t^\infty K_{|p|^{r-2}_+(s-t)} \mathcal{E}_s \, ds \hspace{5mm} \mbox{and} \hspace{5mm} \bar{\mathcal{D}}_t := \int_t^\infty K_{|p|^{r-2}_+ (s-t)} \mathcal{D}_s \, ds.
\end{equation*}
We note that the following identities hold
\begin{equation} \label{eq:02021345}
    - \partial_t \mathcal{E}_t = 2\mathcal{D}_t \hspace{5mm} \mbox{and} \hspace{5mm}  - \partial_t \bar{\mathcal{E}}_t = 2\bar{\mathcal{D}}_t.
\end{equation}
In particular, both $\mathcal{E}_t$ and $\bar{\mathcal{E}}_t$ are decreasing, and thus they are always smaller than $1$ (since $\mathcal{E}_0 = 1$ by definition, and $\bar{\mathcal{E}}_0 \leq 1$ by~\eqref{K*KsmallerthanK} and the inequality $|p|_+ \geq 1$). We next split the proof into different steps.

\medskip

\textit{Step 1. Bounding the $L^2$-norm of the gradient of the heat kernel.}

\medskip

We prove the following inequality: there exists a constant $C := C(d) < \infty$ such that
\begin{equation} \label{eq:1114091222}
    \int_0^\infty \sum_{x \in \mathbb{T}_L} \mathbf{m}(t , x ; p) \left| \nabla P_\a(t , x) \right|^2 dt \leq \frac{C}{|p|^{r-2}_+} .
\end{equation}
The proof of the inequality~\eqref{eq:1114091222} is obtained by first noting that, since the energy $\mathcal{E}_t$ is smaller than $1$, we may integrate the first identity of~\eqref{eq:02021345} over the times $t \in (0 , \infty)$ and obtain
\begin{equation*}
     \int_0^\infty \sum_{x \in \mathbb{T}_L} \nabla P_\a(t , x) \cdot \a(t , x) \nabla P_\a(t , x) \, dt \leq \frac{1}{2}.
\end{equation*}
We next upper bound the left-hand side of~\eqref{eq:1114091222} using Proposition~\ref{prop4.3} (and noting that, for any $x \in \mathbb{T}_L$, there are $C := C(d)$ vertices $y$ satisfying $y \sim_2 x$)
\begin{align*}
    \int_0^\infty \sum_{x \in \mathbb{T}_L} \mathbf{m}(t , x ; p) \left| \nabla P_\a(t , x) \right|^2 dt & \leq C \int_0^\infty \int_t^\infty \sum_{x \in \mathbb{T}_L} K_{|p|^{r-2}_+ (s - t)} \nabla u(s , x) \cdot \a(s , x ; p) \nabla u(s , x) \, ds \, dt \\
    & = C \int_0^\infty \int_0^\infty \sum_{x \in \mathbb{T}_L} \indc_{\{ t \leq s \}} K_{|p|^{r-2}_+ (s - t)} \nabla u(s , x) \cdot \a(s , x ; p) \nabla u(s , x) \, dt \, ds.
\end{align*}
We next note that, by~\eqref{K*KsmallerthanK}, for any $s \geq 0$,
\begin{equation*}
    \int_0^\infty \indc_{\{ t \leq s \}} K_{|p|^{r-2}_+ (s - t)} \, dt \leq \int_0^\infty K_{|p|^{r-2}_+  t} \, dt \leq \frac{1}{|p|_+^{r-2}}.
\end{equation*}
Combining the three previous displays, we obtain
\begin{equation*}
    \int_0^\infty \sum_{x \in \mathbb{T}_L} \mathbf{m}(t , x ; p) \left| \nabla P_\a(t , x) \right|^2 dt  \leq \frac{C}{|p|_+^{r-2}} \int_0^\infty \sum_{x \in \mathbb{T}_L} \nabla P_\a(t , x) \cdot \a(t , x) \nabla P_\a(t , x) \, dt  \leq \frac{C}{|p|_+^{r-2}}.
\end{equation*}

\medskip

\textit{Step 2. Bounding the integral of the heat kernel: the small times.}

\medskip

In this step, we prove the following inequality: for any time $T \geq 0$,
\begin{equation} \label{ineq:1diemnsionalline}
    \int_0^T P_\a (t , 0) \leq \frac{C \sqrt{L T}}{|p|_+^{r/2-1}} \sup_{(t , x) \in [0 , T] \times \mathbb{T}_L} \mathbf{m}(t , x ; p)^{-1/2}.
\end{equation}
The proof of the inequality~\eqref{ineq:1diemnsionalline} is based on the following observation: for any vertex $y \in \mathbb{T}_L$, there exists a path connecting $0$ to $y$ whose length is at most $d L$. We select one of these paths according to an arbitrary criterion and denote it by $\mathcal{L}_y$ (it is thus a collection of less than $dL$ adjacent vertices starting at $0$ and ending at $y$). We may then estimate the difference $P_\a (t , 0) - P_\a (t , y)$ by the sum of the norm of the gradient of the heat kernel over the line $\mathcal{L}_y$ as follows
    \begin{equation*}
    \left| P_\a (t , 0) - P_\a (t , y) \right|  \leq \sum_{x \in \mathcal{L}_y} \left| \nabla P_\a (t , x)\right|.
    \end{equation*}
Applying the Cauchy-Schwarz inequality, we further deduce that
\begin{equation*}
\left| P_\a (t , 0) - P_\a (t , y) \right| \leq C \sqrt{ L} \sqrt{  \sum_{x \in \mathcal{L}_y} \left| \nabla P_\a (t , x)\right|^2 } \leq  C \sqrt{ L} \sqrt{ \sum_{x \in \mathbb{T}_L} \left| \nabla P_\a (t , x)\right|^2 }.   
\end{equation*}
We next sum both sides of the inequality over the vertices $y \in \mathbb{T}_L$ (note that the right-hand side does not depend on $y$) and use the identity $\sum_{y \in \mathbb{T}_L}P_\a (t , y) = 0$ to obtain
\begin{align*}
    \left| P_\a (t , 0) \right| = \left|  P_\a (t , 0) - \frac{1}{\left| \mathbb{T}_L \right|} \sum_{y \in \mathbb{T}_L}  P_\a (t , y) \right| & \leq \frac{1}{\left| \mathbb{T}_L \right|} \sum_{y \in \mathbb{T}_L} \left|  P_\a (t , 0) - P_\a (t , y) \right|  \\
    & \leq C \sqrt{ L} \sqrt{ \sum_{x \in \mathbb{T}_L} \left| \nabla P_\a (t , x)\right|^2 }.
\end{align*}
Integrating the previous inequality over the times $t \in (0 , T)$ and applying the Cauchy-Schwarz inequality, we further deduce that
\begin{align*}
    \left( \int_0^T P_\a (t , 0) \, dt \right)^2 & \leq   L T \int_0^T \sum_{x \in \mathbb{T}_L} \left| \nabla P_\a (t , x)\right|^2 \\
                            & \leq C L T \sup_{(t , x) \in [0 , T] \times \mathbb{T}_L} \mathbf{m}(t , x ; p)^{-1}  \int_0^T \sum_{x \in \mathbb{T}_L} \mathbf{m}(t , x ; p) \left| \nabla P_\a(t , x) \right|^2 dt \\
                            & \leq \frac{C L T}{|p|_+^{r-2}}  \sup_{(t , x) \in [0 , T] \times \mathbb{T}_L} \mathbf{m}(t , x ; p)^{-1}.
\end{align*}

\medskip

\textit{Step 3. Bounding the integral of the heat kernel: the large times.}

\medskip

In this step, we introduce the random variable
    \begin{equation} \label{def.defvariableM}
        \mathbf{M} := \sup_{T \geq 1} \left(  \frac{1}{T} \int_0^T  \inf_{\substack{x \in \mathbb{T}_L\\ s \in (t , t + |p|_+^{2-r}) }} \mathbf{m}(s , x ; p) \, dt \right)^{-1},
    \end{equation}
    and prove the following inequality: there exist two constants $C := C(d) < \infty$ and $c := c(d) > 0$ such that, for any $T \geq 2$,
    \begin{equation} \label{eq:ineqPalargetimes}
        \int_T^\infty P_\a(t , 0) \, dt \leq C L^2\mathbf{M} \exp \left( - c \frac{T}{L^2\mathbf{M}} \right).
    \end{equation}
The core of the proof relies on the following differential inequality on the moderated energy: for any $t \geq 0$,
\begin{equation} \label{eq:diffineqlargetimes}
    -  \partial_t \bar{\mathcal{E}}_t \geq  \frac{c |p|^{r-2}_+}{L^2} \left( \inf_{\substack{x \in \mathbb{T}_L\\ s \in (t , t + |p|_+^{2-r}) }} \mathbf{m}(s , x ; p) \right)  \bar{\mathcal{E}}_t.
\end{equation}

We then decompose this step into two substeps.

\medskip
\textit{Substep 3.1. Proof of the inequality~\eqref{eq:diffineqlargetimes}.}

\medskip

We start with the following computation which relies on Proposition~\ref{prop4.3} and the first inequality of~\eqref{K*KsmallerthanK} (scaled in $|p|_+^{r-2}$): for any $t \geq 0$,
\begin{align*}
    \lefteqn{\int_{t}^\infty K_{|p|_+^{r-2}(s-t)}\sum_{x \in \mathbb{T}_L}  \mathbf{m}(s , x ; p) \left| \nabla P_\a(s , x) \right|^2 \, ds} \qquad & \\
    & \leq C \int_{t}^\infty K_{|p|_+^{r-2}(s-t)} \int_s^{\infty} K_{|p|_+^{r-2}(s'-s)} \sum_{x \in \mathbb{T}_L} \nabla P_\a(s' ,x) \cdot \a(s' , x) \nabla P_\a(s' ,x) \, ds' ds \\
    & \leq \frac{C}{|p|^{r-2}_+}  \int_{t}^\infty K_{|p|^{r-2}_+(s-t)} \sum_{x \in \mathbb{T}_L}  \nabla P_\a(s , x) \cdot \a(s , x) \nabla P_\a(s , x) \, ds \\
    & \leq \frac{C}{|p|^{r-2}_+} \bar{\mathcal{D}}_t .
\end{align*}
Combining the previous inequality with the identity~\eqref{eq:02021345}, we obtain the differential inequality
\begin{equation*}
   -  \partial_t \bar{\mathcal{E}}_t  \geq c|p|^{r-2}_+ \int_{t}^\infty K_{|p|_+^{r-2}(s-t)}\sum_{x \in \mathbb{T}_L}  \mathbf{m}(s , x ; p) \left| \nabla P_\a(s , x) \right|^2 \, ds.
\end{equation*}
We then lower bound the right-hand side by applying the Poincar\'e inequality (on the torus $\mathbb{T}_L$ using the assumption $\sum_{x \in \mathbb{T}_L} P_\a(s , x) = 0$). We obtain
\begin{align*}
        -  \partial_t \bar{\mathcal{E}}_t & \geq c|p|^{r-2}_+  \int_{t}^\infty K_{|p|_+^{r-2}(s-t)}  \left( \inf_{x \in \mathbb{T}_L} \mathbf{m}(s , x ; p) \right) \sum_{x \in \mathbb{T}_L} \left| \nabla P_\a(s , x) \right|^2 \, ds \\
        & \geq \frac{c|p|^{r-2}_+}{L^2}  \int_{t}^\infty K_{|p|_+^{r-2} (s-t)} \left( \inf_{x \in \mathbb{T}_L} \mathbf{m}(s, x ; p) \right) \sum_{x \in \mathbb{T}_L} \left| P_\a(s , x) \right|^2 \, ds.
\end{align*}
We next lower bound the integral on the right-hand side by reducing the interval of integration from $(t , \infty)$ to $(t , t + 1/|p|^{r-2}_+)$
\begin{align} \label{eq:10211903}
    -  \partial_t \bar{\mathcal{E}}_t & \geq \frac{c|p|^{2(r-2)}_+}{L^2}  \int_{t}^{t + |p|_+^{2-r}} K_{|p|_+^{r-2}(s-t)}   \left( \inf_{x \in \mathbb{T}_L} \mathbf{m}(s , x ; p) \right) \sum_{x \in \mathbb{T}_L} \left| P_\a(s , x) \right|^2 \, ds \\
    & \geq \frac{c|p|^{2(r-2)}_+}{L^2} \left( \inf_{\substack{x \in \mathbb{T}_L\\ s \in (t , t + |p|_+^{2-r}) }} \mathbf{m}(s , x ; p) \right) \int_{t}^{t + |p|_+^{2-r}} K_{|p|_+^{r-2}(s-t)}  \sum_{x \in \mathbb{T}_L} \left| P_\a(s , x) \right|^2 \, ds. \notag
\end{align}
We finally lower bound the term on the right hand side by using the two following observations:
\begin{equation*}
     \int_{t}^{t + |p|_+^{2-r}} K_{|p|_+^{r-2}(s-t)} \, ds \geq c \int_{t + |p|_+^{2-r}}^{\infty} K_{|p|_+^{r-2}(s-t)} \, ds ~~\mbox{and} ~~ s \mapsto \sum_{x \in \mathbb{T}_L} \left| P_\a(s , x) \right|^2 ~\mbox{is decreasing.}
\end{equation*}
We obtain
\begin{align} \label{eq:10221903}
    \int_{t}^{t + |p|_+^{2-r}} K_{|p|_+^{r-2}(s-t)}  \sum_{x \in \mathbb{T}_L} \left| P_\a(s , x) \right|^2 \, ds & \geq \sum_{x \in \mathbb{T}_L} \left| P_\a( t + |p|_+^{2-r} , x) \right|^2  \int_{t}^{t + |p|_+^{2-r}} K_{|p|_+^{r-2}(s-t)}  \, ds \\
    & \geq c \sum_{x \in \mathbb{T}_L} \left| P_\a( t + |p|_+^{2-r} , x) \right|^2 \int_{t + |p|_+^{2-r}}^{\infty} K_{|p|_+^{r-2}(s-t)}  \, ds \notag \\
    & \geq c \int_{t + |p|_+^{2-r}}^{\infty} K_{|p|_+^{r-2}(s-t)} \sum_{x \in \mathbb{T}_L} \left| P_\a( s , x) \right|^2 \, ds. \notag
\end{align}
Combining the inequalities~\eqref{eq:10211903} and~\eqref{eq:10221903}, we obtain the lower bound
\begin{align*}
-  \partial_t \bar{\mathcal{E}}_t & \geq \frac{c|p|^{2(r-2)}_+}{L^2} \left( \inf_{\substack{x \in \mathbb{T}_L\\ s \in (t , t + |p|_+^{2-r}) }} \mathbf{m}(s , x ; p) \right) \int_{t}^{\infty} K_{|p|_+^{r-2}(s-t)}  \sum_{x \in \mathbb{T}_L} \left| P_\a(s , x) \right|^2 \, ds \\
                                & = \frac{c|p|^{2(r-2)}_+}{L^2} \left( \inf_{\substack{x \in \mathbb{T}_L\\ s \in (t , t + |p|_+^{2-r}) }} \mathbf{m}(s , x ; p) \right) \bar{\mathcal{E}}_t,
\end{align*}
which is exactly~\eqref{eq:diffineqlargetimes}.

\medskip

\textit{Substep 3.2. Deducing~\eqref{eq:ineqPalargetimes} from~\eqref{eq:diffineqlargetimes}.}

\medskip

By integrating~\eqref{eq:diffineqlargetimes} and using that the initial value $ \bar{\mathcal{E}}_0$ is smaller than $1$, we obtain the inequality
\begin{equation*}
    \bar{\mathcal{E}}_t \leq \exp \left( - \frac{c |p|^{r-2}_+ \int_0^t  \inf_{x \in \mathbb{T}_L,  s \in (t , t + |p|_+^{2-r})} \mathbf{m}(s, x ; p) \, ds}{L^2}\right)  \leq  \exp \left( - \frac{c |p|^{r-2}_+ t}{L^2 \mathbf{M}}\right).
\end{equation*}
We next fix a time $T \geq 1$ and estimate the following quantity (using the previous inequality)
\begin{equation*}
    \int_{T}^\infty \bar{\mathcal{E}}_t \exp \left( \frac{c |p|^{r-2}_+ t}{2 L^2 \mathbf{M}}\right) \, dt \leq \int_T^\infty \exp \left( - \frac{c |p|^{r-2}_+  t}{2\mathbf{M} L^2}\right) \, dt \leq \frac{C L^2 \mathbf{M}}{ |p|^{r-2}_+} \exp \left( - \frac{  c |p|^{r-2}_+ T}{2\mathbf{M} L^2}\right) .
\end{equation*}
We then simplify the term on the left-hand side by writing
\begin{align*}
    \int_T^\infty \bar{\mathcal{E}}_t \exp \left( \frac{c |p|^{r-2}_+ t}{2 L^2 \mathbf{M}}\right) \, dt & = \int_T^\infty  \int_T^\infty \indc_{\{t \leq s\}} K_{|p|^{r-2}_+(s-t)} \mathcal{E}_s \exp \left(  \frac{c |p|^{r-2}_+ t}{2 L^2 \mathbf{M}}\right) \, ds\, dt.
\end{align*}
Using Fubini's theorem and the inequality, for any $s \geq T+1/|p|^{r-2}_+$,
\begin{equation*}
     \exp \left( \frac{c |p|^{r-2}_+ s}{2 L^2 \mathbf{M}}\right)  \leq C |p|^{r-2}_+\int_T^\infty \indc_{\{t \leq s\}} K_{|p|^{r-2}_+(s-t)} \exp \left( \frac{c |p|^{r-2}_+ t}{2 L^2 \mathbf{M}}\right) \, dt,
\end{equation*}
we obtain that
\begin{align*}
     \int_{T+1/|p|^{r-2}_+}^\infty  \mathcal{E}_s \exp \left( \frac{c  |p|^{r-2}_+ s}{2 L^2 \mathbf{M}}\right) \, ds \leq C |p|^{r-2}_+ \int_T^\infty \bar{\mathcal{E}}_t \exp \left( \frac{c  |p|^{r-2}_+ t}{2 L^2 \mathbf{M}}\right) \, dt & 
     \leq C L^2 \mathbf{M} \exp \left( - \frac{c  |p|^{r-2}_+ T}{2\mathbf{M} L^2}\right) \\ &
     \leq C L^2 \mathbf{M} \exp \left( - \frac{c T}{2\mathbf{M} L^2}\right) .
\end{align*}
Using the (trivial) inequality $P_\a(t , 0 )^2 \leq \mathcal{E}_t$ and the Cauchy-Schwarz inequality, we obtain, for any $T \geq 1$,
\begin{align*}
    \int_{T+1/|p|^{r-2}_+}^\infty P_{\a}(t , 0) \, dt & \leq \sqrt{ \int_T^\infty \mathcal{E}_t \exp \left( \frac{c t}{2 L^2 \mathbf{M}}\right)  \, dt  \int_T^\infty \exp \left( - \frac{c t}{2 L^2 \mathbf{M}}\right) \, dt} \\
    & \leq C L^2 \mathbf{M} \exp \left( - \frac{c T}{2\mathbf{M} L^2}\right).
\end{align*}
Noting that $1/|p|_+^{r-2} \leq 1$, we obtain, for any $T \geq 2$ (by reducing the value of the constant $c$)
\begin{equation*}
    \int_{T}^\infty P_{\a}(t , 0) \, dt \leq C L^2 \mathbf{M} \exp \left( -  \frac{c T}{2\mathbf{M} L^2}\right).
\end{equation*}

\medskip

\textit{Step 4. The conclusion.} By the identity~\eqref{eq:HScorrectorsubadditive} we may write, for any $T \geq 2$,
\begin{align} \label{eq:131633}
    \mathrm{Var}_{L , p} \left[ \varphi(0) \right] & = \E \left[ \int_0^\infty P_{\a}(t , 0) \, dt \right] \\
                    & = \E \left[ \int_0^T P_{\a}(t , 0) \, dt +  \int_T^\infty P_{\a}(t , 0) \, dt \right]   \notag \\
                    & \leq C \sqrt{\frac{LT}{|p|_+^{r-2}}} \E \left[ \sup_{(t , x) \in [0 , T] \times \mathbb{T}_L} \mathbf{m}(t , x ; p)^{-1} \right] + C \E \left[ L^2 \mathbf{M} \exp \left( - \frac{c T}{L^2 \mathbf{M}}  \right) \right]. \notag
\end{align}
We then set $T := L^{9/4}$ (for the argument below, the important property of the exponent $9/4$ is that it is slightly larger than $2$) and estimate the two terms on the right-hand side of~\eqref{eq:131633}.

\medskip

\textit{Substep 4.1. Estimating the first term on the right-hand side of~\eqref{eq:131633}.} We first write
\begin{equation*}
    \frac{\sqrt{LT}}{|p|_+^{r/2 - 1}} \E \left[ \sup_{(t , x) \in [0 , T] \times \mathbb{T}_L} \mathbf{m}(t , x ; p)^{-1} \right]  \leq \frac{L^{13/8}}{|p|_+^{r/2 - 1}}  \E \left[ \sup_{(t , x) \in [0 , T] \times \mathbb{T}_L} \mathbf{m}(t , x ; p)^{-1} \right].
\end{equation*}
We then estimate the expectation of the supremum of the inverse of the moderated environment on the right-hand side. To this end, we note that , by the inequality~\eqref{ineq:infmoderated} (of Proposition~\ref{prop:stochintmoderated}), we have the inequality, for any $(t , x) \in \R \times \mathbb{T}_L$,
\begin{equation*}
     \sup_{\substack{s \in (t , t + 1/|p|_+^{r-2}) }} \mathbf{m}(s , x ; p)^{-1} \leq C \mathbf{m}(t , x ; p)^{-1} \leq \mathcal{O}_{\Psi, c} (C).
\end{equation*}
Using Proposition~\ref{prop:prop2.20} ``Maximum" with $N := |p|_+^{r-2} T \left| \mathbb{T}_L \right| = |p|_+^{r-2} (2L+1)^d L^{9/4} $ variables, we obtain
\begin{equation*}
    \sup_{(t , x) \in [0 , T] \times \mathbb{T}_L} \mathbf{m}(t , x ; p)^{-1} \leq \mathcal{O}_{\Psi , c} \left( C e^{C \left| \ln \left( |p|_+^{r-2} L^{d + 9/4} \right) \right|^{(r-2)/r}  } \right).
\end{equation*}
The previous inequality implies that
\begin{align} \label{eq:18352406}
    \frac{\sqrt{LT}}{|p|_+^{r/2 - 1}} \E \left[ \sup_{(t , x) \in [0 , T] \times \mathbb{T}_L} \mathbf{m}(t , x ; p)^{-1} \right] & \leq \frac{C L^{13/8}}{|p|_+^{r/2 - 1}} e^{C \left| \ln \left( |p|_+^{r-2} L^{d + 9/4} \right) \right|^{(r-2)/r}} \\
    & = C L^{7/4} \times \frac{e^{C \left| \ln \left( |p|_+^{r-2} L^{d + 9/4} \right) \right|^{(r-2)/r}}}{L^{1/8} |p|_+^{r/2 - 1}} \notag
    \\ 
    & \leq C L^{7/4}, \notag
\end{align}
where, in the third inequality we used that the term involving the exponential grows subpolynomially fast (so that the second term on the right-hand side is bounded). 

\medskip

\textit{Substep 4.1. Estimating the second term on the right-hand side of~\eqref{eq:131633}.}
For the second term on the right-hand side of~\eqref{eq:131633}, we first use the inequality $\exp(- t) \leq C/t^2$ and obtain
\begin{equation} \label{eq:18042406}
    \E \left[ L^2 \mathbf{M} \exp \left( -  \frac{c T}{L^2 \mathbf{M}}   \right) \right] \leq \frac{C L^6}{T^2} \E \left[ \mathbf{M}^3 \right] \leq C L^{3/2} \E \left[ \mathbf{M}^3 \right].
\end{equation}
There remains to estimate the expectation of the random variable $\mathbf{M}.$ We first recall its definition~\eqref{def.defvariableM} and use the convexity of the function $x \mapsto 1/x$ over the positive real numbers to write
\begin{equation*}
     \mathbf{M} := \sup_{T \geq 1} \left(  \frac{1}{T} \int_0^T  \inf_{\substack{x \in \mathbb{T}_L\\ s \in (t , t + |p|_+^{2-r}) }} \mathbf{m}(s , x ; p) \, dt \right)^{-1} \leq \sup_{T \geq 1}  \frac{1}{T} \int_0^T  \sup_{\substack{x \in \mathbb{T}_L\\ s \in (t , t + |p|_+^{2-r}) }} \mathbf{m}(s , x ; p)^{-1} \, dt.
\end{equation*}
Using Proposition~\ref{prop:prop2.20} ``Maximum" with $N := \left| \mathbb{T}_L \right|= (2L+1)^d$ random variables, we deduce that, for any $t \in \mathbb{R},$
\begin{equation*}
    \sup_{\substack{x \in \mathbb{T}_L\\ s \in (t , t + |p|_+^{2-r}) }} \mathbf{m}(s , x ; p)^{-1} \leq \mathcal{O}_{\Psi , c} \left( C e^{C \left( \ln L \right)^{(r-2) / r} } \right).
\end{equation*}
Using Proposition~\ref{propmaximalineq}, we thus have
\begin{equation*}
     \E \left[ \mathbf{M}^3 \right] \leq C \E \left[ \sup_{\substack{x \in \mathbb{T}_L\\ s \in (0 , |p|_+^{2-r}) }} \mathbf{m}(s , x ; p)^{-3} \right] \leq C e^{C \left( \ln L \right)^{(r-2) / r} } .
\end{equation*}
Combining the previous inequality (and noting that the term on the right-hand side grows subpolynomially fast) with~\eqref{eq:18042406}, we obtain
\begin{equation*}
    \E \left[ L^2 \mathbf{M} \exp \left( -  \frac{c T}{L^2 \mathbf{M}}   \right) \right] \leq C L^{3/2}  e^{C \left( \ln L \right)^{(r-2) / r} } \leq C L^{7/4}  \times \left( L^{-1/4} e^{C \left( \ln L \right)^{(r-2) / r} } \right) \leq C L^{7/4}.
\end{equation*}
Combining the previous inequality with~\eqref{eq:131633} and~\eqref{eq:18352406} completes the proof of Proposition~\ref{prop:sublincorr}.
\end{proof}

\section{Strict convexity of the surface tension}

This section is devoted to the study of the surface tension of the model. In particular, we show quantitatively the convergence of the finite-volume surface tension and establish its strict convexity (as stated in~\eqref{eq:conv.barsigma} of Theorem~\ref{main.thm}). We first recall the definition of the finite-volume surface tension (see Definition~\ref{def.surfacetension}): for any integer $L \in \N$ and any slope $p \in \Rd$
\begin{equation} \label{def.fintevolumesurfacetension}
    \bar \sigma_L(p) := \frac{1}{\left| \mathbb{T}_L \right|} \ln \left( \frac{Z_{L , p}}{Z_{L , 0}} \right) ~~\mbox{with}~~ Z_{L , p} := \int_{\Omega^\circ_L} \exp \left( - \sum_{x \in \mathbb{T}_L} V(p + \nabla \varphi(x)) \right) \, d\varphi.
\end{equation}
Let us remark that the finite-volume surface tension is twice-continuously differentiable (although its regularity can degenerate as $L$ tends to infinity). The following statement shows that it converges in the space $C^{1}_{\mathrm{loc}}(\Rd)$ to the infinite-volume surface tension.

\begin{proposition}[Quantitative convergence of the finite-volume surface tension] \label{prop:quantconvsurftens}
There exists a continuously differentiable convex function $\bar \sigma : \Rd \to \R$ and a constant $C := C(d , V) < \infty$ such that, for any slope $p \in \Rd$ and any $L \in \N$,
\begin{equation} \label{eq:quantconvsurftens}
    \left|  \bar \sigma_{L}(p) - \bar \sigma(p) \right| \leq \frac{C |p|^{r-1}_+ }{L^{1/8}} ~~\mbox{and}~~ \left| D_p  \bar \sigma_{L}(p) - D_p \bar \sigma(p) \right| \leq \frac{C |p|^{r-2}_+ }{L^{1/8}}.
\end{equation}
\end{proposition}

\begin{remark}
    Let us make two remarks about the previous proposition:
    \begin{itemize}
    \item It follows from the definition~\eqref{def.fintevolumesurfacetension} that $\bar \sigma_L(0) = 0$ for any $L \in \N$. From this observation, we deduce that that second inequality of~\eqref{eq:quantconvsurftens} implies the first one.
    \item The rate of convergence is once again not optimal, but we remark that Proposition~\ref{prop:quantconvsurftens} uses Proposition~\ref{prop:sublincorr} as an input, and that any improvement on the right-hand side of~\eqref{eq:sublinearitycorr} would yield an improved rate of convergence on the right-hand sides of~\eqref{eq:quantconvsurftens}.
    \end{itemize}
\end{remark}

The next proposition establishes the strict convexity of the finite-volume surface tension. The upper and lower bounds on the second derivative of the surface tension are uniform in the parameter $L$, and thus implies the strict convexity of the infinite-volume surface tension (stated in~\eqref{eq:conv.barsigma} of Theorem~\ref{main.thm} and in Proposition~\ref{prop:strictconvsurftens} below).

\begin{proposition}[Strict convexity of the finite-volume surface tensions] \label{prop:strictconvfinitevol}
There exist two constants  $\lambda_- := \lambda_- (d , V) > 0$ and $\lambda_+ := \lambda_+(d , V) < \infty$ such that, for any slope $p \in \Rd$, any $L \in \N$ and any $\lambda \in \Rd$,
\begin{equation} \label{eq:strictconvsurfens}
    \lambda_- |p|_+^{r-2} I_d \leq D^2_p \sigma_L (p) \leq \lambda_+ |p|_+^{r-2} I_d.
\end{equation}
\end{proposition}

As mentioned above, taking the limit $L \to \infty$, yields a similar result for the infinite-volume surface tension, with one important difference compared to the Proposition~\ref{prop:strictconvfinitevol}: while the finite-volume surface tensions are twice continuously differentiable (since the potential $V$ is assumed to satisfy this regularity property), we do not a priori know that the infinite-volume surface tension satisfies the same regularity property.  Nevertheless, convexity arguments allow to show that the infinite volume surface tension $\bar \sigma$ is $C^{1 , 1}(\Rd)$, which implies that the Hessian of $\bar \sigma$ is well-defined on a set of full measure and satisfies the same inequalities as~\eqref{eq:strictconvsurfens}.
    
We mention that, while it is not formally proved here, we believe that the techniques of~\cite{AW,armstrong2022quantitative} could be adapted to the present setting to establish the $C^2$-regularity of the infinite-volume surface tension.

We formalize in the statement below the discussion of the previous paragraph.

\begin{proposition}[Strict convexity of the surface tension] \label{prop:strictconvsurftens}
The following statements hold true:
\begin{itemize}
\item The function $p \mapsto D_p \bar \sigma(p)$ is differentiable almost everywhere; we denote its derivative by $D_p^2 \bar \sigma$.
\item There exist two constants $\lambda_- := \lambda_- (d , V) > 0$ and $\lambda_+ := \lambda_+(d , V) < \infty$ such that, for almost every slope $p \in \Rd$ and any $\lambda \in \Rd$,
\begin{equation} \label{ineq:strictconvexity}
    c |p|_+^{r-2} \left| \lambda \right|^2 \leq \left( \lambda , D^2_p \bar \sigma (p) \lambda \right) \leq C |p|_+^{r-2} \left| \lambda \right|^2.
\end{equation}
\end{itemize}
\end{proposition}

\subsection{Two identities for the finite-volume surface tension}

In order to prove Proposition~\ref{prop:quantconvsurftens} and Proposition~\ref{prop:strictconvfinitevol}, we first establish an identity which relates the gradient of the finite-volume surface tension to the Langevin dynamic. We recall that the notation $\E_{L,p}$ refers to the expectation under the Gibbs measure $\mu_{L , p}$ (see Definition~\ref{def:periodicGibbsmeasure}), and that the notation $w_{L , p , \lambda}$ refers to the function introduced in Proposition~\ref{prop:propsLangevin}.

\begin{proposition} \label{prop:identsecondder}
Given $p, \lambda \in \Rd$, the gradient of the finite-volume surface tension satisfies the identities, for any $L \in \N$, any $p \in \Rd$ and  any $(t , x) \in \R \times \mathbb{T}_L$,
\begin{align*}
    D_p \bar \sigma_L(p) & = \E_{L , p} \left[ D_p V(p + \nabla \varphi(x))\right] \\
                        & = \E \left[ D_p V(p + \nabla \varphi_L(t, x ; p))\right].
\end{align*}
For any $\lambda \in \Rd$, the Hessian of the finite-volume surface tension satisfies the identities
    \begin{align} \label{id:secondderivativefintevolsur}
        \left( \lambda , D^2_p \sigma_L (p) \lambda \right) & = \E \left[ \lambda \cdot \mathbf{A}(t , x ; p) \left( \lambda + \nabla w_{{L, p, \lambda}}(t , x) \right) \right] \\
        & = \E \left[ \left(\lambda + \nabla w_{{L, p, \lambda}}(t , x) \right) \cdot \mathbf{A}(t , x ; p) \left( \lambda + \nabla w_{{L, p, \lambda}}(t , x) \right) \right]. \notag
    \end{align}
\end{proposition}

\begin{proof}[Proof of Proposition~\ref{prop:strictconvfinitevol}]
    Let us fix an integer $L \in \N$ and recall the notation introduced in Section~\ref{subsec2.1.6}. We first differentiate the right-hand side of the identity~\eqref{def.fintevolumesurfacetension} and obtain the identity, for any $p \in \Rd$ and any $L \in \N$,
    \begin{equation*}
        D_p \bar \sigma_L(p) = \E_{L , p} \left[ \frac{1}{\left| \mathbb{T}_L \right|} \sum_{x \in \mathbb{T}_L} D_p V(p + \nabla \varphi(x))\right]. 
    \end{equation*}
    Using the translation invariance of the measure $\mu_{L , p}$, we may fix a vertex $x \in \mathbb{T}_L$ and rewrite the previous identity as follows
    \begin{equation*}
       D_p \bar \sigma_L(p)  = \E_{L , p} \left[ D_p V(p + \nabla \varphi(x))\right].
    \end{equation*}
    We next make use of the stationary Langevin dynamics introduced in Section~\ref{sec:introdLangevindyn}. Specifically, we observe that, by the first item of Proposition~\ref{prop:propsLangevin} ``Distribution", for any time $t \in \R$ and any $x \in \mathbb{T}_L$,
    \begin{equation*}
         D_p \bar \sigma_L(p) = \E \left[ D_p V(p + \nabla \varphi_L(t, x ; p))\right].
    \end{equation*}
    We then differentiate a second time the finite-volume surface tension and use the fourth item of Proposition~\ref{prop:propsLangevin} ``Differentiability with respect to the slope". We obtain
    \begin{equation*}
        \left( \lambda , D^2_p \sigma_L (p) \lambda \right) = \E \left[ \lambda \cdot \mathbf{A}(\cdot ; p) \left( \lambda + \nabla w_{L,p, \lambda} \right) \right],
    \end{equation*}
            where $w_{L,p, \lambda}$ is the stationary solution to the parabolic equation
    \begin{equation*}
        \partial_t  w_{L , p , \lambda} - \nabla \cdot \a(\cdot ; p )( \lambda +  \nabla w_{L , p , \lambda} ) = 0   \hspace{5mm} \mbox{in} \hspace{3mm}  \R \times \mathbb{T}_L.
    \end{equation*}
    There only remains to prove the second line of~\eqref{id:secondderivativefintevolsur}. It is equivalent to the identity: for any $(t , x) \in \R \times \mathbb{T}_L$,
    \begin{equation*}
        \E \left[ \nabla w_{{p, \lambda}}(t , x) \cdot \mathbf{A}(t , x ; p) \left( \lambda + \nabla w_{{p, \lambda}}(t , x) \right) \right] = 0.
    \end{equation*}
    Using the spatial stationarity of the Langevin dynamic and a discrete integration by parts, we have the identities
    \begin{align*}
        \E \left[ \nabla w_{{p, \lambda}}(t , x) \cdot \mathbf{A}(t , x ; p) \left( \lambda + \nabla w_{{p, \lambda}}(t , x) \right) \right] & = \frac{1}{\left| \mathbb{T}_L \right|} \E \left[ \sum_{y \in \mathbb{T}_L} \nabla w_{{p, \lambda}}(t , y) \cdot \mathbf{A}(t , y ; p) \left( \lambda + \nabla w_{{p, \lambda}}(t , y) \right) \right] \\
                            & = - \frac{1}{\left| \mathbb{T}_L \right|} \E \left[ \sum_{y \in \mathbb{T}_L} w_{{p, \lambda}}(t , y) \nabla \cdot \mathbf{A}(t , y ; p) \left( \lambda + \nabla w_{{p, \lambda}}(t , y) \right) \right] \\
                            & = - \frac{1}{\left| \mathbb{T}_L \right|} \E \left[ \sum_{y \in \mathbb{T}_L} w_{{p, \lambda}}(t , y)  \partial_t w_{{p, \lambda}}(t , y) \right] \\
                            & =  - \frac{1}{2}\E \left[  \partial_t w_{{p, \lambda}}(t , x)^2 \right],
    \end{align*}
    where in the last line, we used the spatial stationarity of the dynamic. The term on the right-hand side can be further simplified as follows (using the time stationarity to conclude in the third equality)
    \begin{align*}
    \E \left[ \nabla w_{{p, \lambda}}(t , x) \cdot \mathbf{A}(t , x ; p) \left( \lambda + \nabla w_{{p, \lambda}}(t , x) \right) \right] & =  - \frac{1}{2}\E \left[  \partial_t w_{{p, \lambda}}(t , x)^2 \right] =  - \frac{1}{2}\partial_t \E \left[   w_{{p, \lambda}}(t , x)^2 \right] =0.
    \end{align*}
\end{proof}

\subsection{Quantitative convergence of the finite-volume surface tensions}

This section is devoted to the proof of Proposition~\ref{prop:quantconvsurftens}. In the proof below, and to ensure that all the quantities are well-defined, we identify the vertices of the box $\Lambda_L$ with the ones of the torus $\mathbb{T}_L$, and the vertices of $\Lambda_{2L}$ the ones of $\mathbb{T}_{2L}$. This identification allows to identify the vertices of $\mathbb{T}_L$ with a subset of~$\mathbb{T}_{2L}$.

\begin{proof}[Proof of Proposition~\ref{prop:quantconvsurftens}]
We will prove the identity, for any integer $L \in \N$ and any slope $p \in \Rd$,
\begin{equation} \label{eq:finitevolLand2L}
        \left| D_p  \bar \sigma_{L}(p) - D_p \bar \sigma_{2L}(p) \right| \leq \frac{C |p|^{r-2}_+ }{L}.
\end{equation}
The inequality~\eqref{eq:finitevolLand2L} is sufficient to conclude as it implies that the sequence $n \mapsto D_p  \bar \sigma_{2^n L}(p)$ is Cauchy, and thus converges.

To prove~\eqref{eq:finitevolLand2L}, we first use Proposition~\ref{prop:identsecondder} and write, for any $(t , x) \in \R \times \Lambda_L$,
\begin{equation} \label{eq:difffintevolsurftens}
        D_p  \bar \sigma_{L}(p) - D_p \bar \sigma_{2L}(p) = \E \left[ D_p V(p + \nabla \varphi_L(t, x ; p)) - D_p V(p + \nabla \varphi_{2L}(t, x ; p))  \right].
\end{equation}
We then define the symmetric positive matrix and maximal eigenvalue
\begin{equation} \label{def.coeffAL}
    \left\{ \begin{aligned}
    \a_L(t , x ; p) & := \int_0^1 D_p^2 V(p + s \nabla \varphi_L(t, x ; p) + (1-s) \nabla \varphi_{2L}(t, x ; p)) \, ds, \\  
    \Lambda_{+ , L} (t , x ; p) & := \sup_{\substack{\xi \in \Rd \\ | \xi | \leq 1} } \xi \cdot \a_L(t , x ; p) \xi,
    \end{aligned} \right.
\end{equation}
and the function
\begin{equation} \label{def.coeffvL}
    v_{L , p} = \varphi_L(t, x ; p) - \varphi_{2L}(t, x ; p).
\end{equation}
Let us note that the map $v_{L , p}$ solves the parabolic equation
\begin{equation} \label{eq:15500703}
    \partial_t v_{L , p} - \nabla \cdot \a_L(\cdot , \cdot ; p)  \nabla v_{L , p}  = 0 ~~\mbox{in}~~ \R \times \Lambda_L.
\end{equation}
The identity~\eqref{eq:15500703} is obtained by subtracting the two Langevin dynamics used to define the functions $ \varphi_L(\cdot , \cdot ; p)$ and $ \varphi_{2L}(\cdot , \cdot ; p)$, and by noticing that, since the dynamics are driven by the same Brownian motions, these terms disappear when taking the difference. 

Using the notation~\eqref{def.coeffAL} and~\eqref{def.coeffvL}, we may rewrite the identity~\eqref{eq:difffintevolsurftens} as follows, for any $(t , x) \in \R \times \Lambda_L$,
\begin{equation*}
        D_p  \bar \sigma_{L}(p) - D_p \bar \sigma_{2L}(p) := \E \left[ \a_L(t , x ; p) \nabla v_{L , p} (t , x) \right].
\end{equation*}
The Cauchy-Schwarz inequality implies that
\begin{equation} \label{eq:17110703}
    \left| D_p  \bar \sigma_{L}(p) - D_p \bar \sigma_{2L}(p) \right| \leq \sqrt{\E \left[ \Lambda_{+ , L} (t , x ; p) \right] \E \left[ \nabla v_{L , p}(t , x) \cdot \a_L(t , x ; p) \nabla v_{L , p}(t , x) \right]}.
\end{equation}
The first term on the right-hand side is bounded by a constant which does not depend on $L$. Specifically, by Proposition~\ref{prop.prop2.3} and the growth 
Assumption~\eqref{AssPot} on the Hessian of $V$
\begin{equation} \label{eq:real5.12}
    \E \left[ \Lambda_{+ , L} (t , x ; p) \right] \leq C \E \left[ |p + \nabla \varphi_{L}(t , x ; p)|^{r-2} \right] \leq C |p|^{r-2}.
\end{equation}
For the second term on the right-hand side of~\eqref{eq:17110703}, we use the space and time stationarity of the dynamic and combine it with the Caccioppoli inequality to obtain
\begin{align*}
    \lefteqn{ \E \left[ \nabla v_{L , p}(t , x) \cdot \a_L(t , x ; p) \nabla v_{L , p}(t , x) \right] } \qquad & \\ 
    & = \frac{1}{(L/2)^2 \left| \Lambda_{L/2}\right|} \E \left[ \int_{-(L/2)^2}^0 \sum_{y \in \Lambda_{L/2}}  \nabla v_{L , p}(s , y) \cdot \a_L(s , y ; p) \nabla v_{L , p}(s , y) \, ds \right] \\
    & \leq \frac{C}{L^2 \left| \Lambda_{L/2}\right|} \frac{1}{L^2} \E \left[ \int_{-L^2}^0 \sum_{y \in \Lambda_{L}}  \Lambda_{+ , L} (t , x ; p)  \left|v_{L , p}(s , y) \right|^2 \, ds \right] \\
    & \leq \frac{C}{L^2}  \E \left[  \Lambda_{+ , L} (t , x ; p)  \left|v_{L , p}(t , x) \right|^2 \, ds \right],
\end{align*}
where we used the space and time stationarity in the last line. The last term on the right-hand side can be estimated using Proposition~\ref{prop.prop2.3} and the Assumption~\eqref{AssPot} (for the term $\Lambda_{+ , L}$) and Proposition~\ref{prop:sublincorr} (for the term $v_{L , p}$). We obtain
\begin{equation*}
    \E \left[ \nabla v_{L , p}(t , x) \cdot \a_L(t , x ; p) \nabla v_{L , p}(t , x) \right] \leq \frac{C |p|^{r-2} (L^{7/8})^2}{L^2} \leq \frac{C |p|^{r-2}}{L^{1/8}}.
\end{equation*}
Combining the previous inequality with~\eqref{eq:17110703} and~\eqref{eq:real5.12} completes the proof of Proposition~\ref{prop:strictconvfinitevol}.
\end{proof}

\subsection{Strict convexity of the surface tension}

This section is devoted to the proof of Proposition~\ref{prop:strictconvfinitevol}. We note that the techniques used in this proof are similar to the ones of the article of Biskup and Rodriguez~\cite{biskup2018limit}.

\begin{proof}[Proof of Proposition~\ref{prop:strictconvfinitevol}]
We fix $L \in \N$ and $p , \lambda \in \Rd$ and recall the (first) identity of~\eqref{id:secondderivativefintevolsur}
\begin{equation*}
     \left( \lambda , D^2_p \sigma_L (p) \lambda \right) = \E \left[ \lambda \cdot \mathbf{A}(t, x ; p) \left( \lambda + \nabla w_{{L, p, \lambda}}(t,x) \right) \right].
\end{equation*}
We next split the argument into three steps. The first step provides some useful estimates on the $L^2$-norm of the gradient of the map $w_{p , \lambda}$, the second step contains the proof of the upper bound of~\eqref{eq:strictconvsurfens} and the third step is devoted to the proof of the lower bound.

\medskip

\textit{Step 1. Estimates for the gradient of the map $w_{p , \lambda}$.}

\medskip

In this step, we prove the two following estimates: there exists a constant $C := C(d , V) < \infty$ such that, for any $(t , x) \in \R \times \mathbb{T}_L$,
    \begin{equation} \label{ineqseconddernotmoder}
        \E \left[ \nabla w_{L, p, \lambda} (t , x) \cdot  \mathbf{A}(t , x ; p) \nabla w_{L, p, \lambda} (t , x)  \right] \leq C |p|_+ ^{r-2}|\lambda|^2,
    \end{equation}
    and, for any $(t , x) \in \R \times \mathbb{T}_L$,
    \begin{equation} \label{ineqseconddermoder}
        \E \left[ \mathbf{m}(t , x ; p) \left| \lambda +  \nabla w_{L, p, \lambda} (t , x) \right|^2 \right] \leq C|\lambda|^2.
    \end{equation}
We first collect a few properties of the map $w_{L , p, \lambda}$. First, by the two identities of~\eqref{id:secondderivativefintevolsur}, we have, for any $(t , x) \in \R \times \mathbb{T}_L,$
    \begin{equation*}
        \E \left[ \nabla w_{L, p, \lambda} (t , x) \cdot  \mathbf{A}(t , x ; p) \nabla w_{L, p, \lambda} (t , x)  \right] =  - \E \left[ \lambda \cdot  \mathbf{A}(t , x ; p) \nabla w_{L, p, \lambda} (t , x)  \right].
    \end{equation*}
   The right-hand side of the previous identity can be estimated by the Cauchy-Schwarz inequality
    \begin{equation*}
    \left| \E \left[ \lambda \cdot  \mathbf{A}(t , x ; p) \nabla w_{L, p, \lambda} (t , x)  \right] \right| \leq \sqrt{\E \left[ \lambda \cdot  \mathbf{A}(t , x ; p) \lambda   \right] \E \left[ \nabla w_{L, p, \lambda} (t , x)  \cdot  \mathbf{A}(t , x ; p) \nabla w_{L, p, \lambda} (t , x)  \right]}.
    \end{equation*}
    By Proposition~\ref{prop.prop2.3} and the growth Assumption~\eqref{AssPot} on the Hessian of $V$, we have
    \begin{equation*}
        \E \left[ \lambda \cdot  \mathbf{A}(t , x ; p) \lambda   \right] \leq \E \left[ \Lambda_+(t , x ; p) \right] \left| \lambda \right|^2 \leq C |p|_+^{r-2} \left| \lambda \right|^2.
    \end{equation*}
    Combining the three previous displays, we obtain the inequality
    \begin{equation*}
         \E \left[ \nabla w_{L, p, \lambda} (t , x) \cdot  \mathbf{A}(t , x ; p) \nabla w_{L, p, \lambda} (t , x)  \right]  \leq C |p|_+^{r-2} \left| \lambda \right|^2.
    \end{equation*}
   We next prove the inequality~\eqref{ineqseconddermoder}. To this end, we combine the previous inequality with Proposition~\ref{prop4.3} and obtain
    \begin{multline*}
        \mathbf{m}(t , x ; p)^2 \left| \lambda + \nabla w_{L, p, \lambda} (t , x) \right|^2 \\
        \leq C \sum_{y \sim_2 x}  \int_t^\infty K_{|p|_+^{r-2} (t - s)} \left( \lambda + \nabla w_{L, p, \lambda} (t , y) \right) \cdot \a(s , y ; p) \left( \lambda + \nabla w_{L, p, \lambda} (t , y) \right) \, ds.
    \end{multline*}
    Taking the expectation on both sides and using the space and time stationarity of the map $w_{L,p , \lambda}$, we obtain the upper bound
    \begin{align*}
            \E \left[ \mathbf{m}(t , x ; p)^2 \left| \lambda + \nabla w_{L, p, \lambda} (t , x) \right|^2 \right] & \leq \frac{C}{|p|_+^{r-2}} \E \left[ \left( \lambda + \nabla w_{L, p, \lambda} (t , x) \right) \cdot \mathbf{A}(t , x ; p)  \left( \lambda + \nabla w_{L, p, \lambda} (t , x) \right) \right] \\
            & \leq C |\lambda|^2.
    \end{align*}
    This is~\eqref{ineqseconddermoder}.

\medskip

\textit{Step 2. Upper bound of~\eqref{eq:strictconvsurfens}.}

\medskip

We start from the inequality~\eqref{ineqseconddernotmoder} and apply the Cauchy-Schwarz inequality as follows
\begin{align} \label{eq:upperboundsecondder}
     \left( \lambda , D^2_p \sigma_L (p) \lambda \right) & = \E \left[ \lambda \cdot \mathbf{A}(t, x ; p) \left( \lambda + \nabla w_{{L , p, \lambda}}(t,x) \right) \right] \\
     & \leq \sqrt{\E \left[ \mathbf{m}(t  ,x ; p)^{-1} \left| \mathbf{A}(\cdot ; p) \lambda \right|^2  \right] \E \left[  \mathbf{m}(t  ,x ; p) \left| \lambda + \nabla w_{{L , p, \lambda}} \right|^2 \right]}. \notag
\end{align}
We then estimate the two terms on the right-hand side. For the first one, we use the Cauchy-Schwarz inequality a second time
\begin{equation*}
    \E \left[ \mathbf{m}(t  ,x ; p)^{-1} \left| \mathbf{A}(\cdot ; p) \lambda \right|^2  \right] \leq \sqrt{\E \left[ \mathbf{m}(t  ,x ; p)^{-2} \right] \E \left[ \left| \mathbf{A}(\cdot ; p) \lambda \right|^4  \right]}.
\end{equation*}
The first term is estimated by Proposition~\ref{prop:stochintmoderated} (and Remark~\ref{Rem:remark3.13} which provides moments estimate on the moderated environment), and the second term is estimated thanks to Proposition~\ref{prop.prop2.3} and the growth Assumption~\eqref{AssPot} on the Hessian of $V$. We obtain
\begin{equation*}
        \E \left[ \mathbf{m}(t  ,x ; p)^{-1} \left| \mathbf{A}(\cdot ; p) \lambda \right|^2  \right] \leq C |p|^{2(r-2)} |\lambda|^{2}.
\end{equation*}
The second term of~\eqref{eq:upperboundsecondder} is estimated more directly by using~\eqref{ineqseconddermoder}
\begin{equation*}
    \E \left[  \mathbf{m}(t  ,x ; p) \left| \lambda + \nabla w_{{L , p, \lambda}} \right|^2 \right] \leq C \left| \lambda \right|^2.
\end{equation*}
Combining the two previous displays, we obtain that
\begin{equation*}
    \left( \lambda , D^2_p \sigma_L (p) \lambda \right) \leq C |p|^{r-2} \left| \lambda \right|^2,
\end{equation*}
which is the desired inequality.

\medskip

\textit{Step 3. Lower bound of~\eqref{eq:strictconvsurfens}.}

\medskip

We now prove the lower bound of Proposition~\ref{prop:strictconvfinitevol}. We first use Proposition~\ref{prop4.3} to write, for any $(t , x) \in \R \times \mathbb{T}_L$,
\begin{multline*}
     \mathbf{m}(t , x ; p) \left| \lambda + \nabla w_{{L, p, \lambda}}(t , x) \right|^2 \\
     \leq C \sum_{y \sim_2 x}  \int_t^\infty K_{|p|^{r-2}_+ (t - s)} \left(\lambda  + \nabla w_{{L,p, \lambda}}(s , y) \right) \cdot \a(s , y ; p) \left( \lambda + \nabla w_{{L, p, \lambda}}(s , y)\right) \, ds.
\end{multline*}
Taking the expectation on both sides and using the space and time stationarity of the map $w_{{p, \lambda}}$, we obtain the inequality
\begin{equation*}
    \E \left[ \mathbf{m}(t , x ; p) \left| \lambda + \nabla w_{{L, p, \lambda}}(t , x) \right|^2 \right] \leq \frac{C}{|p|_+^{r-2}} \E \left[ \left( \lambda + \nabla w_{L, p, \lambda} (t , x) \right) \cdot \mathbf{A}(t , x ; p)  \left( \lambda + \nabla w_{L, p, \lambda} (t , x) \right) \right].
\end{equation*}
Using the identity~\eqref{id:secondderivativefintevolsur}, we deduce that
\begin{equation*}
    \E \left[ \mathbf{m}(t , x ; p) \left| \lambda + \nabla w_{{L, p, \lambda}}(t , x) \right|^2 \right] \leq \frac{C}{|p|_+^{r-2}}  \left( \lambda , D^2_p \sigma_L (p) \lambda \right).
\end{equation*}
We next lower bound the term on the left-hand side using the spatial stationarity of the function $w_{L , p , \lambda}$ (which implies that the expectation of its gradient is equal to $0$), Proposition~\ref{prop:stochintmoderated}, the triangle and Cauchy-Schwarz inequalities as follows
\begin{align*}
 \left| \lambda \right|^2 & = \left| \E \left[ \lambda + \nabla w_{{L, p, \lambda}}(t , x) \right] \right|^2 \\
                        & \leq \E \left[  \left| \lambda + \nabla w_{{L, p, \lambda}}(t , x) \right| \right]^2  \\
                        & \leq \E \left[ \mathbf{m}(t , x ; p)^{-1} \right] \E \left[ \mathbf{m}(t , x ; p) \left| \lambda + \nabla w_{{L, p, \lambda}}(t , x) \right|^2 \right] \\
                        & \leq \frac{C}{|p|_+^{r-2}}  \left( \lambda , D^2_p \sigma_L (p) \lambda \right).
\end{align*}
Multiplying both sides of the previous inequality by $|p|_+^{r-2}/C$ completes the proof of the lower bound of~\eqref{eq:strictconvsurfens}.
\end{proof}

\section{Sublinearity of the flux of the Langevin dynamic} \label{sec:section5sublinearityflux}

In this section, we build upon the results established in the previous sections to derive a quantitative estimate on the weak norm of the flux of the Langevin dynamic. We recall the notation for the parabolic cylinder $Q_L := (- L^2 , 0) \times \Lambda_L$ and of the norm $\underline{W}^{-1 , r}_{\mathrm{par}} (Q_L)$ introduced in Section~\ref{subsecfunctions}.

\begin{proposition}[Sublinearity for the Langevin dynamic and its flux] \label{prop:sublinearityflux}
There exists a constant $C := C(d , V) < \infty$ such that the following estimate holds: for any integer $L \in \N$ and any slope $p \in \Rd$,
    \begin{equation*}
        \left\| D_p V (p + \nabla \varphi_{L}(\cdot , \cdot ; p)) - D_p \bar \sigma_L(p) \right\|_{\underline{W}^{-1 , r}_{\mathrm{par}} (Q_L)} \leq  \mathcal{O}_{1/2} (C |p|_+^{r-1} L^{15/16}) .
    \end{equation*}
\end{proposition}

\begin{remark}
As in Section~\ref{sec:section5sublinearity}, the exponent $15/16$ and the stochastic integrability are not optimal (but the principal feature of the previous inequality is that $15/16$ is strictly smaller than $1$ which is sufficient to establish Theorem~\ref{main.thm}).
\end{remark}

\begin{remark} \label{rem:remark6.3}
    In this section, we adopt the following point of view: we see the stationary Langevin dynamic (as introduced in Definition~\ref{Prop:Langevin}) as a function defined on the box $\Lambda_L$ which solves the system of stochastic differential equations~\eqref{eq:def.Langevintorus} in the box $\Lambda_L$ with periodic boundary conditions. 

    We may then extend this definition of stationary Langevin dynamics to more general boxes: given a vertex $y \in \mathbb{Z}^d$ and an integer $L \in \N$, we let $\tilde B_t^y(x) :=B_t(x) - \frac{1}{\left| \Lambda_L \right|} \sum_{x_1 \in (y + \Lambda_L)} B_t(x_1)$ and then let $\varphi_{y , L} : \R \times (y + \Lambda_L) \to \R$ be the unique solution to the system of stochastic differential equations (with periodic boundary conditions)
    \begin{equation*}
     d \varphi_{y,L}(t , x ;p) = \nabla \cdot D_p V(p + \nabla \varphi_{y,L}(\cdot , \cdot ; p)) (t , x) \, dt  + \sqrt{2} d \tilde{B}_t^y(x) \hspace{5mm} \mbox{for} \hspace{5mm} (t , x) \in \R \times (y + \Lambda_L),
\end{equation*}
     which satisfies (i) and (ii) of Definition~\ref{Prop:Langevin}. This formalism allows to consider overlapping boxes (while it is not necessarily clear how one can make sense of overlapping tori).
\end{remark}

The proof of Proposition~\ref{prop:sublinearityflux} contains two main steps. We first prove that the space-time average of the flux of the dynamic (i.e., the quantity $\left( D_p V (p + \nabla \varphi_{L}(\cdot , \cdot ; p)) \right)_{Q_{\ell}}$ for $L , \ell \in \N$ with $\ell \leq L$) is well-approximated by a sum of independent random variables and then use a concentration inequality to show that it is concentrated around its mean (which is equal to $D_p \bar \sigma_L(p)$). This is the purpose of Section~\ref{sec:locflux} and Section~\ref{section:concentrationflux}.

We then combine this result with the multiscale Poincar\'e inequality (Proposition~\ref{prop.multiscalePoinc}) and Proposition~\ref{prop:quantconvsurftens} to deduce Proposition~\ref{prop:sublinearityflux}. This final step is carried out in Section~\ref{sec:section6.3}.

\subsection{Localization for the flux} \label{sec:locflux}

An important part of the proof of Proposition~\ref{prop:sublinearityflux} is to establish the following result asserting that the spatial-time average of the flux of the dynamic computed in two different (but overlapping) parabolic cylinders is small. This estimate is then used (in Section~\ref{section:concentrationflux} below) to show that the space-time average of the flux of the dynamic is well-approximated by a sum of independent random variables, and thus concentrated around its mean.

\begin{lemma} \label{lemm:lemmacouplingtwoboxes}
There exists a constant $C:= C(d , V) < \infty$ such that the following holds. For an slope $p \in \Rd$ and any triplet $(y_1, L_1) , (y_2, L_2), (y , \ell) \in \Zd \times \N$ such that $(y + \Lambda_{2\ell}) \subseteq (y_1 + \Lambda_{L_1})$ and $(y + \Lambda_{2\ell}) \subseteq (y_2 + \Lambda_{L_2})$, the following inequality holds
\begin{equation*}
    \left| \left( D_p V \left( p + \nabla \varphi_{y_1, L_1} (\cdot , \cdot ; p) \right) \right)_{y + Q_\ell} - \left( D_p V \left( p + \nabla \varphi_{y_2, L_2} (\cdot , \cdot ; p) \right) \right)_{y + Q_\ell} \right| \leq  \mathcal{O}_{1/2} \left( C |p|^{r-2}_+ \frac{(L_1 + L_2)^{7/8}}{\ell} \right).
\end{equation*}
\end{lemma}

\begin{proof}
    The beginning of the proof is similar to the one of Proposition~\ref{prop:quantconvsurftens}. To simplify the notation in the argument below, we will assume that $y = y_1 = y_2 = 0$. We then fix a slope $p \in \Rd$ and the values of the three integers $L_1 , L_2, \ell$ such that they satisfy the assumption of the statement of the lemma, and define the symmetric positive matrix and maximal eigenvalue
\begin{equation*}
    \left\{ \begin{aligned}
    \overline{\a}(t , x ; p) & := \int_0^1 D_p^2 V(p + s \nabla \varphi_{L_1}(t, x ; p) + (1-s) \nabla \varphi_{L_2}(t, x ; p)) \, ds \\  
    \overline{\Lambda}_{+} (t , x ; p) & := \sup_{\substack{\xi \in \Rd \\ | \xi | \leq 1} } \xi \cdot \overline{\a}(t , x ; p) \xi
    \end{aligned} \right.
\end{equation*}
and the function
\begin{equation*}
    v(t , x) = \varphi_{L_1}(t, x ; p) - \varphi_{L_2}(t, x ; p) ~~\mbox{for}~~ (t , x) \in \R \times \Lambda_{2 \ell}.
\end{equation*}
Let us note that the map $v$ solves the parabolic equation
\begin{equation*}
    \partial_t v - \nabla \cdot  \overline{\a}(\cdot , \cdot ; p)  \nabla v  = 0 ~~\mbox{in}~~ \R \times \Lambda_{2\ell},
\end{equation*}
and that we have the identity
\begin{align*}
    \lefteqn{\left( D_p V \left( p + \nabla \varphi_{L_1} (\cdot , \cdot ; p) \right) \right)_{Q_\ell} - \left( D_p V \left( p + \nabla \varphi_{L_2} (\cdot , \cdot ; p) \right) \right)_{Q_\ell}} \qquad & \\  & 
    = \frac{1}{\left|Q_\ell \right|} \int_{I_\ell} \sum_{x \in \Lambda_\ell}  D_p V \left( p + \nabla \varphi_{L_1} (\cdot , \cdot ; p)  \right) - D_p V \left( p + \nabla \varphi_{L_2} (\cdot , \cdot ; p)  \right) \, dt \\
    & = \frac{1}{\left|Q_\ell \right|} \int_{I_\ell} \sum_{x \in \Lambda_\ell}  \overline{\a}(t  , x ; p)  \nabla v(t , x) \, dt.
\end{align*}
We next estimate the term on the right-hand side by applying first the Cauchy-Schwarz inequality and then the Caccioppoli inequality (Proposition~\ref{prop:paraboliccacciop}). We obtain
\begin{align*}
    \left( \int_{I_\ell} \sum_{x \in \Lambda_\ell}  \overline{\a}(t  , x ; p)  \nabla v(t , x) \, dt \right)^2 & \leq \int_{I_\ell} \sum_{x \in \Lambda_\ell} \overline{\Lambda}_{+} (t , x ; p) \, dt \times \int_{I_\ell} \sum_{x \in \Lambda_\ell} \nabla v(t , x) \cdot \overline{\a}(t  , x ; p)  \nabla v(t , x) \, dt \\
    & \leq \frac{C}{\ell^2} \int_{I_\ell} \sum_{x \in \Lambda_\ell} \overline{\Lambda}_{+} (t , x ; p) \, dt \times  \int_{I_{2\ell}} \sum_{x \in \Lambda_{2\ell}} \overline{\Lambda}_{+}(t , x ; p) \left| v(t , x) \right|^2  \, dt.
\end{align*}
From Proposition~\ref{prop.prop2.3} and the growth Assumption~\eqref{AssPot} (for the term $\overline{\Lambda}_{+}(t , x ; p)$) and Proposition~\ref{prop:sublincorr} (for the term $|v(t , x)|$) together with Proposition~\ref{prop:prop2.20} ``Summation", ``Integration" and ``Product", we obtain that
\begin{equation*}
    \left| \left( D_p V \left( p + \nabla \varphi_{L_1} (\cdot , \cdot ; p) \right) \right)_{Q_\ell} - \left( D_p V \left( p + \nabla \varphi_{L_2} (\cdot , \cdot ; p) \right) \right)_{Q_\ell} \right| \leq \mathcal{O}_{1/2} \left( C |p|^{r-2}_+ \frac{(L_1 + L_2)^{7/8}}{\ell} \right).
\end{equation*}
\end{proof}

\subsection{Concentration for the space-time average of the flux} \label{section:concentrationflux}

In this section, we use the result of Lemma~\ref{lemm:lemmacouplingtwoboxes} to show that the space-time average of the flux $\left( D_p V \left( p + \nabla \varphi_{L} (\cdot , \cdot ; p) \right) \right)_{Q_\ell}$ is well approximated by a sum of independent random variable, and then apply a concentration inequality (specifically, the one stated in Proposition~\ref{prop:prop2.20} ``Concentration") to deduce that it is concentrated around its mean. 

\begin{lemma} \label{lemm:concenineqflux}
There exists a constant $C:= C(d , V) < \infty$ such that the following holds. For $p \in \Rd$ and any $L , \ell \in \N$ with $\ell \leq L$,
\begin{equation*}
    \left| \left( D_p V \left( p + \nabla \varphi_{L} (\cdot , \cdot ; p) \right) \right)_{Q_\ell} - D_p \bar \sigma_L(p)  \right| \leq  \mathcal{O}_{1/2} \left( C |p|^{r-1}_+ \left( \frac{ L^{7/8}}{\ell} \right)^{\frac 12} \right).
\end{equation*}
\end{lemma}

\begin{remark}
This estimate only provides valuable information about the space-time average of the flux of the Langevin dynamic when $\ell$ is large (and specifically, when $\ell \in (L^{7/8} , L)$). For smaller values of $\ell$, we will use the following inequality
\begin{equation} \label{lemm:concenineqfluxsmallscales}
        \left| \left( D_p V \left( p + \nabla \varphi_{L} (\cdot , \cdot ; p) \right) \right)_{Q_\ell} - D_p \bar \sigma_L(p)  \right| \leq  \mathcal{O}_{1} \left( C |p|^{r-1}_+ \right),
\end{equation}
which follows from the Assumption~\eqref{AssPot} on the potential $V$ and the stochastic integrability estimate stated in Proposition~\ref{prop.prop2.3}.
\end{remark}

\begin{proof}
Let us fix two integers $L , \ell \in \N$ such that $L^{7/8} \leq \ell \leq L$ (this can be done without loss of generality as~\eqref{lemm:concenineqfluxsmallscales} can be used if $\ell \leq L^{7/8}$). We then select an integer $\ell_1 \leq \ell$ whose explicit value will be decided later in the proof, and partition the parabolic cylinder $Q_\ell$ into parabolic cylinders of the form $z + Q_{\ell_1}$ with $z \in \mathcal{Z}$ (for some finite suitably chosen set $\mathcal{Z} \subseteq Q_\ell$ whose cardinality is of order $(\ell/\ell_1)^{d}$). For each $z = (s , y) \in \mathcal{Z}$, we will use the shorthand notation (N.B. the function $\varphi_{z}$ is defined on the set $\R \times (y + \Lambda_{2\ell_1})$ and valued in $\R$)
\begin{equation*}
    \varphi_{z} (\cdot , \cdot ; p) :=  \varphi_{y , 2\ell_1} (\cdot , \cdot ; p).
\end{equation*}
We then write
\begin{align} \label{eq:concetrationflux}
    \lefteqn{\left| \left( D_p V \left( p + \nabla \varphi_{L} (\cdot , \cdot ; p) \right) \right)_{Q_\ell} - D_p \bar \sigma_L(p) \right|} \qquad & \\ & 
    \leq  \underset{\eqref{eq:concetrationflux}-(i)}{\underbrace{\left| \left( D_p V \left( p + \nabla \varphi_{L} (\cdot , \cdot ; p) \right) \right)_{Q_\ell} - \frac{1}{\left| \mathcal{Z} \right|} \sum_{z \in \mathcal{Z}} \left( D_p V \left( p + \nabla \varphi_{z} (\cdot , \cdot ; p) \right) \right)_{z + Q_{\ell_1}}  \right|}} \notag \\
    & \quad + \underset{\eqref{eq:concetrationflux}-(ii)}{\underbrace{\left|  \frac{1}{\left| \mathcal{Z} \right|} \sum_{z \in \mathcal{Z}} \left( \left( D_p V \left( p + \nabla \varphi_{z} (\cdot , \cdot ; p) \right) \right)_{z + Q_{\ell_1}}  - D_p \bar \sigma_{2\ell_1}(p) \right) \right|  }} \notag \\ 
    &  \quad +  \underset{\eqref{eq:concetrationflux}-(iii)}{\underbrace{\left| D_p \bar \sigma_{2\ell_1}(p) - D_p \bar \sigma_L(p) \right|}}. \notag
\end{align}
We next estimate the three terms on the right-hand side separately (showing that they are all small). For the term \eqref{eq:concetrationflux}-(i), we use with the identity
\begin{equation*}
    \left( D_p V \left( p + \nabla \varphi_{L} (\cdot , \cdot ; p) \right) \right)_{Q_\ell} = \frac{1}{\left| \mathcal{Z} \right|} \sum_{z \in \mathcal{Z}} \left( D_p V \left( p + \nabla \varphi_{L} (\cdot , \cdot ; p) \right) \right)_{z + Q_{\ell_1}}
\end{equation*}
together with Lemma~\ref{lemm:lemmacouplingtwoboxes} and Proposition~\ref{prop:prop2.20} ``Summation" to obtain
\begin{align} \label{ineq:5111}
    \eqref{eq:concetrationflux}-(i)
    & \leq \frac{1}{\left| \mathcal{Z} \right|} \sum_{z \in \mathcal{Z}} \left| \left( D_p V \left( p + \nabla \varphi_{L} (\cdot , \cdot ; p) \right) \right)_{z + Q_{\ell_1}} - \left( D_p V \left( p + \nabla \varphi_{z} (\cdot , \cdot ; p) \right) \right)_{z + Q_{\ell_1}} \right| \\
    & \leq \mathcal{O}_{1/2} \left( C |p|^{r-2}_+ \frac{ L^{7/8}}{\ell_1} \right). \notag
\end{align}
For the term~\eqref{eq:concetrationflux}-(ii), we will prove the inequality
\begin{equation} \label{ineq:5112}
    \eqref{eq:concetrationflux}-(ii) \leq \mathcal{O}_1 \left( C \left( \frac{\ell_1 }{\ell} \right)^{\frac d2} \right).
\end{equation}
To establish~\eqref{ineq:5112}, we would like to argue that the terms in the sum on the left of~\eqref{eq:concetrationflux}-(ii) are independent random variables and apply the concentration inequality of Proposition~\ref{prop:prop2.20} ``Concentration". This strategy faces the following technical issue: the terms of the sum of~\eqref{eq:concetrationflux}-(ii) are not all independent since two cylinders of the form $(z + Q_{2 \ell_1})$ and $(z' + Q_{2 \ell_1})$ with $z , z' \in \mathcal{Z}$ may have some overlap. 

To overcome this difficulty, we refine the partition of the cylinder $Q_\ell$ as follows: we let $\mathcal{Z}_1 , \ldots, \mathcal{Z}_{2d}$ be a partition of the set $\mathcal{Z}$ (each one of them containing roughly $\left| \mathcal{Z} \right| / (2d)$ points) satisfying the following property:
\begin{equation} \label{eq:15182606}
    \forall i \in \{ 1 , \ldots, 2d \}, \, \forall z , z' \in \mathcal{Z}_i~\mbox{with}~ z = (s , y) ~\mbox{and}~ z' = (s' , y'), ~ (y + \Lambda_{2 \ell_1}) \cap (y' + \Lambda_{2 \ell_1}) = \emptyset.
\end{equation}
Using this refined partition, we may rewrite the sum as follows
\begin{equation*}
 \frac{1}{\left| \mathcal{Z} \right|} \sum_{z \in \mathcal{Z}} \left( D_p V \left( p + \nabla \varphi_{z} (\cdot , \cdot ; p) \right) \right)_{z + Q_{\ell_1}} 
 = \frac{1}{2d} \sum_{i = 1}^{2d} \frac{1}{\left| \mathcal{Z}_i \right|}  \sum_{z \in \mathcal{Z}_i} \left( D_p V \left( p + \nabla \varphi_{z + Q_{2\ell_1}} (\cdot , \cdot ; p) \right) \right)_{z + Q_{\ell_1}} .
\end{equation*}
We then fix an integer $i \in \{ 1 , \ldots, 2d \}$ and show the inequality
\begin{equation} \label{ineq:51133}
    \frac{1}{\left| \mathcal{Z}_i \right|}  \sum_{z \in \mathcal{Z}_i} \left( \left( D_p V \left( p + \nabla \varphi_{z} (\cdot , \cdot ; p) \right) \right)_{z + Q_{\ell_1}}  - D_p \bar \sigma_{2\ell_1}(p) \right) \leq  \mathcal{O}_1 \left( C \left( \frac{\ell_1 }{\ell} \right)^{\frac d2} \right).
\end{equation}
The inequality~\eqref{ineq:5112} then follows by applying Proposition~\ref{prop:prop2.20} ``Summation".

To prove~\eqref{ineq:51133}, we note that the random variables inside the sum are independent (this is a consequence of the property~\eqref{eq:15182606}), that, by the space and time stationarity of the Langevin dynamic, for any $z \in \mathcal{Z}_i$,
\begin{equation*}
    \E \left[ \left( D_p V \left( p + \nabla \varphi_{z} (\cdot , \cdot ; p) \right) \right)_{z + Q_{\ell_1}} \right] =  D_p \bar \sigma_{2\ell_1}(p),
\end{equation*}
and that, by Proposition~\ref{prop.prop2.3} and Proposition~\ref{prop:prop2.20} (since we are averaging the flux over parabolic cylinders)
\begin{equation*}
    \left| \left( D_p V \left( p + \nabla \varphi_{z} (\cdot , \cdot ; p) \right) \right)_{z + Q_{\ell_1}} \right| \leq \mathcal{O}_{r/(r-1)} (C|p|^{r-1}_+).
\end{equation*}
We can thus apply Proposition~\ref{prop:prop2.20} ``Concentration" and obtain
\begin{equation*}
    \left| \frac{1}{\left| \mathcal{Z}_i \right|}  \sum_{z \in \mathcal{Z}_i} \left( D_p V \left( p + \nabla \varphi_{z + Q_{2\ell_1}} (\cdot , \cdot ; p) \right) \right)_{z + Q_{\ell_1}}  - D_p \bar \sigma_{2\ell_1}(p) \right| \leq  \mathcal{O}_{r/(r-1)} \left( \frac{C|p|^{r-1}_+}{|\mathcal{Z}_i|^{1/2}} \right).
\end{equation*}
Using that the cardinality of the set $\mathcal{Z}_i$ is of the same order as the cardinality of $\mathcal{Z}$, which is itself of order $(\ell / \ell_1)^d$, we obtain~\eqref{ineq:51133}.

The term \eqref{eq:concetrationflux}-(iii) is the simplest to estimate, and it follows from Proposition~\ref{prop:quantconvsurftens} that
\begin{equation} \label{ineq:5113}
    \eqref{eq:concetrationflux}-(iii) \leq \frac{C}{\ell_1}.
\end{equation}
Combining the estimates~\eqref{ineq:5111},~\eqref{ineq:5112} and~\eqref{ineq:5113}, we have obtained
\begin{align*}
    \left| \left( D_p V \left( p + \nabla \varphi_{L} (\cdot , \cdot ; p) \right) \right)_{Q_\ell} - D_p \bar \sigma_L(p) \right| & \leq  \mathcal{O}_{1/2} \left( \frac{ C L^{7/8}}{\ell_1} \right) + \mathcal{O}_{r/(r-1)} \left( C \left( \frac{\ell_1 }{\ell} \right)^{\frac d2} \right) + \frac{C}{\ell_1} \\
    & \leq \mathcal{O}_{1/2} \left( C |p|^{r-2}_+ \frac{ L^{7/8}}{\ell_1}  +  C |p|^{r-1}_+ \left( \frac{\ell_1 }{\ell} \right)^{\frac d2} \right) + \frac{C |p|^{r-1}_+}{\ell_1}.
\end{align*}
We then simplify the right-hand side (at the cost of being suboptimal) by using the inequalities $1 \leq L^{7/8}$, $|p|^{r-2}_+ \leq |p|^{r-1}_+$ and $2 \leq d$. We obtain
\begin{equation*}
    \left| \left( D_p V \left( p + \nabla \varphi_{L} (\cdot , \cdot ; p) \right) \right)_{Q_\ell} - D_p \bar \sigma_L(p) \right| \leq  \mathcal{O}_{1/2} \left( C |p|^{r-1}_+ \left( \frac{ L^{7/8}}{\ell_1}  + \frac{\ell_1 }{\ell} \right) \right).
\end{equation*}
We then optimise the previous display by choosing $\ell_1 := \sqrt{ L^{\frac{7}{8}} \ell } = L^{\frac{7}{16}} \ell^{\frac{1}{2}}$ (this value is admissible if $L^{7/8} \leq \ell$ so that $\ell_1 \leq \ell$). We obtain
\begin{equation*}
    \left| \left( D_p V \left( p + \nabla \varphi_{L} (\cdot , \cdot ; p) \right) \right)_{Q_\ell} - D_p \bar \sigma_L(p) \right| \leq  \mathcal{O}_{1/2} \left( C |p|^{r-1}_+ \left( \frac{ L^{7/8}}{\ell} \right)^{\frac 12} \right),
\end{equation*}
which is the desired inequality.
\end{proof}

\subsection{Estimating the weak norm of the flux via the multiscale Poincar\'e inequality}
\label{sec:section6.3}

In this section, we combine the result of Lemma~\ref{lemm:concenineqflux} together with the multiscale Poincar\'e inequality (Proposition~\ref{prop.multiscalePoinc}) to complete the proof of Proposition~\ref{prop:sublinearityflux}.

\begin{proof}[Proof of Proposition~\ref{prop:sublinearityflux}]
We only prove the result when $L = 3^n$ for some integer $n \in \N$. We first apply the multiscale Poincar\'e inequality
\begin{align} \label{eq:12190204}
     \lefteqn{\left\| D_p V (p + \nabla \varphi_{3^n}(\cdot , \cdot ; p)) - D_p \bar \sigma_{3^n}(p) \right\|_{\underline{W}^{-1,r}_{\mathrm{par}} (Q_{3^n})} } \qquad & \\ & \leq  C  \left\| D_p V (p + \nabla \varphi_{3^n}(\cdot , \cdot ; p)) - D_p \bar \sigma_{3^n}(p) \right\|_{\underline{L}^{r} (Q_{3^n})} \notag \\
     & \quad + C \sum_{m = 0}^n 3^m \left( \left| \mathcal{Z}_m \right|^{-1}  \sum_{z \in \mathcal{Z}_m} \left| \left( D_p V (p + \nabla \varphi_{3^n}(\cdot , \cdot ; p)) \right)_{z + Q_{3^m}} - D_p \bar \sigma_{3^n}(p)  \right|^r \right)^{\sfrac 1r}. \notag
\end{align}
We next estimate the two terms on the right-hand side. For the first term, we use Proposition~\ref{prop.prop2.3} (which provides a strong stochastic integrability estimate for the gradient of the Langevin dynamic at any time and any point in space), together with the growth Assumption~\eqref{AssPot} on the potential $V$ and Proposition~\ref{prop:prop2.20} ``Summation" and ``Integration". We obtain the bound
\begin{equation*}
     \left\| D_p V (p + \nabla \varphi_{3^n}(\cdot , \cdot ; p)) - D_p \bar \sigma_L(p) \right\|_{\underline{L}^{r} (Q_{3^n})} \leq \mathcal{O}_1(C |p|_+^{r-1}).
\end{equation*}
For the second term on the right-hand of~\eqref{eq:12190204}, we split it into two cases: whether $m \geq \frac{7}{8} n$ (in which case, we apply Lemma~\ref{lemm:concenineqflux}) or $m \leq 7n/8$ (in which case, we apply~\eqref{lemm:concenineqfluxsmallscales}).

\medskip

\textit{Case 1: Large values of $m$.}  We assume here that $m \geq 7n/8$, which implies $3^m \geq (3^n)^{\frac78}$. By Lemma~\ref{lemm:concenineqflux} and the space and time stationarity of the Langevin dynamic, we have the inequality
\begin{equation*}
     \left| \left( D_p V (p + \nabla \varphi_{3^n}(\cdot , \cdot ; p)) - D_p \bar \sigma_L(p) \right)_{z + Q_{3^m}}  \right|^r \leq \mathcal{O}_{1/(2r)} \left( C |p|^{r(r-2)}_+  3^{7r n/16 - rm/2} \right).
\end{equation*}
Summing over the vertices $z \in \mathcal{Z}_m$, and applying Proposition~\ref{prop:prop2.20} ``Summation", we deduce that
\begin{equation*}
        \left( \left| \mathcal{Z}_m \right|^{-1}  \sum_{z \in \mathcal{Z}_m} \left| \left( D_p V (p + \nabla \varphi_{3^n}(\cdot , \cdot ; p)) - D_p \bar \sigma_L(p) \right)_{z + Q_{3^m}}  \right|^r \right)^{\sfrac 1r} \leq \mathcal{O}_{1/2} \left( C |p|^{r-1}_+  3^{7n/16 - m/2} \right).
\end{equation*}
Multiplying both sides of the previous inequality by $3^m$ and summing over the integers $m \in (7n / 8 , n)$, we obtain
\begin{equation} \label{eq:16230204}
    \sum_{m = \lfloor \frac{7n}{8} \rfloor }^n 3^m \left( \left| \mathcal{Z}_m \right|^{-1}  \sum_{z \in \mathcal{Z}_m} \left| \left( D_p V (p + \nabla \varphi_{3^n}(\cdot , \cdot ; p)) - D_p \bar \sigma_L(p) \right)_{z + Q_{3^m}}  \right|^r \right)^{\sfrac 1r}
    \leq \mathcal{O}_{1/2} \left( C |p|^{r-1}_+ 3^{\frac{15n}{16}} \right),
\end{equation}
where we have used the inequality
$$\sum_{m = \lfloor \frac{7n}{8} \rfloor }^n 3^m 3^{7n/16 - m/2} = 3^{7n/16} \sum_{m = \lfloor \frac{7n}{8} \rfloor }^n 3^{m/2} \leq C 3^{15n/16}.$$

\medskip

\textit{Case 2: Small values of $m$.}  We assume here that $m \geq 7n/8$. By the inequality~\eqref{lemm:concenineqfluxsmallscales} and the space and time stationarity of the Langevin dynamic, we have the inequality
\begin{equation*}
     \left| \left( D_p V (p + \nabla \varphi_{3^n}(\cdot , \cdot ; p)) - D_p \bar \sigma_L(p) \right)_{z + Q_{3^m}}  \right|^r \leq \mathcal{O}_{1/(2r)} \left( C |p|^{r(r-1)}_+ \right).
\end{equation*}
Multiplying both sides of the previous inequality by $3^m$ and summing over the integers $m \in (0 , 7n/8)$, we obtain
\begin{equation*}
    \sum_{m = 0 }^{\lfloor \frac{7n}{8} \rfloor} 3^m \left( \left| \mathcal{Z}_m \right|^{-1}  \sum_{z \in \mathcal{Z}_m} \left| \left( D_p V (p + \nabla \varphi_{3^n}(\cdot , \cdot ; p)) - D_p \bar \sigma_L(p) \right)_{z + Q_{3^m}}  \right|^r \right)^{\sfrac 1r}
    \leq \mathcal{O}_{1/2} \left( C |p|^{r-1}_+ 3^{\frac{7n}{8}} \right).
\end{equation*}
Combining the inequality with the estimates~\eqref{eq:16230204},~\eqref{eq:16230204} and~\eqref{eq:16230204} (and noting that the second one is the largest), we obtain
\begin{equation*}
    \left\| D_p V (p + \nabla \varphi_{3^n}(\cdot , \cdot ; p)) - D_p \bar \sigma_{3^n}(p) \right\|_{\underline{W}^{-1,r}_{\mathrm{par}} (Q_{3^n})} \leq \mathcal{O}_{1/2} \left( C |p|^{r-1}_+  3^{15n/16} \right).
\end{equation*}
\end{proof}

\section{Quantitative hydrodynamic limit} \label{sec:section6}

This section contains the proof of the hydrodynamic limit (Theorem~\ref{main.thm}) and is based on a two-scale expansion (making use of the results established in the previous sections).

It is structured as follows:
\begin{itemize}
    \item As it is more convenient to state the main result, we rescale the problem and work on the torus $\mathbb{T}^\ep$ (i.e., the macroscopic scale is of size $1$ and the microscopic scale is of size $\ep$). For this reason we introduce some (suitably rescaled) notation and norms adapted to the discretized torus $\mathbb{T}^\ep$. In Section~\ref{sec:sec7.2} and~\ref{sec:sec7.3}, we introduce and study an approximation scheme for the homogenized equation which is used to pass from the discrete setting (where the Langevin dynamics are defined) to the continuous one (where the homogenized equation is defined).
    \item Section~\ref{Section.construct2scale} is devoted to the construction of the two-scale expansion. Following a standard technique in stochastic homogenization (see, e.g.,~\cite[Chatper 11]{AKMbook} or~\cite{FN19} for recent quantitative results making use of the technique) we introduce a mesoscopic scale and a partition of unity. In Section~\ref{sec:deferrorterms}, the main error terms are introduced and they are all proved to be small (in suitable norms) in Appendix~\ref{sec:appendixB}.
    \item Finally, Section~\ref{secTh.quantitativehydr} implements the two-scale expansion and proves Theorem~\ref{main.thm}, following mostly standard techniques and making use of the estimates established in the previous sections.
\end{itemize}

For the rest of this section, we select a smooth initial condition $f \in C^\infty (\mathbb{T})$, and allow all the constants to depend on the function $f$, as well as on the dimension $d$ and the potential $V$. We remark that an explicit, quantitative dependence of the constants in the initial condition $f$ can be extracted from the argument.

To measure the stochastic integrability of the various random variables, we will only use the super-polynomial stochastic integrability $\mathcal{O}_{\Psi ,c}$ (which is the weakest introduced in this article). Throughout the section, we allow the constant $C < \infty$ and the exponent $c > 0$ to vary from line to line. These two parameters shall only depend on the dimension $d$, the potential $V$ and the initial condition $f$. 

We finally recall that $r$ is the exponent encoding the growth of the potential $V$ (following Assumption~\eqref{AssPot}), and that we denote by $r' = r/(r-1)$ the conjugate exponent of $r$.

\subsection{Preliminaries: microscopic notation and periodic Sobolev spaces} \label{sec:7.1preliminareis}

In this section, we introduce the notation used in this section.

\subsubsection{Discrete gradient}
For each point $x \in \mathbb{T}^\ep$, each $\ep \in (0 , 1)$ and each $u : \mathbb{T}^\ep \to \R$, we introduce the definition of the (vector-valued) discrete gradient
\begin{equation*}
    \nabla^\ep u (t , x) := \left( \nabla_1^\ep u(t , x) , \ldots, \nabla_d^\ep u(t , x) \right) \in \Rd,
\end{equation*}
with 
\begin{equation*}
    \nabla^\ep u (t , x) := \ep^{-1} \left( u(t , x+e_i) - u(t , x) \right).
\end{equation*}
For a discrete vector field $F = (F_1 , \ldots, F_d) : \mathbb{T}^\ep \to \R$, we define its discrete divergence as follows
\begin{equation*}
    \nabla^\ep \cdot F := \ep^{-1} \sum_{i = 1}^d F_i(t , x) -  F_i(t , x - e_i).
\end{equation*}

\subsubsection{Average value}

We denote the average value of a function $u : \mathbb{T}^\ep \to \R$ by (N.B. we use the convention that the torus $\mathbb{T}^\ep$ contains exactly $\ep^{-d}$ vertices)
\begin{equation*}
    \left( u  \right)_{\mathbb{T}^\ep} := \ep^d \sum_{x \in \mathbb{T}^\ep} u(x).
\end{equation*}
Given a discrete box $\Lambda_\ep \subseteq \mathbb{T}^\ep$, we denote by $\left| \Lambda_\ep \right|$ its cardinality and define the average value of a function $u : \mathbb{T}^\ep \to \R$ according to the identity
\begin{equation*}
    \left( u \right)_{\Lambda_\ep} := \left| \Lambda_\ep \right|^{-1} \sum_{x \in \Lambda_\ep} u(x).
\end{equation*}
Given a parabolic cylinder $Q^\ep := I \times \Lambda_\ep \subseteq (0,1) \times \mathbb{T}^\ep$ (where $I \subseteq (0,1)$ is an interval whose length is denoted by $|I|$), and a function $u : Q^\ep \to \R$, we denote by
\begin{equation*}
     \left( u \right)_{Q^\ep} := |I|^{-1} \int_I \left| \Lambda_\ep \right|^{-1} \sum_{x \in \Lambda^\ep} u(t , x) \, dt.
\end{equation*}

\subsubsection{Sobolev spaces on $\mathbb{T}^\ep$}
In this section, we let $q \in (1 , \infty)$ be an exponent and denote by $q' = q/(q-1)$ its conjugate exponent. We then introduce the Sobolev and parabolic Sobolev spaces adapted to the torus $\mathbb{T}^\ep$
\begin{itemize}
    \item \emph{$L^2$-norm:}
    $
    \left\| u \right\|_{L^2 \left( \mathbb{T}^\ep \right)}^2 :=  \ep^d \sum_{x \in \mathbb{T}^\ep} \left| u(x)\right|^2,
    $ 
    \item \emph{$L^q$-norm:}
    $
    \left\| u \right\|_{L^q \left( \mathbb{T}^\ep \right)}^q := \ep^d \sum_{x \in \mathbb{T}^\ep} \left| u(x)\right|^q,
    $
    \item \emph{$W^{1,q}$-norm:}
    $
        \left\| u \right\|_{W^{1,q}(\mathbb{T}^\ep)} :=  \left\| u \right\|_{L^q (\mathbb{T}^\ep)} + \left\| \nabla^\ep u \right\|_{L^q (\mathbb{T}^\ep)},
    $
    \item \emph{$W^{-1,q}$-norm:}
        $
        \left\| u \right\|_{W^{-1,q}(\mathbb{T}^\ep)} := \sup \left\{\ep^d \sum_{x \in \mathbb{T}^\ep} u(x) v(x) \, : \,   \left\| v \right\|_{W^{1,q'}(\mathbb{T}^\ep)} \leq 1 \right\}.
        $
\end{itemize}

\subsubsection{Parabolic Sobolev spaces on $(0,1) \times \mathbb{T}^\ep$}
We then define the following norms and parabolic Sobolev spaces, for any function $u : (0,1) \times \mathbb{T}^\ep \to \R$,
\begin{itemize}
    \item \emph{$L^2$-norm:}
    $
    \left\| u \right\|_{\underline{L}^2 \left( (0 ,1) \times \mathbb{T}^\ep \right)}^2 :=   \int_{(0 , 1)}\left\| u(t , \cdot)\right\|^2_{L^2(\mathbb{T}^\ep)} \, dt,
    $
    \item \emph{$L^q$-norm:}
    $
    \left\| u \right\|_{\underline{L}^q \left( (0 ,1) \times \mathbb{T}^\ep \right)}^q :=  \int_{(0,1)}\left\| u(t , \cdot)\right\|^q_{L^q(\mathbb{T}^\ep)} \, dt,
    $
    \item \emph{$L^\infty$-norm:} $
    \left\| u \right\|_{L^\infty \left( (0 ,1) \times \mathbb{T}^\ep \right)} := \sup_{(t , x) \in (0 ,1) \times \mathbb{T}^\ep} \left| u(t, x)\right|,
    $
    \item \emph{$L^qW^{1,q}$-norm:}
    $
        \left\| u \right\|_{L^q((0 , 1) ,W^{1,q}(\mathbb{T}^\ep))}^q :=  \int_{(0 , 1)}\left\| u(t , \cdot)\right\|^q_{W^{1,q}(\mathbb{T}^\ep))} \, dt,
    $
    \item \emph{$L^{q}W^{-1,q}$-norm:}
         $
        \left\| u \right\|_{L^{q}((0 , 1) , W^{-1,q}(\mathbb{T}^\ep))}^{q} := \int_{(0 , 1)}\left\| u(t , \cdot)\right\|^{q}_{W^{-1,q}(\mathbb{T}^\ep)} \, dt,
        $
    \item  \emph{$W^{1,q}_{\mathrm{par}}$-norm:}
    $
    \left\| u \right\|_{W^{1,q}_{\mathrm{par}}((0 ,1) \times \mathbb{T}^\ep)} := \left\| u \right\|_{L^q \left( (0 ,1) \times \mathbb{T}^\ep \right)} + \left\| \nabla^\ep u \right\|_{\underline{L}^q \left( (0 ,1) \times \mathbb{T}^\ep  \right)} + \left\| \partial_t u \right\|_{\underline{L}^{q}((0 , 1) , W^{-1,q}(\mathbb{T}^\ep))},
    $
    \item  \emph{$\hat{W}^{-1,q}_{\mathrm{par}}$-norm:}
    $
    \left\| u \right\|_{\hat{W}^{-1,q}_{\mathrm{par}}((0 ,1) \times \mathbb{T}^\ep)} := \sup \left\{ \ep^d \int_{(0,1)}\sum_{x \in \mathbb{T}^\ep} u(t , x) v(t , x) \, dt : \,  \left\| v \right\|_{W^{1,q'}_{\mathrm{par}}((0 ,1) \times \mathbb{T}^\ep)} \leq 1 \right\}.
    $
\end{itemize}

The following statement identifies the structure of the space $\hat{W}^{-1,q}_{\mathrm{par}}$ and is used at the end of the proof of Theorem~\ref{main.thm}. It is discrete version of~\cite[Lemma 3.11]{ABM} with the $L^2$-norm is replaced by the $L^q$-norm.

\begin{lemma}[Identification of $\hat{W}^{-1,q}_{\mathrm{par}}((0,1) \times \mathbb{T}^\ep)$] \label{lem.idenH-1par}
There exists a constant $C := C(d,q) < \infty$ such that, for any $f \in L^{q}((0,1) \times \mathbb{T}^\ep) $, there exist a continuous function $h : (0,1) \times \mathbb{T}^\ep \to \R$ and a function $h^* : (0,1) \times \mathbb{T}^\ep \to \R$ such that
\begin{equation*}
    \left\{ \begin{aligned}
     \partial_t h  + h^* & = f, \\
    \left\| h \right\|_{L^q((0 , 1) ,W^{1,q}(\mathbb{T}^\ep))} & \leq C \left\| f \right\|_{\hat{W}^{-1,q}_{\mathrm{par}}((0 ,1) \times \mathbb{T}^\ep)}, \\
    \left\| h^* \right\|_{L^{q}((0 , 1) , W^{-1,q}(\mathbb{T}^\ep))} & \leq C \left\| f \right\|_{\hat{W}^{-1,q}_{\mathrm{par}}((0 ,1) \times \mathbb{T}^\ep)},
    \\
    h(0, \cdot)  = h(1 , \cdot) & = 0.
    \end{aligned} \right.
\end{equation*}
\end{lemma}

\subsection{Solutions of parabolic equations and regularity estimates} \label{sec:sec7.2}

In this section, we state some regularity estimates on the solution $\bar u$ of the equation~\eqref{def.barumainthm}.

\begin{proposition}[Regularity for solutions of the homogenized equation] \label{prop.reghomogenizedsolution}
    Let $f \in C^\infty(\mathbb{T})$ be a smooth initial condition and let $\bar u : [0 , 1] \times \mathbb{T} \to \R$ be the solution of the parabolic equation
    \begin{equation} \label{parabolicbis}
        \left\{ \begin{aligned}
            \partial_t \bar u - \nabla \cdot D_p \bar \sigma (\nabla \bar u) = 0 &~~\mbox{in}~~ (0 , 1) \times \mathbb{T}, \\
            \bar u(0 , \cdot) = f & ~~\mbox{in}~~ \mathbb{T}.
        \end{aligned} \right.
    \end{equation}
    Then there exists a constant $C := C (d , V, f) < \infty$ such that the map $\bar u$ satisfies the following regularity estimates:
    \begin{itemize}
        \item $W^{1, \infty}$ and $H^2$ regularity estimate:
        \begin{equation*}
            \left\| \bar u \right\|_{L^\infty((0 , 1) \times \mathbb{T})} + \left\| \nabla \bar u \right\|_{L^\infty((0 , 1) \times \mathbb{T})}  + \left\| \nabla^2 \bar u \right\|_{L^2((0 , 1) \times \mathbb{T})} \leq C.
        \end{equation*}
        \item $H^1$ regularity for the time derivative:
        \begin{equation*}
            \left\| \partial_t \bar u \right\|_{L^2((0 , 1) \times \mathbb{T})}  + \left\| \partial_t \nabla \bar u \right\|_{L^2((0 , 1) \times \mathbb{T})}  \leq C.
        \end{equation*}
    \end{itemize}
\end{proposition}

\begin{remark}
    The existence and uniqueness of a solution to the parabolic equation~\eqref{parabolicbis} in the space $L^r ((0 , \infty), W^{1,r}(\mathbb{T})) \, \cap \, C((0 , \infty) , W^{-1 , r'}(\mathbb{T}))$ follows from standard arguments (for instance the monotonicity method of~\cite[Chapter 1, Section 8 and Chapter 2, Section 1]{lions1969quelques}). 
    
    The regularity is obtained by differentiating the parabolic equations with respect to the time and space variables and makes crucial use of the strict convexity of the surface tension established in Proposition~\ref{prop:strictconvsurftens}, see~\cite[Chapter 1, Theorem 8.1 and Chapter 2, Theorem 1.3]{lions1969quelques} for the $L^2$ estimates and~\cite{FriedmanDiBenedetto1984, Friedman1985} for the $L^\infty$-estimates (N.B. These references are concerned with the parabolic $p$-Laplacian, but the same arguments apply to the equation~\eqref{parabolicbis}, and are in fact easier because the equation is not degenerate)
\end{remark}

\subsection{Approximation scheme for nonlinear parabolic equations} \label{sec:sec7.3}

The proof requires an argument to pass from the discrete setting (where the Langevin dynamics are defined) to the continuous setting (where the function $\bar u$ is defined). This is the subject of this section, where we define a discretized approximation of the function $\bar u$ and quantify the error made by this procedure.

\begin{definition}[Discretized parabolic equation]
For $f \in C^\infty(\mathbb{T})$, we let $\bar u^\ep: (0 ,1) \times \mathbb{T}^\ep \to \R$ be the solution of the discrete parabolic equation
\begin{equation} \label{def.discequep}
    \left\{ \begin{aligned}
     \partial_t \bar u^\ep - \nabla^\ep \cdot D_p \bar \sigma (\nabla^\ep \bar u^\ep) & = 0 &~\mbox{in} &~  (0 , 1) \times \mathbb{T}^\ep, \\
     \bar u^\ep (0 , \cdot) & = f &~\mbox{on} &~ \mathbb{T}^\ep.
    \end{aligned} \right.
\end{equation}
\end{definition}
We state below some regularity properties satisfied by the discrete solution $\bar u^\ep$.
\begin{proposition}[Regularity for the discretized solution] \label{prop:prop7.5}
     Then there exists a constant $C := C (d , V, f) < \infty$ such that the map $\bar u^\ep$ satisfies the following regularity estimates:
    \begin{itemize}
        \item $W^{1, \infty}$ and $H^2$ regularity estimate:
        \begin{equation*}
            \left\| \bar u^\ep \right\|_{L^\infty((0 , 1) \times \mathbb{T}^\ep)} + \left\| \nabla^\ep \bar u^\ep \right\|_{L^\infty((0 , 1) \times \mathbb{T}^\ep)}  + \left\| \nabla^{2,\ep} \bar u^\ep \right\|_{L^2((0 , 1) \times \mathbb{T}^\ep)} \leq C.
        \end{equation*}
        \item $H^1$ regularity for the time derivative:
        \begin{equation*}
            \left\| \partial_t \bar u^\ep \right\|_{L^2((0 , 1) \times \mathbb{T}^\ep)}  + \left\| \partial_t \nabla^\ep \bar u^\ep \right\|_{L^2((0 , 1) \times \mathbb{T}^\ep)}  \leq C.
        \end{equation*}
    \end{itemize}
\end{proposition}
The $L^2$-estimates can be obtained directly by differentiating the equation~\eqref{def.discequep} (which should be easier to justify than in the continuous setting since this equation is a high-dimensional ordinary differential equation). The $L^\infty$-estimates can be obtained by adapting the arguments of~\cite{FriedmanDiBenedetto1984, Friedman1985}.

The following proposition quantifies the difference of the $L^2$-norm between the solution $\bar u$ of the continuous parabolic equation~\eqref{def.barumainthm} and the solution $\bar u^\ep$ of the discretized equation~\eqref{def.discequep}. In order to state the result, we extend the map $\bar u^\ep$ and its gradient from the discrete setting to the continuous one into piecewise constant functions by setting
\begin{equation*}
    \bar u^\ep (t , x) := \sum_{y \in \ep \Zd} \bar u^\ep(t , y) \indc_{\{ y \in x +  [-\ep, \ep]^d \}} \hspace{5mm} \mbox{and} \hspace{5mm} \nabla^\ep \bar u^\ep(t , x) := \sum_{y \in \ep \Zd} \nabla^\ep \bar u^\ep (t , y) \indc_{\{ y \in x +  [-\ep, \ep]^d \}}.
\end{equation*}
The $L^2$-norm in Proposition~\ref{prop.approx} then denotes the continuous one on the space $\R \times \mathbb{T}.$
 
\begin{proposition}[Approximation by the discrete equation] \label{prop.approx}
There exists a constant $C := C(d , V , f) < \infty$ such that, for any $\ep > 0$,
\begin{equation*}
    \left\| \bar u^\ep - \bar u  \right\|_{L^2((0, \infty) \times \mathbb{T})} + \| \nabla^\ep \bar u^\ep - \nabla \bar u  \|_{L^2((0, \infty) \times \mathbb{T})} \leq C \ep^{\frac12} .
\end{equation*}
\end{proposition}
The proof is essentially identical to the one of~\cite[Proposition 4.2]{armstrong2022quantitative}, we will thus omit the technical details (N.B. the proof is written under the assumption that the second derivative of the surface tension is bounded from above but the adaptation to the present setting is straightforward). We also refer to~\cite[Appendix I]{FS} for a qualitative version of the result (under the assumption that the Hessian of the surface tension is bounded from above).

\subsection{Construction of the two-scale expansion} \label{Section.construct2scale}

\begin{figure}
        \centering
        \includegraphics[scale=0.7]{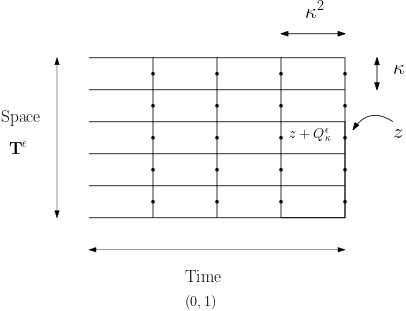}
    \caption{Partitioning the cylinder $(0,1) \times \mathbb{T}^\ep$: the set $\mathcal{Z}_\kappa$ is the collection of black dots and the set $(0,1) \times \mathbb{T}^\ep$ is partitioned into cylinders of the form $\left\{z + Q_\kappa^\ep \, : \, z \in \mathcal{Z}_\kappa \right\}$ (the small rectangles on the picture).\label{fig:partition}}
\end{figure}

\subsubsection{Mesoscopic scale and partition of unity}

We fix $\ep > 0$ and introduce the following notation and definitions:

\begin{itemize}
\item \textit{Mesoscopic scale.} We let $\gamma \in (0,1)$ be an exponent so that $\ep^\gamma$ is the size of a mesoscopic scale used through the argument (one should have in mind $\gamma \simeq 1/(30dr)$ for the argument). We then let
\begin{equation*} 
    \kappa := \ep^{\gamma} ~~\mbox{and}~~ L := \frac{\kappa}{\ep} = \ep^{\gamma - 1}.
\end{equation*} 
\item \textit{Mesoscopic boxes and cylinders.} We then introduce the box, time interval and parabolic cylinder
\begin{equation*}
    \Lambda_{2\kappa}^\ep := \ep \Lambda_{2L} \subseteq \mathbb{T}^\ep, ~~ I_{2\kappa} := (- (2\kappa)^2 , 0]  \subseteq \R ~~\mbox{and}~~ Q_{2\kappa}^\ep := I_{2\kappa} \times \Lambda_{2\kappa}^\ep \subseteq \R \times \mathbb{T}^\ep.
\end{equation*}
\item \textit{Covering of $(0 , 1) \times \mathbb{T}^\ep$ by mesoscopic cylinders.} For practical purposes, we assume that $\kappa$ is the inverse of an integer. We then introduce the set
\begin{equation*}
    \mathcal{Z}_\kappa := \left(\kappa^2 \N_* \times \kappa \Zd \right) \cap \left(  (0 , 1] \times \mathbb{T}^\ep \right).
\end{equation*}
Note that the collection $\left\{ z + Q_{2\kappa}^\ep \ : \, z \in \mathcal{Z}_\kappa \right\}$ is a covering of $(0 , 1) \times \mathbb{T}^\ep$ and that any vertex $z' \in (0 , 1) \times \mathbb{T}^\ep$ belongs to at most $C(d) < \infty$ parabolic cylinders of the set $\left\{ z + Q_{2\kappa}^\ep \ : \, z \in \mathcal{Z}_\kappa \right\}$. In the rest of this section, we make use of the shorthand notation $\sum_{z}$ to refer to the sum $\sum_{z \in \mathcal{Z}_\kappa}.$
\item \textit{Partition of unity.} We next consider a smooth partition of unity $(\chi_z)_{z \in \mathcal{Z}_\kappa} : (0 , 1) \times \mathbb{T}^\ep \to \R$ satisfying the following properties: 
\begin{equation} \label{prop.partofunity2sc}
    0 \leq \chi_z \leq \indc_{\{z + Q_{2 \kappa}^\ep\}}, \hspace{3mm} \sum_{z } \chi_z = 1 \hspace{3mm} \mbox{in} ~(0 , 1) \times \mathbb{T}^\ep,
\end{equation}
and for any $l \in \{0 , 1\}$ and $k \in \{0 , 1 , 2\}$,
\begin{equation} \label{prop.partofunity2scbis}
     \kappa^{2l + k} \| \partial_t^l \nabla^{\ep, k} \chi_z \|_{L^\infty((0 , \infty) \times \mathbb{T}^\ep)}\leq C.
\end{equation}
\end{itemize}

\subsubsection{Definition of the two-scale expansion}

We let $\bar u^\ep$ be the solution of the equation~\eqref{def.discequep} and introduce the following notation:

\begin{itemize}
\item \textit{Slopes.} Given a point $z \in \mathcal{Z}_{\kappa}$, we denote by $p_z$ the average value of the gradient $\nabla^\ep \bar u^\ep$ over the parabolic cylinder $z + Q_{2\kappa}^\ep$, i.e.,
\begin{equation*}
   p_z := \left( \nabla^\ep  \bar u^\ep \right)_{z + Q_{2\kappa}^\ep}.
\end{equation*}
\item \textit{Corrector.} We will use the notation introduced in Remark~\ref{rem:remark6.3}, and, for $y \in \Zd$ and $p \in \Rd$, we let $\varphi_{y , 10L} \left( \cdot , \cdot ; p \right) : \R \times (y + \Lambda_{10L}) \to \R$ be the stationary Langevin dynamics with periodic boundary conditions in the box $(y + \Lambda_{10L})$. For $z := (s , y) \in \mathcal{Z}_{\kappa}$ and $p \in \Rd$, we denote by 
\begin{equation*}
    \varphi_{z} \left( \cdot , \cdot ; p \right) := \varphi_{y/\ep , 10 L} \left( \cdot , \cdot ; p \right) + \frac{\sqrt{2}}{\left| \Lambda_{10 L} \right|} \sum_{x \in y + \Lambda_{10L}} B_t(x).
\end{equation*}
(N.B. The average value of the Brownian motion is added here to take into account that the average value of the Brownian motions on the torus $\mathbb{T}_{10L}$ is assumed to be equal to $0$ in Definition~\ref{Prop:Langevin}, while this property is not satisfied by the Brownian motions on the right hand side of~\eqref{eq:introSDErescaled} (in the box $(y + \Lambda_\kappa^\ep$)). This term only plays a minor role in the rest of the analysis and is (much) smaller than the other error terms.)
\item \textit{Corrected plane.} For $z \in \mathcal{Z}_{\kappa}$, we denote by
\begin{equation*} 
    v_{z}(t , x) := p_z \cdot x + \ep  \varphi_{z} \left( \frac{t}{\ep^2} , \frac{x}{\ep} ;  p_z \right).
\end{equation*}
\item \textit{Two-scale expansion.} We define the two-scale expansion according to the identity, for $(t , x) \in (0 , \infty) \times \mathbb{T}^\ep$,
\begin{equation} \label{def.wL}
    w^\ep(t , x) := \bar u^\ep (t , x) + \ep \sum_{z} \chi_z(t , x) \varphi_{z} \left( \frac{t}{\ep^2} , \frac{x}{\ep} ; p_z\right).
\end{equation}
\end{itemize}

\subsubsection{Estimates for the Langevin dynamic and the two-scale expansion} \label{sec:sec743}

In this section, we collect some estimates pertaining to the two-scale expansion which are used in the argument below (N.B. these estimates are not necessarily the strongest provable and are stated in a convenient form for the argument). The proof of these results is postponed to Appendix~\ref{sec:appendixB}.

\begin{itemize}
\item \textit{Difference between $w^\ep$ and $\bar u^\ep$.} There exist two constants $C , c > 0$ and an exponent $\theta > 0$ (N.B. $\theta$ depends only on the exponents $\gamma$ and $r$ and is strictly positive is $\gamma$ is small enough) such that
\begin{equation} \label{ineq:12171608}
    \left\| \bar u^\ep  - w^\ep \right\|_{L^2((0,1) \times \mathbb{T}^\ep)} \leq \mathcal{O}_{\Psi , c}(C \ep^{2 \theta})
\end{equation}
and 
\begin{equation} \label{ineq:1217160888}
    \left\| \bar u^\ep(0,\cdot)  - w^\ep(0,\cdot) \right\|_{L^2(\mathbb{T}^\ep)} \leq \mathcal{O}_{\Psi , c}(C \ep^{2 \theta}).
\end{equation}
\item \textit{Upper bound for the gradient of the dynamic.} The $L^r$-norm of the Langevin dynamic is controlled as follows
\begin{equation} \label{eq:boundnablauep}
    \left\|  u^\ep \right\|_{L^r ( (0,1) \times \mathbb{T}^\ep )} \leq \mathcal{O}_{\Psi , c} \left( C \right).
\end{equation}
\item \textit{Upper bound for the two-scale expansion.} We have the upper bound
\begin{equation} \label{est:Linftynormw}
    \left\|  w^\ep \right\|_{L^r ( (0,1) \times \mathbb{T}^\ep )} \leq \mathcal{O}_{\Psi , c} \left( C \right).
\end{equation}
\end{itemize}

\subsubsection{Definition of the error terms} \label{sec:deferrorterms}

In this section, we introduce the four main error terms which appear in the proof of Theorem~\ref{main.thm}. They are proved to be small in suitable norms in Appendix~\ref{sec:appendixB} (see~\eqref{eq:estiamteerroreterms} and~\eqref{eq:estiamteerroretermsavecrage} for a summary of the results proved in this appendix)
\begin{equation} \label{def.allthearrorterms}
    \left\{ \begin{aligned}
        \mathcal{E}_1 & := \sum_{z} \partial_t \chi_{z} \left( \ep \varphi_{z} \left( \frac{\cdot}{\ep^2} , \frac{\cdot}{\ep} ; p_{z} \right) - \sqrt{2} B_\cdot^\ep  \right), \\
        \vec{\mathcal{E}}_2 & := D_p V \left( \nabla^\ep w \right)
        - \sum_{z}  \chi_{z} D_p V \left( \nabla^\ep v_{z} \right), \\
        \vec{\mathcal{E}}_3 & := \sum_{z} \chi_{z} \left( D_p \bar \sigma \left( p_{z} \right) - D_p \bar \sigma \left( \nabla^\ep \bar u^\ep \right) \right), \\
        \mathcal{E}_4 & := \sum_{z} \nabla^{\ep} \chi_{z} \cdot \left( D_p V \left( \nabla^{\ep} v_{z} \right) - D_p \bar \sigma \left( p_{z} \right) \right).
    \end{aligned} \right.
\end{equation}
We note that all these terms are functions defined on the set $(0 , 1) \times \mathbb{T}^\ep$, there are real-valued for the terms $\mathcal{E}_1$ and $\mathcal{E}_4$, and vector-valued (specifically, valued in $\R^d$), for the terms $\vec{\mathcal{E}}_2$ and $\vec{\mathcal{E}}_3$ (N.B. For the term $\mathcal{E}_4$ is a scalar product between two vector-valued functions). We then combine these error terms and define
\begin{equation} \label{eq:defEandvexE}
     \vec{\mathcal{E}}= \vec{\mathcal{E}}_2 + \vec{\mathcal{E}}_3 ~~\mbox{and}~~ \mathcal{E} = \mathcal{E}_1 + \mathcal{E}_4.
\end{equation}
These  error terms are proved to be small (in suitable norms) in Appendix~\ref{sec:appendixB}. Specifically, the following estimates are proved: there exist two constants $C , c > 0$ and an exponent $\theta > 0$ (N.B. $\theta$ depends only on the exponent $\gamma$ encoding the size of the mesoscopic scale and is strictly positive is $\gamma$ is small enough) such that
\begin{equation} \label{eq:estiamteerroreterms}
    \| \vec{\mathcal{E}} \|_{L^r((0 , 1) \times \mathbb{T}^\ep)} + \ep \left\| \mathcal{E} \right\|_{L^2((0 , 1) \times \mathbb{T}^\ep)} + \left\| \mathcal{E} \right\|_{\underline{W}^{-1,r}_{\mathrm{par}}\left( (0 , 1) \times \mathbb{T}^\ep \right)} \leq \mathcal{O}_{\Psi , c} \left( C \ep^{2\theta} \right),
\end{equation}
together with the estimate on the average value of the error term $\mathcal{E}$: for any $t \in (0,1),$
\begin{equation} \label{eq:estiamteerroretermsavecrage}
    \left| \int_0^t \left( \mathcal{E}(s , \cdot ) \right)_{\mathbb{T}^\ep} \, ds \right| \leq \mathcal{O}_{\Psi , c}(C \ep^{2 \theta}).
\end{equation}

\subsection{Two-scale expansion and proof of Theorem~\ref{main.thm}} \label{secTh.quantitativehydr}

This section is devoted to the proof of Theorem~\ref{main.thm}. The main objective of the proof is to show that the two-scale expansion $w^\ep$ is almost a solution of the Langevin dynamic. Specifically, we will prove the identity
\begin{equation} \label{eq:16012107}
    \partial_t \left(  w^\ep - \sqrt{2} B_{\cdot}^\ep \right)- \nabla^\ep \cdot D_p V(\nabla^\ep w^\ep)  = \nabla^\ep \cdot \vec{\mathcal{E}} + \mathcal{E} ,
\end{equation}
The identity~\eqref{eq:16012107} is established in Sections~\ref{sec:sec751} and~\ref{sec:sec752}. Once equipped with the identity~\eqref{eq:16012107} and the inequalities~\eqref{eq:estiamteerroreterms} and~\eqref{eq:estiamteerroretermsavecrage}, we show that the $L^{r'}$-norm of the gradient of the difference $u^\ep - w^\ep$ is small and combine this result with the inequality~\eqref{ineq:12171608} to complete the proof of Theorem~\ref{main.thm}. This is the subject of Section~\ref{section:eestimatingthediffernce}.

In the rest of this section, we fix the value of the exponent $\gamma$ encoding the size of the mesoscopic scale and chose it small enough so that the inequalities hold with a strictly positive exponent $\theta > 0$ (e.g., choosing $\gamma \simeq 1/(30dr)$ gives a value $\theta \simeq 1/(100dr))$). This exponent is thus not allowed to change from line to line while the constants $C$ and $c$ are (and they shall only depend on the parameters $d$ and $r$).

\subsubsection{Computing the time derivative of the two-scale expansion $w^\ep$} \label{sec:sec751}

To prove the formula~\eqref{eq:16012107}, we first compute the time derivative of the map $w^\ep -  B^\ep$. The computation is based on the definition~\eqref{def.wL} of the two-scale expansion $w^\ep$ and is straightforward (applying the chain rule to compute the derivative of a product). We obtain the identity
\begin{align} \label{eq:16002107.sec4}
\lefteqn{
\partial_t \left(  w^\ep -  \sqrt{2} B_\cdot^\ep \right)
} \qquad & \notag \\ & 
= \partial_t \bar u^\ep  +  \partial_t \left( \varphi^\ep -  \sqrt{2} B_\cdot^\ep \right) \notag \\
 & = \partial_t  \bar u^\ep   +   \sum_{z} \chi_{z} \partial_t \left( \ep \varphi_{z} \left( \frac{\cdot}{\ep^2} , \frac{\cdot}{\ep} ; p_{z} \right) - \sqrt{2} B_\cdot^\ep \right) +   \sum_{z} \partial_t \chi_{z} \left( \ep \varphi_{z} \left( \frac{\cdot}{\ep^2} , \frac{\cdot}{\ep} ; p_{z} \right) - \sqrt{2} B_\cdot^\ep  \right).
\end{align}
The last term on the right-hand side is exactly the error term $\mathcal{E}_1$. We thus have the identity
\begin{equation*}
    \partial_t \left(  w^\ep -  \sqrt{2} B_\cdot^\ep \right) =  \partial_t  \bar u^\ep   +   \sum_{z} \chi_{z} \partial_t \left( \ep \varphi_{z} \left( \frac{\cdot}{\ep^2} , \frac{\cdot}{\ep} ; p_{z} \right) - \sqrt{2} B_\cdot^\ep \right) + \mathcal{E}_1.
\end{equation*}

\subsubsection{Computing the value of $\nabla^\ep \cdot D_p V \left( \nabla^\ep w^\ep \right)$} \label{sec:sec752}
In this section, we establish the identity
\begin{equation} \label{eq:11370705}
    \nabla^\ep \cdot D_p V \left( \nabla^\ep w^\ep \right) = \sum_{z}  \chi_{z} \nabla^\ep \cdot D_p V ( \nabla^\ep v_{z} )  +  \nabla^\ep \cdot D_p \bar \sigma ( \nabla^\ep \bar u^\ep ) + \mathcal{E}_4 + \nabla^\ep \cdot \vec{\mathcal{E}}.
\end{equation}
The proof of~\eqref{eq:11370705} is decomposed into three steps:
\begin{itemize}
\item In the first step, we prove the identity 
\begin{align} \label{identity6.6}
    \nabla^\ep \cdot D_p V \left( \nabla^\ep w \right)
    & = \nabla^\ep \cdot \left( \sum_{z}  \chi_{z} D_p V \left( \nabla^\ep v_{z} \right) \right)  + \nabla^\ep \cdot \vec{\mathcal{E}}_2.
\end{align}
\item In the second step, we prove the identity
\begin{equation} \label{identity6.67}
    \nabla^\ep \cdot D_p \bar \sigma \left( \nabla^\ep \bar u^\ep \right) = \nabla^\ep \cdot \biggl( \ \sum_{z}\chi_{z} D_p\bar \sigma \left( p_{z} \right) \biggr) - \nabla^\ep \cdot \vec{\mathcal{E}_3}.
\end{equation}
\item In the third step, we prove the identity
\begin{equation} \label{identity6.678}
    \nabla^\ep \cdot \biggl( \sum_{z}  \chi_{z} \left( D_p V \left( \nabla^\ep v_{z} \right) -  D_p\bar \sigma \left( p_{z} \right)\right) \biggr)=  \sum_{z}  \chi_{z} \nabla^\ep \cdot D_p V \left( \nabla^\ep v_{z} \right)  + \mathcal{E}_4.
\end{equation}
\end{itemize}
The identity~\eqref{eq:11370705} is obtained by combining~\eqref{identity6.6},~\eqref{identity6.67} and~\eqref{identity6.678} together with the definition~\eqref{eq:defEandvexE}.

\medskip

\textit{Step 1. The identity~\eqref{identity6.6}.} This identity follows immediately from the definition of the term $\vec{\mathcal{E}}_2$ in~\eqref{def.allthearrorterms} (by applying the discrete divergence to both sides of the definition of $\vec{\mathcal{E}}_2$).

\medskip

\textit{Step 2. The identity~\eqref{identity6.67}.} Using the definition~\eqref{def.allthearrorterms} of the error term $\vec{\mathcal{E}}_3$ and the identity $\sum_{z} \chi_{z}=1$, we have
\begin{equation*}
    \vec{\mathcal{E}}_3  := \sum_{z}\chi_{z} \left( D_p \bar \sigma \left( p_{z} \right) - D_p \bar \sigma \left( \nabla^\ep \bar u ^\ep \right)\right)  = \sum_{z}\chi_{z} D_p \bar \sigma \left( p_{z} \right) - D_p \bar \sigma (\nabla^\ep \bar u^\ep ).
\end{equation*}
The identity~\eqref{identity6.67} follows by applying the discrete divergence on both sides of this identity.

\medskip

\textit{Step 3. The identity~\eqref{identity6.678}.} Expanding the discrete divergence, we may write,
\begin{align} \label{eq:11111608}
     \nabla^\ep \cdot \biggl(  \sum_{z}  \chi_{z} \left( D_p V \left( \nabla v_{z} \right) - D_p \bar \sigma \left( p_{z} \right)\right) \biggr) & =   \sum_{z}  \chi_{z}(t , x) \nabla^\ep \cdot \left( D_p V \left( \nabla v_{z} \right) - D_p  \bar \sigma \left( p_z \right) \right) \\ & \qquad + \sum_{z}  \nabla^{\ep} \chi_{z} \cdot \left( D_p V \left( \nabla^{\ep} v_{z} \right) - D_p \bar \sigma \left( p_z \right) \right). \notag
\end{align}
The previous display can be simplified using that the terms $ D_p \bar \sigma \left( p_z \right)$ are constant. We obtain
\begin{equation*}
    \nabla^\ep \cdot \left( D_p V \left( \nabla^\ep v_{z} \right) - D_p \bar \sigma \left( p_z \right) \right) = \nabla \cdot  D_p V \left( \nabla^\ep v_{z} \right).
\end{equation*}
The second term on the right-hand side of~\eqref{eq:11111608} is exactly equal to the error term $\mathcal{E}_4$. We deduce that
\begin{equation*}
\nabla^\ep \cdot \biggl( \sum_{z}  \chi_{z} \left( D_p V \left( \nabla^\ep v_{z} \right) -  D_p\bar \sigma \left( p_{z} \right)\right) \biggr)=  \sum_{z}  \chi_{z} \nabla^\ep \cdot D_p V \left( \nabla^\ep v_{z} \right)  + \mathcal{E}_4.
\end{equation*}
This is~\eqref{identity6.678}.

\subsubsection{Estimating the norm of the difference $u^\ep - w^\ep$} \label{section:eestimatingthediffernce}

Taking the difference between the equation~\eqref{eq:introSDErescaled} and the equation~\eqref{eq:16012107}, we obtain that the function $v := u^\ep - w^\ep$ solves the parabolic equation
\begin{equation} \label{eq:v2scexp}
    \left\{ \begin{aligned}
        \partial_t  v  - \nabla^\ep \cdot \a \nabla^\ep v & = \nabla^\ep \cdot \vec{\mathcal{E}} + \mathcal{E} & ~~\mbox{in} &~~ (0 , 1) \times \mathbb{T}^\ep, \\
        v(0 , \cdot) & = w^\ep(0, \cdot) -  u^\ep(0, \cdot) &~~ \mbox{in} &~~ \mathbb{T}^\ep,
    \end{aligned} \right.
\end{equation}
where the environment $\a$ is given by the formula
\begin{equation} \label{def:defevtAtwoscale}
    \a(t , x) := \int_0^1 D_p^2 V(s \nabla^\ep u^\ep (t ,x) + (1-s)\nabla^\ep w^\ep (t , x) ) ds.
\end{equation}
Using the linearity of the equation~\eqref{eq:v2scexp}, we may decompose the map $v$ into three terms according to the formula $v = v_0 + v_1 + v_2$, where the functions $v_0$, $v_1$ and $v_2$ are the solutions of the parabolic equations
\begin{equation*}
    \left\{ \begin{aligned}
    \partial_t v_0 - \nabla^\ep \cdot \a \nabla^\ep v_0 & = 0 &~\mbox{in}~&(0 , 1) \times \mathbb{T}^\ep, \\
    v_0(0 , \cdot) & = w^\ep(0, \cdot) - \bar u^\ep(0, \cdot) &~\mbox{in}~& \mathbb{T}^\ep,
    \end{aligned} \right.
\end{equation*}
and
\begin{equation} \label{eq:v22scexp}
    \left\{ \begin{aligned}
    \partial_t v_1 - \nabla^\ep \cdot \a \nabla^\ep v_1 & = \nabla^\ep \cdot \vec{\mathcal{E}} &\mbox{in}~~&(0 , 1) \times \mathbb{T}^\ep, \\
    v_1(0 , \cdot) & = 0 &\mbox{in}~~& \mathbb{T}^\ep,
    \end{aligned} \right.
\end{equation}
and
\begin{equation} \label{eq:v3scexp}
    \left\{ \begin{aligned}
    \partial_t v_2 - \nabla^\ep \cdot \a \nabla^\ep v_2 & = \mathcal{E} &\mbox{in}~~&(0 , 1) \times \mathbb{T}^\ep, \\
    v_2(0 , \cdot) & = 0 &\mbox{in}~~&\mathbb{T}^\ep.
    \end{aligned} \right.
\end{equation}
We then decompose the argument into five steps: the first one collects some important properties of the environment $\a$, Steps 2, 3 and 4 show that the functions $v_0, v_1, v_2$ are small (in suitable norms), Step 5 is the conclusion of the argument.

\medskip

\textit{Step 1. Properties of the environment $\mathbf{A}$.}

\smallskip

In this step, we introduce three quantities related to the environment $\mathbf{A}$ and state two of their properties whose proof can be found in Appendix~\ref{sec:sectionappendixC}. Specifically, introduce:
\begin{itemize}
\item \textit{The maximal and minimal eigenvalues.} We denote by $\mathbf{\Lambda}_+$ and $\mathbf{\Lambda}_-$ to denote the largest and smallest eigenvalue of $\a$, i.e.,
    \begin{equation} \label{def.lambda+-ineq6.44}
    \mathbf{\Lambda}_+(t ,x) := \sup_{\substack{\xi \in \Rd \\ | \xi | = 1}} \xi \cdot \a(s , x) \xi \hspace{5mm}  \mbox{and} \hspace{5mm} \mathbf{\Lambda}_-(t,x) :=  \inf_{\substack{\xi \in \Rd \\ | \xi | = 1}}  \xi \cdot \a(t , x) \xi.
\end{equation}
\item \textit{The moderated environment.} We define the moderated environment associated with the environment $\a$: for any $(t,x) \in [0 , 1] \times \mathbb{T}^\ep$,
\begin{equation} \label{def.modeeratedevt}
     \mathbf{m}(t , x) := 
     \left\{ \begin{aligned}
     \ep^{-2} \int_t^{1} k_{\ep^{-2}(s-t)} \frac{\mathbf{\Lambda}_- (s, x )  \wedge 1 }{( s-t)^{-1} \sum_{y \sim x}\int_t^s \left( 1+ \mathbf{\Lambda}_+ \left( s' , x \right)  \right) \, ds'} \, ds & ~~\mbox{if} ~~ t \in \left[0 , \frac{1}{2}\right], \\
     \ep^{-2} \int_0^{t} k_{\ep^{-2}(s-t)} \frac{\mathbf{\Lambda}_- (s, x )  \wedge 1 }{( s-t)^{-1} \sum_{y \sim x}\int_t^s \left( 1+ \mathbf{\Lambda}_+ \left( s' , x \right)  \right) \, ds'} \, ds & ~~\mbox{if} ~~ t \in \left(\frac{1}{2} , 1 \right].
     \end{aligned} \right.
\end{equation}
We consider a different moderated environment than in Section~\ref{sec:section3moderated} because we only want to use time values between $0$ and $1$ (rather than between $0$ and $\infty$). This is a minor difference compared to Section~\ref{sec:section3moderated} but has the following consequence: the moderated environment can take the value $0$ with positive (although very small) probability (see~\eqref{eq:22572106} below).
\item \textit{The maximal function.} We define the maximal function
\begin{equation*}
    \mathbf{M}_+(t , x) :=
    \left\{ \begin{aligned}
        \sup_{ s \in (t,1)} \frac{1}{(s-t)} \int_t^s 1 + \mathbf{\Lambda}_+(s' ,x) \, ds' & ~~\mbox{if} ~~ t \in \left[0 , \frac{1}{2}\right] ,\\
        \sup_{ s \in (0,t)} \frac{1}{(t-s)} \int_s^t 1 + \mathbf{\Lambda}_+(s' ,x) \, ds' & ~~\mbox{if} ~~ t \in \left(\frac{1}{2} , 1 \right] .
    \end{aligned} \right.
\end{equation*}
\end{itemize}

The important properties about these quantities which will be used in the argument below is that the maximal function cannot be too large values, that the product $\mathbf{m} \times \mathbf{M}_+$ cannot be too small, and that the environment $\mathbf{m}$ can be used to moderate solutions of parabolic equations. Specifically, recalling the definition of the exponent~$\theta$ introduced in~\eqref{eq:estiamteerroreterms}, we prove in Appendix~\ref{sec:sectionappendixC} that the following inequalities hold:
\begin{itemize}
\item \textit{Upper bound for the largest eigenvalue and the maximal function.} We have the inequality (N.B. this inequality is slightly redundant since $\mathbf{\Lambda}_+ \leq \mathbf{M}_+$)
\begin{equation} \label{eq:realD20}
    \left\| \mathbf{\Lambda}_+ \right\|_{L^{\frac{r}{r-2}}((0,1) \times \mathbb{T}^\ep)} + \left\| \mathbf{M}_+ \right\|_{L^{\frac{r}{r-2}}((0,1) \times \mathbb{T}^\ep)} \leq \mathcal{O}_{\Psi , c} (C).
\end{equation}
\item \textit{Lower bound for the product $\mathbf{m} \times \mathbf{M}_+$.} We have the inequality
\begin{equation} \label{eq:22572106}
    \mathbb{P} \left[  \inf_{(t , x) \in (0,1) \times \mathbb{T}^\ep} \mathbf{m}(t , x)  \times \mathbf{M}_+(t , x)  \leq \ep^\theta \right] \leq C \exp \left( - c \left| \ln \ep \right|^{\frac{r}{r-2}}  \right).
\end{equation}
\item \textit{Moderation.} For any function $F : (0 , 1) \times \mathbb{T}^\ep \to \R$, every $t \in (0 , 1)$ and every solution $u : (0 , 1) \times \mathbb{T}^\ep \to \R$ of the parabolic equation
\begin{equation*}
    \partial_t u - \nabla^\ep \cdot \a \nabla^\ep u = F ~~~\mbox{in} ~~~ (0 , 1) \times \mathbb{T}^\ep,
\end{equation*}
one has the inequality
\begin{multline} \label{eq:11140912sec7}
    \int_0^1 \sum_{x \in \mathbb{T}^\ep} \mathbf{m}(t , x) \left| \nabla^\ep u(t , x) \right|^2 \, dt
    \leq C \int_0^1 \sum_{x \in \mathbb{T}^\ep} \nabla^\ep u(t, x) \cdot \a(t , x ) \nabla^\ep u(t , x) \, dt \\ + C \ep^2  \int_0^1 \sum_{x \in \mathbb{T}^\ep} |F(t,x)|^2 \, dt.
\end{multline}
\end{itemize}

\medskip

\textit{Step 2. Estimating the term $v_0$.}

\medskip

This term is the simplest to estimate. An energy estimate for the parabolic problem yields the upper bound
\begin{equation*}
    \sup_{t \in (0,1)} \left\| v_0(t , \cdot) \right\|^2_{L^2(\mathbb{T}^\ep)} + \int_0^1 \ep^d \sum_{x \in \mathbb{T}^\ep} \nabla v_0(s , y) \cdot \a(s , y) \nabla v_0(s , y) \, ds \leq  \left\|  w^\ep(0, \cdot) - \bar u^\ep(0, \cdot) \right\|^2_{L^2(\mathbb{T}^\ep)}.
\end{equation*}
Noting that the second term on the left-hand side is nonnegative and applying the inequality~\eqref{ineq:1217160888}, we obtain
\begin{equation} \label{eq:estiamtev0}
    \left\| v_0 \right\|_{L^2((0,1) \times \mathbb{T}^\ep)} \leq \sup_{t \in (0,1)} \left\| v_0(t , \cdot) \right\|_{L^2(\mathbb{T}^\ep)}  \leq \left\|  w^\ep(0, \cdot) - \bar u^\ep(0, \cdot) \right\|_{L^2(\mathbb{T}^\ep)} \leq \mathcal{O}_{\Psi , c} \left(C \ep^{2 \theta} \right).
\end{equation}

\medskip

\textit{Step 3. Estimating the term $v_1$.}

\medskip

The objective of this step is to prove that the $L^2$-norm of the function $v_1$ over the parabolic cylinder $(0 ,1 ) \times \mathbb{T}^\ep$ is small. This is achieve in two substeps: we first prove that the average value of the function $v_1$ is equal to $0$, and then prove that the $L^2$-norm of its discrete gradient is small. The conclusion then follows from the Poincar\'e inequality.

\medskip

\textit{Substep 3.1. Estimating the average value of the term $v_1$.}

\medskip

We first note that the average value of the map $v_1$ is always equal to $0$, i.e., for any $t \geq 0$,
\begin{equation} \label{eq:13152207}
    \sum_{x \in \mathbb{T}^\ep} v_1(t , x) = 0.
\end{equation}
This result is obtained by summing both sides of the first line of~\eqref{eq:v22scexp} over all the vertices $x \in \mathbb{T}^\ep$ and performing discrete integrations by parts to treat the two terms involving a discrete divergence.

\medskip

\textit{Substep 3.2. Estimating the gradient of the term $v_1$.}

\medskip

We then estimate the $L^2$-norm of the gradient of the map $v_1$. Using the inequality~\eqref{eq:11140912sec7}, we first write
\begin{equation*}
    \int_0^{1} \sum_{x \in \mathbb{T}^\ep}  \mathbf{m}(t , x) \left| \nabla^\ep v_1 (t , x) \right|^2 \, dt \leq C \int_0^{1} \sum_{x \in \mathbb{T}^\ep} \nabla^\ep v_1(t , x) \cdot \a(t , x) \nabla^\ep v_1(t , x) \, dt + C \ep^2 \int_0^{1} \sum_{x \in \mathbb{T}^\ep} | \nabla^\ep \cdot \vec{\mathcal{E}} (t , x) |^2 \, dt.
\end{equation*}
We then use the inequality $ | \nabla^\ep \cdot \vec{\mathcal{E}} (s , x) |^2 \leq C \ep^{-2} \sum_{y \sim x}| \vec{\mathcal{E}} (s , y) |^2$ (i.e., we forget the discrete divergence and pay a factor $\ep^{-2}$) so as to obtain
\begin{equation*}
    \int_0^{1} \sum_{x \in \mathbb{T}^\ep}  \mathbf{m}(t , x) \left| \nabla^\ep v_1 (t , x) \right|^2 \, dt \leq C \int_0^{1} \sum_{x \in \mathbb{T}^\ep} \nabla^\ep v_1(t , x) \cdot \a(t , x) \nabla^\ep v_1(t , x) \, dt + C \int_0^{1} \sum_{x \in \mathbb{T}^\ep} | \vec{\mathcal{E}} (t , x) |^2 \, dt.
\end{equation*}
The previous inequality can be combined with an energy estimate which reads
\begin{equation*}
     \sum_{x \in \mathbb{T}^\ep} \left| v_1(t , x) \right|^2 + \int_0^1 \sum_{x \in \mathbb{T}^\ep} \nabla^\ep v_1(t , x) \cdot \a(t , x) \nabla^\ep v_1(t , x) \, dt \leq \int_0^1 \sum_{x \in \mathbb{T}^\ep}  \vec{\mathcal{E}} (t , x) \cdot \nabla^\ep v_1(t , x) \, dt,
\end{equation*}
so as to obtain the estimates
\begin{align} \label{eq:17041205}
    \int_0^1 \sum_{x \in \mathbb{T}^\ep} \mathbf{m}(t , x) \left| \nabla^\ep v_1 (t , x) \right|^2 \, dt
     & \leq  C \int_0^{1}  \nabla^\ep v_1(t , y) \cdot \a(t , y) \nabla^\ep v_1(t , y) \, dt + C \int_0^{1} \sum_{x \in \mathbb{T}^\ep} | \vec{\mathcal{E}} (t , x) |^2 \, dt \\
     & \leq C \int_0^1 \sum_{x \in \mathbb{T}^\ep}  \vec{\mathcal{E}} (t , x) \cdot \nabla^\ep v_1(t , x) \, dt + C \int_0^{1} \sum_{x \in \mathbb{T}^\ep} | \vec{\mathcal{E}} (t , x) |^2 \, dt. \notag
\end{align}
By applying the H\"{o}lder and Jensen inequalities (and using $r > 2$), the inequality~\eqref{eq:17041205} can be rewritten as follows
\begin{align} \label{eq:12.142608}
    \int_0^1 \ep^d \sum_{x \in \mathbb{T}^\ep} \mathbf{m}(t , x) \left| \nabla^\ep v_1 (t , x) \right|^2 \, dt & \leq C \| \vec{\mathcal{E}} \|_{L^r((0,1) \times \mathbb{T}^\ep)} \left\| \nabla^\ep v_1 \right\|_{L^{r'}((0,1) \times \mathbb{T}^\ep)} + C \| \vec{\mathcal{E}} \|_{L^2((0,1) \times \mathbb{T}^\ep)}^2 \\
    & \leq  C \| \vec{\mathcal{E}} \|_{L^r((0,1) \times \mathbb{T}^\ep)} \left\| \nabla^\ep v_1 \right\|_{L^{r'}((0,1) \times \mathbb{T}^\ep)} + C \| \vec{\mathcal{E}} \|_{L^r((0,1) \times \mathbb{T}^\ep)}^2 . \notag
\end{align}
We then lower bound the term on the left-hand side using H\"{o}lder's inequality (and recalling the identity $r' = r/(r-1)$)
\begin{align} \label{eq:11522608}
    \int_0^1 \sum_{x \in \mathbb{T}^\ep} \left| \nabla^\ep v_1 (t , x) \right|^{r'} \, dt & \leq \left(  \int_0^1 \sum_{x \in \mathbb{T}^\ep} \mathbf{m}(t , x) \left| \nabla^\ep v_1 (t , x) \right|^2 \, dt \right)^{\frac{r'}{2}} \left( \int_0^1 \sum_{x \in \mathbb{T}^\ep} \mathbf{m}(t,x)^{-\frac{r}{r-2}} \, dt \right)^{\frac{r-2}{2r-2}}.
\end{align}
The inequality~\eqref{eq:11522608} can be rewritten using more compact notation as follows
\begin{equation*}
    \left\|\nabla^\ep v_1 \right\|_{L^{r'} ((0,1) \times \mathbb{T}^\ep)}^2 \left\| \mathbf{m}^{-1} \right\|_{L^{\frac{r}{r-2}}((0,1) \times \mathbb{T}^\ep)}^{-1} \leq   \int_0^1 \ep^d \sum_{x \in \mathbb{T}^\ep} \mathbf{m}(t , x) \left| \nabla^\ep v_1 (t , x) \right|^2 \, dt .
\end{equation*}
Combining the previous inequality with~\eqref{eq:12.142608}, we deduce that
\begin{equation*}
    \left\|\nabla^\ep v_1 \right\|_{L^{r'} ((0,1) \times \mathbb{T}^\ep)}^2 \left\| \mathbf{m}^{-1} \right\|_{L^{\frac{r}{r-2}}((0,1) \times \mathbb{T}^\ep)}^{-1}
    \leq C \| \vec{\mathcal{E}} \|_{L^r((0,1) \times \mathbb{T}^\ep)} \left\| \nabla^\ep v_1 \right\|_{L^{r'}((0,1) \times \mathbb{T}^\ep)} + C \| \vec{\mathcal{E}} \|_{L^r((0,1) \times \mathbb{T}^\ep)}^2.
\end{equation*}
This inequality further implies that
\begin{equation} \label{eq:14432608}
    \left\| \nabla^\ep v_1 \right\|_{L^{r'}((0,1) \times \mathbb{T}^\ep)} \leq C \left\| \mathbf{m}^{-1} \right\|_{L^{\frac{r}{r-2}}((0,1) \times \mathbb{T}^\ep)} \| \vec{\mathcal{E}} \|_{L^r((0,1) \times \mathbb{T}^\ep)},
\end{equation}
The term on the right-hand side can be estimated on the (high probability) event $\left\{  \inf \mathbf{m} \times \mathbf{M}_+  > \ep^{\theta/2} \right\}$ by using~\eqref{eq:realD20},
\begin{equation} \label{eq:12342708}
    \left\| \mathbf{m}^{-1} \right\|_{L^{\frac{r}{r-2}}((0,1) \times \mathbb{T}^\ep)} \indc_{\{ \inf \mathbf{m} \times \mathbf{M}_+ > \ep^{\theta} \} }  \leq \ep^{-\theta}\left\| \mathbf{M} \right\|_{L^{\frac{r}{r-2}}((0,1) \times \mathbb{T}^\ep)} \leq \mathcal{O}_{\Psi , c} \left(C \ep^{- \theta} \right).
\end{equation}
Combining the previous inequality with~\eqref{eq:14432608}, using the estimate~\eqref{eq:estiamteerroreterms} on the error terms and Proposition~\ref{prop:prop2.20} ``Product", we deduce that
\begin{align*}
    \left\| \nabla^\ep v_1 \right\|_{L^{r'}((0,1) \times \mathbb{T}^\ep)} \indc_{\{ \inf \mathbf{m} \times \mathbf{M}_+ > \ep^{\theta} \} } \leq \mathcal{O}_{\Psi , c} \left( C \ep^{-\theta} \ep^{2 \theta}\right) \leq \mathcal{O}_{\Psi , c} \left( C \ep^{\theta}\right).
\end{align*}
Combining this inequality with~\eqref{eq:13152207} and the Poincar\'e inequality (for discrete periodic functions), we obtain that
\begin{equation} \label{eq:estiamtev1sansgrad}
    \left\| v_1 \right\|_{L^{r'}((0,1) \times \mathbb{T}^\ep)} \indc_{\{ \inf \mathbf{m} \times \mathbf{M}_+ > \ep^{\theta} \}} \leq \mathcal{O}_{\Psi , c} \left( C \ep^{\theta}\right).
\end{equation}

\medskip

\textit{Step 4. Estimating the term $v_2$.}

\medskip

We show in this step is the $L^2$-norm of the function $v_2$ over the parabolic cylinder $(0 ,1 ) \times \mathbb{T}^\ep$ is small following the same strategy as in Step 3 above. The argument is more involved, especially regarding the average value of the function $v_2$ which is not equal to $0$ and thus has to be properly estimated. 

\medskip

\textit{Substep 4.1. Estimating the average value of the term $v_2$.}

\medskip

We first estimate the $L^{r'}$-norm of the average value of the map $v_2$, i.e., the quantity
\begin{equation*}
    \left\| \left( v_2 \right)_{\mathbb{T}^\ep} \right\|_{L^{r'}((0,1))} :=  \left( \int_0^1 \left| \left( v_2(t , \cdot) \right)_{\mathbb{T}^\ep} \right|^{r'} \, dt \right)^{1/r'}.
\end{equation*}
Specifically, we will prove the estimate
\begin{equation} \label{eq:r'normaveragev2}
    \left\| \left( v_2 \right)_{\mathbb{T}^\ep} \right\|_{L^{r'}((0,1))} \leq \mathcal{O}_{\Psi , c} \left(  C \ep^{2\theta} \right).
\end{equation}
To prove~\eqref{eq:r'normaveragev2}, we first sum the first line of~\eqref{eq:v3scexp} over all the vertices of $x \in \mathbb{T}^\ep$ (and perform an integration by parts) to obtain that
\begin{equation*}
    \partial_t (v_2(t , \cdot))_{\mathbb{T}^\ep} =  \left( \mathcal{E}(t , \cdot ) \right)_{\mathbb{T}^\ep},
\end{equation*}
which, after an integration over the times, yields
\begin{equation*}
    (v_2(t , \cdot))_{\mathbb{T}^\ep} =  \int_0^t \left( \mathcal{E}(s , \cdot ) \right)_{\mathbb{T}^\ep} \, ds.
\end{equation*}
Applying the inequality~\eqref{eq:estiamteerroretermsavecrage}, we obtain that, for any $t \in (0 , 1)$,
\begin{equation*}
    \left| (v_2(t , \cdot))_{\mathbb{T}^\ep} \right| \leq \mathcal{O}_{\Psi , c} \left(  C \ep^{2\theta} \right).
\end{equation*}
The inequality~\eqref{eq:r'normaveragev2} follows from the previous inequality by applying Proposition~\ref{prop:prop2.20} ``Integration".

\medskip

\textit{Substep 4.2. Estimating the $L^2$-norm of the gradient of the term $v_2$.}

\medskip

We then estimate the $L^2$-norm of the gradient of $v_2$. To this end, we first use the same arguments as in Substep 3.2 (which rely on the inequality~\eqref{eq:11140912sec7}) and obtain the upper bound
\begin{equation*}
    \int_0^{1} \sum_{x \in \mathbb{T}^\ep}  \mathbf{m}(t , x) \left| \nabla^\ep v_2 (t , x) \right|^2 \, dt \leq C \int_0^{1}  \sum_{x \in \mathbb{T}^\ep} \nabla^\ep v_2(t , x) \cdot \a(t , x) \nabla^\ep v_2(t , x) \, dt + C \ep^2 \int_0^{1} \sum_{x \in \mathbb{T}^\ep} | \mathcal{E} (t,x)|^2 \, dt.
\end{equation*}
Applying H\"{o}lder's inequality as in~\eqref{eq:11522608}, the previous inequality implies 
\begin{align} \label{eq:estaftermoderation}
    \left\|\nabla^\ep v_2 \right\|_{L^{r'} ((0,1) \times \mathbb{T}^\ep)}^2 \left\| \mathbf{m}^{-1} \right\|_{L^{\frac{r}{r-2}}((0,1) \times \mathbb{T}^\ep)}^{-1}
    & \leq \int_0^{1} \ep^d \sum_{x \in \mathbb{T}^\ep}  \mathbf{m}(t , x) \left| \nabla^\ep v_2 (t , x) \right|^2 dt \\
    & \leq C \int_0^{1} \ep^d \sum_{x \in \mathbb{T}^\ep} \nabla^\ep v_2(t , x) \cdot \a(t , x) \nabla^\ep v_2(t , x) \, dt \notag \\
    & \qquad + C \ep^2 \| \mathcal{E} \|_{L^2 \left((0,1) \times \mathbb{T}^\ep \right)}^2 . \notag
\end{align}
We next apply Lemma~\ref{lem.idenH-1par}, we obtain the existence of a pair of functions~$(h, h^*) \in L^r((0,1) ; W^{1,r}(\mathbb{T}^\ep)) \times L^{r}((0,1) ; W^{-1, r}(\mathbb{T}^\ep))$ such that
\begin{equation} \label{eq:decompE}
    \mathcal{E} = \partial_t h + h^*.
\end{equation}
Additionally, the map $h$ is continuous in the time variable $t$, satisfies $h(0 , \cdot) = h(1 , \cdot) = 0$ and the pair $(h , h^*)$ satisfies
\begin{equation} \label{eq:uppboundww*}
    \left\| h \right\|_{L^r((0,1) ; W^{1,r}(\mathbb{T}^\ep))} + \left\| h^* \right\|_{L^{r}((0,1) ; W^{-1,r}(\mathbb{T}^\ep))}  \leq C \left\| \mathcal{E} \right\|_{W^{-1,r}_{\mathrm{par}}((0,1)  \times \mathbb{T}^\ep)}.
\end{equation}
Combining the equations \eqref{eq:v3scexp} and~\eqref{eq:decompE}, we deduce that
\begin{equation} \label{eq:eqforv2}
    \partial_t v_2 -  \nabla^\ep \cdot \a \nabla^\ep v_2 = \partial_t h + h^* ~\mbox{in}~ (0,1) \times \mathbb{T}^\ep.
\end{equation}
Multiplying the equation by $v_2$, integrating over the parabolic cylinder $(0 , 1) \times \mathbb{T}^\ep$, and using that the map $h$ is equal to $0$ at times $t=0$ and $t = 1$, we obtain the identity
\begin{multline} \label{eq:iden.2scexp001}
     \sum_{x \in \mathbb{T}^\ep} \left| v_2 (1 , x) \right|^2 +  \int_0^{1} \sum_{x \in \mathbb{T}^\ep} \nabla^\ep v_2(t , x) \cdot \a(t , x) \nabla^\ep v_2 (t , x) \, dt = \\
     - \int_0^{1} \sum_{x \in \mathbb{T}^\ep} h(t , x) \partial_t v_2(t , x) \, dt + \int_0^{1} \sum_{x \in \mathbb{T}^\ep} h^*(t , x) v_2(t , x) \, dt.
\end{multline}
Since the first term on the left-hand side of~\eqref{eq:iden.2scexp001} is nonnegative, we deduce from the previous identity that
\begin{align} \label{eq:10122708}
     \lefteqn{\int_0^{1} \ep^d \sum_{x \in \mathbb{T}^\ep} \nabla^\ep v_2(t , x) \cdot \a(t , x) \nabla^\ep v_2 (t , x) \, dt } \qquad & \\ &
     \leq - \int_0^{1} \ep^d \sum_{x \in \mathbb{T}^\ep} h(t , x) \partial_t v_2(t , x) \, dt + \int_0^{1}  \ep^d \sum_{x \in \mathbb{T}^\ep} h^*(t , x) v_2(t , x) \, dt \notag \\
     & \leq  - \int_0^{1} \ep^d \sum_{x \in \mathbb{T}^\ep} h(t , x) \partial_t v_2(t , x) \, dt + C \left\| h^* \right\|_{L^{r}((0,1) ; W^{-1,r}(\mathbb{T}^\ep))} \left\| v_2 \right\|_{L^{r'}((0,1) ; W^{1,r'}(\mathbb{T}^\ep))},\notag \\
     & \leq - \int_0^{1} \ep^d \sum_{x \in \mathbb{T}^\ep} h(t , x) \partial_t v_2(t , x) \, dt +  \left\| \mathcal{E} \right\|_{W^{-1,r}_{\mathrm{par}}((0,1)  \times \mathbb{T}^\ep)} \left\| v_2 \right\|_{L^{r'}((0,1) ; W^{1,r'}(\mathbb{T}^\ep))}. \notag
\end{align}
We next focus on the first term on the right-hand side, perform an integration by part in time (using that $h(0, \cdot) = h(1, \cdot) = 0$) and use the identity~\eqref{eq:eqforv2} to write
\begin{align*}
    \int_0^{1} \sum_{x \in \mathbb{T}^\ep}  h(t , x) \partial_t v_2(t , x) \, dt
    & = - \int_0^{1} \sum_{x \in \mathbb{T}^\ep}  h(t , x) \left( \nabla^\ep \cdot \a(t,x) \nabla^\ep v_2(t,x)+ \partial_t h(t,x) + h^*(t,x) \right)\, dt \\
    & = \int_0^{1} \sum_{x \in \mathbb{T}^\ep}  \nabla^\ep h(t , x) \cdot \a(t,x) \nabla^\ep v_2(t,x) - \frac12 \partial_t h^2(t , x) - h^*(t,x) h(t , x) \, dt \\
    & = \int_0^{1} \sum_{x \in \mathbb{T}^\ep}  \nabla^\ep h(t , x) \cdot \a(t,x) \nabla^\ep v_2(t,x) - h^*(t,x) h(t , x) \, dt.
\end{align*}
Using the inequality~\eqref{eq:uppboundww*} together with the observation that $r > r'$ (since $r> 2$), we further obtain that
\begin{align*}
    \int_0^{1} \sum_{x \in \mathbb{T}^\ep}  h(t , x) \partial_t v_2(t , x) \, dt & \leq \int_0^{1} \sum_{x \in \mathbb{T}^\ep}  \nabla^\ep h(t , x) \cdot \a(t,x) \nabla^\ep v_2(t,x) \, dt \\
    & \qquad + \left\| h \right\|_{L^r((0,1) ; W^{1,r}(\mathbb{T}^\ep))} \left\| h^* \right\|_{L^{r'}((0,1) ; W^{-1,r'}(\mathbb{T}^\ep))} \\
    & \leq \int_0^{1} \sum_{x \in \mathbb{T}^\ep}  \nabla^\ep h(t , x) \cdot \a(t,x) \nabla^\ep v_2(t,x) \, dt \\
    & \qquad + \left\| h \right\|_{L^r((0,1) ; W^{1,r}(\mathbb{T}^\ep))} \left\| h^* \right\|_{L^{r}((0,1) ; W^{-1,r}(\mathbb{T}^\ep))} \\
    & \leq \int_0^{1} \sum_{x \in \mathbb{T}^\ep}  \nabla^\ep h(t , x) \cdot \a(t,x) \nabla^\ep v_2(t,x) \, dt  + C \left\| \mathcal{E} \right\|_{W^{-1,r}_{\mathrm{par}}((0,1)  \times \mathbb{T}^\ep)}^2.
\end{align*}
Using the Cauchy-Schwarz inequality and the same computation as in~\eqref{eq:10122708}, we deduce that
\begin{multline*}
    \int_0^{1} \ep^d \sum_{x \in \mathbb{T}^\ep}  h(t , x) \partial_t v_2(t , x) \, dt 
    \\
    \leq \left( \int_0^{1} \ep^d \sum_{x \in \mathbb{T}^\ep}  \nabla^\ep h \cdot \a \nabla^\ep h  \right)^{\frac 12} \left( \int_0^{1} \ep^d \sum_{x \in \mathbb{T}^\ep}  \nabla^\ep v_2 \cdot \a \nabla^\ep v_2  \right)^{\frac 12}  
    + C \left\| \mathcal{E} \right\|_{W^{-1,r}_{\mathrm{par}}((0,1)  \times \mathbb{T}^\ep)}^2. 
\end{multline*}
Combining the previous inequality with~\eqref{eq:10122708}, we obtain
\begin{multline} \label{eq:11512708}
    \int_0^{1} \ep^d \sum_{x \in \mathbb{T}^\ep} \nabla^\ep v_2(t , x) \cdot \a(t , x) \nabla^\ep v_2 (t , x) \, dt \\
    \leq  C \int_0^{1} \ep^d \sum_{x \in \mathbb{T}^\ep}  \nabla^\ep h(t ,x) \cdot \a(t,x) \nabla^\ep h(t,x) \, dt 
    \\ + C \left\| \mathcal{E} \right\|_{W^{-1,r'}_{\mathrm{par}}((0,1)  \times \mathbb{T}^\ep)}^2 +  \left\| \mathcal{E} \right\|_{W^{-1,r}_{\mathrm{par}}((0,1)  \times \mathbb{T}^\ep)} \left\| v_2 \right\|_{L^{r'}((0,1) ; W^{1,r'}(\mathbb{T}^\ep))}.
\end{multline}
The first term on the right-hand side is then estimated using Jensen's inequality. We obtain
\begin{align*}
    \int_0^{1} \ep^d \sum_{x \in \mathbb{T}^\ep}  \nabla^\ep h(t ,x) \cdot \a(t,x) \nabla^\ep h(t,x) \, dt  & \leq \int_0^{1} \ep^d \sum_{x \in \mathbb{T}^\ep} \mathbf{\Lambda}_+(t , x) \left| \nabla^\ep h(t ,x) \right|^2  \, dt \\
    & \leq \left\| \mathbf{\Lambda}_+ \right\|_{L^{\frac{r}{r-2}} ((0,1) \times \mathbb{T}^\ep)} \left\| \nabla^\ep h \right\|_{L^r ((0,1) \times \mathbb{T}^\ep)}^2.
\end{align*}
Combining the previous inequality with~\eqref{eq:estaftermoderation} and~\eqref{eq:11512708} (and noting that $\mathbf{m} \leq 1$ by definition), we obtain
\begin{align*}
    \left\|\nabla^\ep v_2 \right\|_{L^{r'} ((0,1) \times \mathbb{T}^\ep)}^2  &
    \leq C \left\| \mathbf{m}^{-1} \right\|_{L^{\frac{r}{r-2}}((0,1) \times \mathbb{T}^\ep)} \left( 1 + \left\| \mathbf{\Lambda}_+ \right\|_{L^{\frac{r}{r-2}} ((0,1) \times \mathbb{T}^\ep)} \right) \left\| \mathcal{E} \right\|_{W^{-1,r}_{\mathrm{par}}((0,1)  \times \mathbb{T}^\ep)}^2 \\
    & \qquad  + C \left\| \mathbf{m}^{-1} \right\|_{L^{\frac{r}{r-2}}((0,1) \times \mathbb{T}^\ep)} \left\| \mathcal{E} \right\|_{W^{-1,r}_{\mathrm{par}}((0,1)  \times \mathbb{T}^\ep)} \left\| v_2 \right\|_{L^{r'}((0,1) ; W^{1,r'}(\mathbb{T}^\ep))} \\
    & \qquad +   C \ep^2 \| \mathcal{E} \|_{L^2 \left((0,1) \times \mathbb{T}^\ep \right)}^2.
\end{align*}
We then use the Poincar\'e inequality (on the torus with respect to the spatial variable) to write
\begin{equation*}
    \left\| v_2 \right\|_{L^{r'}((0,1) ; W^{1,r'}(\mathbb{T}^\ep))} \leq C \left\| \nabla^\ep v_2 \right\|_{L^{r'}((0,1) \times \mathbb{T}^\ep)} + C \left\| \left( v_2 \right)_{\mathbb{T}^\ep} \right\|_{L^{r'}((0,1))}.
\end{equation*}
Combining the two previous inequalities, we obtain the upper bound
\begin{align*}
    \lefteqn{ \left\|\nabla^\ep v_2 \right\|_{L^{r'} ((0,1) \times \mathbb{T}^\ep)} } \qquad & \\ &
    \leq C \left\| \mathbf{m}^{-1} \right\|_{L^{\frac{r}{r-2}}((0,1) \times \mathbb{T}^\ep)} \left( 1 + \left\| \mathbf{\Lambda}_+ \right\|_{L^{\frac{r}{r-2}} ((0,1) \times \mathbb{T}^\ep)} \right) \left\| \mathcal{E} \right\|_{W^{-1,r}_{\mathrm{par}}((0,1)  \times \mathbb{T}^\ep)} \\
    & \qquad + C \left\| \mathbf{m}^{-1} \right\|_{L^{\frac{r}{r-2}}((0,1) \times \mathbb{T}^\ep)}^{1/2} \left\| \left( v_2 \right)_{\mathbb{T}^\ep} \right\|_{L^{r'}((0,1))}^{1/2} \left\| \mathcal{E} \right\|_{W^{-1,r}_{\mathrm{par}}((0,1)  \times \mathbb{T}^\ep)}^{1/2} \\
    & \qquad + C \ep \| \mathcal{E} \|_{L^2 \left((0,1) \times \mathbb{T}^\ep \right)}.
\end{align*}
On the high-probability event event $\left\{  \inf \mathbf{m} \times \mathbf{M}_+  > \ep^{\theta} \right\}$, the term on the left-hand side can be estimated using the inequalities~\eqref{eq:realD20},~\eqref{eq:12342708}, the upper bound for the average value of $v_2$ stated in~\eqref{eq:r'normaveragev2}, the upper bound on the error terms~\eqref{eq:estiamteerroreterms} and Proposition~\ref{prop:prop2.20} ``Product". We obtain
\begin{equation*}
    \left\|\nabla^\ep v_2 \right\|_{L^{r'} ((0,1) \times \mathbb{T}^\ep)} \indc_{ \left\{  \inf \mathbf{m} \times \mathbf{M}_+  > \ep^{\theta} \right\} } \leq \mathcal{O}_{\Psi , c} \left( \ep^\theta \right).
\end{equation*}
Combining the previous inequality with the Poincar\'e inequality (on the torus with respect to the spatial variable) and the inequality~\eqref{eq:r'normaveragev2}, we deduce that
\begin{align} \label{eq:13540609}
    \left\| v_2 \right\|_{L^{r'} ((0,1) \times \mathbb{T}^\ep)} \indc_{ \left\{  \inf \mathbf{m} \times \mathbf{M}_+  > \ep^{\theta} \right\} } & \leq \left( \left\|\nabla^\ep v_2 \right\|_{L^{r'} ((0,1) \times \mathbb{T}^\ep)}  +  \left\| \left( v_2 \right)_{\mathbb{T}^\ep} \right\|_{L^{r'}((0,1))} \right)\indc_{ \left\{  \inf \mathbf{m} \times \mathbf{M}_+  > \ep^{\theta} \right\} } \\
    & \leq \mathcal{O}_{\Psi , c} \left( \ep^\theta \right). \notag
\end{align}

\medskip

\textit{Step 6. The conclusion.}

\medskip

Combining~\eqref{eq:estiamtev0} of Step 2,~\eqref{eq:estiamtev1sansgrad} of Step 3 and~\eqref{eq:13540609} of Step 4 with Jensen's inequality (using that $r' < 2$), we obtain
\begin{align} \label{eq:19080609}
    \lefteqn{ \left\|  u^\ep - w^\ep \right\|_{L^{r'}((0,1) \times \mathbb{T}^\ep)} \indc_{ \left\{  \inf \mathbf{m} \times \mathbf{M}_+  > \ep^{\theta} \right\} } } \qquad & \\ & \leq \left( \left\|  v_0 \right\|_{L^2((0,1) \times \mathbb{T}^\ep)} + \left\|  v_1 \right\|_{L^{r'}((0,1) \times \mathbb{T}^\ep)} + \left\|  v_2 \right\|_{L^{r'}((0,1) \times \mathbb{T}^\ep)} \right) \indc_{ \left\{  \inf \mathbf{m} \times \mathbf{M}_+  > \ep^{\theta} \right\} }  \notag
    \\
    & \leq \mathcal{O}_{\Psi, c} \left( C \ep^{\theta} \right). \notag
\end{align}
To complete the proof of  Theorem~\ref{main.thm}, we need to remove the indicator of the event $\left\{  \inf \mathbf{m} \times \mathbf{M}_+  > \ep^{\theta} \right\}$ on the left-hand side. This is achieved by first noting that the inequality~\eqref{eq:22572106} implies
\begin{equation*}
       \indc_{ \left\{  \inf \mathbf{m} \times \mathbf{M}_+  < \ep^{\theta} \right\}} \leq \mathcal{O}_{\Psi , c} \left( C \ep \right).
\end{equation*}
Combining the previous inequality with~\eqref{eq:boundnablauep} and~\eqref{est:Linftynormw}, using Jensen's inequality (together with $r' < r$) and Proposition~\ref{prop:prop2.20} ``Product", we obtain
\begin{align*}
    \left\|  u^\ep - w^\ep \right\|_{L^{r'}((0,1) \times \mathbb{T}^\ep)} \indc_{ \left\{  \inf \mathbf{m} \times \mathbf{M}_+  \leq \ep^{\theta} \right\} } & \leq \left( \left\|  u^\ep \right\|_{L^{r}((0,1) \times \mathbb{T}^\ep)}  + \left\|  w^\ep \right\|_{L^{r}((0,1) \times \mathbb{T}^\ep)} \right) \indc_{ \left\{  \inf \mathbf{m} \times \mathbf{M}_+   \leq \ep^{\theta} \right\} } \\
    & \leq \mathcal{O}_{\Psi , c} \left( C \ep \right) \\
    & \leq \mathcal{O}_{\Psi , c} \left( C \ep^\theta \right). 
\end{align*}
Combining the previous inequality with~\eqref{eq:19080609}, we deduce that
\begin{equation*}
    \left\|  u^\ep - w^\ep \right\|_{L^{r'}((0,1) \times \mathbb{T}^\ep)} \leq \mathcal{O}_{\Psi , c} \left( C \ep^\theta \right).
\end{equation*}
The result can finally be improves from the $L^{r'}$ to the $L^2$-norm by interpolating the space $L^2$ between the spaces $L^{r'}$ and $L^r$ and using Proposition~\ref{prop:prop2.20} ``Product" and ``Powers". Specifically, using the identity $1/2 = / (2r') + 1/(2r)$, we may write
\begin{align*}
    \left\|  u^\ep - w^\ep \right\|_{L^{2}((0,1) \times \mathbb{T}^\ep)} & \leq \left\|  u^\ep - w^\ep \right\|_{L^{r'}((0,1) \times \mathbb{T}^\ep)}^{1/2} \left\|  u^\ep - w^\ep \right\|_{L^{r}((0,1) \times \mathbb{T}^\ep)}^{1/2} \\
    & \leq  \left\|  u^\ep - w^\ep \right\|_{L^{r'}((0,1) \times \mathbb{T}^\ep)}^{1/2} \left( \left\|  u^\ep \right\|_{L^{r}((0,1) \times \mathbb{T}^\ep)} + \left\|  w^\ep \right\|_{L^{r}((0,1) \times \mathbb{T}^\ep)} \right)^{1/2} \\
    & \leq \mathcal{O}_{\Psi , c} (C \ep^{\theta/2}).
\end{align*}
Redefining the value of the exponent $\theta$ completes the proof of Theorem~\ref{main.thm}.

\newpage

\appendix

\section{Properties of the Langevin dynamic} \label{App.B}

In this section, we give an outline of the proof of the results stated in Section~\ref{sec:introdLangevindyn}. The proof of these results is greatly simplified when the potential $V$ is assumed to satisfy a strict convexity assumption of the form:
\begin{equation*}
    \exists c > 0, \, \forall x \in \Rd, \, \mathrm{Hess} \, V (x) \geq c I_n.
\end{equation*}
We will thus make this assumption to simplify the proofs (only) in the rest of this section, and explain how the results can be extended to the potentials satisfying Assumption~\eqref{AssPot}. We will also allow, only in this section, all the constants to depend on the sidelength $L$ of the torus (N.B. This is allowed because all the properties proved in this section are qualitative).

\begin{proof}[Proof of the result stated in Definition~\ref{Prop:Langevin}]
    We fix a slope $p \in \Rd$ and an integer $L \in \N$ and recall the notation introduced in Section~\ref{sectionBM} (in particular, the fact that the Brownian motions are defined for $t \in \R$). We let $\varphi$ be a random variable distributed according to the Gibbs measure $\mu_{L,p}$ which is independent of the Brownian motions. 
    
    For a fixed negative integer $K \in \Z \setminus \N$, we let $\varphi^K : (K , \infty) \times \mathbb{T}_L \to \R$ be the solution of the system of stochastic differential equations (with periodic boundary condition)
    \begin{equation} \label{eq:introSDEappB}
   \left\{ \begin{aligned}
    d \varphi^K(t , x;p) & = \nabla \cdot D_p V(\nabla \varphi^K(t , x;p)) dt + \sqrt{2} d \tilde B_t(x) & \mbox{for} &~ (t , x) \in (K , \infty) \times \mathbb{T}_L, \\
    \varphi^K(K , x;p) & = \varphi & \mbox{for} &~  x \in \mathbb{T}_L.
    \end{aligned}
    \right.
\end{equation}
(N.B. Since we are working in finite-volume, the existence and uniqueness of the solution follows standard arguments; one can even prove, using a Picard iteration scheme, that for \emph{any} realization of the trajectories of the Brownian motions and \emph{any} realization of the initial condition $\varphi$, there exists a unique solution to~\eqref{eq:introSDEappB}).

Let us note that, for any $K \in \Z \setminus \N$, the dynamic~\eqref{eq:introSDEappB} is stationary and ergodic. In particular, for any $t \in (K , \infty)$, the random variable $\varphi^K(t , \cdot;p)$ is distributed according to the Gibbs measure $\mu_{L,p}$. 

We then consider the difference $w := \varphi^{K-1}(\cdot , \cdot ; p ) - \varphi^K (\cdot , \cdot ; p )$. Using the Definition~\eqref{eq:introSDEappB}, we see that the map $w$ solves the parabolic equation
\begin{equation} \label{eq:introSDEappB22}
   \left\{ \begin{aligned}
    \partial_t w (t , x) & = \nabla \cdot \a^K \nabla w(t , x) & \mbox{for} &~ (t , x) \in (0 , \infty) \times \mathbb{T}_L, \\
     w(K , \cdot) & = \varphi^{K-1}(K , x ;p) - \varphi^{K}(K , x ;p) & \mbox{for} &~  x \in \mathbb{T}_L.
    \end{aligned}
    \right.
\end{equation}
with
\begin{equation} \label{eq:17421406}
    \a^K(t , x ; p) := \int_0^1 D_p^2 V (p + s \nabla \varphi^{K-1}(t , x) + (1-s) \nabla \varphi^{K}(t , x)) \, ds \geq c I_d.
\end{equation}
The parabolic equation~\eqref{eq:introSDEappB22} can be solved by appealing to the heat kernel $P_\a$ and we have
\begin{equation*}
    w (t , x) = \sum_{y \in \mathbb{T}_L} P_{\a^K(\cdot ; p)}(t , x ; K , y) (\varphi^{K-1}(K,y ;p) - \varphi^{K}(K , y ;p)).
\end{equation*}
Using the Definition~\eqref{eq:17421406} (and the lower bound $\a^K(t , x ; p) \geq c I_d$), the following deterministic (and suboptimal) upper bound on the heat kernel can be established
\begin{equation} \label{uppboundheatkernel}
   \forall (t , x) \in (T , \infty) \times \mathbb{T}_L, \, \left|  P_{\a^K (\cdot ; p)}(t , x ; K , y) \right| \leq \exp \left( - \frac{t - K}{CL^2} \right).
\end{equation}
A combination of the two previous displays then yields the bound (denoting by $C_L$ a constant which is allowed to depend on $L$): for any $(t , x) \in (K , \infty),$
\begin{align*}
    \mathbb{E} \left[ \sup_{t' \geq t} \left| w (t' , x) \right| \right] & \leq \sum_{y \in \mathbb{T}_L} \exp \left( - \frac{t - K}{C_L} \right) \mathbb{E} \left[ \left| \varphi^{K-1}(K,y ;p) - \varphi^{K}(K , y ;p) \right| \right] \\
    & \leq C_L \exp \left( - \frac{t - K}{C_L} \right).
\end{align*}
We thus deduce that, for any $(t , x) \in \R \times \mathbb{T}_L$
\begin{equation*}
    \E \left[ \sum_{\substack{K \in \mathbb{Z} \\ K \leq t}} \sup_{t' \geq t} \left| \varphi^{K-1}(t ,x ;p) - \varphi^{K}(t , x ;p) \right| \right] \leq C_L  \sum_{\substack{K \in \mathbb{Z} \\ K \leq t}} \exp \left( - \frac{t - K}{C_L} \right) < \infty.
\end{equation*}
This implies that, for almost every realization of the initial condition $\varphi$ and almost every realization of the Brownian motions, the random function $(t , x) \mapsto \varphi^{K}(t , x)$ converges (locally uniformly) as $K \to - \infty$ to a limit which we denote by $(t , x) \mapsto \varphi_L(t , x ; p)$.

One can then verify that the function $\varphi_L(\cdot , \cdot ; p)$ satisfies the desired properties: the properties ``(i) Average value" and ``(iii) Stochastic differential equations" of Definition~\ref{Prop:Langevin}, and the properties ``Distribution", ``Stationarity" and ``Ergodicity" are inherited from the dynamics $\varphi^{K-1}$. The condition ``(ii) Growth" is obtained by using the stationarity property together with an application of the Borel-Cantelli Lemma, and the uniqueness part of the statement is obtained by using the same computation as in~\eqref{eq:introSDEappB22} (and making use of the upper bound~\eqref{uppboundheatkernel} on the heat kernel).

For the differentiability with respect to the slope, we consider the dynamic started from $0$: for $K \in \Z \setminus \N$,
\begin{equation} \label{eq:introSDEappB0init}
   \left\{ \begin{aligned}
    d \tilde \varphi^K(t , x;p) & = \nabla \cdot D_p V(p + \nabla \tilde \varphi^K(t , x;p)) dt + \sqrt{2} d \tilde B_t(x) & \mbox{for} &~ (t , x) \in (K , \infty) \times \mathbb{T}_L, \\
    \tilde \varphi^K(K , x;p) & = 0 & \mbox{for} &~  x \in \mathbb{T}_L.
    \end{aligned}
    \right.
\end{equation}
The same arguments as in~\eqref{eq:introSDEappB22} shows that, for almost every realization of the Brownian motions, the random function $(t , x) \mapsto \tilde \varphi^K(t , x ; p)$ converges (locally uniformly) to the dynamic $(t , x) \mapsto \tilde \varphi_L(t , x ; p)$.

Using the results of~\cite[Chapter 5, Theorem 3.1]{hartman2002ordinary}, the dynamic~\eqref{eq:introSDEappB0init} is differentiable with respect to the slope $p \in \Rd$, and its derivative at the slope $p \in \Rd$ in the direction $\lambda \in \Rd$ is the solution of the parabolic equation
\begin{equation*}
   \left\{ \begin{aligned}
   \partial_t \tilde w_{L , p , \lambda}^K - \nabla \cdot \tilde \a^K (\cdot ; p )( \lambda +  \nabla \tilde w_{L , p , \lambda}^K ) & = 0   & \hspace{5mm} \mbox{in} \hspace{3mm}  \R \times \mathbb{T}_L,\\
   \tilde w_{L , p , \lambda}^K(K , \cdot) & = 0  & \hspace{5mm} \mbox{in} \hspace{3mm} \mathbb{T}_L,
    \end{aligned}
    \right.
\end{equation*}
where
\begin{equation*}
    \tilde \a^K (t , x ; p ) := D_p^2 V (p +  \nabla \tilde \varphi^{K}(t , x)). 
\end{equation*}
One can then show that, as $K \to -\infty$, the function $\tilde w_{L , p , \lambda}^K$ converges (locally uniformly over the space, time and slope variables) to a function $(t , x) \mapsto w_{L , p , \lambda}$ which is solution to the equation~\eqref{eqstatpardifferentiated} (and that under some moment condition, this solution is unique, this can be established using the same technique as in the computation~\eqref{eq:introSDEappB22}). We deduce from these observations (and the almost sure uniform convergence) that the map $\varphi(t , x ; p)$ is differentiable with respect to $p \in \Rd$ and that its derivative in the direction $\lambda \in \Rd$ is given by the function $w_{L , \lambda , p}.$ 

The last step of the proof is the differentiablity with respect to the Brownian motions. To this end, we consider the dynamic~\eqref{eq:introSDEappB0init}. Let us fix a large negative integer $K \in \Z \setminus \N$ and let $t , s \in \R$ with $s < t$ and $x \in \mathbb{T}_L$. We then fix a realization of the Brownian motions $\{ \tilde B_t(x) \, : \, t \in \R, \, x \in \mathbb{T}_L \}$ (N.B. this can be done since the dynamic can be solved for every realization of the Brownian motions) and, for $\xi \in \R$, denote by (N.B. Note that $\tilde \varphi^K_0 = \tilde \varphi^K$)
\begin{equation*}
    \tilde \varphi^K_\xi (t' , y;p) := \tilde \varphi^K_\xi (t' , y;p) \left( \{ \tilde B_t(x) + \xi \delta_y(x) g_{s,t}(t) \, : \, t \in \R, \, x \in \mathbb{T}_L \} \right).
\end{equation*}
Let us note that dynamic $\tilde \varphi^K_\xi (t , x;p)$ is a solution to the equation
\begin{equation*}
   \left\{ \begin{aligned}
    d \tilde \varphi^K_\xi(t' , x;p) & = \nabla \cdot D_p V(p + \nabla \tilde \varphi^K_\xi(t' , y;p)) dt + \sqrt{2} d \tilde B_{t'}(x) + \sqrt{2} \xi \delta_x(y) \indc_{(s , t)}(t') & \mbox{in} &~  (K , \infty) \times \mathbb{T}_L, \\
    \tilde \varphi^K_\xi(K , x;p) & = 0 & \mbox{in} &~   \mathbb{T}_L.
    \end{aligned}
    \right.
\end{equation*}
Differentiating each term in the equation above at $\xi = 0$ (following the definition~\eqref{eq:diffwithrespecttoBM}) and noting that the Brownian motions do not depend on $\xi$, we obtain that the derivative $w^K := \partial \tilde \varphi_{\xi}^K (\cdot , \cdot ; p)/\partial X_{t , s} (x)$ solves the parabolic equation
\begin{equation*}
     \left\{ \begin{aligned}
    \partial_t w^K (t' , y) &= \nabla \cdot \a^K \nabla w^K (t' , y) + \frac{\sqrt{2}}{t - s}  \delta_{y}(x)  \indc_{(s , t)}(t') &~\mbox{for}~& (t' , y) \in [0 , \infty] \times \mathbb{T}_L, \\
    w (K , y) &= 0 &~\mbox{for}~& x \in \mathbb{T}_L,
    \end{aligned} \right.
\end{equation*}
with the environment $\a^K(t' , y) := D_p^2 V( p + \tilde \varphi^K(t' , y;p) )$.
Applying Duhamel's principle, we obtain the identity
\begin{equation*}
    w(t' , y) = \frac{\sqrt{2}}{t - s} \int_{t}^{s} P_{\a^K} (t' , x ; s' , y)  \, ds'. 
\end{equation*}
The result is then deduced for the limiting dynamic $(t , x) \mapsto \varphi(t , x)$ by taking the limit $K \to - \infty$ in the identities above (and by verifying that a limit and a derivative can be exchanged).
\end{proof}

\section{Parabolic multiscale Poincar\'e inequality} \label{App.multiscale}

This section is devoted to the proof of the parabolic version of the multiscale Poincar\'e inequality and the identification of the dual parabolic space $\hat{W}^{-1,q}_{\mathrm{par}}((0,1) \times \mathbb{T}^\ep)$. We will prove the results for a general exponent $q \in (1 , \infty)$ (but only use it when $q = r$), and we recall that we denote by $q' = q/(q-1)$ the conjugate exponent of $q$.

\subsection{Parabolic multiscale Poincar\'e inequality}

\begin{proposition}[Parabolic multiscale Poincar\'e inequality] 
For any $q \in (1 , \infty)$, there exists a constant $C := C(d , q) < \infty$ such that, for any integer $n \in \N$ and any function $f : Q_{3^n} \to \R$,
\begin{equation*}
    \left\| f \right\|_{\underline{\hat{W}}^{-1,q}_{\mathrm{par}} (Q_{3^n})} \leq C  \left\| f \right\|_{\underline{L}^{q} (Q_{3^n})} + C \sum_{m = 0}^n 3^m \left( \left| \mathcal{Z}_{m,n} \right|^{-1}  \sum_{z \in \mathcal{Z}_{m,n}} \left| \left( f \right)_{z + Q_{3^m}}  \right|^{q} \right)^{\frac 1{q}}.
\end{equation*}
\end{proposition}

\begin{proof}
We fix an exponent $q \in (1 , \infty)$ and, following the argument of~\cite[Proposition 3.6]{ABM}, we decompose the proof into 3 Steps.

\medskip

\textit{Step 1.} We first prove the following inequality: for any integer $m \in \N$, and any function $g \in W^{1,q}_{\mathrm{par}}(Q_{3^m})$,
\begin{equation} \label{eq:18260209}
    \left\| g - \left( g \right)_{Q_{3^m}} \right\|_{\underline{L}^{q} \left( Q_{3^m} \right)} \leq C 3^m \left\| \nabla g \right\|_{\underline{L}^{q} \left( Q_{3^m} \right)} + C 3^m \left\| \partial_t g \right\|_{\underline{L}^{q} \left( (- 3^{2m}, 0) \, , \, \underline{W}^{-1,q}(\Lambda_{3^m}) \right)}.
\end{equation}
We first estimate the term on the left-hand side by introducing the spatial average of the function $g$ and by writing
\begin{align} \label{eq:15300209}
    \left\| g - \left( g \right)_{Q_{3^m}} \right\|_{\underline{L}^{q} \left( Q_{3^m} \right)}^{q} & = 3^{-2m} \int_{(- 3^{2m}, 0)} \left| \Lambda_{3^m} \right|^{-1} \sum_{x \in \Lambda_{3^m}} \left| g(t,x) - \left( g \right)_{Q_{3^m}} \right|^{q} \, dt \\
    & \leq C 3^{-(d+2)m} \int_{(- 3^{2m}, 0)} \sum_{x \in \Lambda_{3^m}} \left| g(t,x) - \left( g(t , \cdot) \right)_{\Lambda_{3^m}} \right|^{q} \, dt \notag  \\
    & \qquad + C 3^{-2m} \int_{(- 3^{2m}, 0)} \left| \left( g(t , \cdot) \right)_{\Lambda_{3^m}} - \left( g \right)_{Q_{3^m}} \right|^{q} \, dt. \notag
\end{align}
The first term on the right-hand side can be estimated using the Poincar\'e inequality
\begin{equation*}
    3^{-(d+2)m} \int_{(- 3^{2m}, 0)} \sum_{x \in \Lambda_{3^m}} \left| g(t,x) - \left( g(t , \cdot) \right)_{\Lambda_{3^m}} \right|^{q} \, dt \leq C  3^{qm} \left\| \nabla g \right\|_{\underline{L}^{q} \left( Q_{3^m} \right)}^q.
\end{equation*}
It is thus sufficient, in order to prove~\eqref{eq:18260209}, to show the inequality
\begin{multline} \label{eq:18320209}
    3^{-2m} \int_{(- 3^{2m}, 0)} \left| \left( g(t , \cdot) \right)_{\Lambda_{3^m}} - \left( g \right)_{Q_{3^m}} \right|^{q} \, dt \\
    \leq C 3^{qm} \left\| \nabla g \right\|_{\underline{L}^{q} \left( z + Q_{3^m} \right)}^{q} + C  3^{qm} \left\| \partial_t g \right\|_{\underline{L}^{q} \left( (- 3^{2m}, 0) , \underline{W}^{-1,q}(\Lambda_{3^m}) \right)}^{q}.
\end{multline}
We consider a non-negative function $\chi : \Lambda_{3^m} \to \R$ supported in $\Lambda_{3^m/2}$ satisfying
\begin{equation*}
    \frac{1}{\left| \Lambda_{3^m} \right|} \sum_{x \in \Lambda_{3^m}} \chi(x) = 1 ~~ \mbox{and}~~ \forall x \in  \Lambda_{3^m}, ~ \left| \nabla \chi (x) \right| \leq C 3^{-m}.
\end{equation*}
We then let $\psi$ be the solution of the discrete Neumann problem (N.B. the first assumption on the function $\chi$ ensures that this function is well-defined)
\begin{equation*}
    \left\{ \begin{aligned}
    - \Delta \psi & = 1 - \chi & ~\mbox{in}~ \Lambda_{3^m}, \\
    \mathbf{n} \cdot \nabla \psi & = 0 & ~\mbox{on}~ \partial \Lambda_{3^m}.
    \end{aligned} \right.
\end{equation*}
The discrete version of the Calder\'on-Zygmund estimates implies that
\begin{equation} \label{eq:16550209}
    \left\| \nabla \psi \right\|_{\underline{L}^{q'}(\Lambda_{3^m})} \leq C 3^{m}.
\end{equation}
Equipped with these functions, we may estimate the second term on the right-hand side of~\eqref{eq:15300209} as follows
\begin{align} \label{eq:17320209}
    3^{-2m} \int_{(- 3^{2m}, 0)} \left| \left( g(t , \cdot) \right)_{\Lambda_{3^m}} - \left( g \right)_{Q_{3^m}} \right|^{q} \, dt & \leq C 3^{-2m} \int_{(- 3^{2m}, 0)} \left| \left( g(t , \cdot) (1 - \chi)\right)_{\Lambda_{3^m}} \right|^{q} \, dt \\\
    & \qquad + C 3^{-2m} \int_{(- 3^{2m}, 0)} \left| \left( g(t , \cdot) \chi \right)_{\Lambda_{3^m}} - \left( g \chi \right)_{Q_{3^m}} \right|^{q} \, dt \notag \\
    & \qquad + C \left| \left( g (1- \chi) \right)_{Q_{3^m}}\right|^{q}. \notag
\end{align}
The first term can be estimated by using the definition of the function $\psi$, by performing a discrete integration by parts and by using H\"{o}lder's inequality as follows
\begin{align} \label{eq:17320209bis}
   3^{-2m} \int_{(- 3^{2m}, 0)} \left| \left( g(t , \cdot) (1 - \chi)\right)_{\Lambda_{3^m}} \right|^{q} \, dt & = 3^{-2m} \int_{(- 3^{2m}, 0)} \left| \left( \nabla g(t , \cdot) \cdot \nabla \psi \right)_{\Lambda_{3^m}} \right|^{q} \, dt \\
   & \leq \left\| \nabla g \right\|_{\underline{L}^{q}(Q_{3^m})}^{q} \left\| \nabla \psi \right\|_{\underline{L}^{q'}(\Lambda_{3^m})}^{q} \notag \\
   & = C 3^{qm} \left\| \nabla g \right\|_{\underline{L}^{q}(Q_{3^m})}^{q}. \notag
\end{align}
The third term on the right-hand side of~\eqref{eq:17320209} can be estimated similarly and we have
\begin{equation} \label{eq:17320209ter}
   \left| \left( g (1- \chi) \right)_{Q_{3^m}}\right| \leq  C 3^{m} \left\| \nabla g \right\|_{\underline{L}^{q}(Q_{3^m})}.
\end{equation}
For the second term on the right-hand side of~\eqref{eq:17320209}, we first note that
\begin{equation*}
    3^{-2m} \int_{(- 3^{2m}, 0)}  \left( g(t , \cdot) \chi \right)_{\Lambda_{3^m}} \, dt =  \left( g \chi \right)_{Q_{3^m}}.
\end{equation*}
Applying the Poincar\'e inequality with respect to the time variable together with the estimate~\eqref{eq:16550209}, we obtain
\begin{align*}
    \lefteqn{3^{-2m} \int_{(- 3^{2m}, 0)} \left| \left( g(t , \cdot) \chi \right)_{\Lambda_{3^m}} - \left( g \chi \right)_{Q_{3^m}} \right|^{q} \, dt } \qquad & \\ & \leq 3^{2qm} 3^{-2m} \int_{(- 3^{2m}, 0)}  \left| \left( \partial_t g(t , \cdot) \chi \right)_{\Lambda_{3^m}} \right|^{q} \, dt  \\
    & \leq 3^{2qm}3^{-2m} \int_{(- 3^{2m}, 0)}  \left\| \partial_t g(t , \cdot) \right\|_{\underline{W}^{-1,q}(\Lambda_{3^m})}^{q} \left\| \chi \right\|_{\underline{W}^{1,q'}(\Lambda_{3^m})}^{q} \, dt  \\
    & \leq 3^{qm} \left\| \partial_t g \right\|_{\underline{L}^{q} \left( (- 3^{2m}, 0) , \underline{W}^{-1,q}(\Lambda_{3^m}) \right)}^{q}.
\end{align*}
Combining the previous inequality with~\eqref{eq:17320209bis} and~\eqref{eq:17320209ter} completes the proof of~\eqref{eq:18320209}.

\bigskip

\textit{Step 2.} In this step, we fix a function $g \in W^{1,q'}_{\mathrm{par}}(Q_{3^n})$ such that
    \begin{equation} \label{eqassumptiongmultscale}
        \left\| g \right\|_{\underline{W}_{\mathrm{par}}^{1,q'}(Q_{3^n})} \leq 1
    \end{equation}
and prove the inequality
\begin{equation} \label{eqassumptiongmultscaleresult}
    \left| \mathcal{Z}_{m,n} \right|^{-1} \sum_{z \in \mathcal{Z}_{m,n}} \left\| g - \left( g \right)_{z + Q_{3^m}} \right\|_{\underline{L}^{q'} \left( z + Q_{3^m} \right)}^{q'} \leq C 3^{q'm} ,
\end{equation}
To this end, we first translate the inequality~\eqref{eq:18260209} and rewrite it as follows (with the exponent $q'$ instead of $q$): for any $m \in \{0 , \ldots, n\}$ and any $z = (s , y) \in \mathcal{Z}_{m,n},$
\begin{equation*}
\left\| g - \left( g \right)_{z + Q_{3^m}} \right\|_{\underline{L}^{q'} \left( z + Q_{3^m} \right)} \leq C 3^m \left\| \nabla g \right\|_{\underline{L}^{q'} \left( z + Q_{3^m} \right)} + C 3^m \left\| \partial_t g \right\|_{\underline{L}^{q'} \left( s + (- 3^{2m}, 0) , \underline{W}^{-1,q'}(y + \Lambda_{3^m}) \right)}.
\end{equation*}
From this inequality, we deduce that
\begin{multline} \label{eq:09190309}
    \left| \mathcal{Z}_{m,n} \right|^{-1} \sum_{z \in \mathcal{Z}_{m,n}} \left\| g - \left( g \right)_{z + Q_{3^m}} \right\|_{\underline{L}^{q'} \left( z + Q_{3^m} \right)}^{q'} \\ \leq C 3^{q'm} \left| \mathcal{Z}_{m,n} \right|^{-1} \sum_{z \in \mathcal{Z}_{m,n}}  \left\| \nabla g \right\|_{\underline{L}^{q'} \left( z + Q_{3^m} \right)}^{q'} 
    + C 3^{q'm} \left| \mathcal{Z}_{m,n} \right|^{-1} \sum_{z \in \mathcal{Z}_{m,n}} \left\| \partial_t g \right\|_{\underline{L}^{q'} \left( s + (- 3^{2m}, 0) , \underline{W}^{-1,r'}(y + \Lambda_{3^m}) \right)}^{q'}.
\end{multline}
The first term on the right-hand side is easily estimated
\begin{equation} \label{eq:eq:09190309}
    \left| \mathcal{Z}_{m,n} \right|^{-1} \sum_{z \in \mathcal{Z}_{m,n}}  \left\| \nabla g \right\|_{\underline{L}^{q'} \left( z + Q_{3^m} \right)}^{q'} = \left\| \nabla g \right\|_{\underline{L}^{q'} \left( Q_{3^n} \right)}^{q'} \leq \left\| \nabla g \right\|_{\underline{L}^{q'} \left( Q_{3^n} \right)}^{q'} \leq 1.
\end{equation}
For the second term on the right-hand side of~\eqref{eq:09190309}, for any $z = (s , y) \in \mathcal{Z}_{m,n}$, we select a function $h_z : (z + Q_{3^m}) \to \R$ satisfying the three following properties
\begin{equation*}
    \left\{ \begin{aligned}
        h_z = 0 ~\mbox{on}~ (s + (- 3^{2m} , 0)) \times  \partial^+ & (y  + \Lambda_{3^m}), \\ 
         \left\| h_z \right\|_{\underline{L}^{q} \left( s + (- 3^{2m}, 0) , \underline{W}^{1,q}(y + \Lambda_{3^m}) \right)} & \leq 1, \\
         \left\| \partial_t g \right\|_{\underline{L}^{q'} \left( s + (- 3^{2m}, 0) , \underline{W}^{-1,q'}(y + \Lambda_{3^m}) \right)}  & = \left( \partial_t g h_z \right)_{z + Q_{3^m}}.
         \end{aligned}
         \right.
\end{equation*}
We extend the function $h_z$ by $0$ outside the parabolic cylinder $(z + Q_{3^m})$ and introduce the shorthand notation 
$$a_z := \left\| \partial_t g \right\|_{\underline{L}^{q'} \left( s + (- 3^{2m}, 0) , \underline{W}^{-1,q'}(y + \Lambda_{3^m}) \right)} .$$
Equipped with this collection of functions, we may write
\begin{align*}
    \left| \mathcal{Z}_{m,n} \right|^{-1} \sum_{z \in \mathcal{Z}_{m,n}} \left\| \partial_t g \right\|_{\underline{L}^{q'} \left( s + (- 3^{2m}, 0) , \underline{W}^{-1,q'}(y + \Lambda_{3^m}) \right)}^{q'}
    & = \left| \mathcal{Z}_{m,n} \right|^{-1} \sum_{z \in \mathcal{Z}_{m,n}}   \left( \partial_t g (a_z^{q'-1} h_z) \right)_{z + Q_{3^m}} \\
    & = \left( \partial_t g  \left(  \sum_{z \in \mathcal{Z}_{m,n}} a_z^{q'-1} h_z \right) \right)_{Q_{3^n}}.
\end{align*}
The term on the right-hand can be estimated using the assumption~\eqref{eqassumptiongmultscale}
\begin{align*}
    \left( \partial_t g  \left(  \sum_{z \in \mathcal{Z}_{m,n}} a_z^{q'-1} h_z \right) \right)_{Q_{3^n}} & \leq \left\| \partial_t g \right\|_{\underline{L}^{q'} \left(  (- 3^{2n}, 0) , \underline{W}^{-1,q'}( \Lambda_{3^n}) \right)} \left\| \sum_{z \in \mathcal{Z}_{m,n}} a_z^{q'-1} h_z \right\|_{\underline{L}^{q} \left((- 3^{2n}, 0) , \underline{W}^{1,q}(\Lambda_{3^n}) \right)} \\
    & \leq \left\| \sum_{z \in \mathcal{Z}_{m,n}} a_z^{q'-1} h_z \right\|_{\underline{L}^{q} \left((- 3^{2n}, 0) , \underline{W}^{1,q}(\Lambda_{3^n}) \right)}.
\end{align*}
Since the functions $(h_z)_{z \in \mathcal{Z}_{m,n}}$ have disjoint supports, we have
\begin{align*}
    \left\| \sum_{z \in \mathcal{Z}_{m,n}} a_z^{q'-1} h_z \right\|_{\underline{L}^{q} \left((- 3^{2n}, 0) , \underline{W}^{1,q}(\Lambda_{3^n}) \right)}^q & = \left| \mathcal{Z}_{m,n} \right|^{-1} \sum_{z \in \mathcal{Z}_{m,n}} a_z^{q(q'-1)} \left\| h_z \right\|_{\underline{L}^{q} \left( s + (- 3^{2m}, 0) , \underline{W}^{1,q}(y + \Lambda_{3^m}) \right)}^q \\
    & \leq \left| \mathcal{Z}_{m,n} \right|^{-1} \sum_{z \in \mathcal{Z}_{m,n}} a_z^{q(q'-1)}. 
\end{align*}
Combining the three previous displays and noting that $q(q'-1) = q'$, we deduce that
\begin{equation*}
    \left| \mathcal{Z}_{m,n} \right|^{-1} \sum_{z \in \mathcal{Z}_{m,n}} a_z^{q'} \leq \left(  \left| \mathcal{Z}_{m,n} \right|^{-1} \sum_{z \in \mathcal{Z}_{m,n}} a_z^{q'} \right)^{\frac{1}{q}}.
\end{equation*}
From this inequality, we further deduce that
\begin{equation*}
    \left| \mathcal{Z}_{m,n} \right|^{-1} \sum_{z \in \mathcal{Z}_{m,n}} \left\| \partial_t g \right\|_{\underline{L}^{q'} \left( s + (- 3^{2m}, 0) , \underline{W}^{-1,q'}(y + \Lambda_{3^m}) \right)}^{q'} = \left| \mathcal{Z}_{m,n} \right|^{-1} \sum_{z \in \mathcal{Z}_{m,n}} a_z^{q'} \leq 1.
\end{equation*}
Combining the previous inequality with~\eqref{eq:09190309} and~\eqref{eq:eq:09190309} completes the proof of~\eqref{eqassumptiongmultscaleresult}.

\medskip

\textit{Step 3.} As in the previous step, we fix a function $g \in W^{1,q'}_{\mathrm{par}}(Q_{3^n})$ such that
    \begin{equation} \label{eqassumptiongmultscalebis}
        \left\| g \right\|_{\underline{W}_{\mathrm{par}}^{1,q'}(Q_{3^n})} \leq 1.
    \end{equation}
and aim to prove
\begin{multline} \label{eqassumptiongmultscaleter}
    \left| Q_{3^n} \right|^{-1} \int_{(-3^{2n}, 0)} \sum_{x \in \Lambda_{3^n}} f(t,x) g(t , x) \, dt \\ \leq  C  \left\| f \right\|_{\underline{L}^{q} (Q_{3^n})} + C \sum_{m = 0}^n 3^m \left( \left| \mathcal{Z}_{m,n} \right|^{-1}  \sum_{z \in \mathcal{Z}_{m,n}} \left| \left( f \right)_{z + Q_{3^m}}  \right|^{q} \right)^{\frac 1{q}}.
\end{multline}
The proposition follows by taking the supremum over all the functions $g$ satisfying~\eqref{eqassumptiongmultscalebis}. To prove~\eqref{eqassumptiongmultscaleter}, we use the decomposition
\begin{align*}
    \left| Q_{3^n} \right|^{-1} \int_{(-3^{2m}, 0)} \sum_{x \in \Lambda_{3^m}} f(t,x) g(t , x) \, dt & = 
    \left| Q_{3^n} \right|^{-1} \sum_{z \in \mathcal{Z}_{0,n}}  \int_{(-1, 0)} \sum_{x \in z + \Lambda_{1}} f(t,x) ( g(t , x) - (g)_{z + Q_1})  \, dt \\
    & \qquad + \sum_{m = 0}^{n-1} \left| \mathcal{Z}_{m,n} \right|^{-1} \sum_{y \in \mathcal{Z}_{m,n}} \left( \left( g \right)_{y + Q_{3^m}} -  \left( g \right)_{z_y + Q_{3^{m+1}}}\right) \left( f \right)_{y + Q_{3^m}} \\
    & \qquad + \left( g \right)_{Q_{3^n}} \left( f \right)_{Q_{3^n}},
\end{align*}
where, for $y \in \mathcal{Z}_{m,n}$, we denote by $z_y$ the unique element of $\mathcal{Z}_{m+1,n}$ such that $y \in (z_y + Q_{3^{m+1}}).$ Applying H\"{o}lder's inequality, we deduce that
\begin{align*}
    \lefteqn{\left| Q_{3^n} \right|^{-1} \int_{(-3^{2n}, 0)} \sum_{x \in \Lambda_{3^n}} f(t,x) g(t , x) \, dt } \qquad & \\ 
    & \leq \left\| f \right\|_{\underline{L}^q(Q_{3^n})} \left( \left| Q_{3^n} \right|^{-1} \sum_{z \in \mathcal{Z}_{0,n}} \left\| g - \left( g \right)_{z + Q_{1}} \right\|_{\underline{L}^{q'} \left( z + Q_{1} \right)}^{q'} \right)^{1/q'} \\
    & + \sum_{m = 0}^{n-1}  \left(  \left| \mathcal{Z}_{m,n} \right|^{-1} \sum_{y \in \mathcal{Z}_{m,n}} \left| \left( g \right)_{y + Q_{3^m}} -  \left( g \right)_{z_y + Q_{3^{m+1}}}\right|^{q'} \right)^{1/q'} \left(   \left| \mathcal{Z}_{m,n} \right|^{-1} \sum_{y \in \mathcal{Z}_{m,n}} \left|  \left( f \right)_{y + Q_{3^m}} \right|^{q} \right)^{1/q} \\
    & + \left| \left( g \right)_{Q_{3^n}} \right| \left| \left( f \right)_{Q_{3^n}} \right|.
\end{align*}
Using the inequality~\eqref{eqassumptiongmultscaleresult} and the assumption~\eqref{eqassumptiongmultscalebis}, we obtain
\begin{equation*}
     \left| Q_{3^n} \right|^{-1} \sum_{z \in \mathcal{Z}_{0,n}} \left\| g - \left( g \right)_{z + Q_{1}} \right\|_{\underline{L}^{q'} \left( z + Q_{1} \right)}^{q'}   \leq C,
\end{equation*}
as well as
\begin{align*}
    \left| \mathcal{Z}_{m,n} \right|^{-1} \sum_{y \in \mathcal{Z}_{m,n}} \left| \left( g \right)_{y + Q_{3^m}} -  \left( g \right)_{z_y + Q_{3^{m+1}}}\right|^{q'} & \leq C \left| \mathcal{Z}_{m,n} \right|^{-1} \sum_{z \in \mathcal{Z}_{m+1,n}} \left\| g - \left( g \right)_{z + Q_{3^{m+1}}} \right\|_{\underline{L}^{q'} \left( z + Q_{3^{m+1}} \right)}^{q'}  \\
    & \leq C  3^{q'm} 
\end{align*}
and
\begin{equation*}
    \left| \left( g \right)_{Q_{3^n}} \right|  \leq C 3^n.
\end{equation*}
Combining the four previous displays completes the proof of~\eqref{eqassumptiongmultscaleter}.
\end{proof}

\subsection{Identification of the dual parabolic space}

\begin{lemma}[Identification of $\hat{W}^{-1,q}_{\mathrm{par}}((0,1) \times \mathbb{T}^\ep)$] \label{lem.idenH-1parapp}
For any exponent $q \in (1 , \infty)$, there exists a constant $C := C(d,q) < \infty$ such that, for any $f \in L^{q}((0,1) \times \mathbb{T}^\ep) $, there exist a continuous function $h : (0,1) \times \mathbb{T}^\ep \to \R$ and a function $h^* : (0,1) \times \mathbb{T}^\ep \to \R$ such that
\begin{equation*}
    \left\{ \begin{aligned}
     \partial_t h  + h^* & = f \\
    \left\| h \right\|_{L^q((0 , 1) ,W^{1,q}(\mathbb{T}^\ep))} & \leq C \left\| f \right\|_{\hat{W}^{-1,q}_{\mathrm{par}}((0 ,1) \times \mathbb{T}^\ep)}, \\
    \left\| h^* \right\|_{L^{q}((0 , 1) , W^{-1,q}(\mathbb{T}^\ep))} & \leq C \left\| f \right\|_{\hat{W}^{-1,q}_{\mathrm{par}}((0 ,1) \times \mathbb{T}^\ep)},
    \\
    h(0, \cdot)  = h(1 , \cdot) & = 0.
    \end{aligned} \right.
\end{equation*}
\end{lemma}

\begin{proof}
We denote by $W^{1,q'}_{\mathrm{par}}((0,1) \times \mathbb{T}^\ep)$ the set of functions $u : (0,1) \times \mathbb{T}^\ep \to \R$ whose $W^{1,q'}_{\mathrm{par}}((0,1) \times \mathbb{T}^\ep)$-norm is finite (N.B. since the discretized torus $\mathbb{T}^\ep$ contains finitely many vertices, this space is in fact equal to the space of functions $u \in L^{q'}((0,1)\times \mathbb{T}^\ep)$ such that $\partial_t u \in L^{q'}((0,1)\times \mathbb{T}^\ep)$; in particular these functions are continuous in time).

\medskip

We consider the Banach space $\mathfrak{B} := L^{q'}((0,1) \times \mathbb{T}^\ep) \times L^{q'}((0,1) \times \mathbb{T}^\ep)^d \times L^{q'}((0,1) \times \mathbb{T}^\ep)$ equipped with the norm: for $(u , v , w) \in \mathfrak{B}$,
\begin{equation*}
    \left\| (u , v , w) \right\|_{\mathfrak{B}} := \left\| u \right\|_{L^{q'}((0 , 1) \times \mathbb{T}^\ep)} + \left\| v \right\|_{L^{q'}((0 , 1) \times \mathbb{T}^\ep)} + \left\| w \right\|_{L^{q'}((0 , 1), W^{-1,q'}(\mathbb{T}^\ep))},
\end{equation*}
and consider the injection $i : W^{1,q}_{\mathrm{par}}((0,1) \times \mathbb{T}^\ep) \to \mathfrak{B}$ given by
\begin{equation*}
    i : u \mapsto (u , \nabla^\ep u , \partial_t u) \in \mathcal{B}.
\end{equation*}
Note that this injection preserves the norms: for any $u \in W^{1,q'}_{\mathrm{par}}((0,1) \times \mathbb{T}^\ep)$,
\begin{equation*}
    \left\| i(u) \right\|_{\mathcal{B}} = \left\| u \right\|_{W^{1,q}_{\mathrm{par}}((0,1) \times \mathbb{T}^\ep)}.
\end{equation*}
Now fix $f \in L^{q}((0,1) \times \mathbb{T}^\ep)$ and consider the linear form $u \mapsto \int_{(0,1)} \sum_{x \in \mathbb{T}^\ep} f(t , x) u(t , x)$. By the Hahn-Banach extension theorem, we may extend it from $W^{1,q'}_{\mathrm{par}}((0,1) \times \mathbb{T}^\ep)$ to $\mathcal{B}$, i.e., there exists a continuous linear form $F : \mathcal{B} \to \R$ such that
\begin{equation*}
    \forall (u , v , w) \in \mathcal{B}, ~ \left| F((u , v , w)) \right| \leq  \left\| f \right\|_{\hat{W}^{-1,q}_{\mathrm{par}}((0 ,1) \times \mathbb{T}^\ep)} \left\| (u , v , w) \right\|_{\mathcal{B}}
\end{equation*}
and
\begin{equation*}
    \forall u \in W^{1,q'}_{\mathrm{par}}((0,1) \times \mathbb{T}^\ep), ~ F(i(u)) = \int_{(0,1)} \sum_{x \in \mathbb{T}^\ep} f(t , x) u(t , x).
\end{equation*}
Since the dual of the space $\mathcal{B}$ is the space $\mathcal{B}^* := L^{q}((0,1) \times \mathbb{T}^\ep) \times L^{q}((0,1) \times \mathbb{T}^\ep)^d \times L^{q}((0,1) \times \mathbb{T}^\ep)$ equipped with the norm
\begin{equation*}
    \left\| (u^* , v^* , w^*) \right\|_{\mathfrak{B}^*} := \max \left(  \left\| u^* \right\|_{L^{q}((0 , 1) \times \mathbb{T}^\ep)} , \left\| v^* \right\|_{L^{q}((0 , 1) \times \mathbb{T}^\ep)} , \left\| w^* \right\|_{L^{q}((0 , 1), W^{1,q}(\mathbb{T}^\ep))} \right),
\end{equation*}
we obtain that there exists a triplet $(u^* , v^* , w^*) \in \mathcal{B}^*$ such that
\begin{multline} \label{eq:11520409}
    \forall u \in W^{1,q}_{\mathrm{par}}((0,1) \times \mathbb{T}^\ep), ~  \int_{(0,1)} \sum_{x \in \mathbb{T}^\ep} f(t , x) u(t , x) \\
    = \int_{(0,1)} \sum_{x \in \mathbb{T}^\ep} \left( u^*(t , x) u(t , x) + v^*(t , x) \cdot \nabla^\ep u(t , x) + w^*(t , x) \partial_t u(t , x)  \right) \, dt
\end{multline}
and
\begin{equation*}
    \left\| u^* \right\|_{L^{q}((0 , 1) \times \mathbb{T}^\ep)} + \left\| v^* \right\|_{L^{q}((0 , 1) \times \mathbb{T}^\ep)} + \left\| w^* \right\|_{L^{q}((0 , 1), W^{1,q}(\mathbb{T}^\ep))} \leq C \left\| f \right\|_{\hat{W}^{-1,q}_{\mathrm{par}}((0 ,1) \times \mathbb{T}^\ep)}.
\end{equation*}
The result of the lemma is then obtained by setting $h = w^*$ and $h^* = u^* + \nabla^\ep \cdot v^*$ (N.B. the continuity in time of $h$ follows from the observation that $\partial_t h = f - h^* \in L^{q}((0,1) \times \mathbb{T}^\ep)$, the identity $h(0 , \cdot) = h(1, \cdot) = 0$ is obtained by performing an integration by part with respect to the time variable in~\eqref{eq:11520409}).
\end{proof}

\section{Estimating the error terms in the two-scale expansion} \label{sec:appendixB}

This section of the appendix contains the estimates of the four error terms which appear when performing the two scale expansion (see Sections~\ref{sec:deferrorterms} and~\ref{secTh.quantitativehydr}). All the results proved here are summarized by the inequalities~\eqref{eq:estiamteerroreterms} and~\eqref{eq:estiamteerroretermsavecrage}. We divide this section into four subsections, each one is dedicated to one of the four error terms. We recall the notation used in all the section (having in mind that $\gamma \simeq 1/(30rd)$)
\begin{equation*}
    \kappa = \ep^\gamma ~~\mbox{and}~~ L = \kappa/\ep = \ep^{\gamma - 1}.
\end{equation*}

\subsection{Preliminary estimates on the Langevin dynamic and the two-scale expansion.}

In this section, we prove some estimates on the solution of the Langevin dynamic $u^\ep$ and the two-scale expansion $w^\ep$ which are used in Section~\ref{sec:section6}, in this appendix and in Appendix~\ref{sec:sectionappendixC}. They are stated in the following list:
\begin{itemize}
\item \textit{Difference between $w^\ep$ and $\bar u^\ep$.} For any $t \in (0,1)$, one has the estimate
\begin{equation} \label{eq:AppC11}
    \left\| \bar u^\ep(t,\cdot)  - w^\ep(t,\cdot) \right\|_{L^2(\mathbb{T}^\ep)} \leq \mathcal{O}_{\Psi , c}(C \ep^{\frac{1 + 7\gamma}{8}}).
\end{equation}
\item \textit{Upper bound for the dynamic.} The $L^r$-norm of the Langevin dynamic is controlled as follows
\begin{equation} \label{eq:AppC12}
    \left\|  u^\ep \right\|_{L^r ( (0,1) \times \mathbb{T}^\ep )} +   \left\|  \nabla u^\ep \right\|_{L^r ( (0,1) \times \mathbb{T}^\ep )} \leq \mathcal{O}_{\Psi , c} \left( C \right).
\end{equation}
\item \textit{Upper bound for the two-scale expansion.} We have the upper bound
\begin{equation} \label{eq:AppC13}
    \left\|  w^\ep \right\|_{L^\infty ( (0,1) \times \mathbb{T}^\ep )} \leq \mathcal{O}_{\Psi , c} \left( C \right).
\end{equation}
Together with the upper bounds on the gradient of the two-scale expansion
\begin{equation} \label{eq:AppC14}
    \forall (t , x) \in (0,1) \times \mathbb{T}^\ep, \left| \nabla^{\ep} w^\ep(t , x) \right| \leq \mathcal{O}_{\Psi , c} \left( C \right) ~~\mbox{and}~~ \left\| \nabla^\ep  w^\ep \right\|_{L^\infty ( (0,1) \times \mathbb{T}^\ep )} \leq \mathcal{O}_{\Psi , c} \left( C \left| \ln \ep \right|\right).
\end{equation}
The first inequality implies (by using Proposition~\ref{prop:prop2.20}) that $\left\| \nabla^\ep  w^\ep \right\|_{L^r ( (0,1) \times \mathbb{T}^\ep )} \leq \mathcal{O}_{\Psi , c} \left( C\right)$.
\item \textit{Upper bound for the corrected plane.} We have the upper bound
\begin{equation}  \label{eq:AppC15}
    \sup_{z \in \mathcal{Z}_\kappa} \left\| \nabla^\ep  v_z \right\|_{L^\infty ( z + Q^\ep_{2 \kappa})} \leq \mathcal{O}_{\Psi , c} \left( C \left| \ln \ep \right|\right).
\end{equation}
\end{itemize}
Note that these estimates imply all the results stated in Section~\ref{sec:sec743}. The rest of this section is divided into three subsections corresponding to three items of the list above.

\subsubsection{Estimating the difference between $w^\ep$ and $\bar u^\ep$}

In this section, we show the inequality~\eqref{eq:AppC11}. From the definition of the two-scale expansion $w^\ep$, we have the identity: for any $(t,x) \in (0,1) \times \mathbb{T}^\ep$,
\begin{equation} \label{eq:14460909}
    \left| \bar u^\ep(t,x)  - w^\ep(t,x) \right| \leq \ep \sum_{z} \chi_z(t , x) \left| \varphi_{z} \left( \frac{t}{\ep^2} , \frac{x}{\ep} ; p_z\right) \right|.
\end{equation}
Using Proposition~\ref{prop:sublincorr} together with the observation that, for any $t \in (0 , \ep^{-2}),$ (N.B. we use here that the sum of independent Brownian motions is a multiple of a Brownian motion)
\begin{equation*}
    \frac{1}{\left| \Lambda_{10L} \right|} \sum_{x \in \Lambda_{10L}} B_t(x) \leq \mathcal{O}_2 \left( C \ep^{-1} L^{- d/2} \right) \overset{(d \geq 2)}{\leq} \mathcal{O}_2(C \ep^{-1} L^{-1}),
\end{equation*}
we obtain 
\begin{equation*}
    \left| \varphi_{z} \left( \frac{t}{\ep^2} , \frac{x}{\ep} ; p_z\right) \right| \leq \mathcal{O}_1 \left( C L^{7/8} + C \ep^{-1} L^{-1} \right).
\end{equation*}
Combining the previous inequality with~\eqref{eq:14460909}, we obtain (choosing $\gamma$ sufficiently small for the second inequality)
\begin{align*}
    \left| \bar u^\ep(t,x)  - w^\ep(t,x) \right| & \leq \mathcal{O}_1 \left( C \ep L^{7/8} + C L^{-1} \right) \\
    & \leq \mathcal{O}_{\Psi , c} \left( C \ep^{\frac{1 + 7\gamma}{8}} \right).
\end{align*}

\subsubsection{Upper bound for the dynamic and its gradient} In this section, we show the inequality~\eqref{eq:AppC12}. It essentially follows from an energy estimate on the equation~\eqref{eq:introSDErescaled} defining the function $u^\ep$ (by making use of the growth Assumption~\eqref{AssPot} on the potential $V$). To be more specific, we let $\varphi : \R \times \mathbb{T}_{\ep} \to \mathbb{R}$ the stationary Langevin dynamic with slope $p = 0$ on the torus $\mathbb{T}_{\ep}$  and denote by $\varphi^\ep := \ep \varphi \left( \frac{\cdot}{\ep^2} , \frac{\cdot}{\ep} \right)$. We then consider the difference $v^\ep := u^\ep - \varphi^\ep$ and observe that this solution solves the parabolic equation
\begin{equation*}
   \left\{ \begin{aligned}
    \partial_t v^\ep & = \nabla^\ep \cdot \a \nabla^\ep v & \mbox{in}~ (0 , 1) \times \mathbb{T}^\ep, \\
    v^\ep(0 , \cdot) & = f -  \varphi^\ep \left( 0 , \cdot \right) & \mbox{on}~  \mathbb{T}^\ep,
    \end{aligned}
    \right.
\end{equation*}
with
\begin{equation*}
     \a (t , x) := \int_0^1 D_p^2 V \left(s \nabla^\ep  u^\ep + (1-s) \nabla^\ep \varphi^\ep \right) \, ds.
\end{equation*}
An energy estimate thus yields the upper bound
\begin{equation*}
    \sum_{x \in \mathbb{T}^\ep} \left| v^\ep(1 , x) \right|^2 + \int_0^1 \sum_{x \in \mathbb{T}^\ep} \nabla^\ep v^\ep (t , x) \cdot \a(t , x) \nabla^\ep v^\ep (t , x) \, dt \leq \sum_{x \in \mathbb{T}^\ep} \left| f(x) - \varphi^\ep \left( 0 , x \right) \right|^2.
\end{equation*}
Using the Assumption~\eqref{AssPot} on the growth of the Hessian of $V$, we have the upper bound
\begin{multline*}
    \int_0^1 \ep^d \sum_{x \in \mathbb{T}^\ep} \left| \nabla^\ep u^\ep (t , x) \right|^r \, dt \\ \leq C \left(  1+ \int_0^1 \ep^d \sum_{x \in \mathbb{T}^\ep} \nabla^\ep v^\ep (t , x) \cdot \a(t , x) \nabla^\ep v^\ep (t , x) \, dt + \int_0^1 \ep^d \sum_{x \in \mathbb{T}^\ep} \left|\nabla^\ep \varphi^\ep \left( t , x \right) \right|^r \, dt \right).
\end{multline*}
Combining the two previous inequalities with the properties of the stationary Langevin dynamics established in Proposition~\ref{prop.prop2.3} and Proposition~\ref{prop:sublincorr}, we obtain
\begin{equation*}
     \left\|  \nabla^\ep u^\ep \right\|_{L^r ( (0,1) \times \mathbb{T}^\ep )} \leq \mathcal{O}_1 (C) \leq \mathcal{O}_{\Psi , c}(C).
\end{equation*}
To establish the bound on the $L^r$-norm of $u^\ep$, we note that the average value of $u^\ep$ is given by the identity (which is obtained by summing the first line of~\eqref{eq:introSDErescaled} over all the vertices $x \in \mathbb{T}^\ep$)
\begin{equation*}
    \left( u^\ep (t , \cdot) \right)_{\mathbb{T}^\ep} = \left( f \right)_{\mathbb{T}^\ep} + \int_0^t \ep^d \sum_{x \in \mathbb{T}^\ep} \sqrt{2} B_t^\ep(x) \, dt.
\end{equation*}
This term can be estimated using that the second term on the right-hand side is a Gaussian random variable whose variance can be explicitly computed, and we obtain, for any $t \in (0,1)$,
\begin{equation*}
    \left|  \left( u^\ep (t , \cdot) \right)_{\mathbb{T}^\ep}  \right| \leq \mathcal{O}_2(C).
\end{equation*}
A combination of the two previous inequalities with the Poincar\'e inequality completes the proof of~\eqref{eq:AppC12}.

\subsubsection{Upper bound for the two-scale expansion and corrected plane}

We only show the upper bound~\eqref{eq:AppC14} as the inequality~\eqref{eq:AppC13} is essentially a consequence of~\eqref{eq:14460909} and the proof of the inequality~\eqref{eq:AppC15} is essentially a consequence of Proposition~\ref{propositionsubdynamic} (and a union bound). 

Applying the discrete gradient to the definition of the two-scale expansion stated in~\eqref{def.wL}, we obtain the upper bound
\begin{equation*}
 \left| \nabla^\ep w^\ep(t , x) \right| \leq \left| \nabla^\ep \bar u^\ep (t , x) \right| + \ep \sum_{z} \left| \nabla^\ep \chi_z(t , x) \right| \left| \varphi_{z} \left( \frac{t}{\ep^2} , \frac{x}{\ep} ; p_z\right) \right| + \ep \sum_{z} \chi_z(t , x) \left| \nabla^\ep \varphi_{z} \left( \frac{t}{\ep^2} , \frac{x}{\ep} ; p_z\right) \right|.
\end{equation*}
Using the properties of the stationary Langevin dynamics stated in Proposition~\ref{prop.prop2.3} and Proposition~\ref{prop:sublincorr}, we obtain, for any $(t, x) \in (0,1) \times \mathbb{T}^\ep$,
\begin{equation*}
    \left| \nabla^\ep w^\ep(t , x) \right| \leq \mathcal{O}_1 \left(  C \right) \leq \mathcal{O}_{\Psi , c}(C).
\end{equation*}
The upper bound on the $L^\infty$-norm of the gradient of $w^\ep$ can be obtained with a similar argument, using this time Proposition~\ref{propositionsubdynamic} on the supremum of the gradient of the stationary dynamics.

\subsection{Estimating the error term $\mathcal{E}_1$}

In this section, we prove that the error term $\mathcal{E}_1$ is small. Specifically, we prove the inequality, for any $(t , x) \in (0 , \infty) \times \mathbb{T}^\ep$,
\begin{equation} \label{eq:estiamteerrortermE1}
       \left| \mathcal{E}_1(t , x) \right|  \leq \mathcal{O}_{\Psi , c} \left(C \ep^{\frac{1- 9\gamma}{8}} \right).
\end{equation}
This estimate is stronger than all the estimates needed on this term. In particular, by applying Proposition~\ref{prop:prop2.20} ``Summation" and ``Integration", we obtain
\begin{equation*}
 \ep \left\| \mathcal{E}_1 \right\|_{L^2((0 , 1) \times \mathbb{T}^\ep)} + \left\| \mathcal{E}_1 \right\|_{\underline{W}^{-1,r}_{\mathrm{par}}\left( (0 , 1) \times \mathbb{T}^\ep \right)} \leq C \left\| \mathcal{E}_1 \right\|_{L^r((0 , 1) \times \mathbb{T}^\ep)} \leq \mathcal{O}_{\Psi , c} \left( C \ep^{\frac{1- 9\gamma}{8}} \right).
\end{equation*}
Similarly, we deduce the estimate on the average value of the error term $\mathcal{E}_1$: for any $t \in (0,1),$
\begin{equation*}
    \left| \int_0^t \left( \mathcal{E}_1(s , \cdot ) \right)_{\mathbb{T}^\ep} \, ds \right| \leq \mathcal{O}_{\Psi , c}(C \ep^{\frac{1- 9\gamma}{8}}).
\end{equation*}
Note that the exponent is strictly positive is $\gamma$ is chosen small enough. To prove~\eqref{eq:estiamteerrortermE1}, we first simplify the term $\mathcal{E}_1$ by noting that the Brownian motions do not contribute to the term. Specifically, we have the identity, for any $(t , x) \in (0 , \infty) \times \mathbb{T}^\ep$,
\begin{equation*}
            \sum_{z} \partial_t \chi_{z}(t , x)  \sqrt{2} B_t^\ep(x) = \left( \sqrt{2} B_t^\ep(x) \right) \sum_{z} \partial_t \chi_{z}(t , x)  = 0,
\end{equation*}
where the first identity is obtained by noting that the term involving the Brownian motions does not depend on the parameter $z$, and the second one is obtained by differentiating both sides of the identity in~\eqref{prop.partofunity2sc} with respect to the time. From the previous computation, we deduce that, for any $(t , x) \in (0 , \infty) \times \mathbb{T}^\ep$,
\begin{equation*}
     \mathcal{E}_1(t , x) =  \ep \sum_{z} \partial_t \chi_{z}(t , x) \varphi_{z} \left( \frac{t}{\ep^2} , \frac{x}{\ep} ; p_{z} \right).
\end{equation*}
We then estimate the $L^2$-norm of this term using the bound~\eqref{prop.partofunity2scbis} on the time derivative of the functions~$(\chi_z)_{z \in \mathcal{Z}_\kappa}$ together with Proposition~\ref{prop:sublincorr} and the observation that, for any fixed pair $(t , x) \in (0,1) \times \mathbb{T}^\ep$, there are only $C := C(d)$ terms which are not equal to $0$ in the collection $(\chi_z(t  ,x))_{z \in \mathcal{Z}_\kappa}$. We obtain
\begin{equation*}
    \left| \mathcal{E}_1(t , x) \right| \leq \mathcal{O}_1 \left( C \ep \kappa^{-2} L^{7/8} \right).
\end{equation*}
Using the identities $L = \ep^{\gamma - 1}$ and $\kappa = \ep^{\gamma}$ and Proposition~\ref{prop:prop2.20} ``Comparison", we further obtain
\begin{equation*}
    \left| \mathcal{E}_1(t , x) \right| \leq \mathcal{O}_{\Psi , c} \left( C \ep^{\frac{1- 9\gamma}{8}} \right).
\end{equation*}

\subsection{Estimating the error term $\vec{\mathcal{E}}_2$}

The objective of this section is to establish the following inequality
\begin{equation} \label{eq:15602107}
    \| \vec{\mathcal{E}}_2 \|_{L^r((0 , 1) \times \mathbb{T}^\ep)} \leq \mathcal{O}_{\Psi , c} \left( C \ep^{\frac{1 - \gamma}{20 r} \wedge \frac{\gamma}{2 r}} \right). 
\end{equation}
To prove the inequality~\eqref{eq:15602107}, we decompose the term $\vec{\mathcal{E}}_2$ in two terms as follows
\begin{equation} \label{eq:12001208}
    \vec{\mathcal{E}}_2
     = \underset{\eqref{eq:12001208}-(i)}{\underbrace{D_p V \left( \nabla^\ep w^\ep \right) -  D_p V \left( \sum_{z} \chi_{z} \nabla^\ep v_{z} \right)}}
     +  \underset{\eqref{eq:12001208}-(ii)}{\underbrace{D_p V \left( \sum_{z} \chi_{z} \nabla^\ep v_{z} \right)  - \sum_{z} \chi_{z} D_p V \left( \nabla^\ep v_{z} \right)}}
\end{equation}
and estimate the two terms~$\eqref{eq:12001208}-(i)$ and~$\eqref{eq:12001208}-(ii)$ in two distinct substeps (N.B. In the steps below, we first estimate the $L^2$-norm of these two terms and upgrade the result from an $L^2$-norm to the $L^r$-norm at the end by using an interpolation argument).

\medskip

\textit{Substep 1.1. Estimating the term~\eqref{eq:12001208}-(i).} 

\medskip

We first show that the discrete gradient of the map $w^\ep$ is close to the function $\sum_{z} \chi_{z} \nabla^\ep v_{z}$. Specifically, we prove the inequality
\begin{equation} \label{eq:09222504}
    \left\| \nabla^\ep w^\ep - \sum_{z} \chi_{z} \nabla^\ep v_{z} \right\|_{L^2((0, 1) \times \mathbb{T}^\ep)} \leq\mathcal{O}_{\Psi , c} (C \ep^{\frac{1-\gamma}{8} \wedge \gamma}).
\end{equation}
We next compute the gradient of the map $w^\ep$. Using the definition~\eqref{def.wL}, we have the identity,
\begin{equation} \label{formulawL.2sc}
   \nabla^\ep w^\ep  = \sum_{z} \chi_{z} \nabla^\ep v_{z} + \vec{\mathcal{E}}_1 ,
\end{equation}
where we introduce a vector-valued error term $\vec{\mathcal{E}}_1$ defined by the identity
\begin{equation*}
            \vec{\mathcal{E}}_1 := \ep \sum_{z} \nabla^\ep \chi_{z} \varphi_{z} \left( \frac{\cdot}{\ep^2} , \frac{\cdot}{\ep} ; p_z\right) + (\nabla \bar u^\ep - p_z).
\end{equation*}
The first term can be estimated using the sublinearity of the Langevin dynamic (Proposition~\ref{prop:sublincorr}) and the bound~\eqref{prop.partofunity2scbis}, the second term can be estimated using the regularity estimate on the gradient of the solution $\bar u^\ep$ (Proposition~\ref{prop:prop7.5}). We obtain
\begin{equation} \label{eq:11471208}
    \| \vec{\mathcal{E}}_1 \|_{L^2(Q^\ep)} \leq \mathcal{O}_1 (C L^{-1/8}) + C \ep L \leq \mathcal{O}_{\Psi , c} (C \ep^{\frac{1-\gamma}{8}}) + C \ep^\gamma.
\end{equation}
This inequality implies~\eqref{eq:09222504} (by using that $C \ep^\gamma \leq \mathcal{O}_{\Psi , c}(C \ep^\gamma)$).
We next combine the identity~\eqref{formulawL.2sc} with the growth assumption on the potential $D_p V$ to obtain
\begin{equation*}
    \left| D_p V \left( \nabla^\ep w^\ep \right) -  D_p V \left( \sum_{z} \chi_{z} \nabla^\ep v_{y} \right) \right| \leq C \left( \left| \nabla^\ep w^\ep  \right|^{r-2}_+ +  \bigg| \sum_{z} \chi_{z} \nabla^\ep v_{y} \bigg|^{r-2}_+   \right) | \vec{\mathcal{E}}_1 |.
\end{equation*}
Applying the estimate~\eqref{eq:11471208} together with the estimates~\eqref{eq:AppC14} and~\eqref{eq:AppC15}, we obtain
\begin{equation} \label{eq:15161508}
    \left\| \eqref{eq:12001208}-(i) \right\|_{L^2((0, 1) \times \mathbb{T}^\ep)} \leq \mathcal{O}_{\Psi , c} \left( C |\ln \ep| \ep^{\frac{1-\gamma}{8} \wedge \gamma}  \right) \leq  \mathcal{O}_{\Psi , c} \left( C \ep^{\frac{1-\gamma}{16} \wedge \frac{\gamma}{2}}  \right).
\end{equation}

\medskip

\textit{Substep 1.2. Estimating the term~\eqref{eq:12001208}-(ii).}

\medskip

To analyse this term, we recall the notation
\begin{equation*}
    Q_{\kappa}^\ep := (-\kappa^2 , 0) \times \ep \Lambda_{L},
\end{equation*}
and note that the collection $\left\{ z + Q_\kappa^\ep \, : \, z \in \mathcal{Z}_\kappa \right\}$ is a partition of $(0,1) \times \mathbb{T}^\ep$ (see Figure~\ref{fig:partition}). We will make use of the following notation: for any fixed $z \in \mathcal{Z}_\kappa$, we write $\sum_{z' \sim z}$ to refer to the sum over the vertices $z' \in \mathcal{Z}_\kappa$ such that $(z + Q_{ \kappa}) \cap (z' + Q_{2 \kappa}) \neq \emptyset$ (Note that this set contains at most $C := C(d) < \infty$ elements). This notation is useful for the following reason: for any $z \in \mathcal{Z}_\kappa$ and any $(t , x) \in (z + Q_\kappa)$, we have the identities
\begin{equation*}
    \sum_{z'} \chi_{z'}(t , x) \nabla^\ep v_{z'} (t , x) = \sum_{z' \sim z} \chi_{z'}(t , x) \nabla^\ep v_{z'} (t , x)
\end{equation*}
and
\begin{equation*}
    \sum_{z'} \chi_{z'}(t , x) D_p V \left( \nabla^\ep v_{z'} \right)(t , x) = \sum_{z'\sim z} \chi_{z'}(t , x) D_p V \left( \nabla^\ep v_{z'} \right)(t , x).
\end{equation*}
We first use the growth assumption on the map $D_p V$ to write, for any $z\in \mathcal{Z}_\kappa$,
\begin{align*}
    \left|  D_p V \biggl( \sum_{z' \sim z} \chi_{z' } \nabla v_{z'} \biggr)  - D_p V \left( \nabla v_z \right)  \right|  & \leq C \left( \bigg|\sum_{z' \sim z} \chi_{z'} \nabla v_{z'} \bigg|_+^{r-2} + \left| \nabla v_z \right|_+^{r-2} \right)
    \left|  \sum_{z'\sim z} \chi_{z'} \nabla v_{z'} - \nabla v_{z}  \right|.
\end{align*}
We next estimate the two terms on the right-hand side. For the first one, we use the $L^\infty$ estimate on the gradient of the function $v_z$ stated in~\eqref{eq:AppC15} to obtain
\begin{equation*}
     \left( \bigg|\sum_{z' \sim z} \chi_{z'} \nabla v_{z'} \bigg|_+^{r-2} + \left| \nabla v_z \right|_+^{r-2} \right) \leq \mathcal{O}_{\Psi , c} \left(  C \left| \ln \ep \right|^{r-2} \right).
\end{equation*}
For the second term, we first state the following inequality proved in Lemma~\ref{lemmavzvzprime} (with the exponent $\alpha = \gamma/4$): for any $z , z' \in \mathcal{Z}_\kappa$ such that $(z + Q_{ \kappa}) \cap (z' + Q_{2 \kappa}) \neq \emptyset$ (N.B. note that this implies that the map $v_{z'}$ is defined on the cylinder $z + Q^\ep_\kappa$)
\begin{equation*}
    \left\| \nabla v_{z'} - \nabla v_{z}  \right\|_{\underline{L}^2(z + Q_\kappa^\ep)} \leq \mathcal{O}_{\Psi , c} \left( C L^{\frac{\gamma}{4} - \frac{1}{16}} + C  L^{\frac{\gamma}{4}}   |p_z - p_{z'}|  \right) . 
\end{equation*}
This inequality asserts that the difference of the gradient between the two maps $v_z$ and $v_{z'}$ is small, and this is quantified in terms of the size of the mesoscopic scale and the distance between the two slopes $p_z$ and $p_{z'}$. Based on this inequality, we deduce the upper bound
\begin{align*}
    \left\|  \sum_{z'\sim z} \chi_{z'} \nabla v_{z'} - \nabla v_{z}  \right\|_{\underline{L}^2(z + Q_\kappa^\ep)} & \leq \sum_{z' \sim z } \left\|  \nabla v_{z'} - \nabla v_{z} \right\|_{{\underline{L}^2(z + Q_\kappa^\ep)}}  \\
    & \leq \mathcal{O}_{\Psi , c} \left( C  L^{\frac{\gamma}{4} - \frac{1}{16}}  + C  L^{\frac{\gamma}{4}} \sum_{z'\sim z}  |p_z - p_{z'}| \right).
\end{align*}
A combination of the two previous displays yields the bound
\begin{equation*}
    \left\|  D_p V \biggl( \sum_{z' \sim z} \chi_{z' } \nabla v_{z'} \biggr)  - D_p V \left( \nabla v_z \right)  \right\|_{\underline{L}^2(z + Q_\kappa^\ep)} \leq  \mathcal{O}_{\Psi , c} \left( \frac{C L^{\gamma/4} \left| \ln \ep \right|}{L^{1/16}} + C \left| \ln \ep \right| L^{\gamma/4}  \sum_{z'\sim z}  |p_z - p_{z'}| \right).
\end{equation*}
Building upon the previous inequality, we obtain (the first identity uses that the collection $(\chi_z)_{z \in \mathcal{Z}_\kappa}$ is a partition of unity, the second one uses that, for any fixed pair $(t , x)$, there are only $C := C(d) < \infty$ terms which are not equal to $0$ in the collection $(\chi_z(t  ,x))_{z \in \mathcal{Z}_\kappa}$)
\begin{align*}
    \lefteqn{\left\| D_p V \left( \sum_{z} \chi_{z} \nabla^\ep v_{z} \right)  - \sum_{z} \chi_{z} D_p V \left( \nabla^\ep v_{z} \right) \right\|_{L^2 \left( (0,1) \times \mathbb{T}^\ep \right)}^2} \qquad & 
    \\ & =  \left\|  \sum_{z} \chi_z \left( D_p V \left( \sum_{z} \chi_{z} \nabla^\ep v_{z} \right)  - D_p V \left( \nabla^\ep v_{z} \right) \right) \right\|_{L^2 \left( (0,1) \times \mathbb{T}^\ep \right)}^2
    \\ & \leq \frac{1}{\left| \mathcal{Z}_\kappa \right|} \sum_{z \in \mathcal{Z}_\kappa} \left\|  D_p V \biggl( \sum_{z' \sim z} \chi_{z' } \nabla v_{z'} \biggr)  - D_p V \left( \nabla v_z \right)  \right\|_{\underline{L}^2(z + Q_\kappa^\ep)}^2 \\
    & 
    \leq \mathcal{O}_{\Psi , c} \left( C\left| \ln \ep \right|^2L^{\frac{\gamma}{4} - \frac{1}{8}}  + C \left| \ln \ep \right|^2 L^{\frac{\gamma}{4}} \sum_{z \in \mathcal{Z}_\kappa} \sum_{z'\sim z}  |p_z - p_{z'}|^2 \right).
\end{align*}
The definition of the terms $(p_z)_{z \in \mathcal{Z}_\kappa}$ as the average value of the gradient of the map $\bar u^\ep$ together with the $H^2$-regularity estimate stated in Proposition~\ref{prop:prop7.5} yields the inequality
\begin{equation*}
    \sum_{z \in \mathcal{Z}_\kappa} \sum_{z'\sim z}  |p_z - p_{z'}|^2 \leq C \kappa^2 \left\| \nabla^{2,\ep} \bar u^\ep  \right\|_{L^2((0,1) \times \mathbb{T})}^2 \leq C \ep^{2 \gamma}.
\end{equation*}
We may then combine the two previous displays with the identity $L = \ep^{\gamma - 1}$ (and having in mind that $\gamma \simeq 1/(30dr)$), we obtain
\begin{align*}
    \left\| D_p V \left( \sum_{z} \chi_{z} \nabla^\ep v_{z} \right)  - \sum_{z} \chi_{z} D_p V \left( \nabla^\ep v_{z} \right) \right\|_{L^2 \left( (0,1) \times \mathbb{T}^\ep \right)} & \leq \mathcal{O}_{\Psi , c} \left( C \ep^{\frac{1 - \gamma}{20}} + C \ep^{\frac{\gamma}{2}} \right) \\
    & \leq \mathcal{O}_{\Psi , c} \left( C \ep^{\frac{1 - \gamma}{20} \wedge \frac{\gamma}{2}} \right).
\end{align*}
We have thus obtained
\begin{equation} \label{eq:15171508}
    \left\|  \eqref{eq:12001208}-(ii) \right\|_{L^2( (0,1) \times \mathbb{T}^\ep )} \leq  \mathcal{O}_{\Psi , c} \left( C \ep^{\frac{1 - \gamma}{20} \wedge \frac{\gamma}{2}} \right).
\end{equation}
Combining the inequalities~\eqref{eq:15161508} and~\eqref{eq:15171508} yields the upper bound
\begin{equation*}
    \| \vec{\mathcal{E}}_2 \|_{L^2((0 , 1) \times \mathbb{T}^\ep)} \leq \mathcal{O}_{\Psi , c} \left( C \ep^{\frac{1 - \gamma}{20} \wedge \frac{\gamma}{2}} \right).
\end{equation*}
We then upgrade the $L^2$-norm into an $L^r$-norm by interpolating the space $L^r$ between the spaces $L^2$ and $L^\infty$ and writing
\begin{align*}
    \| \vec{\mathcal{E}}_2 \|_{L^r((0 , 1) \times \mathbb{T}^\ep)} & \leq \| \vec{\mathcal{E}}_2 \|_{L^2((0 , 1) \times \mathbb{T}^\ep)}^{\frac{2}{r}} \| \vec{\mathcal{E}}_2 \|_{L^\infty((0 , 1) \times \mathbb{T}^\ep)}^{\frac{r-2}{r}} \\
    & \leq \mathcal{O}_{\Psi , c} \left( C \ep^{\frac{1 - \gamma}{20 r} \wedge \frac{\gamma}{2r}} \right),
\end{align*}
where the $L^\infty$-norm of the term $ \vec{\mathcal{E}}_2 $ is estimated using the inequalities~\eqref{eq:AppC14} and~\eqref{eq:AppC15} (and yields a logarithmic term in $\ep$ which can be absorbed by reducing the value of the power of $\ep$).

\subsection{Estimating the error term $\vec{\mathcal{E}}_3$}

The objective of this section is to prove the inequality (N.B. contrary to the other error terms, this one is deterministic)
\begin{equation*}
   \| \vec{\mathcal{E}}_3 \|_{L^r((0,1) \times \mathbb{T}^\ep)} \leq C \ep^{2\gamma/r}.
\end{equation*}
This estimate is the simplest among the four error terms. We first use the strict convexity of the surface tension stated in Proposition~\ref{prop:strictconvsurftens} (and specifically the upper bound in the inequality~\eqref{ineq:strictconvexity}), to write
\begin{equation*}
    \left| D_p \bar \sigma \left( \nabla^\ep \bar u^\ep \right) -  D_p \bar \sigma \left( p_{z} \right) \right| \leq C \left( \left| \nabla^\ep \bar u^\ep \right|^{r-2}_+ + \left| p_z \right|^{r-2}_+  \right) \left| \nabla^\ep \bar u^\ep - p_z \right|.
\end{equation*}
Using that the first term on the right-hand side is bounded together with the definition of the term $p_z$ (as the average value of the gradient of $\nabla^\ep \bar u^\ep$ over a mesoscopic box) and the $H^2$-regularity estimate stated in Proposition~\ref{prop.reghomogenizedsolution}, we obtain
\begin{align*}
        \left\|  D_p \bar \sigma \left( \nabla^\ep \bar u^\ep \right) -  D_p \bar \sigma \left( p_{z} \right) \right\|_{\underline{L}^2 \left( z + Q_\kappa^\ep \right)} & \leq C \left\|  \nabla^\ep \bar u^\ep -  p_{z} \right\|_{\underline{L}^2 \left( z + Q_\kappa^\ep \right)} \\
        & \leq C \kappa \left\|  \nabla^{2,\ep} \bar u^\ep \right\|_{\underline{L}^2 \left( z + Q_\kappa^\ep \right)}.
\end{align*}
From the previous estimate, we deduce the inequality (using once again that, for any fixed pair $(t , x)$, there are only $C := C(d) < \infty$ terms which are not equal to $0$ in the collection $(\chi_z(t  ,x))_{z \in \mathcal{Z}_\kappa}$)
\begin{align*}
   \| \vec{\mathcal{E}}_3 \|_{L^2((0,1) \times \mathbb{T}^\ep)}^2 & \leq  \frac{1}{\left| \mathcal{Z}_\kappa \right|} \sum_{z} \left\|  D_p \bar \sigma \left( \nabla^\ep \bar u^\ep \right) -  D_p \bar \sigma \left( p_{z} \right) \right\|_{\underline{L}^2 \left( z + Q_\kappa^\ep \right)}^2 \\
   & \leq C \kappa^2 \frac{1}{\left| \mathcal{Z}_\kappa \right|} \sum_{z}  \left\|  \nabla^{2, \ep} \bar u^\ep \right\|_{\underline{L}^2 \left( z + Q_\kappa^\ep \right)}^2 \\
   & \leq C \kappa^2.
\end{align*}
The identity $\kappa = \ep^\gamma$ thus implies
\begin{equation*}
    \| \vec{\mathcal{E}}_3 \|_{L^2((0,1) \times \mathbb{T}^\ep)} \leq  C \ep^{\gamma}.
\end{equation*}
Interpolating the space $L^r$ between the spaces $L^2$ and $L^\infty$, we obtain
\begin{align*}
    \| \vec{\mathcal{E}}_3 \|_{L^r((0,1) \times \mathbb{T}^\ep)} & \leq  \| \vec{\mathcal{E}}_3 \|_{L^2((0,1) \times \mathbb{T}^\ep)}^{\frac{2}{r}}  \| \vec{\mathcal{E}}_3 \|_{L^\infty((0,1) \times \mathbb{T}^\ep)}^{\frac{r-2}{r}} \\
    & \leq C \ep^{2\gamma / r}.
\end{align*}

\subsection{Estimating the error term $\mathcal{E}_4$}

In this section, we will prove three inequalities on the error term $\mathcal{E}_4$:
\begin{itemize}
    \item The upper bound on the spatial average of $\mathcal{E}_4$
        \begin{equation*}
            \left| \int_0^t \left( \mathcal{E}_4(s , \cdot) \right)_{\mathbb{T}^\ep} \, ds \right| \leq \mathcal{O}_{\Psi , c} \left( C \ep^{\frac{1-9\gamma}{8}} \right).
        \end{equation*}
        \item The pointwise upper bound on $\mathcal{E}_4$: for any $(t , x) \in (0,1) \times \mathbb{T}^\ep$,
        \begin{equation*}
            \left| \mathcal{E}_4 (t , x) \right| \leq \mathcal{O}_{\Psi , c} \left( C \ep^{- \gamma} \right).
        \end{equation*}
    Note that, contrary to the other terms, the right-hand side is large, but it implies (using Proposition~\ref{prop:prop2.20} ``Integration", ``Summation")
        \begin{equation*}
            \ep \left\| \mathcal{E}_4 \right\|_{L^2\left( (0,1) \times \mathbb{T}^\ep \right)} \leq \mathcal{O}_{\Psi , c} \left( C \ep^{1-\gamma} \right).
        \end{equation*}
    \item The upper bound on the $\underline{W}^{-1,r}_{\mathrm{par}}((0,1) \times \mathbb{T}^\ep)$-norm of $\mathcal{E}_4$
        \begin{equation} \label{eq:appestimateE}
            \left\| \mathcal{E}_4 \right\|_{\underline{W}^{-1,r}_{\mathrm{par}}((0,1) \times \mathbb{T}^\ep)}  \leq \mathcal{O}_{\Psi , c} \left( C \ep^{\frac{1 - 17 \gamma}{16} - (d+2) \gamma} \right).
        \end{equation}
\end{itemize}

We decompose the rest of this section into three steps.

\medskip

\textit{Step 1. \textit{Upper bound on the spatial average of $\mathcal{E}_4$.}}

\medskip

We first recall the identity~\eqref{identity6.678}
\begin{equation} \label{eq:13482207}
\nabla^\ep \cdot \biggl( \sum_{z}  \chi_{z} \left( D_p V \left( \nabla^\ep v_{z} \right) -  D_p\bar \sigma \left( p_{z} \right)\right) \biggr)=  \sum_{z}  \chi_{z} \nabla^\ep \cdot D_p V \left( \nabla^\ep v_{z} \right)  + \mathcal{E}_4.
\end{equation}
Summing the previous identity over all the vertices of the torus and performing a discrete integration by parts in space (which cancel the term on the left-hand side of~\eqref{eq:13482207} since it is the divergence of a vector field) and an integration by parts in time, we obtain
\begin{align*}
     \int_0^t  \left( \mathcal{E}_4(s , \cdot) \right)_{\mathbb{T}^\ep} \, ds  = \int_0^t \ep^d \sum_{x \in \mathbb{T}^\ep} \mathcal{E}_4(s , x) \, ds & =  \int_0^t \ep^d \sum_{x \in \mathbb{T}^\ep} \sum_{z}  \chi_{z}(s , x) \nabla^\ep \cdot D_p V \left( \nabla^\ep v_{z} \right)(s , x) \, ds \\
     & =  \sum_{z} \int_0^t \ep^d \sum_{x \in \mathbb{T}^\ep}  \chi_{z}(s , x) \partial_t \left( \ep \varphi_z \left( \frac{s}{\ep^2} , \frac{x}{\ep} ; p_z\right) - B_s^\ep(x) \right) \, ds \notag \\
     & = - \sum_{z}  \int_0^t \ep^d \sum_{x \in \mathbb{T}^\ep}  \partial_t \chi_{z}(s, x) \left( \ep \varphi_z \left( \frac{s}{\ep^2} , \frac{x}{\ep} ; p_z\right) - B_s^\ep(x) \right) \, ds \notag \\
     & \qquad +  \sum_{z} \ep^d \sum_{x \in \mathbb{T}^\ep}  \chi_{z}(t , x) \left( \ep \varphi_z \left( \frac{t}{\ep^2} , \frac{x}{\ep} ; p_z\right) - B_t^\ep(x) \right) \notag
     \\
     & \qquad -  \sum_{z} \ep^d \sum_{x \in \mathbb{T}^\ep}  \chi_{z}(0 , x) \left( \ep \varphi_z \left( 0 , \frac{x}{\ep} ; p_z\right) \right) .\notag
\end{align*}
We then prove that the three terms on the right-hand side are small. For the first one, we use that $\sum_{z}  \partial_t \chi_{z}(t , x) = 0$ (which follows by differentiating in time the equality $\sum_{z}  \chi_{z} = 1$) together with the upper bound $\left| \partial_t \chi_z \right| \leq C \kappa^{-2}$ and Proposition~\ref{prop:sublincorr}. We obtain
\begin{align*}
     \left| \sum_{z}  \int_0^t \ep^d \sum_{x \in \mathbb{T}^\ep}  \partial_t \chi_{z}(t , x) \left( \ep \varphi_z \left( \frac{s}{\ep^2} , \frac{x}{\ep} ; p_z\right) - B_s^\ep(x) \right) \, ds \right| & = \left| \sum_{z}  \int_0^t \ep^d \sum_{x \in \mathbb{T}^\ep}  \partial_t \chi_{z}(t , x) \ep \varphi_z \left( \frac{s}{\ep^2} , \frac{x}{\ep} ; p_z\right) \, ds \right| \\
     & \leq \mathcal{O}_1 \left( C \ep \kappa^{-2} L^{7/8} \right) \\
     & \leq  \mathcal{O}_{\Psi , c} \left( C \ep^{\frac{1-9\gamma}{8}} \right).
\end{align*}
For the second term, we use that $\sum_{z}  \chi_{z} = 1$ together with the observation that, for any time $t \in (0,1)$, the sum $\ep^d \sum_{x \in \mathbb{T}^\ep} B_t^\ep(x)$ is distributed according to a normal distribution whose variance is equal to $ \ep^{-d}$ which implies
\begin{equation*}
    \forall t \in (0,1), ~ \left| \ep^d \sum_{x \in \mathbb{T}^\ep} B_t^\ep(x) \right| \leq \mathcal{O}_2 \left( C  \ep^{-d/2} \right).
\end{equation*}
We thus obtain, for any $t \in (0,1)$,
\begin{align*}
    \left|  \sum_{z} \ep^d \sum_{x \in \mathbb{T}^\ep}  \chi_{z}(t , x) \left( \varphi_z \left( \frac{t}{\ep^2} , \frac{x}{\ep} ; p_z\right) - B_t^\ep(x) \right) \right| & = \left| \sum_{z} \ep^d \sum_{x \in \mathbb{T}^\ep}  \chi_{z}(t , x) \varphi_z \left( \frac{t}{\ep^2} , \frac{x}{\ep} ; p_z\right) \right| + \left| \ep^d \sum_{x \in \mathbb{T}^\ep} B_t^\ep(x) \right| \\
    & \leq \mathcal{O}_1 \left( C \ep L^{7/8} + C \ep^{-d/2}\right) \\
    & \leq  \mathcal{O}_{\Psi , c} \left( C \ep^{\frac{1+7\gamma}{8}} \right).
\end{align*}
A combination of the previous displays yields the upper bound
\begin{equation*}
    \left| \int_0^t \left( \mathcal{E}_4(s , \cdot) \right)_{\mathbb{T}^\ep} \, ds \right| \leq \mathcal{O}_{\Psi , c} \left( C \ep^{\frac{1-9\gamma}{8}} \right).
\end{equation*}

\medskip

\textit{Step 2. \textit{The pointwise upper bound on $\mathcal{E}_4$.}}

\medskip

The proof is relatively straightforward: using the bound~\eqref{prop.partofunity2scbis} on the partition of unity, and the bound stated in Proposition~\ref{prop.prop2.3} on the gradient of the Langevin dynamic. We obtain, for any $(t , x) \in (0,1) \times \mathbb{T}^\ep$,
 \begin{align*}
    \left| \mathcal{E}_4 (t , x ) \right| & \leq \mathcal{O}_{\frac{r}{r-1}} \left( C \kappa^{-1} \right) \\
    & \leq  \mathcal{O}_{\Psi , c} \left( C \ep^{-\gamma} \right).
 \end{align*}

\medskip

\textit{Step 3. The upper bound on the $\underline{W}^{-1,r}_{\mathrm{par}}((0,1) \times \mathbb{T}^\ep)$-norm of $\mathcal{E}_4$.}

\medskip

This step is the most intricate. We first define a collection of suitable Sobolev spaces for the argument (which are similar to the ones introduced in Section~\ref{sec:7.1preliminareis} but are defined on a parabolic cylinder rather than on the set $(0,1) \times \mathbb{T}^\ep$). We denote by $\left| \Lambda_{2\kappa}^\ep \right|$ the cardinality of the box $\Lambda_{2\kappa}^\ep$ (which is of order $\kappa^d/\ep^d = \ep^{d(\gamma-1)}$), let $z = (s , y) \in \mathcal{Z}_\kappa$ be fixed and let $q \in (1 , \infty)$ be an exponent (with $q' = q/(q-1)$ its conjugate). We introduce the following norms:
\begin{itemize}
    \item \emph{$L^{q}$-norm:}
    $
    \left\| u \right\|_{\underline{L}^{q} \left( y + \Lambda_{2\kappa}^\ep \right)}^{q} := \left| \Lambda_{2\kappa}^\ep \right|^{-1} \sum_{x \in y + \Lambda_{2\kappa}^\ep} \left| u(x)\right|^{q},
    $
    \item \emph{$W^{1,q}$-norm:}
    $
        \left\| u \right\|_{\underline{W}^{1,q}(y + \Lambda_{2\kappa}^\ep)} := \kappa^{-1} \left\| u \right\|_{\underline{L}^{q} (y + \Lambda_{2\kappa}^\ep)} + \left\| \nabla^\ep u \right\|_{\underline{L}^{q} (y + \Lambda_{2\kappa}^\ep)},
    $
    \item \emph{$W^{-1,q}$-norm:}
        $$
        \left\| u \right\|_{\underline{W}^{-1,q}(y + \Lambda_{2\kappa}^\ep)} := \sup \left\{\left| \Lambda_{2\kappa}^\ep \right|^{-1} \sum_{x \in y + \Lambda_{2\kappa}^\ep} u(x) v(x) \, : \,  v = 0 \mbox{ on } \partial^+ (y + \Lambda_{2\kappa}^\ep), \,  \left\| v \right\|_{\underline{W}^{1,q'}(y + \Lambda_{2\kappa}^\ep)} \leq 1 \right\},
        $$
    \item \emph{Parabolic $L^{q}$-norm:}
    $
    \left\| u \right\|_{\underline{L}^{q} \left( z + Q_{2\kappa}^\ep \right)}^{q} :=  (2\kappa)^{-2} \int_{s + I_{2\kappa}}\left\| u(t , \cdot)\right\|^{q}_{\underline{L}^{q}(y + \Lambda_{2\kappa}^\ep)} \, dt,
    $
    \item \emph{$L^{q} W^{1,q}$-norm:}
    $
        \left\| u \right\|_{\underline{L}^q(s + I_{2\kappa}, \underline{W}^{1,q}(y + \Lambda_{2\kappa}^\ep))}^q :=  \kappa^{-2} \int_{s + I_{2\kappa}}\left\| u(t , \cdot)\right\|^q_{\underline{W}^{1,q}(y + \Lambda_{2\kappa}^\ep)} \, dt,
    $
    \item \emph{$L^{q}W^{-1,q}$-norm:}
         $
        \left\| u \right\|_{\underline{L}^{q}(s + I_{2\kappa} , W^{-1,q}(y + \Lambda_{2\kappa}^\ep))}^{q} := \kappa^{-2} \int_{s + I_{2\kappa}}\left\| u(t , \cdot)\right\|^{q}_{\underline{W}^{-1,q}(y + \Lambda_{2\kappa}^\ep)} \, dt,
        $
    \item  \emph{$W^{1,q}_{\mathrm{par}}$-norm:}
    $
    \left\| u \right\|_{\underline{W}^{1,q}_{\mathrm{par}}(z + Q_{2\kappa}^\ep)} := \kappa^{-1} \left\| u \right\|_{\underline{L}^q \left( z + Q_{2\kappa}^\ep \right)} + \left\| \nabla^\ep u \right\|_{\underline{L}^q \left(z + Q_{2\kappa}^\ep  \right)} + \left\| \partial_t u \right\|_{\underline{L}^{q}((0 , 1) , W^{-1,q}(z + Q_{2\kappa}^\ep))},
    $
    \item  \emph{$\hat{W}^{-1,q}_{\mathrm{par}}$-norm:}
    $$
    \left\| u \right\|_{\hat{\underline{W}}^{-1,q}_{\mathrm{par}}(z + Q_{2\kappa}^\ep)} := \sup \left\{ \kappa^{-2} \int_{s + I_{2\kappa}}\left| \Lambda_{2\kappa}^\ep \right|^{-1} \sum_{x \in y + \Lambda_{2\kappa}^\ep} u(t , x) v(t , x) \, dt : \,  \left\| v \right\|_{\underline{W}^{1,q'}_{\mathrm{par}}(z + Q_{2\kappa}^\ep)} \leq 1 \right\}.
    $$
\end{itemize}
    This notation is related to the ones introduced in Section~\ref{subsecfunctions} through scaling identities. In particular we will use the following one: for any function $u : Q_{L} \to \R$,
    \begin{equation} \label{eq:rescalingnegativenorm}
    \left\| u \left( \frac{\cdot}{\ep^2} , \frac{\cdot}{\ep} \right) \right\|_{\hat{\underline{W}}^{-1,q}_{\mathrm{par}}\left(Q^{\ep}_{\kappa} \right)} = \ep \left\| u \right\|_{\hat{\underline{W}}^{-1,q}_{\mathrm{par}}\left(Q_{L} \right)}.
    \end{equation}
For the rest of the argument, we select a function $f \in W^{1,r'}_{\mathrm{par}} \left( (0,1) \times \mathbb{T}^\ep \right)$ such that $ \left\| f \right\|_{\underline{W}^{1 , r'}_{\mathrm{par}} ((0,1) \times \mathbb{T}^\ep)} \leq 1$ and estimate the term
\begin{equation*}
    \int_0^1 \ep^d \sum_{x \in \mathbb{T}^\ep} \mathcal{E}_4 (t , x) f(t , x) \, dt.
\end{equation*}
We first use the definition of the error term $\mathcal{E}_4$ and write
\begin{equation*}
    \int_0^1 \ep^d \sum_{x \in \mathbb{T}^\ep} \mathcal{E}_4 (t , x) f(t , x) \, dt = \sum_{z} \int_0^1 \ep^d \sum_{x \in \mathbb{T}^\ep}  \nabla^{\ep} \chi_{z}(t , x) \cdot \left( D_p V \left( \nabla^{\ep} v_{z} (t , x) \right) - D_p \bar \sigma \left( p_{z} \right) \right) f(t , x) \, dt.
\end{equation*}
The strategy is then to regroup the partition of unity $\chi_z$ with the test function $f$ as follows: for any $z \in \mathcal{Z}_\kappa$,
\begin{align*}
    \lefteqn{\int_0^1 \ep^d \sum_{x \in \mathbb{T}^\ep}  \nabla^{\ep} \chi_{z}(t , x) \cdot \left( D_p V \left( \nabla^{\ep} v_{z} (t , x) \right) - D_p \bar \sigma \left( p_{z} \right) \right) f(t , x) \, dt } \qquad & \\ & 
    =  \int_0^1 \ep^d \sum_{x \in \mathbb{T}^\ep} \left( D_p V \left( \nabla^{\ep} v_{z} (t , x) \right) - D_p \bar \sigma \left( p_{z} \right) \right) \cdot  \left( \nabla^{\ep} \chi_{z}(t , x) f(t , x) \right) \, dt \\
    & \leq \left\| D_p V \left( \nabla^{\ep} v_{z}  \right) - D_p \bar \sigma \left( p_{z} \right) \right\|_{\underline{W}^{-1, r}_\mathrm{par}(z + Q_{2\kappa}^\ep)} \left\| (\nabla^{\ep} \chi_{z}) f\right\|_{\underline{W}^{1, r'}_\mathrm{par}(z + Q_{2\kappa}^\ep)}.
\end{align*}
Combining the two previous displays and applying the H\"{o}lder inequality, we obtain the upper bound
\begin{multline} \label{eq:14162604}
    \int_0^1 \ep^d \sum_{x \in \mathbb{T}^\ep} \mathcal{E}_4 (t , x) f(t , x) \, dt \leq \left( \sum_{z} \left\| D_p V \left( \nabla^{\ep} v_{z} \right) - D_p \bar \sigma \left( p_{z} \right) \right\|_{\underline{W}^{-1, r}_\mathrm{par}(z + Q_{2\kappa}^\ep)}^r \right)^{1/r} \\
    \times \left( \sum_{z} \left\|(\nabla^{\ep} \chi_{z}) f \right\|_{\underline{W}^{1,r'}_\mathrm{par}(z + Q_{2\kappa}^\ep)}^{r'} \right)^{1/r'} .
\end{multline}
The first term on the right-hand side can be estimated using Proposition~\ref{prop:sublinearityflux} which (after suitable rescaling as in~\eqref{eq:rescalingnegativenorm}) reads as follows: for any $z \in \mathcal{Z}_\kappa$,
\begin{align*}
    \left\| D_p V \left( \nabla^{\ep} v_{z} \right) - D_p \bar \sigma \left( p_{z} \right) \right\|_{\underline{W}^{-1, r}_\mathrm{par}(z + Q_{2\kappa}^\ep)} & \leq \mathcal{O}_{1/2} \left( C \ep L^{15/16} \right) \\
    & \leq \mathcal{O}_{\Psi , c} \left( C \ep^{\frac{1+15\gamma}{16}} \right).
\end{align*}
From the previous inequality, we deduce that
\begin{equation*}
     \left( \frac{1}{\left| \mathcal{Z}_\kappa \right|} \sum_{z} \left\| D_p V \left( \nabla^{\ep} v_{z} (t , x) \right) - D_p \bar \sigma \left( p_{z} \right) \right\|_{\underline{W}^{-1, r}_\mathrm{par}(z + Q_{2\kappa}^\ep)}^r \right)^{1/r} \leq \mathcal{O}_{\Psi,c} \left( C \ep^{\frac{1+15\gamma}{16}} \right).
\end{equation*}
We next estimate the (deterministic) second term on the right-hand side of~\eqref{eq:14162604}. Specifically, we will prove the inequality
\begin{equation} \label{eq:14512604}
    \frac{1}{\left| \mathcal{Z}_\kappa \right|} \sum_{z} \left\| (\nabla^{\ep} \chi_{z}) f \right\|_{\underline{W}^{1, r'}_\mathrm{par}(z + Q_{2\kappa}^\ep)}^{r'} \leq \frac{C}{\kappa^{2r'}}. 
\end{equation}
A combination of the two previous displays with~\eqref{eq:14162604} (taking the supremum over all the functions $f$ satisfying $\left\| f \right\|_{\underline{W}^{1,r'}_\mathrm{par}((0 , 1) \times \mathbb{T}^\ep)} \leq 1$) yields
\begin{equation*}
       \left\| \mathcal{E}_4 \right\|_{\underline{W}^{-1,r}_{\mathrm{par}}\left( (0 , 1) \times \mathbb{T}^\ep \right)}  \leq \mathcal{O}_{\Psi , c} \left( C \kappa^{-2} \left| \mathcal{Z}_\kappa \right|\ep^{\frac{1+15\gamma}{16}} \right).
\end{equation*}
Using the identity $\kappa = \ep^\gamma$ and noting that the cardinality of the set $\mathcal{Z}_\kappa$ is of order $\kappa^{-(d+2)}$, we obtain~\eqref{eq:appestimateE}.

There remains to prove the inequality~\eqref{eq:14512604}, we first write, for any fixed $z = (s , y) \in \mathcal{Z}_\kappa$,
\begin{align} \label{eq:17082604}
    \left\| (\nabla^{\ep} \chi_{z}) f \right\|_{\underline{W}^{1,r'}_\mathrm{par}(z + Q_{2\kappa}^\ep)} & = \frac{1}{\kappa}  \left\|  (\nabla^{\ep} \chi_{z}) f  \right\|_{\underline{L}^{r'}(z + Q_{2\kappa}^\ep)} + \left\| \nabla^\ep \left( (\nabla^{\ep} \chi_{z}) f \right) \right\|_{\underline{L}^{r'}(z + Q_{2\kappa}^\ep)} \\
    & \qquad + \left\| \partial_t \left( (\nabla^{\ep} \chi_{z}) f \right) \right\|_{\underline{L}^{r'}(s + I_{2\kappa} , \underline{W}^{-1,r'} \left( y + \Lambda_{2\kappa}^\ep \right) )} \notag
\end{align}
and estimate the three terms on the right-hand side. The first one is the easiest to estimate: using the upper bound on the gradient of the partition of unity stated in~\eqref{prop.partofunity2scbis}, we obtain
\begin{equation*}
     \left\| (\nabla^{\ep} \chi_{z}) f  \right\|_{\underline{L}^{r'}(z + Q_{2\kappa}^\ep)} \leq \frac{C}{\kappa} \left\|  f  \right\|_{\underline{L}^{r'}(z + Q_{2\kappa}^\ep)}.
\end{equation*}
Summing over $z \in \mathcal{Z}_\kappa$, we obtain
\begin{equation} \label{eq:11570305}
    \frac{1}{\left| \mathcal{Z}_\kappa \right|} \sum_z \left\|  (\nabla^{\ep} \chi_{z}) f  \right\|_{\underline{L}^{r'}(z + Q_{2\kappa}^\ep)}^{r'} \leq \frac{C}{\kappa^{r'}} \left\|  f  \right\|_{L^{r'}((0 , 1) \times \mathbb{T}^\ep)}^{r'} \leq \frac{C}{\kappa^{r'}}.
\end{equation}
For the second term on the right-hand side of~\eqref{eq:17082604}, we expand the discrete derivative and use~\eqref{prop.partofunity2scbis} a second time to obtain
\begin{align*}
    \left\| \nabla^\ep \left( (\nabla^{\ep} \chi_{z}) f \right) \right\|_{\underline{L}^{r'}(z + Q_{2\kappa}^\ep)} & \leq C \left\| (\nabla^{\ep, 2} \chi_{z}) f  \right\|_{\underline{L}^{r'}(z + Q_{2\kappa}^\ep)} + C\left\|(\nabla^{\ep} \chi_{z}) \nabla^\ep f \right\|_{\underline{L}^{r'}(z + Q_{2\kappa}^\ep)} \\
    & \leq \frac{C}{\kappa^{2}} \left\| f  \right\|_{\underline{L}^{r'}(z + Q_{2\kappa}^\ep)} + \frac{C}{\kappa} \left\| \nabla^\ep f \right\|_{\underline{L}^{r'}(z + Q_\kappa^\ep)}.
\end{align*}
Summing over $z \in \mathcal{Z}_\kappa$, we obtain
\begin{align} \label{eq:12000305}
    \frac{1}{\left| \mathcal{Z}_\kappa \right|} \sum_z \left\|\nabla^\ep \left( (\nabla^{\ep} \chi_{z}) f \right)\right\|_{\underline{L}^{r'}(z + Q_{2\kappa}^\ep)}^{r'} & \leq \frac{C}{\kappa^{2r'}} \left\| f  \right\|_{L^{r'}((0 , 1) \times \mathbb{T}^\ep)}^{r'} + \frac{C}{\kappa^{r'}}  \left\| \nabla^\ep f \right\|_{L^{r'}((0 , 1) \times \mathbb{T}^\ep)}^{r'} \\
    & \leq \frac{C}{\kappa^{2r'}}. \notag
\end{align}
The third term on the right-hand side of~\eqref{eq:17082604} is the most intricate to study. We first expand it as follows
\begin{multline} \label{eq:17472604}
    \left\| \partial_t \left( (\nabla^{\ep} \chi_{z}) f \right) \right\|_{\underline{L}^{r'}(s + I_{2\kappa} , \underline{W}^{-1,r'} \left( \Lambda_{2\kappa}^\ep \right) )}^{r'} \\
    \leq C \left\| \left( \partial_t  \nabla^{\ep} \chi_{z} \right) f  \right\|_{\underline{L}^{r'}(s + I_{2\kappa} , \underline{W}^{-1,r'} \left( y + \Lambda_{2\kappa}^\ep \right) )}^{r'}  + C \left\| (\nabla^{\ep} \chi_{z})\partial_t  f  \right\|_{\underline{L}^{r'}(s + I_{2\kappa} , \underline{W}^{-1,r'} \left( y + \Lambda_{2\kappa}^\ep \right) )}^{r'}.
\end{multline}
The first term on the right-hand side can be upper bounded using that the $W^{-1,r'}$-norm is smaller than the $L^{r'}$-norm (with suitable rescaling) and the upper bound~\eqref{prop.partofunity2scbis} on the time derivative of the gradient of the partition of unity. We obtain
\begin{align*}
    \left\| \left( \partial_t  \nabla^{\ep} \chi_{z} \right) f  \right\|_{\underline{L}^{r'}(s + I_{2\kappa} , \underline{W}^{-1,r'} \left( y + \Lambda_{2\kappa}^\ep \right) )} & \leq C \kappa \left\| \left( \partial_t  \nabla^{\ep} \chi_{z} \right) f  \right\|_{\underline{L}^{r'}(z +  Q_{2\kappa}^\ep )} \\
    & \leq \frac{C}{\kappa^{2}} \left\| f  \right\|_{\underline{L}^{r'}(z +  Q_{2\kappa}^\ep )}.
\end{align*}
To estimate the second term on the right-hand side of~\eqref{eq:17472604}, we introduce a test function $g \in L^r(s + I_{2\kappa}, W^{1,r} \left( y +  \Lambda_{2\kappa}^\ep \right))$ and write
\begin{equation*}
    \int_0^1 \ep^d \sum_{x \in \mathbb{T}^\ep} \nabla^{\ep} \chi_{z}(t , x) \partial_t  f(t , x) g(t , x) dt \leq \left\| \partial_t  f \right\|_{L^{r'}((0,1) , W^{-1,r'}(\mathbb{T}^\ep))}  \left\| (\nabla^{\ep} \chi_{z}) g \right\|_{L^r((0,1) , W^{1,r}( \mathbb{T}^\ep))}.
\end{equation*}
The second term on the right-hand side is finally estimated using~\eqref{prop.partofunity2scbis} and we obtain
\begin{equation*}
     \left\| (\nabla^{\ep} \chi_{z}) g \right\|_{L^r((0,1) , W^{1,r}( \mathbb{T}^\ep))} \leq \frac{C}{\kappa} \left\| g \right\|_{ \underline{L}^r(s + I_{2\kappa}, \underline{W}^{1,r} \left( y +  \Lambda_{2\kappa}^\ep \right))}.
\end{equation*}
A combination of the two previous inequalities yields the upper bound
\begin{equation*}
        \left\| (\nabla^{\ep} \chi_{z}) \partial_t  f  \right\|_{\underline{L}^{r'}(s + I_{2\kappa} , \underline{W}^{-1,r'} \left( y + \Lambda_{2\kappa}^\ep \right) )} \leq \frac{C}{\kappa} \left\| \partial_t  f \right\|_{L^{r'}((0,1) , W^{-1,r'}(\mathbb{T}^\ep))}
\end{equation*}
and thus
\begin{equation*} 
    \left\| \partial_t \left( (\nabla^{\ep} \chi_{z}) f \right) \right\|_{\underline{L}^{r'}(s + I_{2\kappa} , \underline{W}^{-1,r'} \left( y+\Lambda_{2\kappa}^\ep \right) )}^{r'} \leq \frac{C}{\kappa^{2r'}} \left\| f  \right\|_{\underline{L}^{r'}(z +  Q_{2\kappa}^\ep )}^{r'} + \frac{C}{\kappa^{r'}} \left\| \partial_t  f \right\|_{L^{r'}((0,1) , W^{-1,r'}(\mathbb{T}^\ep))}^{r'}.
\end{equation*}
Summing over $z \in \mathcal{Z}_\kappa$, we obtain
\begin{align} \label{eq:14100305}
    \frac{1}{\left| \mathcal{Z}_\kappa \right|}\sum_z  \left\| \partial_t \left((\nabla^{\ep,*} \chi_{z}) f \right) \right\|_{\underline{L}^{r'}(s + I_{2\kappa} , \underline{W}^{-1,r'} \left( \Lambda_{2\kappa}^\ep \right) )}^{r'} & \leq \frac{C}{\kappa^{2r'}} \left\| f  \right\|_{L^{r'}((0, 1) \times \mathbb{T}^\ep)}^{r'} + \frac{C}{\kappa^{r'}} \left\| \partial_t  f \right\|_{L^{r'}((0,1) , W^{-1,r'}(\mathbb{T}^\ep))}^{r'} \notag \\
    & \leq \frac{C}{\kappa^{2r'}}. 
\end{align}
Combining~\eqref{eq:17082604},~\eqref{eq:11570305},~\eqref{eq:12000305} and~\eqref{eq:14100305}, we obtain the inequality~\eqref{eq:14512604}.

\section{Moderating the two-scale expansion} \label{sec:sectionappendixC}

In this section, we show the three inequalities~\eqref{eq:realD20},~\eqref{eq:22572106} and~\eqref{eq:11140912sec7} stated in Step 1 of Section~\ref{section:eestimatingthediffernce}. We will make use of the notation introduced there, and in particular, we recall the definitions:
\begin{itemize}
    \item We let $\a$ be the environment defined in~\eqref{def:defevtAtwoscale};
    \item We let $\mathbf{\Lambda}_+$ and $\mathbf{\Lambda}_-$ be the maximal and minimal eigenvalues of the environment $\a$ as defined in~\eqref{def.lambda+-ineq6.44};
    \item We let $\mathbf{m}$ be the moderated environment defined in~\eqref{def.modeeratedevt}.
\end{itemize}

\subsection{Upper bound for the largest eigenvalue and the maximal function.}

This section is devoted to the proof of the inequality~\eqref{eq:realD20} restated below:
\begin{equation} \label{eq:D201159}
    \left\| \mathbf{\Lambda}_+ \right\|_{L^{\frac{r}{r-2}}((0,1) \times \mathbb{T}^\ep)} + \left\| \mathbf{M}_+ \right\|_{L^{\frac{r}{r-2}}((0,1) \times \mathbb{T}^\ep)} \leq \mathcal{O}_{\Psi , c} (C).
\end{equation}
We will first prove the inequality
\begin{equation} \label{eq:11250909}
    \left\| \mathbf{\Lambda}_+ \right\|_{L^{\frac{r}{r-2}}((0,1) \times \mathbb{T}^\ep)} \leq \mathcal{O}_{\Psi , c} (C).
\end{equation}
To prove the inequality~\eqref{eq:11250909}, we first use the definition of the environment $\a$ stated in~\eqref{def:defevtAtwoscale} together with the Assumption~\eqref{AssPot} to obtain the upper bound: for any $(t , x) \in (0,1) \times \mathbb{T}^\ep$,
\begin{equation*}
    \left| \mathbf{\Lambda}_+(t , x) \right| \leq C \left( \left| \nabla w^\ep (t , x) \right|^{r-2}_+ + \left| \nabla u^\ep (t , x) \right|^{r-2}_+  \right).
\end{equation*}
We thus have
\begin{equation*}
    \left\| \mathbf{\Lambda}_+ \right\|_{L^{\frac{r}{r-2}}((0,1) \times \mathbb{T}^\ep)} \leq C \left\| \nabla w^\ep \right\|_{L^{r}((0,1) \times \mathbb{T}^\ep)}^{r-2} + \left\|  \nabla u^\ep \right\|_{L^{r}((0,1) \times \mathbb{T}^\ep)}^{r-2}.
\end{equation*}
We may then apply the inequalities~\eqref{eq:AppC12} and~\eqref{eq:AppC13} to obtain
\begin{equation*}
    \left\| \mathbf{\Lambda}_+ \right\|_{L^{\frac{r}{r-2}}((0,1) \times \mathbb{T}^\ep)} \leq \mathcal{O}_{\Psi , c} (C).
\end{equation*}
By the Hardy-Littlewood maximal inequality (with respect to the time variable), we have, for any $x \in \mathbb{T}^\ep$,
\begin{equation*}
    \int_0^1 \left| \mathbf{M}_+ (t , x) \right|^{\frac{r}{r-2}} \, dt \leq C  \int_0^1  1 + \left|\mathbf{\Lambda}_+ (t , x) \right|^{\frac{r}{r-2}} \, dt.
\end{equation*}
Summing over the vertices $x \in \mathbb{T}^\ep$, we obtain
\begin{equation*}
    \left\| \mathbf{M}_+ \right\|_{L^{\frac{r}{r-2}}((0,1) \times \mathbb{T}^\ep)} \leq C \left\| \mathbf{\Lambda}_+ \right\|_{L^{\frac{r}{r-2}}((0,1) \times \mathbb{T}^\ep)} \leq \mathcal{O}_{\Psi , c} (C).
\end{equation*}
From the previous inequality, we deduce~\eqref{eq:D201159}.

\subsection{Stochastic integrability estimates for the moderated environment} \label{sec:sectionD2}

The objective of this section is to prove the inequality~\eqref{eq:22572106} restated below

\begin{equation} \label{eq:22572106App}
    \mathbb{P} \left[  \inf_{(t , x) \in (0,1) \times \mathbb{T}^\ep} \mathbf{m}(t , x)  \times \mathbf{M}_+(t , x)  \leq \ep^\theta \right] \leq C \exp \left( - c \left| \ln \ep \right|^{\frac{r}{r-2}}  \right).
\end{equation}

The proof is decomposed into two subsections and follows the same strategy as in Section~\ref{sec:section3moderated} (with essentially minor adaptations to prove the fluctuations of the gradient of the two-scale expansion $\nabla w^\ep$ instead of the gradient of the corrector).

\subsubsection{A fluctuation estimate for the two-scale expansion}

As in Section~\ref{sec:section3moderated}, the main step of the arguments is to show that the probability for the gradient of the two-scale expansion to remain in a bounded set for a long time is small (see Proposition~\ref{prop3.4}). This result is stated in the following lemma.

\begin{lemma}[Fluctuation for the two-scale expansion] \label{prop3.42scD2}
There exist two constants $C := C(d , V,f) < \infty$ and $c := c(d , V,f) >0$ such that, for any $T \geq 1$ and any vertex $x \in \mathbb{T}^\ep$, 
\begin{equation} \label{upperboundfluctuation2sc}
    \mathbb{P}\left[ \, \forall t \in [0 , \ep^{2 - \theta/(d+5)}], \, \left| \nabla^\ep w^\ep (t , x )  \right| \leq R_1 \, \right] \leq C \exp \left( - c \left| \ln \ep \right|^{\frac{r}{r-2}} \right).
\end{equation}
\end{lemma}

\begin{proof}
    The arguments are essentially identical to the ones of Section~\ref{sec:section3moderated} (and specifically of Proposition~\ref{prop3.4}), but are more technical due to the more involved definition of the two-scale expansion $w^\ep$. We only provide a detailed sketch of the proof. 
    
    We first rescale the problem and define the function
    \begin{equation*}
     w(t , x) := \ep^{-1} \bar u^\ep (\ep^2 t , \ep x) + \sum_{z \in \mathcal{Z}_\kappa} \chi_z( \ep^2 t , \ep x) \varphi_{z} \left( t , x ; p_z\right),
    \end{equation*}
    as well as the environment and maximal eigenvalue
    \begin{equation*}
        \a_{z} (t , x) := D_p^2 V ( p_z + \nabla \varphi_{z} (t , x ; p_z)) \hspace{3mm} \mbox{and} \hspace{3mm} \mathbf{\Lambda}_{+,z}(t , x) := \sup_{\substack{\xi \in \Rd \\ \left| \xi \right| = 1 }} \xi \cdot  \a_{z} (t , x) \xi.
     \end{equation*}
    With this definition, the inequality~\eqref{upperboundfluctuation2sc} is equivalent to the following result: for any $x \in \mathbb{T}_L$,
    \begin{equation*} 
    \mathbb{P}\left[ \, \forall t \in [0 , \ep^{ - \theta/(d+5)}], \, \left| \nabla w (t , x )  \right| \leq R_1 \, \right] \leq C \exp \left( - c \left| \ln \ep \right|^{\frac{r}{r-2}} \right).
\end{equation*}
    From now on and for the rest of the proof, we fix a vertex $x \in \mathbb{T}_L$.
    We next note that, for any $z \in \mathcal{Z}_\kappa$,
    \begin{equation*}
        \left| \nabla \varphi_{z} \left( t , x ; p_z\right) \right| \leq \mathcal{O}_r(C).
    \end{equation*}
    This result is a direct consequence of Proposition~\ref{prop.prop2.3}. 
    
    We then set set $N := \theta \left| \ln \ep \right| / ((d+5) R_1^2)$ and recall the notation introduced in the proof of Proposition~\ref{prop3.4} (and specifically the ones related to the Brownian bridges~\eqref{notation:Brownianbridges} and increments~\eqref{notation:Increments}). As in the proof of Proposition~\ref{prop3.4}, For any~$l \in \N$, we introduce the following random subset of $\R$ (depending on the collection $\mathcal{R}_{l}$),
\begin{equation*}
    \mathcal{A}_l (\mathcal{R}_{l}) := \left\{  X \in \R \, : \, \left| \nabla w \left(\frac{l+1}{N} , x ;p \right) \left( X, \mathcal{R}_{l} \right) \right| \leq R_1  \right\} \subseteq \R,
\end{equation*}
We then introduce the event $A_l \subseteq \Omega$ defined as follows
\begin{equation*}
    A_l := \left\{ \mathcal{R} := (X_l(x),  \mathcal{R}_{l}) \in \Omega \, : \,  X_l(x) \in \mathcal{A}_l (\mathcal{R}_{l}) ~~\mbox{and}~~ \frac{1}{\sqrt{2 \pi N}}\int_{\mathcal{A}_l (\mathcal{R}_{l})} e^{- \frac{x^2}{2N}} \, dx \leq 1- \ep^{9\theta/(10(d+5))} \right\}.
\end{equation*}
The same proof as in Proposition~\ref{prop3.4} shows the stretched exponential decay estimate
\begin{equation} \label{eq:11492811App}
    \mathbb{P}\left[ \bigcap_{l = 0}^{ \lfloor N \ep^{ - \theta/(d+5)} \rfloor} A_l \right] \leq \exp \left( - \ep^{-\theta/(10(d+5)} \right) \leq C \exp \left( - c \left| \ln \ep \right|^{\frac{r}{r-2}} \right).
\end{equation}
We next note that there exist two constants $\ep_G := \ep_G(d , V,f) < \infty$ and $C_G := C_G(d , V,f) < \infty$ such that the following implication holds: for any $\ep \leq \ep_G $, and any $z  \in \mathcal{Z}_\kappa$,
\begin{equation*}
    \left| p_z + \nabla \varphi_{z} (t , x ; p_z) \right| \leq \frac{ \left|\ln \ep \right|^{\frac{1}{r-2}}}{C_G} \implies \mathbf{\Lambda}_{+, z} (t , x) \leq \frac{N}{8d}.
\end{equation*}
We then define the interval
\begin{equation*}
    I_\ep := \left[ - \frac{\left| \ln \ep \right|^{\frac{1}{r-2}}}{(\sqrt{2}(4d)^2)C_G} , \frac{\left| \ln \ep \right|^{\frac{1}{r-2}}}{(\sqrt{2}(4d)^2)C_G}  \right],
\end{equation*}
as well as the bad event
\begin{multline*}
    G_\ep^c := \Bigg\{ \mathcal{R} \in \Omega \, : ~~ \exists z = (s , y) \in \mathcal{Z}_\kappa, ~ \ep^{-1} y \in x + \Lambda_L ~~ \mbox{and}  \\ \sup_{t \in [0,\ep^{ - \theta/(d+5)}]} \sum_{x' \sim x} \left| p_z + \nabla \varphi_{y , L} (t , x' ; p_z)(\mathcal{R}) \right| \geq \frac{\left| \ln \ep \right|^{\frac{1}{r-2}}}{2 C_G} \Bigg\} 
    \bigcup \bigcup_{k=0}^{\lfloor N \ep^{ - \gamma/(d+5)} \rfloor} \left\{  X_k(x) \notin I_\ep \right\}.
\end{multline*}
The same computation as in~\eqref{eq:18092811} (with an additional union bound giving a polynomial term in $\ep$ which can then be absorbed in the super-polynomial right-hand side of~\eqref{eq:19032805App}) yields the upper bound
\begin{equation} \label{eq:19032805App}
    \mathbb{P} \left( G_\ep^c \right) \leq C \exp \left( - c \left| \ln \ep \right|^{\frac{r}{r-2}} \right).
\end{equation}
Following the proof of Proposition~\ref{prop3.4}, we next show the inclusion of events
\begin{equation} \label{eq:18082811App}
    \left\{ \mathcal{R} \in \Omega \, : \, \forall t \in [0 , \ep^{ - \theta/(d+5)}], \, \left| \nabla w (t , x )  \right| \leq R_1 \right\} \subseteq \bigcap_{l = 0}^{ \lfloor  N \ep^{ - \theta/(d+5)}  \rfloor} A_l \cup G_T^c.
\end{equation}
Lemma~\ref{prop3.42scD2} is then obtained by combining~\eqref{eq:11492811App},~\eqref{eq:19032805App},~\eqref{eq:18082811App} and a union bound.

The proof of~\eqref{eq:18082811App} first relies on the observation that the derivative of the function $w$ with respect to the increment $X_l(x)$ is given by Proposition~\ref{prop:propsLangevin}, we know that, for $y \in \mathbb{T}_L$, the derivative of the $\nabla \varphi_{L} (t , y; p)$ with respect to the increment $X_l(x)$ is given by the following identity: for any $y \in \mathbb{T}_L$,
\begin{equation*}
    \frac{\partial w (t , y)}{\partial X_l(x)} =  \sqrt{2} N \int_{\frac{l}{N}}^{\frac{l+1}{N}}  \sum_{z} \chi_z( \ep^2 t , \ep x) P_{\a_{z}} (t , y ; s , x) \, ds.
\end{equation*}
Applying the discrete gradient on both sides of the identity, we obtain that
\begin{equation*}
    \frac{\partial \nabla w (t , y)}{\partial X_l(x)} =  \sqrt{2} N \int_{\frac{l}{N}}^{\frac{l+1}{N}}  \sum_{z} \nabla \left( \chi_z( \ep^2 t , \ep x) P_{\a_{z}} (t , y ; s , x) \right) \, ds.
\end{equation*}
The rest of the proof is then essentially identical to the proof of Proposition~\ref{prop3.4}: using the same arguments we may prove the monotonicity property, for any $\mathcal{R} :=  \left( X_l(x), \mathcal{R}_l \right) \in G_\ep$,
\begin{equation*}
    X \mapsto - \nabla_1 w \left( \frac{l+1}{N} , x \right)(X , \mathcal{R}_{l}) - \frac{3}{4} X ~~\mbox{is increasing on the interval}~I_\ep.
\end{equation*}
This property then implies the identity, for any $l \in \{ 1 , \ldots, N \ep^{- \gamma/(d+5)} \}$,
\begin{equation*}
          G_\ep \cap A_l = G_\ep \cap \left\{ \mathcal{R} \in \Omega \, : \,  \left|\nabla w \left(\frac{l+1}{N} , x \right) \left(\mathcal{R} \right) \right| \leq R_1 \right\}.
\end{equation*}
Taking the intersection over the integers $l \in \{ 1 , \ldots, N \ep^{- \theta/(d+5)} \}$ completes the proof of~\eqref{eq:18082811App}.
\end{proof}

\subsubsection{Stochastic integrability for the moderated environment}

\begin{proof}[Proof of the inequality~\eqref{eq:22572106App}]
We split the proof into two steps.

\medskip

\textit{Step 1. Proof of the inequality~\eqref{eq:probmoderatedevt1037}.} 

\medskip

This step is devoted to the proof of the following inequality: for any $\ep \in (0,1)$, any $t \in (0 , 1- \ep^{2})$ and any $x \in \mathbb{T}^\ep$,
\begin{equation} \label{eq:probmoderatedevt1037}
    \mathbb{P} \left[ \inf_{s \in [0 , \ep^2]} \mathbf{m}(t + s , x) \times \mathbf{M}_+(t +s , x) \leq \ep^\theta \right] \leq C \exp \left( - c \left| \ln \ep \right|^{\frac{r}{r-2}}  \right).
\end{equation}
To simplify the notation, we only prove the result when $t = 0$. For technical reasons, we will in fact prove the following estimate: there exists a constant $c_1 := c_1(d , V , f) > 0$ such that for any $\ep \in (0,1)$,
\begin{equation} \label{eq:probmoderatedevt10378}
    \mathbb{P} \left[ \inf_{s \in [0 , \ep^2]} \mathbf{m}( s , x) \times \mathbf{M}_+(s , x) \leq c_1 \ep^\theta \right] \leq C \exp \left( - c \left| \ln \ep \right|^{\frac{r}{r-2}}  \right).
\end{equation}
The inequality~\eqref{eq:probmoderatedevt1037} can be easily deduced by applying~\eqref{eq:probmoderatedevt10378} with $c_1^{-1/\theta} \ep$ instead of $\ep$ (and by increasing the value of $C$ and reducing the value of $c$).
To prove the inequality~\eqref{eq:probmoderatedevt10378}, we first note that, by the definition of the moderated environment and of the maximal function, we have the inequality: for any $(s , x) \in (0,1/2) \times \mathbb{T}^\ep$,
\begin{equation*}
     \mathbf{m}(s , x) \times \mathbf{M}_+(s , x) \geq 
     \ep^{-2} \int_s^{1} k_{\ep^{-2}(s'-t)} \left( \mathbf{\Lambda}_- (s', x )  \wedge 1 \right) \, ds'.  
\end{equation*}
It is thus sufficient to estimate the probability of the right-hand side to be small. We proceed as in the proof of Proposition~\ref{prop3.4} and first show the inclusion of events: for any $\ep \in (0,1)$ (N.B. this is the same constant $c_1$ as in~\eqref{eq:probmoderatedevt10378}, and all the events on the right-hand side have small probability)
\begin{align} \label{eq:11030712app}
    \left\{ \inf_{s \in (0 , \ep^2)} \mathbf{m}(s , x) \times \mathbf{M}_+(s , x) \leq c_1 \ep^\theta \right\} & \subseteq \left\{ \sup_{t \in [0,\ep^{2 - \theta/6}]} \left| \nabla^\ep w^\ep (t , x)  \right| \leq  R_1 \right\} \\
    & \quad \bigcup \left\{ \sup_{t \in [0,\ep^{2 - \theta/6}]} \sum_{x' \sim x} \sum_{z} \chi_z(t , x') \left| \nabla^\ep \cdot D_p V \left(  \nabla^\ep v_z \right) \left( t ,x' \right) \right| \geq \ep^{-1 - \theta/6}  \right\} \notag \\
    & \quad \bigcup \left\{ \sup_{\substack{t , t' \in [0,\ep^{2-\theta/6}]\\ |t - t'|  \leq  \ep^{2+\theta/6}}}  \sum_{x' \sim x}  \left| \nabla^\ep B_{t'}^\ep \left( x\right) - \nabla^\ep B_{t}^\ep \left( x \right) \right| \geq \frac{1}{2} \right\} \notag \\
    & \quad \bigcup \left\{ \sup_{t \in [0,\ep^{2 - \theta/6}]} \sum_{x' \sim x} \sum_{z} \left|\partial_t \chi_{z}(t , x') \right| \ep \left| \varphi_{z} \left( \frac{t}{\ep^2} , \frac{x'}{\ep} ; p_{z} \right) \right| \geq \ep^{-1 - \theta/6} \right\}. \notag
\end{align}
The inclusion~\eqref{eq:11030712app} is equivalent to the following implication
\begin{equation} \label{eq:17260712app}
    \left. \begin{aligned}
        \sup_{t \in [0,\ep^{2 - \theta/6}]} \left| \nabla^\ep w^\ep (t , x)  \right| & \geq  R_1 \\ 
        \sup_{t \in [0,\ep^{2 - \theta/6}]}   \sum_{x' \sim x} \sum_{z} \chi_z(t , x') \left| \nabla^\ep \cdot D_p V \left(  \nabla^\ep v_z \right) \left( t , x' \right) \right| 
        & \leq \ep^{-1 - \theta/6} \\
        \sup_{\substack{t , t' \in [0,\ep^{2-\theta/6}]\\ |t - t'|  \leq  \ep^{2+\theta/6}}} \sum_{x' \sim x} \left| \nabla^\ep B_{t'}^\ep \left( x'\right) - \nabla^\ep B_{t}^\ep \left( x' \right) \right| & \leq \frac{1}{2} \\
        \sup_{t \in [0,\ep^{2 - \theta/6}]} \sum_{x' \sim x} \sum_{z} \left|\partial_t \chi_{z}(t , x') \right| \ep \left| \varphi_{z} \left( \frac{t}{\ep^2} , \frac{x'}{\ep} ; p_{z} \right) \right| & \leq \ep^{-1 - \theta/6} 
    \end{aligned} \right\} 
    \implies \inf_{s \in (0 , \ep^2)} \mathbf{m}(s , x) \geq c_1 \ep^\theta.
\end{equation}
Let us assume that all the conditions on the left-hand side of~\eqref{eq:17260712app} are satisfied and recall the identity~\eqref{eq:16002107.sec4}
\begin{align*} 
\lefteqn{
\partial_t \left(  w^\ep -  \sqrt{2} B_\cdot^\ep \right)
} \qquad & \notag \\
 & = \partial_t  \bar u^\ep   +   \sum_{z} \chi_{z} \partial_t \left( \ep \varphi_{z} \left( \frac{\cdot}{\ep^2} , \frac{\cdot}{\ep} ; p_{z} \right) - \sqrt{2} B_\cdot^\ep \right) +   \sum_{z} \partial_t \chi_{z} \left( \ep \varphi_{z} \left( \frac{\cdot}{\ep^2} , \frac{\cdot}{\ep} ; p_{z} \right) - \sqrt{2} B_\cdot^\ep  \right) \\
 & = \partial_t  \bar u^\ep   +  \sum_{z} \chi_z \nabla^\ep \cdot D_p V \left(  \nabla^\ep v_z \right)  +   \sum_{z} \left(\partial_t \chi_{z} \right) \ep \varphi_{z} \left( \frac{\cdot}{\ep^2} , \frac{\cdot}{\ep} ; p_{z} \right).
 \end{align*}
Using the assumptions of~\eqref{eq:17260712app}, we deduce that, for any $t \in  [0 , \ep^{2 - \theta/6}],$
\begin{equation*}
   \sum_{x' \sim x} \left| \partial_t \left(  w^\ep -  \sqrt{2} B_\cdot^\ep \right)(t ,x') \right| \leq C \ep^{-1 - \theta/6}.
\end{equation*}
The previous inequality implies the upper bound, for any $t \in  [0 , \ep^{2 - \theta/6}],$
\begin{equation} \label{eq:17311605}
    \left| \partial_t \left( \nabla^\ep w^\ep -  \sqrt{2} \nabla^\ep B_\cdot^\ep \right)(t , x) \right| \leq C \ep^{-2 - \theta/6}.
\end{equation}
We then set $c_0 := (2C)^{-1}$ where $C$ is the constant on the right-hand side of the previous inequality. 
We next fix a time $t \in [0 , \ep^{2 - \theta/6}]$ such that $ \left| \nabla^\ep w^\ep (t , x)  \right|  \geq  R_1$. By~\eqref{eq:17311605}, for any time $s \geq 0 $ with $|s - t| \leq c_0 \ep^{2+\theta/6}$,
\begin{align*}
    \left| \nabla^\ep w^\ep(s , x) -  \nabla^\ep w^\ep(t , x) \right| &
    \leq c_0 \ep^{2 + \theta} C \ep^{-2 - \theta/6} + \left| \nabla^\ep B_{s}^\ep \left( x\right) - \nabla^\ep B_{t}^\ep \left( x \right) \right| \\
    & \leq \frac12 + \frac 12 \\
    & \leq 1.
\end{align*}
Using the assumption $R_1 \geq 2$, we deduce that, for any time $s \geq 0 $ with $|s - t| \leq c_0 \ep^{2+\theta/6}$
\begin{equation*}
    \left| \nabla^\ep w^\ep (s , x ) \right| \geq \frac{R_1}{2}.
\end{equation*} 
From the definition of $R_1$ in Lemma~\ref{lemma.upperandlowerboundLambda} and the environment $\a$ in~\eqref{def:defevtAtwoscale}, we obtain the lower bound, for any  $s \in \R$ with $|s - t| \leq c_1 \ep^{2+\theta/6}$,
\begin{equation*}
    \mathbf{\Lambda}_-(s ,x ; p) \geq 1.
\end{equation*}
Combining the two previous displays with the definitions of the moderated environment $\mathbf{m}$, the maximal function $\mathbf{M}_+$, and the definition of the function~$k$, we deduce that
\begin{align*}
    \inf_{s \in (0 , \ep^2)} \mathbf{m}(s, x ) \times \mathbf{M}_+(s , x) & \geq \ep^{-2} \int_0^{1} k_{\ep^{-2} s} \left( \mathbf{\Lambda}_- (s, 0)  \wedge 1 \right) \, ds \\
    & \geq \ep^{-2} \int_{t - c_0 \ep^{2+\theta/6}}^{t +c_0 \ep^{2+\theta/6}} k_{\ep^{-2} s} \left( \mathbf{\Lambda}_- (s, 0)  \wedge 1 \right) \, ds  \\
    & \geq \ep^{\theta/6} \ep^{-2}  \int_{t - c_0 \ep^{2+\theta/6}}^{t +c_0 \ep^{2+\theta/6}} k_{\ep^{-2} s} \, ds \\
    & \geq  \ep^{-2+\theta/6} (2 c_0 \ep^{2+\theta/6}) k_{\ep^{-\theta/6} +c_0 \ep^{\theta/6}}.
\end{align*}
Using the definition of the function $k$ (and in particular that it decays asymptotically like the function $t \mapsto t^{-4}$), we obtain
\begin{equation*}
    \inf_{s \in (0 , \ep^2)} \mathbf{m}(s, x ) \times \mathbf{M}_+(s , x) \geq c_1 \ep^{\theta}.
\end{equation*}
The proof of~\eqref{eq:17260712app}, and thus of~\eqref{eq:11030712app} is complete. 

The last step of the proof consists in showing that all the events on the right-hand side of~\eqref{eq:11030712app} have small probability. The proof is essentially identical to the one of Proposition~\ref{prop:stochintmoderated} (we only need to use Lemma~\ref{prop3.42scD2} instead of Proposition~\ref{prop3.4}) and thus omit the technical details.

\medskip

\textit{Step 2. Proof of the inequality~\eqref{eq:22572106App}.}

\medskip

The inequality~\eqref{eq:22572106App} is then deduced from~\eqref{eq:probmoderatedevt1037} and a union bound as follows
\begin{align*}
   \mathbb{P} \left[ \inf \mathbf{m} \times \mathbf{M}_+ \leq \ep^\theta \right]
    & = \mathbb{P} \left[ \exists x \in \mathbb{T}^\ep, \, \exists k \in \{ 0, \ldots, \lfloor \ep^{-2} \rfloor \}, ~\inf_{s \in (0 , \ep^2)} \mathbf{m}(k \ep^2 + s , x) \times \mathbf{M}_+(k \ep^2 + s , x) \leq \ep^\theta \right] \\
    & \leq \sum_{x \in \mathbb{T}^\ep} \sum_{k = 0}^{\lfloor \ep^{-2} \rfloor} \mathbb{P} \left[ \inf_{s \in (0 , \ep^2)} \mathbf{m}(k \ep^2 + s , x) \times \mathbf{M}_+(k \ep^2 + s , x) \leq \ep^\theta\right].
\end{align*}
We next use the observation that the cardinality of $\mathbb{T}^\ep$ is equal to $\ep^{-d}$ and combine it with the inequality~\eqref{eq:probmoderatedevt1037} to obtain
\begin{align*}
    \mathbb{P} \left[ \inf \mathbf{m} \times \mathbf{M}_+ \leq \ep^\theta \right] & \leq \sum_{x \in \mathbb{T}^\ep} \sum_{k = 0}^{\lfloor \ep^{-2} \rfloor} C \exp \left( - c \left| \ln \ep \right|^{\frac{r}{r-2}}  \right) \\
    & \leq C \ep^{-d-2} \exp \left( - c \left| \ln \ep \right|^{\frac{r}{r-2}}  \right) \\
    & \leq C \exp \left( - c \left| \ln \ep \right|^{\frac{r}{r-2}}  \right),
\end{align*}
where in the last line we increased the constant $C$ and reduced the exponent $c$ to absorb the polynomial factor $\ep^{-d-2}$. The proof of the inequality~\eqref{eq:22572106App} is complete.
\end{proof}

\subsection{Moderation in finite time with nonzero right-hand side}

The proof of Proposition~\ref{propD1} below is essentially an adaptation of the one of Proposition~\ref{prop4.3} with two differences: the result is written with a nonzero right-hand side (the result is used with $F = \nabla^\ep \cdot \vec{\mathcal{E}}$ or $F = \mathcal{E}$ in Section~\ref{section:eestimatingthediffernce} and is essentially identical to Lemma~2.11 of the article of Biskup and Rodriguez~\cite{biskup2018limit}), and the integral on the right-hand side stops at time $1$ instead of infinity (this is the reason why the moderated environment introduced in~\eqref{def.modeeratedevt} is different from the one introduced in Definition~\ref{def:moderatedLangevin}).

\begin{proposition}[Moderation with nonzero right-hand side] \label{propD1}
Let $F : (0 , 1) \times \mathbb{T}^\ep \to \R$ be a continuous function (with respect to the time variable). There exists a constant $C := C(d) < \infty$ such that, for every solution $u : (0 , 1) \times \mathbb{T}^\ep \to \R$ of the parabolic equation\begin{equation*}
    \partial_t u - \nabla^\ep \cdot \a \nabla^\ep u = F ~~~\mbox{in} ~~~ (0 , 1) \times \mathbb{T}^\ep,
\end{equation*}
one has the inequality
\begin{multline*}
    \int_0^1 \sum_{x \in \mathbb{T}^\ep} \mathbf{m}(t , x) \left| \nabla^\ep u(t , x) \right|^2 \, dt
    \leq C \int_0^1 \sum_{x \in \mathbb{T}^\ep} \nabla^\ep u(t, x) \cdot \a(t , x ) \nabla^\ep u(t , x) \, dt \\ + C \ep^2  \int_0^1 \sum_{x \in \mathbb{T}^\ep} |F(t,x)|^2 \, dt.
\end{multline*}
\end{proposition}

\begin{remark}
    As mentioned above, this result has already been proved and used in the article of Biskup and Rodriguez~\cite[Lemma 2.11]{biskup2018limit}. The proof is added to this article for completeness.
\end{remark}

\begin{proof}
In order to prove Proposition~\ref{propD1}, it is enough to prove the two inequalities
 \begin{multline} \label{eq:17221305}
     \int_0^{1/2} \sum_{x \in \mathbb{T}^\ep}  \mathbf{m}(t , x) \left| \nabla^\ep u (t , x) \right|^2 \, dt 
     \leq C \int_0^{1} \sum_{x \in \mathbb{T}^\ep} \nabla^\ep u(t , x) \cdot \a(t , x) \nabla^\ep u(t , x) \, dt \\ + C \ep^2 \int_0^{1} \sum_{x \in \mathbb{T}^\ep} | F (t , x) |^2 \, dt.
\end{multline}
and
\begin{multline} \label{eq:17231305}
     \int_{1/2}^{1} \sum_{x \in \mathbb{T}^\ep}  \mathbf{m}(t , x) \left| \nabla^\ep u (t ,x) \right|^2 \, dt
     \leq C \int_0^{1} \sum_{x \in \mathbb{T}^\ep} \nabla^\ep u(t , x) \cdot \a(t , x) \nabla^\ep u(t , x) \, dt \\
     + C \ep^2 \int_0^{1} \sum_{x \in \mathbb{T}^\ep} | F (t , x) |^2 \, dt.
\end{multline}
The proof of the inequalities~\eqref{eq:17221305} and~\eqref{eq:17231305} are almost identical, we thus only present the proof of~\eqref{eq:17221305} and split the argument into two steps.

\medskip

\textit{Step 1. Proof of the inequality~\eqref{ineq:moderationv1}.} In this step, we fix a time $t \in (0 , 1/2)$, a vertex $x \in \mathbb{T}^\ep$ and prove the inequality
    \begin{align} \label{ineq:moderationv1}
    \mathbf{m}(t , x) \left| \nabla^\ep u (t , x) \right|^2 & \leq C \ep^{-2} \sum_{y \sim x}  \int_t^{1} K_{\ep^{-2}(s - t)} \nabla^\ep u(s , y) \cdot \a(s , y) \nabla^\ep u(s , y) \, ds \\
    & \qquad + C \sum_{y \sim x}  \int_t^{1} K_{\ep^{-2}(s - t)}  | F (s , x) |^2  \, ds. \notag
\end{align}
To prove~\eqref{ineq:moderationv1}, we first use the definition of the moderated environment
\begin{align} \label{eq:moderatedwithepandrighthand}
     \lefteqn{\mathbf{m}(t , x) \left| \nabla^\ep u(t , x) \right|^2} \qquad &  \\ & = \ep^{-2} \int_t^{1} k_{\ep^{-2}(s-t)} \frac{\mathbf{\Lambda}_- (s, x )  \wedge 1 }{( s-t)^{-1} \sum_{y \sim x}\int_t^s \left( 1+ \mathbf{\Lambda}_+ \left( s' , x \right)  \right) \, ds'} \left| \nabla^\ep u(t , x) \right|^2 \, ds \notag \\
     & = 2 \ep^{-2} \int_t^{1} k_{\ep^{-2}(s-t)} \frac{\mathbf{\Lambda}_- (s, x )  \wedge 1 }{( s-t)^{-1} \sum_{y \sim x}\int_t^s \left( 1+ \mathbf{\Lambda}_+ \left( s' , x \right)  \right) \, ds'} \left| \nabla^\ep u(s , x) \right|^2 \, ds \notag \\
     & \qquad + 2 \ep^{-2} \int_t^{1} k_{\ep^{-2}(s-t)} \frac{\mathbf{\Lambda}_- (s, x )  \wedge 1 }{( s-t)^{-1} \sum_{y \sim x}\int_t^s \left( 1+ \mathbf{\Lambda}_+ \left( s' , x \right)  \right) \, ds'} \left| \nabla^\ep u(s , x) - \nabla^\ep u(t , x)\right|^2 \, ds. \notag
\end{align}
For the first term on the right-hand side, we have the upper bound
\begin{multline*}
\ep^{-2} \int_t^{1} k_{\ep^{-2}(s-t)} \frac{\mathbf{\Lambda}_- (s, x )  \wedge 1 }{( s-t)^{-1} \sum_{y \sim x}\int_t^s \left( 1+ \mathbf{\Lambda}_+ \left( s' , x \right)  \right) \, ds'} \left| \nabla^\ep u(s , x) \right|^2 \, ds \\ \leq \ep^{-2} \int_t^{1} k_{\ep^{-2}(s-t)} \nabla^\ep u(s , y) \cdot \a(s , y ) \nabla^\ep u(s , y) \, ds.
\end{multline*}
For the second term on the right-hand side of~\eqref{eq:moderatedwithepandrighthand}, we use the identity $\partial_t u = \nabla^\ep \cdot \a \nabla^\ep u + F$ to write
\begin{align*}
    (\nabla^\ep u(s , x) - \nabla^\ep u(t , x) )^2 & \leq C \ep^{-2} \sum_{y \sim x} ( u(s , y) - u(t , y) )^2   \\
    & \leq C \ep^{-2} \sum_{y \sim x} \left(  \int_t^s \nabla^\ep \cdot \a \nabla^\ep u(s' , x) + F(s',x)\, ds' \right)^2 \\
    & \leq C \ep^{-2} \sum_{y \sim x} \left(  \int_t^s \nabla^\ep \cdot \a \nabla^\ep u(s' , y) \, ds' \right)^2 + C \ep^{-2} \sum_{y \sim x} \left(  \int_t^s F(s',y)\, ds' \right)^2.
\end{align*}
The first two terms are estimated as in the proof of Proposition~\ref{prop4.3} (taking into account that the gradients have been scaled by a factor $\ep^{-1}$, this accounts for an additional multiplicative factor $\ep^{-2}$ on the right-hand side below) and the second one can be upper bounded by using the Cauchy-Schwarz inequality. We obtain
\begin{multline*}
    \frac{\mathbf{\Lambda}_- (s, x )  \wedge 1 }{( s-t)^{-1} \sum_{y \sim x}\int_t^s \left( 1+ \mathbf{\Lambda}_+ \left( s' , x \right)  \right) \, ds'} (\nabla^\ep u(s , x) - \nabla^\ep u(t , x) )^2 \\
    \leq C  (s-t)\sum_{y \sim x} \int_t^s \left( \ep^{-4} \nabla^\ep u(s' , y) \cdot \a(s', y) \nabla^\ep u(s' , y)  + \ep^{-2} |F(s',y)|^2 \right) \, ds'.
\end{multline*}
From the previous inequality, we deduce that
\begin{align*}
   \lefteqn{\ep^{-2} \int_t^{1} k_{\ep^{-2}(s-t)} \frac{\mathbf{\Lambda}_- (s, x )  \wedge 1 }{( s-t)^{-1} \sum_{y \sim x}\int_t^s \left( 1+ \mathbf{\Lambda}_+ \left( s' , x \right)  \right) \, ds'} (\nabla^\ep u(s , x) - \nabla^\ep u(t , x) )^2 \, dt } \qquad & \\ &
   \leq  C \ep^{-2} \sum_{y \sim x} \int_t^{1} k_{\ep^{-2}(s-t)} (s-t) \int_t^s \left( \ep^{-4} \nabla^\ep u(s' , y) \cdot \a(s', y) \nabla^\ep u(s' , y)  + \ep^{-2} |F(s',y)|^2 \right) \, ds' \\
   & \leq C \sum_{y \sim x} \ep^{-2} \int_t^1  K_{\ep^{-2}(s-t)} \nabla^\ep u(s , y) \cdot \a(s, y) \nabla^\ep u(s , y) \, ds +  C \sum_{y \sim x}  \int_t^1  K_{\ep^{-2}(s-t)} |F(s,y)|^2 \, ds.
\end{align*}
The proof of the inequality~\eqref{ineq:moderationv1} is complete.

\medskip

\textit{Step 2. Proof of~\eqref{eq:17221305}.}

\medskip

Summing both sides of the inequality~\eqref{ineq:moderationv1} over the vertices of $\mathbb{T}^\ep$ and integrating over the times $t \in (0 , 1/2)$ gives the inequality
\begin{align*}
    \int_0^{1/2} \sum_{x \in \mathbb{T}^\ep} \mathbf{m}(t , x) \left| \nabla^\ep u (t , x) \right|^2 \, dt & \leq C \ep^{-2} \int_0^{1/2} \int_t^{1} \sum_{x \in \mathbb{T}^\ep}   K_{\ep^{-2}(s - t)} \nabla^\ep u(s , y) \cdot \a(s , y) \nabla^\ep u(s , y) \, ds \, dt \\
    & \qquad + C \int_0^{1/2} \int_t^{1} \sum_{x \in \mathbb{T}^\ep} K_{\ep^{-2}(s - t)}  | F (s , x) |^2   \, ds \, dt.
\end{align*}
Applying Fubini's theorem together with the inequality
\begin{equation*}
    \ep^{-2} \int_0^{s } K_{\ep^{-2}(s-t)} \, dt \leq \ep^{-2} \int_0^{1} K_{\ep^{-2}t} \, dt
     \leq \int_0^{\infty} K_{t} \, dt \leq C,
\end{equation*}
we obtain the inequality~\eqref{eq:17221305}.
\end{proof}

\subsection{Moderation for the error term $\vec{\mathcal{E}}_2$}

In this section, we prove the following technical inequality which is used in the estimate of the error term $\mathcal{E}_2$ and whose proof relies on the techniques developed in this appendix.

\begin{lemma} \label{lemmavzvzprime}
    For any exponent $\alpha > 0$ and any pair $z , z' \in \mathcal{Z}_\kappa$ such that $(z + Q_{\kappa}^\ep) \cap (z' + Q_{2\kappa}^\ep) \neq \emptyset$, there exist two constants $C := C(d , V , \alpha) < \infty$ and $c := c(d , V , \alpha) > 0$ such that
\begin{equation*}
    \left\| \nabla^\ep v_{z'} - \nabla^\ep v_{z}  \right\|_{\underline{L}^2(z + Q_\kappa^\ep)} \leq \mathcal{O}_{\Psi , c} \left( C L^{\alpha-\frac{1}{16}} + C  L^\alpha  |p_z - p_{z'}|  \right) . 
\end{equation*}

\begin{proof}
We fix an exponent $\alpha > 0$ and allow all the constants in this proof to depend on the parameter~$\alpha$. In order to simplify the argument,  and will prove the following statement. Let $\Lambda_1 , \Lambda_2$ be two boxes of $\Zd$ such that $\Lambda_{2L} \subseteq \Lambda_1 \cap \Lambda_2$ and $\Lambda_1 \cup \Lambda_2 \subseteq \Lambda_{8L}.$ Let $\varphi_1(\cdot , \cdot ; p_z)$ be the stationary Langevin dynamic on the cylinder $\R \times \Lambda_1$ with slope $p_z$, let $\varphi_2(\cdot , \cdot ; p_{z'})$ be the stationary Langevin dynamic on the cylinder $\R \times \Lambda_2$ with slope $p_{z'}$ and set
\begin{equation*}
    v_1 (t , x) := p_z \cdot x + \varphi_1(t , x ; p_z) ~~\mbox{and}~~  v_2 (t , x) := p_{z'} \cdot x + \varphi_2(t , x ; p_{z'}).
\end{equation*}
Then we have 
\begin{equation} \label{eq:23071557}
     \left\| \nabla v_{1} - \nabla v_{2}  \right\|_{\underline{L}^2(Q_L)} \leq \mathcal{O}_{\Psi , c} \left( C L^{\alpha-\frac{1}{16}} + C  L^\alpha    |p_z - p_{z'}|  \right) .
\end{equation}
The lemma can then be deduced from the previous inequality by a rescaling argument.

To prove~\eqref{eq:23071557}, we first note that the map $w := v_{1} - v_{2}$ solves the following parabolic equation
\begin{equation*}
        \partial_t w + \nabla \cdot \a \nabla w = 0 ~~\mbox{in}~~ Q_L,
\end{equation*}
with
\begin{equation*}
    \a(t , x) := \int_0^1 D^2_p V \left( s \nabla v_1 + (1  - s) \nabla v_{2}  \right) \, ds. 
\end{equation*}
We also denote by $\mathbf{\Lambda}^+$ and $\mathbf{\Lambda}^-$  the maximal eigenvalue of $\a$, i.e.,
\begin{equation*}
    \mathbf{\Lambda}^+(t , x) := \sup_{\substack{\xi \in \Rd \\ |\xi| = 1}} \xi \cdot \a(t , x) \xi ~~\mbox{and}~~ \mathbf{\Lambda}^-(t , x) := \inf_{\substack{\xi \in \Rd \\ |\xi| = 1}} \xi \cdot \a(t , x) \xi,
\end{equation*}
and note that we have the stochastic integrability estimate: for $(t , x) \in Q_L$,
\begin{equation} \label{eq:15132307}
    \mathbf{\Lambda}^+(t , x) \leq \mathcal{O}_{r/(r-2)}(C).
\end{equation}
By the Caccioppoli inequality, we have
\begin{equation*}
    \int_{I_L} \sum_{x \in \Lambda_L } \nabla w (t , x) \cdot \a(t , x) \nabla w(t , x) \, dt \leq \frac{C}{L^2} \int_{I_{2L}} \sum_{x \in \Lambda_{2L} } \mathbf{\Lambda}^+(t , x) \left| w(t , x) \right|^2 \, dt.
\end{equation*}
The right-hand side is then estimated as follows: by using the identity $w = v_{1} - v_{2}$, we may write
\begin{multline*}
    \int_{I_L} \sum_{x \in \Lambda_L } \nabla w (t , x) \cdot \a(t , x) \nabla w(t , x) \, dt  \leq C |p_z - p_{z'}|^2 \int_{ I_{2L}} \sum_{x \in  \Lambda_{2L} }  \mathbf{\Lambda}^+(t , x) \, dx  \\
    + \frac{C}{L^2}  \int_{I_{2L}} \sum_{x \in \Lambda_{2L} }  \mathbf{\Lambda}^+(t , x) \left( \left| \varphi_{1} \left( t,x ;  p_z \right) \right|^2 + \left|  \varphi_{2} \left(t , x;  p_{z'} \right) \right|^2 \right) \, dx.
\end{multline*}
The first terms on the right-hand side can be estimated using~\eqref{eq:15132307} and Proposition~\ref{prop:prop2.20} ``Summation" and ``Integration". The second term can be estimated using ~\eqref{eq:15132307}, Proposition~\ref{prop:sublincorr} as well as Proposition~\ref{prop:prop2.20} ``Product", ``Summation" and ``Integration" (and recalling the identity $L = \kappa/\ep$). We obtain
\begin{equation} \label{eq:17012307}
     \frac{1}{|I_{2L}|}\int_{I_{2L}} \frac{1}{\left| \Lambda_{2L} \right|} \sum_{x \in \Lambda_{2L}} \nabla w (t , x) \cdot \a(t , x) \nabla w(t , x) \, dt \leq \mathcal{O}_{\Psi , c} \left( C L^{-1/4} + C  |p_z - p_{z'}|^2 \right).
\end{equation}
This estimate is already quite close to the conclusion of the Lemma, but we need to take into account that the environment $\a$ can be degenerate in order to conclude. This is done by using a moderation argument. Since the argument is very similar to the ones presented in this section (and is in fact simpler), we only present a detailed sketch of the proof.

We first define the following moderated environment
    \begin{equation*}
        \mathbf{m}(t , x) := 
     \left\{ \begin{aligned}
     \int_t^{0} k_{(s-t)} \frac{\mathbf{\Lambda}^- (s, x )  \wedge 1 }{( s-t)^{-1} \sum_{y \sim x}\int_t^s \left( 1+ \mathbf{\Lambda}_+ \left( s' , x \right)  \right) \, ds'} \, ds & ~~\mbox{if} ~~ t \in \left[- L^2, - \frac{L^2}{2}\right], \\
     \int_{-L^2}^{t} k_{(s-t)} \frac{\mathbf{\Lambda}^- (s, x )  \wedge 1 }{( s-t)^{-1} \sum_{y \sim x}\int_t^s \left( 1+ \mathbf{\Lambda}^+ \left( s' , x \right)  \right) \, ds'} \, ds & ~~\mbox{if} ~~ t \in \left( - \frac{L^2}{2}, 0\right],
     \end{aligned} \right.
    \end{equation*}
    as well as its infimum over the parabolic cylinder $ Q_L$
    \begin{equation*}
        \mathbf{m}_{-} := \inf_{(t , x) \in Q_L} \mathbf{m}(t , x)
    \end{equation*}
and collect the following properties and estimates:
\begin{itemize}
    \item Tail estimate: the probability for the moderated environment $\mathbf{m}_{-}$ to be small is controlled by the following inequality: for any exponent $\alpha > 0$, there exist two constants $C := C(d , V , \alpha,f) < \infty$ and $c := c(d , V , \alpha,f) > 0$ such that
    \begin{equation} \label{eq:estmz-}
       \mathbb{P} \left[  \mathbf{m}_{-} \leq L^{-\alpha} \right] \leq C \exp \left( - c \left( \ln L \right)^{\frac{r}{r-2}} \right). 
    \end{equation}
    (N.B. The proof is essentially the same as the one presented in Section~\ref{sec:sectionD2}, we note that the maximal function does not need to be incorporated because, in the setting of this section, we have the stochastic integrability estimate~\eqref{eq:15132307} which provides very strong control over the probability of the environment $\a$ to have a large eigenvalue.)
    \item Moderation: one has the following inequality
    \begin{equation*}
        \int_{I_L} \sum_{x \in \Lambda_L} \mathbf{m}(t , x) \left| \nabla w (t , x) \right|^2 \, dt \leq C \int_{I_L} \sum_{x \in \Lambda_L} \nabla w (t , x) \cdot \a(t , x) \nabla w(t , x) \, dt.
    \end{equation*}
\end{itemize}
Combining the previous inequality with~\eqref{eq:17012307}, we obtain that
\begin{align*}
    \mathbf{m}_{-} \left\|  \nabla w  \right\|_{\underline{L}^2(Q_L)}^2 & \leq  \frac{1}{|I_{L}|}\int_{I_{L}} \frac{1}{\left| \Lambda_{L} \right|}  \sum_{x \in \Lambda_L} \mathbf{m}(t , x) \left| \nabla w (t , x) \right|^2 \, dt \\
    & \leq  C  \frac{1}{|I_{L}|}\int_{I_{L}} \frac{1}{\left| \Lambda_{L} \right|}  \sum_{x \in \Lambda_L} \nabla w (t , x) \cdot \a(t , x) \nabla w(t , x) \, dt.
\end{align*}
We thus have
\begin{align*}
    \left\|  \nabla w  \right\|_{\underline{L}^2(Q_L)}^2 \indc_{\{  \mathbf{m}_{-} \geq L^{-\alpha} \}} & \leq  C L^\alpha  \frac{1}{|I_{L}|}\int_{I_{L}} \frac{1}{\left| \Lambda_{L} \right|} \sum_{x \in \Lambda_L} \nabla w (t , x) \cdot \a(t , x) \nabla w(t , x) \, dt \\
    & \leq \mathcal{O}_{\Psi , c} \left( C L^{\alpha-\frac14} + C L^\alpha  |p_z - p_{z'}|^2 \right) .
\end{align*}
Using the inequality~\eqref{eq:estmz-} and Proposition~\ref{prop.prop2.3} (which gives super-exponential stochastic integrability estimates on the gradients of $v_1$ and $v_2$), one can verify that the following upper bound holds (N.B. the factor $C/L$ is somewhat arbitrary and could be replaced by the inverse of a polynomial of higher degree)
\begin{equation*}
    \left\|  \nabla w  \right\|_{\underline{L}^2(Q_L)} \indc_{\{  \mathbf{m}_{-} \leq L^{-\alpha} \}} \leq \mathcal{O}_{\Psi , c} \left( \frac{C}{L} \right).
\end{equation*}
Combining the two previous inequalities completes the proof of Lemma~\ref{lemmavzvzprime}.
\end{proof}
\end{lemma}

{\small
\bibliographystyle{abbrv}
\bibliography{degeneratehydrodynamical.bib}
}

\end{document}